%% file: ltcs_2query.tex
\documentclass[11pt]{article}
\PassOptionsToPackage{numbers, compress}{natbib}
\usepackage{amsmath,amssymb,amsthm,fullpage,mathrsfs,pgf,tikz,caption,subcaption,mathtools}
\usepackage[ruled,vlined,linesnumbered]{algorithm2e}
\usepackage{natbib}
\usepackage{amsmath,amssymb,amsthm,mathtools}
\usepackage[utf8]{inputenc} 
\usepackage[T1]{fontenc}    
\usepackage{hyperref}       
\usepackage{url}            
\usepackage{booktabs}       
\usepackage{amsfonts}       
\usepackage{nicefrac}       
\usepackage{microtype}      
\usepackage[margin=1in]{geometry}
\usepackage{bbm}
\usepackage{enumitem}
\usepackage{xcolor}
\usepackage{cleveref}

\usepackage[T1]{fontenc}
\usepackage{lmodern}

\usepackage{chngcntr} 
\input{amsart_header_18_05}

\usepackage[nottoc]{tocbibind}

\numberwithin{equation}{section} 

\usepackage{enumitem}

\theoremstyle{definition}
\newtheorem{alg}[thm]{Algorithm}
\theoremstyle{plain}

\DeclareMathOperator{\res}{res} 
\DeclareMathOperator{\dist}{dist}

\newcommand{\from}{\leftarrow}

\DeclareMathOperator{\cse}{cse} 
\DeclareMathOperator{\ccd}{ccd} 
\DeclareMathOperator{\cbe}{cbe} 
\DeclareMathOperator{\se}{se} 
\DeclareMathOperator{\cd}{cd}

\DeclareMathOperator{\nat}{nat} 
\DeclareMathOperator{\uni}{uni} 

\DeclareMathOperator{\ugr}{Gr}

\DeclareMathOperator{\Inf}{Inf}

\newcommand{\Fmax}{F^{\max}}
\newcommand{\Fmin}{F^{\min}}

\newcommand{\ContZwZUp}[3]{\frac{\Fmax_{#1,#2,#3}\Fmax_{#1,#3}}{\Fmin_{#2,#3}}} 
\newcommand{\ContZwZLo}[3]{\frac{\Fmin_{#1,#2,#3}\Fmin_{#1,#3}}{\Fmax_{#2,#3}}}

\newcommand{\wZContZLo}[3]{\frac{\Fmin_{#2,#3}}{\Fmax_{#1,#2,#3}\Fmax_{#1,#3}}}

\newcommand{\LinkAllUp}[3]{\frac{\Fmax_{#2,#3}}{\Fmin_{#1,#2,#3}}} 
\newcommand{\LinkAllLo}[3]{\frac{\Fmin_{#2,#3}}{\Fmax_{#1,#2,#3}}} 
\newcommand{\AllLinkLo}[3]{\frac{\Fmin_{#1,#2,#3}}{\Fmax_{#2,#3}}} 
\newcommand{\AllLinkUp}[3]{\frac{\Fmax_{#1,#2,#3}}{\Fmin_{#2,#3}}} 

\newcommand{\Sci}{\frac{
		\Fmax_{t,k+1,d} 
		\Fmax_{t,d}
		\Fmax_{i,t,d}
	}{
		2
		\Fmin_{k+1,d}
		\Fmin_{t,d}
	}} 
\newcommand{\Scii}[1]{\frac{
		\Fmax_{i,#1,d}
		\Fmax_{i,#1}
	}{
		\Fmin_{#1,d}
	}} 

\newcommand{\qc}{} 

\newcommand{\Duc}[1]{
	\frac{
		\Fmin_{#1,k+1,d}
		\Fmin_{k,d}
	}{
		\Fmax_{#1,k,d}
		\Fmax_{k+1,d}
	}
} 
\newcommand{\Ducprime}[1]{
	\frac{
		\Fmin_{#1,k+2,d}
		\Fmin_{k+1,d}
	}{
		\Fmax_{#1,k+1,d}
		\Fmax_{k+2,d}
	}
}

\newcommand{\CorR}{
\max\left\{\ContZwZUp{i}{k}{d}\where i\in\{0,\dots,k\} \right\}
}

\DeclareMathOperator{\AG}{AG}

\newcommand{\NI}{{\mathrm{NIH}}} 
\newcommand{\NNI}{{\mathrm{NIG}}} 

\begin{document}
\title{Cosystolic Expansion of Sheaves on Posets \\ with Applications to Good 2-Query 
Locally Testable Codes
and Lifted Codes}

\date{}

\author{Uriya A. First \\
Department of Mathematics\\
University of Haifa
\and
Tali Kaufman\\
Computer Science Department\\ 
Bar-Ilan University}

%%%%%%%%%%%%%%%%%%%%%%%%%%%%%%%%%%%%%%%%%%%%%%%%%%%%%%%%%%%%%%%%%%%%

\maketitle

\begin{abstract}

We show that cosystolic expansion 
of sheaves on posets can be derived from local expansion conditions
of the sheaf and the poset.
When the poset at hand is a cell complex --- typically
a high dimensional expander --- a sheaf may be thought of as generalizing coefficient
groups used for defining homology and cohomology, by letting the coefficient group
vary along the cell complex.
Previous works, e.g.~\cite{Kaufman_2016_isoperimetic_inequalities}, \cite{Evra_2016_cosystolic_expanders},
established local  criteria for
cosystolic expansion only for \emph{simplicial complexes}
and with respect to \emph{constant coefficients}.
Our main technical contribution is providing a criterion
that is more general in two ways: it applies  to \emph{posets} 
and \emph{sheaves}, respectively. 

The importance of working with sheaves
on posets (rather than constant coefficients and simplicial complexes)
stems from applications to locally testable codes (LTCs). It has been observed
\cite{Kaufman_2014_high_dim_expanders_property_testing} that cosystolic expansion 
is related to property testing in the context of simplicial complexes and constant coefficients,
but
unfortunately, this special case does not give rise to interesting LTCs.
We observe that this relation also exists in the much more general  
setting of sheaves on posets. As the language of sheaves 
is   more expressive, it allows us   to put this relation 
to  use.
Specifically, we apply our criterion for cosystolic expansion 
in two ways.

First, we show the existence of good $2$-query LTCs.
These codes are actually related   to the good $q$-query LTCs of \cite{Dinur_2022_ltcs_const_rate}
and \cite{Panteleev_2022_good_quantum_codes}, being the formers' so-called \emph{line codes},
but we get them from a new, more illuminating perspective. By realizing these
codes as cycle codes of sheaves on   posets, we can derive their good properties
directly from our criterion for cosystolic expansion. 
The local expansion conditions that our criterion requires 
unfold to the conditions on the ``small codes'' in \cite{Dinur_2022_ltcs_const_rate}
and \cite{Panteleev_2022_good_quantum_codes}, and hence give a conceptual explanation to why conditions
such as \emph{agreement testability} are required. 

Second, we show that local testability of a \emph{lifted code}
can be derived solely from local expansion and testability conditions. 
In the  work \cite{Dikstein_2020_ltcs_via_hdxs_preprint}, 
it was shown
that one can obtain local testability of lifted codes
from a mixture of local and global conditions, namely, from local testability of the local codes and \emph{global} agreement expansion
of an auxiliary  $3$-layer system called a \emph{multilayered agreement sampler}.
Our result achieves the same, but using genuinely local conditions and a simpler 
$3$-layer structure. It is derived neatly from our local criterion
for cosystolic expansion, by interpreting the situation in the language of sheaves on posets.

\end{abstract}

\bigskip

\newpage

\tableofcontents

\newpage

\section{Overview}

\subsection{General}

We prove a local criterion
for \emph{cosystolic expansion} of sheaves
on finite partially ordered sets, called \emph{posets} for short.
This   extends the reach of known similar criteria established in \cite{Kaufman_2016_isoperimetic_inequalities}, \cite{Evra_2016_cosystolic_expanders} (see
also \cite{Evra_2024_cosystolic_expanders_bdd_deg}), \cite{Kaufman_Mass_2021_cosystolic_non_abelean},
\cite{Kaufman_Mass_2022_improved_cosystole},
\cite{Dikstein_2023_cbe_cse_without_dep_on_dim_deg}
which (in our terminology) apply only to \emph{constant} sheaves
on \emph{simplicial complexes}.

Criteria for establishing  cosystolic expansion are motivated
by applications to locally testable codes (LTCs).
The relation between cosystolic expansion and property
testing was first observed in \cite{Kaufman_2014_high_dim_expanders_property_testing},
in the context of constant sheaves on simplicial complexes,
and implicitly for general sheaves on posets in \cite{Panteleev_2022_good_quantum_codes}.
We use the expressive language of sheaves  and our criterion
for cosystolic expansion of sheaves on posets to construct good $2$-query  locally testable codes (while previous works provided  good locally testable codes that use many queries) and to get a genuinely local   criterion for testability of lifted codes (while prior work  used a non-trivial mixture of local and global conditions to derive testability of lifted codes).

Our main result may also be seen a unifying mechanism through
which one can   recover   many known results about cosystolic expansion and testability.
For example, it   recovers in part   the main results of
\cite{Kaufman_2021_amplified_local_testability_preprint}
and
\cite{Dinur_2024_expansion_cube_complexes}.

\subsection{Posets}

A poset is a finite set $X$ endowed with a transitive anti-reflexive relation $<$.
For $x,y\in X$, we  write $x\leq y$ to denote that $x<y$ or $x=y$.
The posets that we consider in this work will always be equipped
with a \emph{dimension function} (also called a \emph{rank function})
$\dim =\dim_X:X\to \Z$ which is required to satisfy 
$\dim x<\dim y$ whenever $x<y$, and $\dim x+1=\dim y$
if in addition no elements of $X$ lie strictly between $x$ and $y$.
See \S\ref{subsec:graded-posets} for further details.

Our main example of a poset 
will be the poset of faces of a regular cell complex\footnote{
	Also called a regular CW complex.
} together
with the dimension function assigning every
face its usual dimension. (There is also a single empty face of dimension $-1$.) 
This includes  simplicial complexes and cube complexes.
Another example of a poset with a dimension function is the affine Grassmannian of a finite vector space.

Following the notation for simplicial complexes, given
a poset $X$, we write $X(i)$ for the elements of $X$ of dimension $i$
and call such elements $i$-faces of $X$.
A poset $X$ is \emph{pure of dimension $d$} if it is nonempty and each of its
faces is contained in a $d$-face.
It is called a \emph{$d$-poset} if it moreover has 
a a unique $(-1)$-face, denoted $\emptyset_X$, which is a face of every
other face in $X$.
Examples of $d$-posets include pure $d$-dimensional simplicial and cubical complexes,
and the poset of affine spaces in $\F^n$ ($\F$ is a finite field)
of dimension $d$ or less plus the empty set. The \emph{degree} of a $d$-poset $X$ is the largest possible
number of faces containing a $0$-face.

All our posets will carry a \emph{weight function} 
and an \emph{orientation}, which we  
suppress in this overview  for the sake of simplicity. For details,
see \S\ref{subsec:weights} and \S\ref{subsec:orientation}.

\subsection{Sheaves}

Broadly speaking,
a sheaf is a layer of linear-algebra data that is put on top of a poset. 
When the poset   comes from a geometric source, a sheaf on it may
also be viewed as a generalization of the group of coefficients
that is used in the definition of homology and cohomology.
We shall survey the history of sheaves after we present them.
For simplicity, we only consider here sheaves of $\F_2$-vector spaces --- called
$\F_2$-sheaves later in the paper --- and call them sheaves for brevity.

A \emph{sheaf} $\calF$ on a poset $X$ consists of the following data:
\begin{enumerate}[label=(\arabic*)]
	\item an $\F_2$-vector space $\calF(x)$ for every   face $x\in X$;
	\item an  $\F_2$-linear map $\res^{\calF}_{y\from x}:\calF(x)\to \calF(y)$
	for every   $x,y\in X$ with $x< y$;
\end{enumerate}
subject to the requirement
$
\res^{\calF}_{z\from y}\circ \res^{\calF}_{y\from x}=\res^{\calF}_{z\from x}
$
whenever $x< y< z$.\footnote{
	In this case, it is   convenient to define $\res^{\calF}_{x\from x}=\id_{\calF(x)}$
	so that $
\res^{\calF}_{z\from y}\circ \res^{\calF}_{y\from x}=\res^{\calF}_{z\from x}
$ holds whenever $x\leq y\leq z$.
} The maps $\res^{\calF}_{y\from x}$ are called \emph{restriction maps}; the superscript
$\calF$ will be dropped when it is clear from context.

By reversing the direction of the restriction maps one gets the notion of a \emph{cosheaf}
(called a \emph{local system} in some works).
A cosheaf on $X$ is essentially the same thing as a sheaf on the opposite poset of $X$.

Here is a simple example of a sheaf on $X$: Take an $\F_2$-vector space $V$
and define 
$\calF(x)=V$ and $\res^{\calF}_{y\from x}=\id_V$ for all $x$ and $y$.
This sheaf is denoted $V_X$. Sheaves of this form (up to isomorphism) 
are called \emph{constant}.

If $X$ is a $1$-poset, then one can get a sheaf $\calF$ on $X$
by setting $\calF(\emptyset_X)=0$ and $\res^{\calF}_{x\from \emptyset}=0$
for every $x\in X-\{\emptyset_X\}$, and choosing the remaining spaces
$\calF(x)$ and restrictions map $\res^{\calF}_{y\from x}$ arbitrarily.
The condition $
\res^{\calF}_{z\from y}\circ \res^{\calF}_{y\from x}=\res^{\calF}_{z\from x}
$ holds automatically for any $x<y<z$, because we must have $x=\emptyset_X$.

Further examples will be given later in this overview.

\paragraph{Brief History and Related Notions.}

In   topology and algebraic geometry,
sheaves are defined on \emph{topological spaces},
and this is the common definition in the literature.
They were studied since the 1950s and their definition is more involved, e.g.,
see \cite{Iversen_1986_cohomology_of_sheaves}, \cite[Chp.~II]{MacLane_1994_intro_to_topos_theory}.

The sheaves defined in this paper may be seen as discrete,   elementary versions
of sheaves on  topological spaces. When $X$ is the poset of a cell complex,
they are known as \emph{cellular sheaves}. They were first considered
by Shepard \cite{Shepard_1985_cellular_descrip_of_der_cat_PhD},
and their theory was further developed by Curry \cite{Curry_2014_sheaves_cosheaves_PhD},
who also considered cosheaves.
A more concise treatment (over regular cell complexes)
appears in \cite{Hansen_2019_spectral_thy_of_sheaves}.
The definition of sheaves on posets given here
is a natural generalization of (cellular) sheaves on cell complexes, 
and was briefly considered in \cite[\S4.2.2]{Curry_2014_sheaves_cosheaves_PhD}.

We remark that   despite the differences between   sheaves on posets
and sheaves on topological spaces, the former is actually a special case of the latter
\cite[\S4.2]{Curry_2014_sheaves_cosheaves_PhD}.

The special case of sheaves on graphs was considered independently
in several other works, e.g., the local systems of \cite[\S2]{Jordan_1997_Ramanujan_local_systems}.
The sheaves on graphs considered by Friedman \cite{Friedman_2015_sheaves_on_graphs} 
are cosheaves in our notation.

\subsection{Sheaf Cohomology.}

Let $X$ be a regular cell complex  
and let  $i\in\Z$.
Recall that the space of $i$-cochains on $X$ (with coefficients in $\F_2$)
is   $C^i=C^i(X,\F_2):=\F_2^{X(i)}$.
One then defines  the $i$-th coboundary
map  $d_i:C^i\to C^{i+1}$ by
\begin{equation}\label{EQ:coboundary-F-two}
(d_i f)(y)=\sum_{x\in y(i)} f(x)\qquad\forall f\in C^i,\,y\in X(i+1),
\end{equation}
where here, we wrote $y(i)$ for the set of $i$-faces of the face $y$. 
It is well-known that $d_{i}\circ d_{i-1}=0$, and so $Z^i=Z^i(X,\F_2):=\ker d_i$
contains $B^i=B^i(X,\F_2):=\im d_{i-1}$. The spaces $Z^i$, $B^i$
and the quotient $\HH^i(X,\F_2):=Z^i/B^i$ are the
$\F_2$-spaces of $i$-cocycles, $i$-coboundaries and $i$-th cohomology
of $X$, respectively. They are all well-studied.

In   the same manner,
with every sheaf $\calF$ on a (graded, oriented) poset $X$, we can associate
$\F_2$-spaces of cochains, cocycles, coboundaries and cohomology:
First, set $C^i=C^i(X,\calF)=\prod_{x\in X(i)} \calF(x)$.
That is, elements $f\in C^i$ are collections $\{f(x)\}_{x\in X(i)}$
with $f(x)\in \calF(x)$ for every $x\in X(i)$.
The $i$-th coboundary map is defined as in \eqref{EQ:coboundary-F-two},
but using the restriction maps of $\calF$:\footnote{
	When $X$ is not a regular cell complex, or
	$\F_2$ is replaced by a field of characteristic
	not $2$, one needs to introduce signs into this formula,
	see \S\ref{subsec:sheaf-coh}.
}
\begin{equation}\label{EQ:coboundary-map-sheaf}
(d_i f)(y)=\sum_{x\in y(i)} \res_{y\from x} f(x) \in \calF(y) \qquad\forall f\in C^i,\,y\in X(i+1).
\end{equation}
The spaces $Z^i=Z^i(X,\calF)$, $B^i=B^i(X,\calF)$ and $\HH^i(X,\calF)$ are then
defined to be $\ker d_i$, $\im d_{i-1}$ and $\ker d_i/\im d_{i-1}$, respectively.\footnote{
	Caution: At this level of generality, one can have $\HH^i(X,\calF)\neq 0$
	for $i<0$. Also, for a general poset, $\{\HH^i(X,-)\}_{i\geq 0}$
	are in general not the \emph{right derived functions} of $\HH^0(X,-)$.
}

Observe that if $\calF$ is the constant sheaf $(\F_2)_X$, then
$C^i(X,\calF)$ is just $C^i(X,\F_2)$ and the coboundary maps agree.

\subsection{Locally Testable Codes From Sheaves on Posets}
\label{subsec:intro-ltcs-from-sheaves}

Our interest in sheaves on posets and their cohomology
is motivated by the fact that they   give rise to locally testable
codes (LTCs). 

\paragraph*{Locally Testable Codes.}

Let $\Sigma$ be a finite alphabet and let $C\subseteq\Sigma^n$ 
be a code with block length $n$.
We write $\delta(C)$ and $r(C)$ for the relative distance and rate of $C$,
respectively; the definitions are recalled in \S\ref{subsec:codes}. 
As usual, when $C\subseteq \Sigma^n$
ranges across a family $\{C_i\subseteq \Sigma^{n_i}\}_{i\geq 0}$ with $n_i\to \infty$,
we say that $C$ is \emph{good} if its relative distance
and rate are  bounded away from $0$.

Recall that a \emph{tester} for $C\subseteq \Sigma^n$
is a randomized algorithm $T$
which, given access to a word $f\in\Sigma^n$,
can decide with high probability
whether it is close to a codeword or not
by querying just a few 
(i.e.\ $O(1)$) of its letters. Formally, a $q$-query tester $T$ 
may probe at most $q$ letters from the input $f$,
and must accept all words $f\in C$.
The tester $T$ has soundness $\mu$ ($\mu\geq 0)$ if for every
$f\in \Sigma^n$, 
\[
\mathrm{Prob}(\text{$T$ rejects $f$})\geq \mu\cdot\dist(f,C).
\]
Here, $\dist(\cdot,\cdot)$ is the normalized Hamming distance in $\Sigma^n$.
A \emph{locally testable code} (LTC) is a family of codes-with-testers
with block length tending to $\infty$ such that all the testers
have the same query size $q$ and the same positive   soundness $\mu$.
See \cite{Goldreich_2011_locally_testable_codes_proofs} for a survey.

The question of whether there exist 
good  LTCs was open until  it was recently answered
on the positive by Dinur, Evra, Livne, Lubotzky and Mozes \cite{Dinur_2022_ltcs_const_rate}
and Panteleev and Kalachev \cite{Panteleev_2022_good_quantum_codes}
(independently);
see also  
\cite{Leverrier_2022_quantum_tanner_codes}.
The many works on the subject predating this result are surveyed
in \cite[\S1.2]{Dinur_2022_ltcs_const_rate}.

\paragraph*{Cocycle Codes and Cosystolic Expansion.}

Let $\calF$ be a sheaf on a poset $X$  and let $i\in \Z$.
Suppose further that there exists an $\F_2$-vector space
$\Sigma$ such that $\calF(x)=\Sigma$ for every $x\in X(i)$.
Then $C^i=C^i(X,\calF)=\Sigma^{X(i)}$, and so we may view
$Z^i=Z^i(X,\calF)$ as a code inside $\Sigma^{X(i)}$.
We call $Z^i$ the \emph{$i$-cocycle code} of $(X,\calF)$.
The $i$-cocycle code $Z^i\subseteq \Sigma^{X(i)}$ 
also has a natural tester: Given $f\in \Sigma^{X(i)}$,
choose $y\in X(i+1)$ uniformly at random, probe $f(x)$ for every
$x\in y(i)$ and accept $f$ if and only if 
$d_i f(y)=\sum_{x\in y(i)}\res_{y\from x}f(x)=0$
(cf.\ \eqref{EQ:coboundary-map-sheaf}).
The query size of the this tester is the largest number of $i$-faces that an $(i+1)$-face
of $X$ can have. For example, when $i=0$ and $X$ is a regular cell complex,
the natural tester probes only $2$ letters.

Write $\|\cdot\|$ 
for the normalized Hamming
norm 
on $C^i=\Sigma^{X(i)}$ or $C^{i+1}=\prod_{y\in X(i+1)} \calF(y)$.
Then the natural tester of $Z^i$ has soundness $\mu\geq 0$ if and only if
\begin{equation}\label{EQ:cse-intro}
\|d_i f\|\geq \mu \dist(f,Z^i)\qquad\forall f\in C^i.
\end{equation}
This condition may also be regarded as an expansion condition for $i$-cochains,
and indeed, the \emph{$i$-cosystolic expansion} of $\calF$, denoted
\[
\cse_i(X,\calF)
\]
is  defined   to be the supremum of the set of $\mu\geq 0$ for which \eqref{EQ:cse-intro}
holds.\footnote{
	Actually, the definition of cosystolic expansion involves the weight function
	on $X$, so $\cse_i(X,\calF)$ and the soundness of the tester
	of $Z_i$ are the same only up to a constant. See Lemma~\ref{LM:cse-to-ltc}.
} Note that this makes sense even without requiring that $\calF(x)=\Sigma$
for every $x\in X(i)$.

Observe further that the relative distance of $Z^i$ is the largest
$\delta\geq 0$ such that
\[
\|f\|\geq \delta\qquad\forall f\in Z^i-\{0\}.
\]
Again, this may be viewed as a condition on the expansion of $i$-cochains.
For a general sheaf $\calF$ on $X$, 
we define the \emph{$i$-cocycle distance} of $\calF$ to be\footnote{
	Again, the actual definition of $\ccd_i(X,\calF)$ involves the weights on $X$.
}
\[
\ccd_i(X,\calF)=\min\{\|f\|\where f\in Z^i-B^i\}.
\]
The reason why we let $f$ range on $Z^i-B^i$ and not on $Z^i-\{0\}$ is because
$B^i$ typically contains vectors of small support (e.g. the coboundary of a small
$(i-1)$-cochain). However, when $B^i=0$, we have $\delta(Z^i)=\ccd_i(X,\calF)$.

Following \cite{Kaufman_2016_isoperimetic_inequalities},
\cite{Evra_2024_cosystolic_expanders_bdd_deg}  and similar
sources, we say that $(X,\calF)$ is a  \emph{$(\mu,\delta)$-cosystolic expander}
in dimension $i$ if $\cse_i(X,\calF)\geq \mu$ and $\ccd_i(X,\calF)\geq \delta$.

\medskip

\emph{To conclude, provided that $B^i=0$, the  
$i$-cocycle code $Z^i=Z^i(X,\calF)\subseteq\Sigma^{X(i)}$
is locally testable and has linear distance if and only if $(X,\calF)$
is a  $(\mu,\delta)$-cosystolic expander in dimension $i$ for some $\mu,\delta>0$.}

\begin{remark}
	Cosheaves on posets similary give rise to \emph{cycle} codes, and
	their distance and testability are governed by the \emph{cycle distance}  
	and \emph{systolic expanion} of the cosheaf at hand. This is completely dual to the 
	case of sheaves.
	For example,
	the famous \emph{expander codes} of \cite{Sipser_1996_expander_codes} may be realized
	as $1$-cycle codes on graphs \cite{Meshulam_2018_graph_codes_arxiv_version}.
\end{remark}

\paragraph{Coboundary Expansion}

Coboundary expansion is a stronger version of cosystolic expansion that will be
needed to state our main result. Given a sheaf $\calF$ on $X$,
its \emph{$i$-coboundary expansion}, denoted
\[
\cbe_i(X,\calF),
\]
is the supremum of the set of $\mu\geq 0$ such that
\begin{equation*}
\|d_i f\|\geq \mu \dist(f,B^i)\qquad\forall f\in C^i.
\end{equation*}
In the context of $i$-cocycle codes, $\cbe_i(X,\calF)$ is the soundness of the natural
tester of $Z^i$, but when used a tester for the smaller code $B^i$.

\paragraph*{Brief History of     Coboundary and Cosystolic Expansion.}

Coboundary expansion     originated in
the works of
Linial--Meshulam \cite{Linial_2006_homological_connectivity}
and Meshulam--Wallach \cite{Meshulam_2009_homological_connectivity}
on the cohomology of random simplicial complexes, and the work
of Gromov \cite{Gromov_2010_expanders_and_top_II} on the minimal amount of overlapping
forced by mapping
a simplicial complex to $\R^n$. These works did not mention
sheaves explicitly, and in our terminology, only considered
the case of  constant sheaves on simplicial complexes.
Within this restricted setting,
cosystolic expansion  
was 
developed in   
\cite{Dotterrer_2018_topological_overlap},
\cite{Kaufman_2016_isoperimetic_inequalities}, \cite{Evra_2016_cosystolic_expanders}
as  a
relaxed version of coboundary expansion
meant  to extend the reach of Gromov's methods.
The first connections between cosystolic expansion and property
testing were observed and studied in \cite{Kaufman_2014_high_dim_expanders_property_testing}.

\subsection{A Criterion For Cosystolic Expansion of Sheaves}
\label{subsec:intro-lgp}

Our main result is a criterion for bounding the $i$-cosystolic expansion
and $i$-cocycle distance of a sheaf by means of mostly-local expansion conditions.

To state it, we need four more pieces of notation.
Let $\calF$ be a sheaf on a $d$-poset $X$.

\medskip

{\noindent\sf Lower Irregularity.}
For integers $-1\leq i<j<k\leq d$, let $\Fmax_{i,j,k}$ (resp.\ $\Fmin_{i,j,k}$)
denote the largest (resp.\ smallest) possible number of $j$-faces lying between
an $i$-face and a $k$-face that are incident. The ratio $L_{i,j,k}(X):=
\Fmax_{i,j,k}/\Fmin_{i,j,k}$
is the $(i,j,k)$-lower irregularity of $X$. The 
maximum of $L_{i,j,k}(X)$ over all $i,j,k$
is  is called the \emph{lower irregularity} of $X$ and denoted $L(X)$.
For example, simplicial complexes and cube complexes have the lowest possible irregularity,
which is $1$.

\medskip

{\noindent\sf Links.} 
Let $z\in X$. The \emph{link} of $X$
at $z$ is  $X_z:=\{x\in X\suchthat x\geq z\}$
together with the partial order inherited from $X$
and the dimension function $\dim_z:X_z\to \Z$ given by $\dim_z(x)=\dim x-\dim z-1$.
Note that $X_{\emptyset_X}$ is just $X$.
The sheaf $\calF$ restricts to a sheaf $\calF_z$  on $X_z$ defined
by $\calF_z(x)=\calF(x)$ and $\res^{\calF_z}_{y\from x}=\res^{\calF}_{y\from x}$
for all $x,y\in X_z$.

\medskip

{\noindent\sf No-Intersection Graph.}
Let $i,j,k$ be integers with $-1\leq i,j\leq k$.
The $(i,j,k)$-no-intersection graph of $X$, denoted $\NNI^{i,j,k}(X)$,
is a graph with vertex set $X(i)\cup X(j)$ (the vertex set is just $X(i)$ if $i=j$),
where for every triple $(x,y,z)\in X(i)\times X(j)\times X(k)$
with $x\neq y$, $x,y\leq z$ and $\inf\{x,y\}=\emptyset_X$, one adds an edge
between $x$ and $y$.\footnote{The implicit weight function
on $X$ induces a weight function on the no-intersection graph;
see \S\ref{subsec:no-intersect}.}

For example, if $X$ is a regular cell complex, then $\NNI^{0,0,1}(X)$
is just the underlying graph of $X$, denoted $\ugr(X)$.
Also, if $X$ is a cube complex, then $\NNI^{1,1,2}(X)$ is the graph
whose vertices are the edges of $X$ and in which two edges (viewed as vertices of the graph)
are connected when they are the opposite sides of a square in $X$.

\medskip

{\noindent\sf Skeleton Expanders.}
Let $\alpha,\beta\geq 0$. A weighted graph $(G,w)$ 
is called an \emph{$(\alpha,\beta)$-skeleton expander} if 
for every set of vertices $A\subseteq G(0)$, we have
$w(E(A))\leq \alpha w(A)+\beta w(A)^2$, where $E(A)$ is the
set of edges having both their vertices in $A$. 
Informally, the smaller $\alpha$ is, the more expanding $(G,w)$ is considered.

For example, if $G$ is a regular graph, $w$ assigns every vertex
(resp.\ edge) the weight $\frac{1}{|G(0)|}$ (resp.\ $\frac{1}{|G(1)|}$) 
and the second largest normalized eigenvalue of $G$ is $\lambda\in[-1,1]$,
then $G$ is a $(\lambda,1-\lambda)$-skeleton expander
(Proposition~\ref{PR:eml-special-case}).

\medskip

Given $i,j,k\in\Z$ as above and $z\in X$
with $\ell:=\dim z<i$, it will be convenient to write $\NNI^{i,j,k}_z(X)$ for the graph
$\NNI^{i-\ell  -1,j-\ell-1,k-\ell-1}(X_z)$.
Our main result states the following:

\begin{thm}[Simplified; see Theorem~\ref{TH:lgp-simple-version}]
	\label{TH:lgp-intro}
	For every $k\in\N$, $F\in\N$, $L\in[1,\infty)$ and $B\in\R_+$,
	there are constants $K,K'\in (0,1]$
	such that the following holds: Let $X$ be a $d$-poset ($d\geq k+2$)
	with $L(X)\leq L$ and such that every
	$(k+2)$-face of $X$ has at most $F$-subfaces, let $\calF$ be a sheaf on 
	$X$
	and let $\veps\in (0,1]$. Suppose that:
	\begin{enumerate}
		\item[(1a)] 
		$\cbe_{k-\dim z-1}(X_z,\calF_z)\geq \veps$  for every
		$z\in X(0)\cup\dots \cup X(k)$;
		\item[(1b)]  $\cbe_{k-\dim z}(X_z,\calF_z)\geq \veps$ for every
		$z\in X(0)\cup\dots \cup X(k+1)$;
		\item[(2)] $\NNI^{i,j,t}_z(X)$ is a $((K\veps)^{2^{k+1-i}},B)$-skeleton
		expander for every $z\in X(-1)\cup\dots \cup X(k)$ and  $i,j,t$
		with $\dim z< i\leq j<t\leq k+2$.
	\end{enumerate}
	Then 
	$\cse_k(\calF)\geq K'(K\veps)^{2^{k+2}-1}$ and
	$\ccd_k(\calF)\geq  K'(K\veps)^{2^{k+1}-1}$.
\end{thm}

We encourage the reader to think of $F$, $L$, $B$ and $\veps$ as constant,
and of $X$ and $\calF$ as varying.
In typical applications of Theorem~\ref{TH:lgp-intro}, as the  degree  of $X$ grows,
$F$, $L$, $B$ and $\veps$ will remain constant while the skeleton expansion
of the $\NNI^{i,j,t}_z(X)$ will tend to $(0,c)$ for some constant $c>0$. 
Thus, once
the degree of $X$ is  large enough (but constant), all three conditions (1a), (1b)
and (2)  
will be satisfied.

\medskip

A few remarks are now in order. 

First, if the sheaf $\calF$ from the theorem also satisfies
$\calF(x)=0$ for all $x\in X(k-1)$ (so that $B^k=0$)
and $\calF(x)=\Sigma$ for all $x\in X(k)$, then
the theorem says that (up to scaling caused by non-uniform weights)
the $k$-cocycle code $Z^k(X,\calF)\subseteq C^k=\Sigma^{X(k)}$
has relative distance $ \Theta(\veps^{2^{k+1}-1})$
and its natural tester has soundness $\Theta(\veps^{2^{k+2}-1})$.
Moreover, in this case, $Z^k$ has a linear-time decoding algorithm
able to correct words   that are $\Theta(\veps^{2^{k+2}-1})$-close
to $Z^k$;
see Corollary~\ref{CR:ltcs-from-sheaves}.

Second, assumption (2) in the theorem   can often be relaxed; 
it is in general  not necessary
to bound the skeleton expansion 
of   \emph{all}
the no-intersection graphs   $\NNI^{i,j,t}_z(X)$. 
Here are three such notable examples:
\begin{itemize}
\item When $X$ is a simplicial
complex, we only need (2)  to hold for the graphs $\NNI^{i,i,i+1}_z(X)$ with $\dim z=i-1$,
i.e., the underlying graph  of every $X_z$ with $z\in X(-1)\cup\dots\cup X(k)$. 
See Theorem~\ref{TH:lgp-simplicial-case}.
\item When $X$ is a cube complex, it is enough that (2)  holds 
for every $\ugr(X_z)$   as in the simplicial case
and also for the graphs $ \NNI^{1,1,2}(X),\dots,\NNI^{k+1,k+1,k+2}(X)$.
The graph $\NNI^{i,i,i+1}(X)$
is   obtained from $X$ by taking
the $i$-cubes as vertices and connecting two $i$-cubes by an edge whenever
they are the opposite sides of an $(i+1)$-cube. See Theorem~\ref{TH:lgp-cube-case}.
\item When $k=0$, we   need to consider in (2) only the graphs 
$\NNI^{0,0,1}(X)$, 
$\NNI^{1,1,2}(X)$
and $\NNI^{0,0,1}(X_v)$
for every $v\in X(0)$.
\end{itemize}
In the general case, the  
graphs that we need to consider in (2) are specified by an \emph{intersection profile}
for the poset $X$ --- 
a novel notion that we introduce in Section~\ref{sec:non-intersect}.

Third, assumptions (1a), (1b) and assumption (2) in the case $z\neq\emptyset_X$
are \emph{local} in the sense that they care only about the structure of $X_z$ and $\calF_z$
for $\emptyset\neq z\in X$ and not about the global structure of $X$
and $\calF$. Thus, Theorem~\ref{TH:lgp-intro} may be informally summarized as:
If 
\begin{itemize}
\item  $\calF_z$  is a good coboundary expander for every $z \neq\emptyset$ (``$\calF$ has
good \emph{local} coboundary expansion''),
\item $\NNI^{i,j,t}_z(X)$ is a   good skeleton expander for all $z\neq \emptyset$, $i$, $j$, $t$ (``$X$ is 
\emph{locally} expanding'')
and
\item $\NNI^{i,j,t}(X)$ is a   good skeleton expander for all   $i$, $j$, $t$ (``$X$ is \emph{globally} expanding''),
\end{itemize}
then $\calF$ is a good cosystolic expander in dimension $k$ (a global condition).
For special $X$, we can   make Theorem~\ref{TH:lgp-intro} into a purely local criterion
for cosystolic expansion. For example, if $X$ is a simplicial complex, then by our previous remark,
the only global
expansion condition that $X$ needs to satisfy is that $\ugr(X)$ is good skeleton expander,
and this can be deduced from expansion of the proper links of $X$ by
Oppenheim's Trickling Down Theorem \cite[Thm.~1.4]{Oppenheim_2018_local_spectral_expansion_I}.
The Trickling Down Theorem was extended to some non-simplicial posets in 
\cite{Kaufman_2023_Garland_for_Posets}, so  with more work,
there may  likely be a purely local criteria
for cosystolic expansion of sheaves on such posets.

Finally, 
we actually
prove a  more flexible  and technical 
version of   Theorem~\ref{TH:lgp-intro} --- Theorem~\ref{TH:main-very-technical}.
In the special case of $0$-cosystolic expansion,
this stronger  version 
admits a neat and accessible formulation, which we find useful to state here explicitly.

\begin{thm}[Criterion for $0$-Cosystolic Expansion; simplified; see {Theorem~\ref{TH:lgp-zero-detailed}}]
	\label{TH:lgp-zero-simple}
	For every $F\in\N$ and $L\in [1,\infty)$, 
	there are   constants $E,E',E'',E'''>0$ such that the following hold:
	Let $X$ be a $d$-poset\footnote{
		Recall that
		all our posets carry 
		a weight function and an orientation
		and those should be taken into account;
		see \S\ref{subsec:weights}, \S\ref{subsec:links}, \S\ref{subsec:orientation} and Definition~\ref{DF:nig}.
	} 
	($d\geq 2$) with $L(X)\leq L$
	and such that every $2$-face of $X$ contains at most
	$F$ subfaces, and let $\calF$ be a sheaf on $X$.
	Let $\veps\in\R_+$ and $\alpha_0,\beta_0,\alpha_ {-1},\beta_{-1},\alpha_{||},\beta_{||}\in[0,\infty)$
	and suppose that:
	\begin{enumerate}[label=(\arabic*)]
		\item[(1a)]
		$\cbe_{-1}(X_z,\calF_z)\geq \veps$
		for every $z\in X(0)\cup X(1)$;
		\item[(1b)] $\cbe_0(X_v,\calF_v)\geq \veps$ 
		for every $v\in X(0)$;
		\item[(2a)] 
		$\NNI^{0,0,1}(X_v)$ is an $(\alpha_0,\beta_0)$-skeleton expander for all $v\in X(0)$;
		\item[(2b)] $\NNI^{0,0,1}(X)$ is an $(\alpha_{-1},\beta_{-1})$-skeleton expander;	
		\item[(2c)] $\NNI^{1,1,2}(X)$
		is an $(\alpha_{||},\beta_{||})$-skeleton expander.	
	\end{enumerate}
	Suppose further that 
	\[
	\alpha_{-1} < E\veps
	\]
	and one can find $h_{-1},h_0,h_{||}\in (0,1]$ satisfying the   inequality 
\begin{align*}
	(\alpha_0+\beta_0 h_0) + (\alpha_{||}+\beta_{||}h_{||})+\frac{\alpha_{-1}+\beta_{-1}h_{-1}}{h_0}
	\leq  E' \veps.
	\end{align*}
	Then $\ccd_0(X,\calF)\geq \frac{E'''(E\veps-\alpha_{-1})}{\beta_{-1}}$ and 
	$\cse_0(X,\calF)\geq  \frac{E''}{h_0^{-1}h_1^{-1}+h_{||}^{-1}}$.
	When $X$ is a simplicial (resp.\ cube) complex,
	we can take $E=E'''=1$, $E'=\frac{1}{12}$ (resp.\ $E'=\frac{1}{16}$) and $E''=\frac{1}{2}$.
\end{thm}

The constants $E,E',E'',E'''$ are explicit and may be found in Table~\ref{TB:lgp-zero-constants}
on page \pageref{TB:lgp-zero-constants}.

\paragraph*{Relation to Other Works.}

Local criteria for establishing cosystolic expansion 
of \emph{constant sheaves} on \emph{simplicial complexes}
appeared
in \cite{Kaufman_2016_isoperimetic_inequalities} ($\dim X\leq 3$, $\calF=(\F_2)_X$),
\cite{Evra_2016_cosystolic_expanders}
($\calF=(\F_2)_X$, see also \cite{Evra_2024_cosystolic_expanders_bdd_deg}),
\cite{Kaufman_Mass_2021_cosystolic_non_abelean}
(any constant sheaf),
\cite{Dikstein_2023_cbe_cse_without_dep_on_dim_deg} (same).
Our Theorem~\ref{TH:lgp-intro} is an improvement of these results in two
ways. First, it applies to general posets, and second, it applies to all sheaves.
It further improves 
\cite{Kaufman_2016_isoperimetic_inequalities},
\cite{Evra_2016_cosystolic_expanders}, 
\cite{Kaufman_Mass_2021_cosystolic_non_abelean} by providing a lower bound 
on the cosystolic expansion which is independent of the dimension and the degree of $X$.
On the other hand, in the simplicial case,
\cite{Dikstein_2023_cbe_cse_without_dep_on_dim_deg} gives
a better lower bound on the $k$-cosystolic expansion
and  
\cite[Thm.~7]{Kaufman_Mass_2022_improved_cosystole} establishes a  better lower bound
on the $k$-cocycle distance when $\calF=(\F_2)_X$;
both bounds are $\Theta(\veps^{k+1})$.
The reason
for this  difference
seems to stem from the fact that the arguments
in 
\cite{Kaufman_Mass_2022_improved_cosystole},
\cite{Dikstein_2023_cbe_cse_without_dep_on_dim_deg}  make  critical use of  the fact that the sheaf is constant and the poset is simplicial.
A new feature of our result that did not appear in previous works on cosystolic expansion is the use of no-intersection graphs, which turned out to be  necessary in  treating the non-simplicial case. 
No-intersection graphs were already studied in  \cite{Kaufman_2021_amplified_local_testability_preprint}
in the context of \emph{amplified testability}, but not as a way to get cosystolic expansion.

A local criterion for establishing cosystolic expansion 
of \emph{general} sheaves on \emph{simplicial complexes} first appeared in an earlier work of the authors
\cite[\S8]{First_2023_sheaves_on_complexes_preprint} now
subsumed by this work.

The main result of \cite{Dinur_2024_expansion_cube_complexes} gives   sufficient conditions
for certain  non-constant  sheaves on certain cube complexes
to have good    cosystolic and \emph{systolic} 
expansion.
It is  likely that
the cosystolic part of this result follows    from our Theorem~\ref{TH:lgp-intro},
possibly with different constants.
Indeed, the \emph{two-way robustness} requirement in \cite{Dinur_2024_expansion_cube_complexes} is 
essentially the same as saying 
that the proper links of the sheaved cube complex at hand are good
coboundary and boundary expanders, and therefore satisfy assumptions
(1a) and (1b) of Theorem~\ref{TH:lgp-intro},
and the expansion conditions in \cite{Dinur_2024_expansion_cube_complexes}
imply the necessary expansion condition (2) in Theorem~\ref{TH:lgp-intro}
when $X$ is a cube complex  \cite[Claim~5.11, Lem.~5.12]{Dinur_2024_expansion_cube_complexes}.

Our main theorem is also related to the main result
of \cite{Kaufman_2021_amplified_local_testability_preprint} (see also
\cite{Kaufman_2021_amplified_local_summary}). There,
the authors consider \emph{codes modelled over $2$-layer   expanding systems},
and show that these codes are locally testable if the underlying system
satisfies some global and local expansion conditions.
In our language, the $2$-layer system is a $2$-poset $X$, and the code modelled over it, while not strictly
being a sheaf, is very similar to an $\F$-sheaf $\calF$ on $X$ in which $\calF(x)=\F$ for all $x\in X$
and all the restriction maps are isomorphisms. 
The conditions under which    \cite[Thm.~1.17]{Kaufman_2021_amplified_local_testability_preprint}
applies
resemble assumptions (1a)--(2c) of Theorem~\ref{TH:lgp-zero-simple} and   suggest that   conditions of this flavor may be  necessary in general.
In fact, the codes modelled on $2$-layer
systems   of \cite{{Kaufman_2021_amplified_local_testability_preprint}}
are examples of \emph{constraint systems} on a poset --- a variation on the definition of a sheaf to
which our main result still applies, see \S\ref{subsec:constraint-system} ---,
and so our Theorem~\ref{TH:lgp-intro} recovers the testability
part of \cite[Thm.~1.17]{Kaufman_2021_amplified_local_testability_preprint}.
On the other hand, \cite{{Kaufman_2021_amplified_local_testability_preprint}} establishes a stronger kind of testability called
\emph{amplified local testability}.

Finally, we note that Kaufman and Tessler \cite{Kaufman_2023_Garland_for_Posets}
extended Garland's Method and   Oppenheim's Trickling Down Theorem,
which are examples of local criteria  for other types of expansion, from
simplicial complexes to general posets.

\paragraph*{About The Proof.}

The proof of Theorem~\ref{TH:lgp-intro} 
is loosely based on the 
\emph{fat machinery} method of
\cite{Kaufman_2016_isoperimetic_inequalities} and
\cite{Evra_2024_cosystolic_expanders_bdd_deg}, called \emph{heavy machinery} here.
Broadly speaking, the idea is to first   reduce  the problem into showing that a \emph{locally
minimal} $k$-cochain $f\in C^k:=C^k(X,\calF)$ with small support must expand under $d_k:C^k\to C^{k+1}$.
Being locally minimal means that for every $z\in X$ of dimension $i\geq 0$,
the restriction $f_z:=f|_{X_z(k-i-1)}\in C^{k-i-1}(X_z,\calF_z)$ satisfies
$\|f_z\|=\dist(f_z,B^{k-i-1}(X_z,\calF_z))$. Thus, if   we assume
that $\cbe_{k-i-1}(X_z,\calF_z)\geq \veps$, then we would know that $\|d_{k-i-1}f_z\|\geq \veps\|f_z\|$.
One would like to take advantage of this to show that $\|d_kf\|$ is at least proportional to $\|f\|$, 
but in general, $d_{k-i-1}f_z$ and $(d_k f)_z$ may differ.
The heavy machinery is a method of keeping track of faces $z$ such that
$d_{k-i-1}f_z=(d_k f)_z$, ultimately showing that the contribution of
faces for which this equality fails is negligible.

We follow this general strategy, but introduce many new ingredients.
For example, we use information from no-intersection graphs (that are not underlying hypergraphs of links),
which is necessary to make the argument work for general   posets,
and introduce \emph{intersection profiles} to keep track of the types of no-intersection graphs that we need.
We also introduce \emph{terminal faces} and use them to make the delicate summation process
over the faces $z$ above more efficient and streamlined. Furthermore, following ideas from \cite{Kaufman_2021_amplified_local_testability_preprint} and \cite{Dikstein_2023_cbe_cse_without_dep_on_dim_deg}, we replace locally
minimal cochains with a variation --- \emph{mock $q$-locally minimal cochains} ---,
which allows us to make the lower bound on $\cse_k(X,\calF)$
independent of the dimension and the degree of $X$.
Instead, the bound  depends on the lower irregularity of $X$ (in all dimensions); this dependence was
transparent in works concerning with simplicial and cube complexes, because they
have lower irregularity $1$.

\subsection{First Application: Good $2$-Query LTCs}
\label{subsec:intro-ltc}

As an application of   Theorem~\ref{TH:lgp-intro} and its finer version
Theorem~\ref{TH:lgp-zero-simple},
we give an example of good $2$-query LTCs arising from sheaves on square complexes.
These codes are in fact the \emph{line codes} of the   good LTCs of \cite{Dinur_2022_ltcs_const_rate}. 
By interpreting these codes as $0$-cocycle codes of sheaves, we can apply
Theorem~\ref{TH:lgp-zero-simple} to neatly deduce that they form a $2$-query LTC.
This offers a new perspective on the LTCs
\cite{Dinur_2022_ltcs_const_rate}, showing that their 
testability may be seen as a consequence  of cosystolic expansion.
It also shows that the \emph{agreement testability}
requirement appearing in \cite[Thm.~4.5]{Dinur_2022_ltcs_const_rate} is actually
a manifestation of coboundary expansion.
(We remark that while  the good  LTCs of \cite{Dinur_2022_ltcs_const_rate}
are related to the good LTCs of \cite{Panteleev_2022_good_quantum_codes},  we do not know how to directly
relate our LTCs  to those of \cite{Panteleev_2022_good_quantum_codes}.)

We shall first describe our  good $2$-query LTCs, and after that explain their relation   
to the good LTCs of \cite{Dinur_2022_ltcs_const_rate}.

\paragraph*{The Example.}

We take our base poset $X$ to be a \emph{left-right Cayley complex}.
Let $G$ be a finite group and let $A,B\subseteq G$ be two
sets of generators for $G$ such that $A=A^{-1}$,
$B=B^{-1}$, $1\notin A\cup B$
and no element of $A$
is a conjugate of an element of $B$.
Recall that $X=\mathrm{Cay}(A,G,B)$ is a square complex with faces determined as follows:
\begin{itemize}
	\item $X(0)=\{\{g\}\where g\in G\}$, 
	\item $X(1)=\{\{g,ag\}\where g\in G,\, a\in A\}\cup \{\{g,gb\}\where g\in G,\, b\in B\}$,
	\item $X(2)=\{\{g,ag,gb,agb\}\where g\in G,\,a\in A,\,b\in B\}$.
\end{itemize}
(We also have $X(-1)=\{\emptyset\}$.)
Our assumptions imply that for every $\{g\}\in X(0)$ and $e\in X(1)$
containing $\{g\}$, there is a unique $x\in A\cup B$ such that $e=\{g,xg\}$ if
$x\in A$ and $e=\{g,gx\}$ if $x\in B$.
A similar claim applies to edges and squares.

Let   $C_A\subseteq \F_2^A$ and $C_B\subseteq \F_2^B$ be linear codes.
It will be convenient to view $\F_2^A\otimes \F_2^B$ as the space $\nMat{\F_2}{A\times B}$
of matrices with rows indexed by $A$ and columns indexed by $B$.
Given a $m\in  \nMat{\F_2}{A\times B}$, we write $r_a(m)$ for the $a$-th row of $m$
and $c_b(m)$ for the $b$-th column of $m$.
The tensor code $C_A\otimes C_B\subseteq \F_2^A\otimes \F_2^B=\nMat{\F_2}{A\times B}$
may now be viewed as the space of $A\times B$-matrices $m$ with
$r_a(m)\in C_B$ and $c_b(m)\in C_A$ for all $a\in A$ and $b\in B$.

We define a sheaf $\calF$ on $X=\mathrm{Cay}(A,G,B)$ as follows:
\begin{itemize}
	\item $\calF(\emptyset)=0$,
	\item $\calF(\{g\})=C_A\otimes  C_B$,
	\item $\calF(\{g,ag\})=C_B$,
	\item $\calF(\{g,gb\})=C_A$,
	\item $\calF(\{g,ag,gb,agb\})=\F_2$,
	\item $\res_{\{g,ag\}\from \{g\}}=r_a:C_A\otimes  C_B\to  C_B$,
	\item $\res_{\{g,gb\}\from \{g\}}=c_b:C_A\otimes  C_B\to C_A$,
	\item $\res_{\{g,ag,gb,agb\}\from \{g,ag\}}:C_B\to \F_2$ is projection on the $b$-component,
	\item $\res_{\{g,ag,gb,agb\}\from \{g,gb\}}:C_A\to \F_2$ is projection on the $a$-component,
\end{itemize}
where here, $g\in G$, $a\in A$, $b\in B$.
It is straightforward 
to check that this is well-defined.
Observe further that if we put $\Sigma=C_A\otimes C_B$,
then $\calF(v)=\Sigma$ for every $v\in X(0)$.
We may therefore form the $0$-cocycle
code   $Z^0=Z^0(X,\calF)\subseteq C^0(X,\calF)=\Sigma^{X(0)}=\Sigma^G$.

Upon unfolding the definition,
one sees that the   natural $2$-query tester of $Z^0$ works as follows: Given
$f\in \Sigma^G$, it chooses an edge $\{g,h\}\in X(1)$ and probes
$f(\{g\})$ and $f(\{h\})$. If $h=ag$ for some $a\in A$,
then $f$ is accepted if and only if $r_a(f(g))=r_{a^{-1}}(f(h))$,
and if $h=gb$ for some $b\in B$, then $f$ is accepted
if and only if $c_b(f(g))=c_{b^{-1}}(f(h))$.

We  use Theorem~\ref{TH:lgp-zero-simple} to give
sufficient conditions on $Z^0\subseteq \Sigma^G$ 
to be locally testable and have linear distance.
To phrase them, recall   \cite[Dfn.~2.8]{Dinur_2022_ltcs_const_rate} 
that the tensor code $C_A\otimes C_B\subseteq \nMat{\F_2}{A\times B}$
is said to be 
\emph{$\kappa$-agreement
testable} if for all $m_1\in C_A\otimes \F_2^B$ and $m_2\in \F_2^A\otimes C_B$, there is $m\in  C_A\otimes C_B$
such that
\begin{align*}
	\kappa\cdot & \left[\frac{\#\{a\in A\suchthat r_a(m_2)\neq r_a(m)\}}{2|A|}
	+\frac{\#\{b\in B\suchthat c_b(m_1)\neq c_b(m)\}}{2|B|}\right]
	\\
	&\qquad\qquad\qquad\qquad\qquad\qquad\qquad\qquad \leq 
	\frac{\#\{(a,b)\in A\times B\suchthat (m_1)_{a,b}\neq (m_2)_{a,b}\}}{|A||B|}.
\end{align*}
Informally, this means that if $m_1$
and $m_2$
agree on nearly all entries, then there is $m\in C_A\otimes C_B$
which agrees with $m_1$ and $m_2$
on nearly all columns and rows, respectively.
See \cite{Dinur_2022_ltcs_const_rate} for more information and examples.

\begin{thm}[Simplified; see Theorem~\ref{TH:ltc-square}]\label{TH:ltc-intro}
	For every $\veps>0$ there are   real
	numbers $\lambda,\mu,\delta_0,\eta>0$ (polynomial in $\veps$) 
	such that the following hold:  
	Let $G,A,B,X,\calF$ and $Z^0\subseteq \Sigma^G$ 
	be as above. Suppose that $|A|\leq |B|$ and 
	the following conditions are met:
	\begin{enumerate}
		\item[(1a$'$)] $\delta(C_A)\geq \veps$,
		\item[(1b$'$)] $\delta(C_B)\geq \veps$,
		\item[(1c$'$)] $C_A\otimes C_B$ is $\veps$-agreement testable,
		\item[(2\,$'$)] the Cayley graphs $\mathrm{Cay}(A,G)$ and $\mathrm{Cay}(G,B)$
		are $\lambda$-expanders, i.e.,
		the second largest eigenvalue of their normalized adjacency operator is at most $\lambda$.
	\end{enumerate}
	Then $\delta(Z^0)\geq \delta_0$ and the natural $2$-query tester of $Z_0$
	has soundness $\frac{|A| }{|B|^2}\cdot  \mu$.
	Moreover,
	$r(Z^0)\geq \frac{4r(C_A)r(C_B)-3}{4r(C_A)r(C_B)}$ 
	and $Z^0$ admits  a linear-time decoding algorithm for words
	that are $\eta$-close to $Z^0$.
\end{thm}

We derive Theorem~\ref{TH:ltc-intro} by applying Theorem~\ref{TH:lgp-zero-simple} 
to the $X$ and $\calF$ we constructed. The full details
are given in Section~\ref{sec:ltc-example}.
Briefly,  
conditions (1a$'$) and (1b$'$) are equivalent to saying that $\cse_{-1}(X_e,\calF_e)\geq \veps$
for every $e\in X(1)$,
and condition (1c$'$) is equivalent to having $\cse_0(X_e,\calF_e)\geq \veps$.\footnote{
	Checking that this follows  readily from the definitions is a recommended exercise.
} 
With a little more work, one further derives from (1a$'$) and (1b$'$)
that $\cse_{-1}(X_v,\calF_v)\geq \veps$ for all $v\in X(0)$,
so (1a$'$)--(1c$'$) imply conditions (1a) and (1b) of Theorem~\ref{TH:lgp-zero-simple}.
One can further show that (2$'$) implies
that $\ugr(X)=\NNI^{0,0,1}(X)$ is a $(\lambda,1)$-skeleton expander
and $\NNI^{1,1,2}(X)$ is a $(2\lambda,4(|A|+|B|))$-skeleton expander.
Moreover, $\ugr(X_v)$ is a $(0,1)$-skeleton expander for every $v\in X(0)$,
because it is a complete bipartite graph. Now, one can readily check  that the inequalities
in Theorem~\ref{TH:lgp-zero-simple} are solvable if $\lambda$ is small enough
and deduce Theorem~\ref{TH:ltc-intro}.

It was observed in \cite{Dinur_2022_ltcs_const_rate} that
there is $\veps>0$ and $n_0\in\N$ such that whenever $|A|,|B|\geq n_0$,
there exist codes $C_A\subseteq \F_2^A$ and $C_B\subseteq \F_2^B$
satisfying conditions (1a$'$)--(1c$'$)
and also $r(C_A),r(C_B)>\sqrt{\frac{3}{4}}$.
Let $\lambda$ be the constant obtained by applying Theorem~\ref{TH:ltc-intro} to 
that $\veps$. It is further known that there is  $n_1\geq n_0$ for
which there 
are infinitely many examples $(G_i,A_i,B_i)_{i\in\N}$
of $G,A,B$ as above such that $|A_i|=|B_i|=n_1$
and both $\mathrm{Cay}(A_i,G_i)$ and $\mathrm{Cay}(G_i,B_i)$
are $\lambda$-expanders.
By applying Theorem~\ref{TH:ltc-intro} to the family
$(G_i,A_i,B_i)_{i\in\N}$ and suitable codes $C_A,C_B\subseteq \F_2^{n_1}$,
we obtain a family of good $2$-query  LTC. 

\paragraph*{Relation Between Lifted Codes and Line Codes.}

To better describe the relation between
our $2$-query LTCs and the   LTCs of \cite{Dinur_2022_ltcs_const_rate},
we need to  a briefly digress and discuss the relation between
lifted codes and their so-called line codes.
See Section~\ref{sec:lifted-vs-line} for an extensive discussion.

Recall that a \emph{lifted code} or a \emph{Tanner code}
$C\subseteq \Sigma^n$ is determined by a family $S$ of subsets
of $[n]:=\{1,\dots,n\}$ covering $[n]$ and, for every $s\in S$,
a code $C_s\subseteq \Sigma^s$.
The lifted code that the family $\{C_s\}_{s\in S}$
determines is 
\[
C=C(\{C_s\}_{s\in S}):=\{f\in \Sigma^n\suchthat \text{$f|_s\in C_s$ for all $s\in S$}\}.
\]
When all the $C_s$ are the same code $D\subseteq \Sigma^{m}$
(or, more generally, whenever $|C_s|=|D|$ for  
all $s\in S$),
we may further associate with $C$ a
code $L\subseteq D^{S}$ with alphabet $D$
known as its   \emph{line code}; it is defined
by
\[
L=\{f=(f_s)_{s\in S}\in D^S\suchthat 
\text{$f_s|_{s\cap s'}=f_{s'}|_{s\cap s'}$ for all $s,s'\in S$}\}.
\]
There is a bijection $f\mapsto (f|_s)_{s\in S}:C\to L$,
so both $C$ and $L$ have proportional rates. Under mild assumptions,
$\delta(C)$ and $\delta(L)$ are also proportional (Proposition~\ref{PR:line-code-rate-and-dist}).

The presentation of $C\subseteq\Sigma^n$ as a lifted code $C=C(\{C_s\}_{s\in S})$ 
gives rise to a natural tester:
Given $f\in \Sigma^n$, choose $s\in S$ uniformly at random and accept $f$
if and only if $f|_s\in C_s$.
This tester usually has poor soundness,
which is why LTCs are considered difficult to construct.
However,
we show in Theorem~\ref{TH:line-code-testability} that
if the line code $L$ of a lifted code $C$ (varying in a family) 
is $2$-query locally testable, then same holds
for   the original code $C$ with its natural tester. 
Also, if $L$ has a linear-time decoding algorithm, then the same holds for $C$
(Proposition~\ref{PR:line-code-eff-decoding}).

\paragraph*{Relation to \cite{Dinur_2022_ltcs_const_rate}.}

Let $G,A,B,X,C_A$ and $C_B$ be as before.
For $g\in G$, we write $s(g)$ for the set of squares containing the $0$-face $\{g\}$ of $X$.
There is a bijection from $A\times B$ to $s(g)$ given by sending $(a,b)$
to $\{g,ag,gb,agb\}$, and we use it to identify
$\F_2^{s(g)}$ with $\F_2^{A\times B}=\F_2^A\otimes \F_2^B$.
Now let $C(A,G,B)$
denote the lifted code  $C\subseteq \F_2^{X(2)}$, 
determined by the sets $\{s(v)\where v\in X(0)\}$ and the codes
\[
C_{s(v)}=C_A\otimes C_B\subseteq \F_2^A\otimes \F_2^B\cong \F_2^{s(v)}.
\]
We endow  $C(A,G,B)$ with its natural $|A\times B|$-query tester.
In \cite{Dinur_2022_ltcs_const_rate}, it was shown
that the codes $C(A,G,B)$ form a good LTC if conditions
(1a$'$)--(2$'$) of Theorem~\ref{TH:ltc-intro} hold  
and $C_A$ and $C_B$ have sufficiently large rate.

It is straightforward to see 
that code $Z^0(X,\calF)$ considered in Theorem~\ref{TH:ltc-intro}
is the line code of $C(A,G,B)$.
Thus, our earlier discussion implies that we can
derive the fact that $C(A,G,B)$
is a good LTC from the fact that $Z^0(X,\calF)$ is a good LTC.

In fact, the fact that the line code of $C(A,G,B)$ 
is locally testable is 
already proved implicity in  \cite{Dinur_2022_ltcs_const_rate},
and the testability of $C(A,G,B)$ 
is derived from it
(look at Algorithm~1 in \cite{Dinur_2022_ltcs_const_rate}, which is also   a correction
algorithm for the line code of $C(A,G,B)$). Therefore, a variant
of Theorem~\ref{TH:ltc-intro}
is  already implicit in \cite{Dinur_2022_ltcs_const_rate}.
Our discussion here is meant to highlight the role of the line code $Z^0(X,\calF)$
in the proof that $C(A,G,B)$ is locally testable,
the fact that the testability of $Z^0(X,\calF)$ 
is a consequence of $\calF$ being a good cosystolic expander
in dimension $0$, and that this can be shown
using our main theorem.

\paragraph*{Relation to Expander Codes.}

Sipser and Spielman's   \emph{expander
codes} \cite{Sipser_1996_expander_codes} are famous examples of good lifted codes.
It turns out that the  line code of an expander code may be realized as the $0$-cocycle code  of a sheaf
on a regular graph. We can use this observation and the same reasoning from earlier
to recover the good properties
of expander codes. Specifically,   a relaxed form of our main result (Theorem~\ref{TH:ccd-lower-bound}) implies
that the $0$-cocycle code has linear distance, and that is enough to deduce that the corresponding
expander code has linear distance.
In fact, we get a slighly better lower bound on the distance.
See \S\ref{subsec:expander-codes} and Example~\ref{EX:sheafy-expander-codes-distance} for details.

\subsection{Second Application: A Local Criterion for Local Testablity of $2$-Layer Lifted Codes}
\label{subsec:two-later-Tanner-intro}

In \cite{Dikstein_2020_ltcs_via_hdxs_preprint},
Dikstein, Dinur, Prahladh and Ron-Zewi give a criterion for a lifted  code with additional
structure to be locally
testable. 
When working inside $\Sigma^n$, the system of sets used
to define the lifted code is required to be embedded in 
an   auxiliary $3$-layer system of subsets
of $[n]=\{1,\dots,n\}$,  which needs to satisfy some expansion conditions
and a  \emph{global} condition on  \emph{agreement testability}.

We   apply Theorem~\ref{TH:lgp-intro} to give
a simpler,  purely \emph{local} criterion for establishing the local testability
of a lifted code.
The ``small'' codes defining our lifted code are  
required to be lifted codes themselves;
we call this $2$-layered structure, defined below,
a   \emph{$2$-layer lifted code}.
A third layer is needed to apply our criterion, but not to define the code;
it is required   in order to be able to talk about agreement testability for the ``small'' lifted codes
defining our global code.

\paragraph*{Agreement Testability.}

The notion of
\emph{agreement expansion} was first considered
in 
\cite{Dinur_2017_hdx_imply_agreement_exp}
and studied further in
\cite{Dikstein_2019_agreement_testing_two_layered_sys}.
Informally, an agreement expander consists of a collection of subsets
$S$ of $[n]$ such that for any finite set $\Sigma$
and any ensemble of functions $\{f_s:s\to \Sigma\}_{s\in S}$
such that $f_s|_{s\cap s'}=f_{s'}|_{s\cap s'}$ for almost all
$s,s'\in S$, there is $g:[n]\to \Sigma$ such that $g|_s=f_s$
for almost all $s\in S$.
Here, we will consider a   refined version of this notion
where each $f_s$ is required to be in a code $C_s\subseteq \Sigma^s$
and $g$ comes from $C=C(\{C_s\}_{s\in S})$.
In the special case of tensor codes, realized as lifted codes 
(Example~\ref{EX:agreement-test-for-tensor-codes}),  
this already appeared in \cite[Dfn.~2.8]{Dinur_2022_ltcs_const_rate} 
under the name \emph{agreement testability}, which we also use here.
We give here a simplified version of the definition   
and refer to \S\ref{subsec:agreement}
for the general definition.

Let  $S$ be a collection of subsets of $[n]$  and
let $C=C(\{C_s\}_{s\in S})\subseteq\Sigma^n$ be a lifted code.
Suppose further that we are given a collection
$T$ of $2$-element subsets of $S$ 
such that $s\cap s'\neq\emptyset$  for every $\{s,s'\}\in T$.
The agreement testability of the lifted code
$C$  w.r.t.\ $T$  
measures how far is an ensemble of local views $\{f_s\in C_s\}_{s\in S}$
such that $f_s|_{s\cap s'}=f_{s'}|_{s\cap s'}$ for almost all 
$\{s,s'\}\in T$ from   being induced by some $g\in C$.
Formally, we say that $C=C(\{C_s\}_{s\in S})$ is \emph{$\kappa$-agreement
testable} w.r.t.\ $T$ if 
for every $(f_s)_{s\in S}\in\prod_{s\in S}C_s$,
	there is $g\in C$ such that
	\[
	\kappa \cdot \frac{\#\{s\in S\suchthat g|_s\neq f_s\}}{|S|}
	\leq  \frac{\#\{\{s,s'\}\in T\suchthat f_s|_{s\cap s' }\neq f_{s'}|_{s\cap s'}\}}{|T|}.
	\]

\begin{remark}
We show in Corollary~\ref{CR:agreement-equiv-loc-test}
that under mild assumptions on $S$, the   lifted
code $C=C(\{C_s\}_{s\in S})\subseteq \Sigma^n$ has positive
agreement testability
w.r.t.\ to \emph{some}
$T$
if and only if it is locally
testable w.r.t.\ to its natural tester. However, 
since the soundness $\kappa$ is scaled down by a constant depending
on $\max_{s\in S}|s|$ when
moving from local testability 
to agreement testability,  interchanging between these notions
is often impractical.
\end{remark}

\paragraph{$2$-Layer Lifted Codes.}

Again, for the sake of simplicity, we give a special case of the general definition,
which may be found     in \S\ref{subsec:two-layer}.

Let $n\in\N$ and let $\Sigma$ be a finite alphabet.
A \emph{$2$-layer lifted code} inside $\Sigma^n$
is a triple
$(S,T,\{C_{s,s'}\}_{s,s'})$
consisting of:
\begin{itemize} 
\item a collection $S$ of subsets of $[n]$;
\item a collection $T$ of $2$-element subsets of $S$;
\item a code $C_{s,s'}\subseteq \Sigma^{s\cap s'}$ for every $\{s,s'\}\in T$;
\end{itemize}
such that:
\begin{enumerate}[label=(\arabic*)]
	\item $S$ covers $[n]$,
	\item for every $s\in S$, the collection $S_s:=\{s\cap s'\in S\suchthat \text{$s\in S$ and $\{s,s'\}\in T$}\}$
	covers
	$s$, and
	\item $s\cap s'\neq\emptyset$ for every $\{s,s'\}\in T$.
\end{enumerate} 
We then associate with $(S,T,\{C_{s,s'}\}_{s,s'})$ a \emph{local} lifted code
\[
C_s=C(\{C_{s,s'}\}_{s'\in S_s})\subseteq \Sigma^s
\]
for every $s\in S$,  
and a \emph{global} lifted code
\[
C=C(\{C_s\}_{s\in S})=C(\{C_{s,s'}\}_{\{s,s'\}\in T})\subseteq \Sigma^n.
\]
The \emph{natural tester} of $C$ is its natural tester when realized as 
a lifted code using the codes $\{C_s\}_{s\in S}$.

\paragraph{Local Testability of $2$-Layer Lifted Codes.}

Let $(S,T,\{C_{s,s'}\}_{s,s'})$ be a $2$-layer
lifted code inside $\Sigma^n$. 
Our local criterion for local testability of the lifted code
$C=C(\{C_s\}_{s\in S})$ requires an additional layer of subsets of $[n]$.
Specifically,
suppose that we are further given
a collection $U$ of $3$-element subsets of $S$
such that:
\begin{enumerate}[label=(\arabic*)]
	\item $\{s,s'\},\{s,s''\},\{s',s''\}\in T$ for every  $\{s,s',s''\}\in U$;
	\item $s\cap s'\cap s''\neq\emptyset$ for every  $\{s,s',s''\}\in U$;
	\item for every $\{s,s'\}\in T$, the sets of the form $s\cap s'\cap s''$ with $\{s,s',s''\}\in U$ cover $s\cap s'$.
\end{enumerate}
In particular, the  union of $T$, $U$ and
$\{\{s\}\where s\in S\}\cup\{\emptyset\}$ forms a $2$-dimensional
simplicial complex denoted $X$.

Given $s\in S$, let $S_s:=\{s\cap s'\in S\suchthat \text{$s\in S$ and $\{s,s'\}\in T$}\}$
as before,
and let  $T_s$
denote the set of pairs $\{s\cap s',s\cap s''\}$
where $\{s,s',s''\}\in U$.
We say that the system $(S,T,U)$ is \emph{lower-regular} if
for every $-1\leq i<j\leq 2$ and every $x\in X(i)$ and $a\in [n]$,
the number of $y\in X(j)$ with $x\subseteq y$ and $a\in \bigcap_{s\in y}s$
is independent of $x$ and $a$.
We say that $(S,T,U)$ is \emph{upper-regular} if for
every $s\in S$ (resp.\ $\{s,s'\}\in T$,
$\{s,s',s''\}\in U$), the number  $\#s$ (resp.\ $\#(s\cap s')$, $\#(s\cap s'\cap s'')$)
is independent of $s$ (resp.\ $\{s,s'\}$, $\{s,s',s''\}$).

\begin{thm}[Simplified; see Theorem~\ref{TH:two-layer-local}]\label{TH:local-two-layer-Tanner-code}
	There are   constants $K,K'>0$ such that the following hold:
	Let $(S,T,\{C_{s,s'}\}_{s,s'})$  be a $2$-layer lifted
	code 
	whose alphabet $\Sigma$ is an $\F_2$-vector space
	and such that every $C_{s,s'}$ is a subspace of $\Sigma^{s\cap s'}$.	
	Let $U$ and $X$ be as above. Suppose
	that $(S,T,U)$ is both lower- and upper-regular and satisfies:
	\begin{enumerate}
		\item[(0)] For every $s,s'\in S$ and $i\in s\cap s'$,
		there are $s_0,s_1,\dots,s_m\in S$ such that $s=s_0$, $s'=s_m$,
		$\{s_0,s_1\},\dots,\{s_{m-1},s_m\}\in T$ and $i\in s_0\cap\dots\cap s_m$.
	\end{enumerate}
	Let $\veps>0$ and suppose in addition that:
	\begin{enumerate}
		\item[(1a)] $\delta(C_{s,s'})\geq \veps$ for every $\{s,s'\}\in T$;
		\item[(1b)] the lifted code $C_s=C(\{C_{s,s'}\})_{s'\in S_s})\subseteq\Sigma^s$
		is  $\veps$-agreement testable w.r.t.\   $T_s$ for every $s\in S$;
		\item[(2)] $\ugr(X_v)$ is a   $ K\veps^2$-expander
		(i.e., its second largest normalized eigenvalue is at most $K\veps^2$)
		for every $v\in X(0)$.
	\end{enumerate}
	Then the natural tester of the lifted code $C=C(\{C_s\}_{s\in S})\subseteq \Sigma^n$
	has soundness $K' \veps^3$ and the relative distance of $C$ is at least $K'\veps^2$.
	Moreover, writing $D=\max_{s\in S}|s|$,
	the code $C$ has a linear-time decoding algorithm
	for words  that are $\frac{K'\veps^3}{D}$-close to $C$.
\end{thm}

Observe that assumptions (1a)--(2) are all \emph{local} in the sense that they
care only about the local structure of $X$ and   the small codes $C_s$
and $C_{s,s'}$.
We actually prove a more general version of this theorem
where no regularity assumptions are necessary and $(S,T,U)$ may be replaced
with a general three-layered system of subsets of $[n]$
organized into a
pure $2$-dimensional regular cell complex; 
see \S\ref{subsec:subset-labelling} and 
Theorem~\ref{TH:two-layer-general}. 
In this more general setting,
one needs to require that the underlying graph of $X$ and its $(1,1,2)$-no-intersection
graph  are sufficiently good skeleton expanders.

\medskip

In order to prove Theorem~\ref{TH:local-two-layer-Tanner-code}, we define a sheaf $\calF$ on $X$ as follows:
\begin{itemize}
	\item $\calF(\emptyset)=0$,
	\item $\calF(\{s\})=C_s$ for all $s\in S$,
	\item $\calF(\{s,s'\})=C_{s,s'}$ for all $\{s,s'\}\in T$,
	\item $\calF(\{s,s',s''\})=\Sigma^{s\cap s'\cap s''}$ for all $\{s,s',s''\}\in U$,
	\item $\res_{\{s,s'\}\from \{s\}}(f)=f|_{s\cap s'}$,
	\item $\res_{\{s,s',s''\}\from \{s,s'\}}(f)=f|_{s\cap s'\cap s''}$.
\end{itemize}
Condition (0) of Theorem~\ref{TH:local-two-layer-Tanner-code}
implies that the $0$-cocycle code $Z^0=Z^0(X,\calF)$ is precisely
the line code of the lifted code $C=C(\{C_s\}_{s\in S})$.
Thus, as   noted earlier in
\S\ref{subsec:intro-ltc},
in order to prove that $C$ is locally testable w.r.t.\ its natural
tester, it is enough to show that $Z^0 $ is locally
testable. 
To that end, we   apply Theorem~\ref{TH:lgp-intro}
or Theorem~\ref{TH:lgp-zero-simple}.
The prerequisites of those theorems can be derived from
conditions   (1a)--(2) and Oppenheim's Trickling Down Theorem \cite[Thm.~4.1]{Oppenheim_2018_local_spectral_expansion_I}, thanks to the fact that $X$
is simplicial and $(S,T,U)$ is  lower-regular.

\paragraph*{Comparison with \cite{Dikstein_2020_ltcs_via_hdxs_preprint}}

As we noted earlier, \cite{Dikstein_2020_ltcs_via_hdxs_preprint} also 
provides
a criterion for a lifted code to be locally testable.
Both the main result
of \cite{Dikstein_2020_ltcs_via_hdxs_preprint}  and
our Theorem~\ref{TH:local-two-layer-Tanner-code} 
assume  that a $3$-layered system 
of subsets of $[n]$ is provided, but otherwise, they differ
both in the setting and the assumptions. 
To  state these differences, let $S,T,U$
be as above and put $\tilde{T}=\{s\cap s'\where s,s'\in T\}$
and $\tilde{U}=\{s\cap s'\cap s''\where \{s,s',s''\}\in U\}$.\footnote{
	In \cite{Dikstein_2020_ltcs_via_hdxs_preprint}, the collections
	$S,\tilde{T},\tilde{U},[n]$ are denoted $S,K,T,V$.
}

In  \cite{Dikstein_2020_ltcs_via_hdxs_preprint}, one starts from
small codes $C_{u}\subseteq \Sigma^{u}$
for every $u\in \tilde{U}$, using which one constructs
bigger lifted codes $C_t\in \Sigma^t$ and $C_s\in \Sigma^s$
for every $t\in \tilde{T}$ and $s\in S$.
The main result of \cite{Dikstein_2020_ltcs_via_hdxs_preprint} may now be loosely
summarized as saying 
that $C=C(\{C_u\}_{u\in\tilde{U}})\subseteq\Sigma^n$
is locally testable when the following conditions are met: (1) each  lifted code
$C_t\subseteq\Sigma^t$ on the layer $\tilde{T}$ has   linear distance,
(2) each lifted code  $C_s\subseteq\Sigma^s$ on the layer
$S$
is locally testable w.r.t.\ its natural tester,
(3) the incidence graph of $(\tilde{T},\tilde{U})$ satisfies an expansion   condition,
and (4) the pair $(S,\tilde{T})$ satisfies an agreement testability condition
(for \emph{$\delta$-ensembles}). 
Note that conditions (3) and (4) concern with the global structure
of the collections $S,\tilde{T},\tilde{U}$.
By contrast, our criterion for local testability 
(Theorem~\ref{TH:local-two-layer-Tanner-code}) starts with ``bigger'' small
codes $C_{t}\subseteq \Sigma^{t}$
on the layer $\tilde{T}$ and replaces requirements
(2),(3),(4) with two \emph{local} requirements which may be loosely   summarized
as saying that for every $s\in S$,
the incidence graph of the $t\in \tilde{T}$
and $u\in\tilde{U}$ contained in $s$ is a good expander,
and the lifted code $C_s=C(\{C_t\}_{ t\in\tilde{T}: t\subseteq s})\subseteq \Sigma^s$
is agreement testable.

That said,  Theorem~\ref{TH:local-two-layer-Tanner-code} applies only
when the sets $S,\tilde{T},\tilde{U}$ may be organized into a $2$-dimensional
simplicial complex and satisfy  some regularity assumptions.
In addition, it requires that the alphabet $\Sigma$ is an $\F_2$-vector space
and the small codes $C_{s,s'}$ are $\F_2$-linear.
No such requirements are imposed
in \cite{Dikstein_2020_ltcs_via_hdxs_preprint}. 
As we noted earlier, our approach
gives a more general version of Theorem~\ref{TH:local-two-layer-Tanner-code} applying to more general $3$-layer collections of subsets of $[n]$.
It requires some additional global expansion assumptions, but no global agreement testablity as in
\cite{Dikstein_2020_ltcs_via_hdxs_preprint}.

\subsection{Conclusion}

By using the relation between lifted codes and their line codes,
we can translate questions about local testability
to statements about cosystolic expansion of sheaves.
Our main theorem (Theorem~\ref{TH:lgp-intro})  serves as a powerful
tool to establish the desired cosystolic expansion.

\subsection{Structure of This Paper}

The remainder of this paper is structured as follows:
Section~\ref{sec:prelim} is preliminary and recalls
relevant facts about expander graphs and error correcting codes, setting some notation along the way.
Section~\ref{sec:lifted-vs-line} concerns with line codes of lifted codes; in this section,
we 
relate the rate, distance, testability and decodability of a lifted code and its line code.
In Section~\ref{sec:posets}, we recall   posets and introduce additional structure
on them that will be needed for this work, e.g., weight functions and orientation.
Sheaves on posets and their cohomology are then discussed in Section~\ref{sec:sheaves}.
The subject matter of Section~\ref{sec:cse} is cocycle codes of sheaves and their relation
to cosystolic expansion.
Section~\ref{sec:non-intersect} concerns with no-intersection (hyper)graphs and their skeleton expansion,
and introduces the notion of an intersection profile.
We then give simplified versions of our main result in Section~\ref{sec:main-simple}.
The results of Section~\ref{sec:main-simple} are applied  in Section~\ref{sec:ltc-example} to give examples of
good $2$-query LTCs, and in Section~\ref{sec:two-layer} to give a local criterion for a $2$-layer
lifted code to be locally testable.
In Section~\ref{sec:tech-versions}, we  formulate our main result in its general  
form (Theorem~\ref{TH:main-very-technical})  and derive  the simpler versions of Section~\ref{sec:main-simple} from it.
The remaining Sections~\ref{sec:proof-of-tech}
and~\ref{sec:proof-of-exp-small-coch} are dedicated to proving the main result ---
Section~\ref{sec:proof-of-tech} reduces it to a result about the expansion of 
(mock) locally minimal cochains (Theorem~\ref{TH:expansion-of-loc-min-cochains}),
which is then proved in Section~\ref{sec:proof-of-exp-small-coch}.

\section{Preliminaries}
\label{sec:prelim}

We begin by recalling relevant definitions and facts concerning expander graphs,
locally testable codes,
lifted codes and agreement testability.

\subsection{General Conventions}

The set of natural numbers $\N$ does not include $0$.
We write $[n]$ for the set $\{1,\dots,n\}$.
The power set of a set $S$ is denoted $P(S)$.
A ring means a commutative (unital, associative)   ring,
and a module means a left module. The group of invertible elements
in a ring $R$ is denoted $\units{R}$.

A (regular) cell complex means a (regular) CW complex, or more precisely, its underlying partially
ordered set, which we assume to include a unique empty cell.

\subsection{Expander Graphs}
\label{subsec:graphs}

Throughout this paper, graphs are finite
and allowed to have double edges,
but no loops. 
A \emph{simple graph} is a graph with no double edges
and a \emph{pure graph} is a nonempty graph in which every vertex belongs to some edge.

Given a graph $G$, we let $G(0)$ denote its set of vertices
and $G(1)$ its set of edges. We also use $G$ to  denote
the set $G(0)\cup G(1)$.
We write $v<e$ to indicate
that $v$ is a vertex of the edge $e$. The the set of (two) vertices
of an edge $e\in G(1)$ is denoted $e(0)$, and the set of edges
having $v\in G(1)$ as a vertex is denoted $G(1)_v$. 
We will sometimes abuse the notation
and write
$e=\{u,v\}$
to say that $e$ connects the vertices $u$ and $v$, even though there
may be other edges with that property.
Given $A\subseteq G(0)$, we let $E(A)$ or $E_G(A)$ denote the set of edges
$e\in G(1)$ with $e(0)\subseteq A$.

\begin{example}[Cayley Graph]\label{EX:Cayley}
Let $G$ be a group and $A$ be a subset of $G-\{1_G\}$ such that $A=A^{-1}:=\{a^{-1}\where a\in A\}$.
Recall that the \emph{left Cayley graph} of $G$ and $A$, denoted
\label{symdef:Cay}$\mathrm{Cay}(A,G)$ is the simple graph with vertex set
$G$ and edge set   $\{\{g,ag\}\where a\in A,g\in G\}$.
The right Cayley graph of $G$ and $A$ is defined similarly by
replacing $\{g,ag\}$ with $\{g,ga\}$, and is denoted $\mathrm{Cay}(G,A)$.
\end{example}

A weight function on a graph $G$ is a function $w :G(0)\cup G(1)\to \R_+$;
we call $(G,w)$ a weighted graph  and, given $A\subseteq G$,
write $w(A)=\sum_{a\in A} w(a)$.
We make no assumptions on $w$. 
However, we will say that
$w$  is \emph{normalized} if $w(G(0))=w(G(1))=1$  
and \emph{proper} if we moreover have $w(v)=\frac{1}{2}\sum_{e\in G(1)_v}w(e)$
for every $v\in G(0)$ (which forces $G$ to be pure). 
A normalized weight function defines probability
measures on $G(0)$ and $G(1)$. It is proper precisely when the probability
of sampling a vertex $v$ according to $w$ is equal to the probability
of getting $v$
by choosing an edge according to $w$ and then choosing one of its vertices
uniformly at random.

\begin{example}\label{EX:graph-nat-weight}
	Let $G$ be a graph. 
	
	(i) The \emph{uniform weight function} $w_{\uni}:G\to \R_+$ assigns every $v\in G(0)$ the weight
	$\frac{1}{|G(0)|}$ and every $e\in G(1) $ the weight $\frac{1}{|G(1)|}$.
	It is defined   when $G$ has at least one vertex and one edge, and is normalized.
	
	(ii) Suppose that $G$ is pure. The \emph{natural weight function}
	of $G$ is 
	$w_{\nat}:G\to \R_+$  defined by
	\[
	w_{\nat}(e)=\frac{1}{|G(1)|}\qquad\text{and}\qquad w_{\nat}(v)=\frac{|G(1)_v|}{2|G(1)|} 
	\] 
	for all $e\in G(1)$ and $v\in G(0)$. This weight function is proper.
\end{example}

The uniform weight function is not proper in general. 
However, when $G$ is a regular graph
(i.e.\ every vertex belongs to the same number of edges), the natural and uniform
weight functions of $G$ coincide.

All graphs  (and also hypergraphs) in this work will carry a weight function,
which 
by default  will be the natural weight function.

\medskip

Suppose henceforth that $(G,w)$ is a properly weighted graph.
The proofs of the following facts can be found in \cite[\S2C]{First_2022_cbe_sheaves_on_graphs},
for instance.

Let $C^0(G,\R)$ denote the space of functions $f:G(0)\to \R$ 
and let $C^0_{\circ}(G,\R)$ denote its subspace
of functions satisfying $\sum_{v\in G(0)}w(v)f(v)=0$.
As usual, the \emph{weighted adjacency operator} of $(G,w)$ is
$\calA=\calA_{G,w}:C^0(G,\R)\to C^0(G,\R)$ given by
\[
(\calA f)(v)=\sum_{e\in G(1)_v} \frac{w(e)}{2 w(v)} f(e-v)
\qquad\forall v\in G(0),
\]
where $e-v$ denotes the vertex of $e$ which is different from $v$.
For example, if $G$ is $k$-regular and $w$ is its natural weight function,
then $\calA$ is just the usual vertex adjacency operator scaled by $\frac{1}{k}$.

The operator $\calA: C^0 (G,\R)\to C^0 (G,\R)$
is diagonalizable. The constant function $1_{G(0)}$ is an eigenfunction
of $\calA$ with eigenvalue $1$ and all other eigenvalues
lie in the interval $[-1,1]$. The subspace $C^0_{\circ}(G,\R)$
is invariant under $\calA$ and complements $\R\cdot 1_{G(0)}$. 
Given $\lambda\in [-1,1]$,
we call $(G,w)$ a \emph{$\lambda$-expander}   if 
all eigenvalues of $\calA$ on $C^0_{\circ}(G,\R)$ lie in the interval
$[-1,\lambda]$. (Here, the smaller $\lambda$ is the more expanding $(G,w)$
is considered.)

We will need the following special case of the Expander Mixing Lemma
for weighted graphs.

\begin{prp}[{\cite[Thm.~3.2(ii)]{First_2022_cbe_sheaves_on_graphs}}]
	\label{PR:eml-special-case}	
	Let $(G,w)$ be   a properly weighted  graph and let.
	If $(G,w)$
	is a $\lambda$-expander, then
	\[
	w(E(A))\leq \lambda w(A) + (1-\lambda) w(A)^2.
	\] 
\end{prp}

Weighted graphs satisfying the condition $w(E(A))\leq w(A)^2+\lambda w(A)$
for every $A\subseteq G(0)$ are known as \emph{$\lambda$-skeleton expanders}.
Thus, for $\lambda\geq 0$, 
every $\lambda$-expander weighted graph is also a $\lambda$-skeleton expander.

\subsection{Error Correcting Codes}
\label{subsec:codes}

Let $\Sigma$ be a finite alphabet and $n\in\N$.
In this work, an \emph{error correcting code},
or a  \emph{code} for short, with alphabet $\Sigma$ and block length $n$ is a nonempty
subset $C\subseteq \Sigma^n$. We also say that $C$ is a code inside $\Sigma^n$.
As usual, the normalized Hamming distance function on $\Sigma^n$ is denoted \label{symdef:dist-I}$\dist(\cdot,\cdot)$
and is given by $\dist(f,g)=\frac{1}{n}\cdot \#\{i\in\{1,\dots,n\}\suchthat f_i\neq g_i\}$.
When $\Sigma$ is an abelian group, the normalized Hamming norm of $f\in\Sigma^n$
is \label{symdef:norm-I}$\|f\|=\frac{1}{n}\cdot \#\{i\in\{1,\dots,n\}\suchthat f_i\neq 0\}$, so that
$\dist(f,g)=\|f-g\|$.
The   \emph{relative distance} of the code $C\subseteq \Sigma^n$ is\label{symdef:delta} 
\[\delta(C):=
\max\{\dist(f,g)\where f,g\in C,\, f\neq g\}\]
and its \emph{rate} is\label{symdef:rate}
\[
r(C) := \log_{|\Sigma^n|}|C|.
\]
The \emph{distance} of $f$ is $n\cdot\delta(C)$.
Given $\eta\in [0,1]$ and $f\in\Sigma^n$,
We say that $f$ is \emph{$\eta$-close} to $C$ if $\dist(f,C)<\eta$
and \emph{$\eta$-far} from $C$ if $\dist(f,C)\geq \eta$ is smaller than $\eta$.

We will often think of $C\subseteq \Sigma^n$
as being part of a family of codes $\{ C_i\subseteq \Sigma^{n_i}\}_{i\in \N}$ with block
length tending to $\infty$,
and (abusing the notation) sometimes ascribe
properties of the   entire family
$\{ C_i\subseteq \Sigma^{n_i}\}_{i\in \N}$ to $C$.
For example, we will say that 
the code $C$ 
is  \emph{good} 
if there are $\rho,\delta>0$ such that
$r(C_i)\geq   \rho$ and     $\delta(C_i)\geq \delta$ for all $i$.
When  $\delta(C_i)\geq \delta$ for all $i$, we also say that $C$ has linear distance
(as a function of the block length $n$). 

Let $\eta\in [0,\frac{1}{2}\delta(C)]$.
A decoding algorithm for words that are $\eta$-close to $C$
is an algorithm which takes as input some $f\in \Sigma^n$ with $\dist(f,C)<\eta$
and outputs
the unique $f'\in C$ satisfying 
$\dist(f,f')<\eta$. The time complexity of a decoding  algorithm
will always be measured w.r.t.\ the block length $n$; ideally, it should be linear.

\medskip

Let $q\in\N$. 
A $q$-tester
for the code $C$ is a randomized algorithm $T$ which, given oracle access to some 
$f\in\Sigma^n$, reads at most $q$ letters from $f$ and returns `accept' or `reject'
subject to the requirement that all words in $C$ are accepted with probability $1$.
The tester $T$ is said to have soundness $\mu$ ($\mu\geq 0$) if 
\[
\mathrm{Prob}(T(f)=\text{`reject'})\geq \mu\dist(f,C)
\qquad\forall f\in \Sigma^n.
\]
A $q$-query \emph{locally testable code} (LTC) is a family of codes $\{C_i\subseteq\Sigma^{n_i}\}_{i\geq 0}$
such that each $C_i$ is equipped with a $q$-tester and all the $C_i$
have     soundness $\mu$ for some $\mu>0$.

\begin{remark}[Codes with Varying Alphabets]\label{RM:generalized-codes}
	We can relax the definition of an error correcting code
	by considering words in which each letter comes from a different
	alphabet, i.e., the $i$-th letter of a word   comes from an alphabet
	$\Sigma_i$ depending on the position $i$. 
	A code would then be a nonempty subset $C$ of $\prod_{i=1}^n\Sigma_i$.
	All the notions just defined extend verbatim to codes with a varying alphabet.
\end{remark}

\subsection{Lifted Codes}
\label{subsec:tanner-codes}

Let $\Sigma$ be a finite alphabet and $n\in\N$.
Recall that a \emph{lifted code}, or a \emph{Tanner code},
is determined by specifying a collection $S$ of subsets
of $[n]:=\{1,\dots,n\}$ with $[n]=\bigcup_{s\in S}s$
and   a code $C_s\subseteq \Sigma^s$
for every $s\in S$. 
The lifted code defined by the $\{C_s\}_{s\in S}$ is
\[
C=C(\{C_s\}_{s\in S}):=\{g\in \Sigma^n \suchthat \text{
$g|_s\in C_s$ for all $s\in S$}\}\subseteq \Sigma^n.
\]
The codes $\{C_s\}_{s\in S}$
are often called the \emph{small codes} defining $C$.
Typically,   all the sets in 
$S$ will have the same size $k=\Theta(1)$ (as $n$ grows)
and every
$i\in [n]$ will be contained in $D=\Theta(1)$ sets from $S$;
the number   of sets in $S$
will therefore be $\frac{D}{k}n=\Theta(n)$. However,
these extra assumptions are not  
necessary.

The presentation of  $C\subseteq \Sigma^n$ as a lifted code   
gives rise to a natural tester: Given $g\in \Sigma^n$, choose $s\in S$
uniformly at random, probe $g_i$ for every $i\in s$, 
and accept $g$ if and only if $g|_s\in C_s$.

By replacing $[n]$ with an arbitrary set
$M$, we can study lifted codes inside $\Sigma^M$, rather than   $\Sigma^n$.

\subsection{Agreement Testability}
\label{subsec:agreement}

Informally, an \emph{agreement expander} consists of a collection $S$ of subsets
 of $[n]$ such that for any finite set $\Sigma$
and any ensemble of functions $\{f_s:s\to \Sigma\}_{s\in S}$
such that $f_s|_{s\cap s'}=f_{s'}|_{s\cap s'}$ for almost all
$s,s'\in S$, there is $g:[n]\to \Sigma$ such that $g|_s=f_s$
for almost all $s\in S$; see
\cite{Dinur_2017_hdx_imply_agreement_exp}
and 
\cite{Dikstein_2019_agreement_testing_two_layered_sys}.
For this work, we need to   consider a refinement of this notion
where each $f_s$ is required to be in a code $C_s\subseteq \Sigma^s$
and the globally defined function $g:[n]\to \Sigma$ is  required to be in the associated
lifted code $C=C(\{C_s\}_{s\in S})\subseteq \Sigma^n$.
This notion already appeared in \cite[Dfn.~2.8]{Dinur_2022_ltcs_const_rate}
in the special case  of tensor codes, realized as lifted codes 
as in Example~\ref{EX:agreement-test-for-tensor-codes} below,
and 
under the name \emph{agreement testability}, which use as well.

The formal definition of agreement testability
requires us to state which pairs $(s,s')\in S\times S$
are considered and in what probability. It is convenient to encode
this information in a normalized weighted graph whose vertices are in bijection with
$S$ and whose edges are   labelled by subsets of $[n]$.

\begin{dfn}[Agreement Tester]
	\label{DF:agreement-test}
	Let $C=C(\{C_s\}_{s\in S})\subseteq\Sigma^n$ be a lifted code
	as in \S\ref{subsec:tanner-codes}.
	An agreement tester for the lifted code $C$ consists
	of a normalized weighted graph $(G,w)$ and a   function  $\ell:G\to P([n])$ assigning
	every vertex and edge a subset of $[n]$ such that the following hold:
	\begin{enumerate}[label=(\arabic*)]
		\item $\ell$ restricts to a bijection between $G(0)$ and $S$;
		\item for every edge $e\in G(1)$ and $u\in e(0)$,
		we have $\ell(e)\subseteq \ell(u) $.
	\end{enumerate}
	In this case, we also say that $(\{C_s\}_{s\in S},G,w,\ell)$ is an \emph{agreement tester}.
	This agreement tester is said to have \emph{soundness} $\kappa\geq 0$
	if for every ensemble $ (f_s)_{s\in S}\in\prod_{s\in S}C_s$,
	there is $g\in C$ such that
	\[
	\kappa \cdot w(\{v\in G(0)\suchthat g|_{\ell(v)}\neq f_{\ell(v)}\})
	\leq  w(\{e=\{u,v\}\in G(1)\suchthat f_{\ell(u)}|_{\ell(e)}\neq f_{\ell(v)}|_{\ell(e)}\}).
	\]
	We will also say that $(\{C_s\}_{s\in S})$
	is \emph{$\kappa$-agreement testable} w.r.t.\ the labelled weighted graph $(G,w,\ell)$.
\end{dfn}

Agreement expanders with larger soundness are considered better.

%%w.r.t.\ evident distance functions induced by $w$)

\begin{example} 
Any lifted code   $\{C_s\}_{s\in S}$ can be naively 
enriched into an agreement tester
as follows.
Construct $G$ by taking the 
vertex set to be  $S$, and then connect a pair $s,s'\in S$
by an edge if $s\cap s'\neq \emptyset $, or more generally, if $s\cap s'$ has some desired cardinality.
The labelling $\ell$ then maps every $s\in G(0)$  to itself
and every edge $\{s,s'\}$ to $s\cap s'$.
The weight function $w$ can be taken to be the uniform one, for instance.
\end{example}

\begin{example}[Agreement Testability of Tensor Codes]
	\label{EX:agreement-test-for-tensor-codes}
	Let $\F$ be a finite field.
	Let $C_1\subseteq \F^{[n_1]}$ 
	and $C_2\subseteq \F^{[n_2]}$ be linear codes.
	The tensor code $C_1\otimes C_2$ (all tensors are over $\F$) is
	the code $C\subseteq \nMat{\F}{n_1\times n_2}=\F^{[n_1]\times [n_2]}$
	consisting of the matrices $m\in \nMat{\F}{n_1\times n_2}$
	such that every row of $m$ lies in $C_2$ and every column of $m$ lies in $C_1$.
	In  \cite[Dfn.~2.8]{Dinur_2022_ltcs_const_rate},
	the tensor code $C=C_1\otimes C_2\subseteq \F^{[n_1]\times [n_2]}$
	is said to be a \emph{$\kappa$-agreement testable}\footnote{
		Actually, what we define  here is   	$\frac{\kappa}{2}$-agreement testability
		in the setting of \cite[Dfn.~2.8]{Dinur_2022_ltcs_const_rate}.
	} if for every
	choice of codewords $\{f_j\in C_1\}_{j\in [n_2]}$
	and $\{f'_i\in C_2\}_{i\in [n_1]}$, there is a matrix $m\in C$ such that
	\[
	\kappa\cdot \Squares{\frac{\#\{i\in [n_1]\suchthat f'_i\neq r_i(m)\}}{2n_1}
	+\frac{\#\{j\in [n_2]\suchthat f'_j\neq c_j(m)\}}{2n_2}}
	\leq 
	\frac{\#\{(i,j)\in [n_1]\times [n_2]\suchthat f_{i,j}\neq f'_{j,i}\}}{n_1n_2}.
	\]
	Here, $r_i(m)$ is the $i$-th row of $m$ and $c_j(m)$ is the $j$-th column of $m$.
	Informally, this means that if the matrix whose rows are the $\{f'_i\}_i$
	and the matrix whose columns are the $\{f_j\}_{j}$ agree in almost all entries,
	then some matrix in $C_1\otimes C_2$ agrees almost everywhere with both of these matrices.

	We can recover  $\kappa$-agreement testablity for $C=C_1\otimes C_2$ 
	as a special case of Definition~\ref{DF:agreement-test}.
	First we realize $C=C_1\otimes C_2\subseteq \F^{[n_1]\times [n_2]}$ as a lifted code
	by taking $S=\{s_1,\dots,s_{n_1},s'_1,\cdots,s'_{n_2}\}$,
	where $s_i=\{i\}\times [n_2]$ and $s'_j=[n_1]\times \{j\}$,
	and putting $C_{s_i}=\{(f_{j})_{(i,j)\in \{i\}\times [n_2]}\where f\in C_2\}\subseteq \F^{\{i\}\times [n_2]}$ 
	and $C_{s'_j}=\{(f_{i})_{(i,j)\in  [n_1] \times \{j\}}\where f\in C_1\}\subseteq \F^{[n_1] \times \{j\}}$ for all $i$ and $j$.
	Now choose the graph $G$ to be the complete biparatite graph on $\{s_1,\dots,s_{n_1}\}$
	and $\{s'_1,\cdots,s'_{n_2}\}$ endowed with its natural weight function
	(Example~\ref{EX:graph-nat-weight}).
	The labelling $\ell$ maps every vertex to itself, and every edge
	$\{s_i,s'_j\}$ to $s_i\cap s'_j=\{(i,j)\}$.
	It is routine (and a recommend exercise for newcomers) to check
	that $(\{C_s\}_{s\in S},G,\ell)$ has soundness $\kappa$
	if and only if $C_1\otimes C_2$ is $\kappa$-agreement testable.
\end{example}

\section{Lifted Codes and Their Line Codes}
\label{sec:lifted-vs-line}

In this section, we recall  the construction of the (so-called) line code of a lifted
code. We then  establish relations between the  rate, distance, testability and decodability
of these codes. 
The   results
of this section will be important to some applications   of our main result.

\medskip

Let $C\subseteq\Sigma^n$ be a lifted code
determined by small codes $\{C_s\}_{s\in S}$ (\S\ref{subsec:tanner-codes}).
Suppose moreover that all the small codes $C_s$ have the same cardinality $\sigma$ and
choose a set $\Sigma'$ of that cardinality.
The \emph{line code}
of   $C=C(\{C_s\}_{s\in S})$ 
is a code $L=L(\{C_s\}_{s\in S})\subseteq \Sigma'^S$ with alphabet $\Sigma'$  constructed as follows:
For every $s\in S$, choose a bijection $C_s\cong \Sigma'$.
We use these bijections
to freely  identify $\prod_{s\in S}C_s$ with $\Sigma'^S$. We then 
define 
$L\subseteq \Sigma'^S$
to be the code  consisting of the (words in $\Sigma'^S$ corresponding to) 
ensembles $f=(f_s)_{s\in S}\in \prod_{s\in S} C_s$
satisfying $f_s|_{s\cap s'}=f_{s'}|_{s\cap s'}$ for all $s,s'\in S$.
That is,
\[
L=L(\{C_s\}_{s\in S}) =\{(f_s)_{s\in S}\in \prod_{s\in S} C_s\suchthat
\text{$f_s|_{s\cap s'}=f_{s'}|_{s\cap s'}$ for all $s,s'\in S$}\}.
\]
Since the sets in $S$ cover $[n]=\{1,\dots,n\}$,
we have a bijection $C\to L$ given by $g\mapsto (g|_s)_{s\in S}$.

\begin{remark}\label{RM:line-code-generalized-code}
	By allowing codes with a varying alphabet, see Remark~\ref{RM:generalized-codes},
	we may   define
	the line code of $C(\{C_s\}_{s\in S})$ even when the $C_s$
	have varying cardinalities --- 
	simply take $C_s$ to be the alphabet at the $s$-coordinate and define
	$L$ as a subset of $\prod_{s\in S}C_s$.	
	With the exception of Proposition~\ref{PR:line-code-rate-and-dist}(i), 
	the results of  this section 
	can be adapted in a straightforward manner
	to this more general setting. 
\end{remark}

\begin{example}[Line Codes of Reed--M\"uller Codes]
	Let $\F$ be a finite field of cardinality $q$ and characteristic $p$ 
	and let $d$ and $n$ be non-negative integers.
	Let $V_n$ denote an $n$-dimensional $\F$-vector space, e.g., $\F^n$.
	Recall that the  Reed--M\"uller code of degree-$d$ functions on $V_n$ is
	the set $C=\mathrm{RM}(n,d,q)\subseteq \F^{V_n}$ of functions $g:V_n\to \F$ having degree at most $d$.
	It is known \cite[Thm.~2]{Kaufman_2006_testing_poly_general_fields} that when $d\leq  q(1-\frac{1}{p})-1$,
	a function $g:V_n\to \F$ has degree $d$ or less if and only if its restriction to
	every $1$-dimensional affine subspace of $V_n$ --- called a \emph{line} for short ---
	is also of degree $d$ or less.
	Assuming this holds,
	we can   describe $\mathrm{RM}(n,d,q)$ as a lifted code $C(\{C_s\}_{s\in S})$: Let $S$
	be the set of lines in $V_n$, and for every $s\in S$, let $C_s$ be the
	Reed--M\"uller code of degree-$d$ functions on the line $s$.
	By identifying each line $s$ in $V_n$ with $V_1\cong\F$,
	we can identify each $C_s$ with the Reed--M\"uller code 
	$\mathrm{RM}(1,d,q)$.
	Thus, the line code $L$ of $\mathrm{RM}(n,d,q)$,
	realized as a lifted code  as just explained,
	has alphabet $\Sigma'=\mathrm{RM}(1,d,q)$ 
	and it consists of the ensembles $(f_s)_{s\in S}\in \mathrm{RM}(1,d,q)^S$
	such that $f_s|_{s\cap s'}=f_{s'}|_{s\cap s'}$ for any two lines $s,s'$ in $V_n$.
	This   example  is the reason why $L(\{C_s\}_{s\in S})$ is called
	a   line code  in general.
\end{example}

\begin{notation}\label{NT:line-codes}
For the remainder of this section, 
fix a lifted code $C=C(\{C_s\}_{s\in S})\subseteq\Sigma^n$ with block length $n$
and alphabet $\Sigma$. Suppose moreover that each $C_s$ is identified with another alphabet
$\Sigma'$, and let $L=L(\{C_s\}_{s\in S})\subseteq \Sigma'^S$ be the the associated line code.
We will make repeated use of 
the following quantities associated to the family
$S\subseteq P([n])$:
\begin{itemize}
	\item $k_{\min}=\min_{s\in S}|s|$,
	\item $k_{\max}=\max_{s\in S}|s|$,
	\item $D_{\min}=\min_{i\in [n]}\#\{s\in S\suchthat i\in s\}$,
	\item $D_{\max}=\max_{i\in [n]}\#\{s\in S\suchthat i\in s\}$.
\end{itemize} 
Thus, every $s\in S$ contains between $k_{\min}$ and $k_{\max}$
elements, and every   $i\in [n]$ is contained in at least $D_{\min}$
and at most $D_{\max}$ sets from $S$. We also let $\tilde{k}$ denote the minimum
distance of a code $C_s\subseteq \Sigma^s$ with $s\in S$, that is,
\begin{itemize}
	\item $\tilde{k}=\min_{s\in S} (\delta(C_s)|s|)$; note that $\tilde{k}\in \{1,2,\dots,k_{\max}\}$.
\end{itemize}
We encourage the reader think of $k_{\min},k_{\max},D_{\min},D_{\max}$
as being $\Theta(1)$ as $n$ grows. This implies that $|S|=\Theta(n)$,
as the following lemma shows.
\end{notation}

\begin{lem}\label{LM:size-of-S}
	With notation as above, 
	$\frac{D_{\min}}{k_{\max}}n\leq|S|\leq \frac{D_{\max}}{k_{\min}} n$.
\end{lem}

\begin{proof}
	The right inequality
	holds because
	$
	k_{\min} |S|\leq \#\{(i,s)\in [n]\times S\suchthat i\in s\}\leq
	n D_{\max}$. The left inequality is shown similarly.
\end{proof}

The following lemma will   be used   several times.

\begin{lem}\label{LM:line-dist-vs-lifted-dist}
	Using Notation~\ref{NT:line-codes},
	let $g_0\in C$ correspond to   $f_0=(g_0|_s)_{s\in S}\in L$.
	Let $g\in \Sigma^n$, put $S_g=\{s\in S\suchthat g|_s\in C_s\}$ 
	and define $f\in \prod_{s\in S}C_s$
	by letting $f_s=g|_s$ if   $g|_s\in C_s$
	and otherwise  choosing $f_s\in C_s$ arbitrarily.
	Then
	\[
	\frac{D_{\min} k_{\min}}{D_{\max}k_{\max}}\dist(g,g_0) - \frac{|S_g|}{|S|}
	\leq
	\dist(f,f_0)\leq \frac{D_{\max} k_{\max}}{D_{\min} }\dist(g,g_0).
	\]
	If we moreover have $g\in C$, then
	\[
	\frac{D_{\min} k_{\min}}{D_{\max}k_{\max}}\dist(g,g_0)  
	\leq
	\dist(f,f_0)\leq \frac{D_{\max} k_{\max}}{D_{\min}\tilde{k}}\dist(g,g_0).
	\]
\end{lem}

\begin{proof}
	Write $A=\{i\in[n]\suchthat g_i\neq g_{0,i}\}$,
	$B=\{s\in S\suchthat g|_s\neq g_{0}|_s\}$
	and $I=\{(i,s)\in A\times B\suchthat i\in s\}$.
	Every $i\in A$
	has between $D_{\min}$ and $D_{\max}$ preimages under 
	the first projection $\mathrm{pr}_1:I\to A$,
	so $\frac{1}{D_{\max}}|I|\leq |A|\leq \frac{1}{D_{\min}}|I|$,
	and every $s\in B$ has between $1$ and $k_{\max}$ preimages under
	the second projection $\mathrm{pr}_2:I\to B$,
	hence $\frac{1}{k_{\max}}|I|\leq |B|\leq |I|$.
	Together, both inequalities imply that
	\[
	\frac{D_{\min}}{k_{\max}}|A| \leq |B|\leq  D_{\max}  |A|.
	\]

	Observe that if $f_s\neq f_{0,s}$, then we must have $g|_s\neq g_{0}|_s$
	(otherwise $g|_s\in C_s$ and then $f_s=g|_s=g_0|_s=f_{0,s}$).
	Thus,
	\begin{align*}
	\dist(f,f_0)
	=\frac{\#\{s\in S\suchthat f_s\neq f_{0,s}\}}{|S|}
	\leq \frac{|B|}{|S|}	
	\leq
	\frac{D_{\max}|A|}{\frac{D_{\min}}{k_{\max}} n}
	=\frac{D_{\max} k_{\max}}{D_{\min}}\dist(g,g_0),
	\end{align*}
	where in the second inequality we used Lemma~\ref{LM:size-of-S}.
	Next, observe that if $g|_s\neq g_{0}|_s$ and $s\notin S_g$,
	then $f_s\neq f_{0,s}$.
	(Indeed,   $g|_s=f_s$  because $s\notin S_g$, 
	so $f_s=g|_s\neq g_0|_s=f_{0,s}$.)
	Thus,
	\begin{align*}
	\dist(f,f_0)
	&=\frac{\#\{s\in S\suchthat f_s\neq f_{0,s}\}}{|S|}
	\geq 
	\frac{|B|-|S_g|}{|S|}
	\geq
	\frac{\frac{D_{\min}}{k_{\max}}|A|}{\frac{D_{\max}}{k_{\min}} n}
	-\frac{|S_g|}{|S|}
	=\frac{D_{\min} k_{\min}}{D_{\max}k_{\max}}\dist(g,g_0) -\frac{|S_g|}{|S|}.
	\end{align*}
	This proves the first assertion of the lemma.

	Suppose now that $g\in C$. 
	Since the distance of every $C_s$ is at least
	$\tilde{k}$, 
	every $s\in B$ has between $\tilde{k}$ and $k_{\max}$ preimages under
	the second projection $\mathrm{pr}_2:I\to B$, meaning that 
	$\frac{1}{k_{\max}}|I|\leq |B|\leq \frac{1}{\tilde{k}}|I|$.
	By arguing as above using this stronger inequality,
	we get that  $|B|\leq \frac{D_{\max}}{\tilde{k}} |A|$
	and
	$\dist(f,f_0)\leq \frac{D_{\max} k_{\max}}{D_{\min}\tilde{k}}\dist(g,g_0)$.
	That $\frac{D_{\min} k_{\min}}{D_{\max}k_{\max}}\dist(g,g_0)  
	\leq
	\dist(f,f_0)$ follows from the previous paragraph, because $S_g=\emptyset$
	when $g\in C$.
\end{proof}

The following proposition relates the rate and distance of $C$
and $L$.

\begin{prp}\label{PR:line-code-rate-and-dist}
	Using Notation~\ref{NT:line-codes}, we have:
	\begin{enumerate}[label=(\roman*)]
		\item 
		$r(C)= \frac{\gamma \log|\Sigma'|}{ \log|\Sigma|} r(L)$,
		where $\gamma:=\frac{|S|}{n}\in [\frac{D_{\min}}{k_{\max}},\frac{D_{\max}}{k_{\min}}]$.
		\item $\frac{D_{\min}\tilde{k}}{D_{\max} k_{\max}} \delta(L) \leq \delta(C)\leq 
	\frac{D_{\max}k_{\max}}{D_{\min} k_{\min}} \delta(L)$.
	\end{enumerate}	 
\end{prp}

\begin{proof}
	(i) Recall that we have a bijection $g\mapsto (g|_s)_{s\in S}:C\to L$.
	Thus,
	\begin{align*}
		r(L)=\frac{\log |L|}{|S|\log |\Sigma'|}
		=\frac{\log |C|}{\gamma n \log |\Sigma'|}=\frac{ \log|\Sigma|}{\gamma \log|\Sigma'|}r(C).
	\end{align*}
	That $\gamma\in [\frac{D_{\min}}{k_{\max}},\frac{D_{\max}}{k_{\min}}]$ follows from Lemma~\ref{LM:size-of-S}.	
	
	(ii) This follows readily from the last assertion of Lemma~\ref{LM:line-dist-vs-lifted-dist}.	
\end{proof}

The next proposition says that
if the line code $L$ has a   decoding algorithm, then $C$
also has a decoding algorithm of a similar complexity. 
We do not know whether the converse holds in general,
but a partial converse will be given in Theorem~\ref{TH:lifted-to-line-test}(ii)
below.

\begin{prp}\label{PR:line-code-eff-decoding}
	Using Notation~\ref{NT:line-codes},
	let $\eta\in [0,\frac{1}{2}\delta(L)]$ and
	suppose that  the line code $L$ has a decoding algorithm  for words
	that are $\eta$-close to $L$.  Then $C$ has a decoding algorithm for
	words that are $\frac{D_{\min} }{D_{\max}k_{\max}}\eta$-close to  $C$. 
	Provided that $k_{\max}=O(1)$ (as a function of $n$), its time complexity
	is $O(n+|S|)$ plus the   time complexity of the decoding algorithm for $L$.
\end{prp}
	
\begin{proof}
	Consider the following   algorithm,
	which takes $g\in \Sigma^n$ and outputs $g'\in C$.
	\begin{enumerate}[label=(\arabic*)]
		\item For every $s\in S$: If $g|_s\in C_s$, set $f_s=g|_s$; otherwise,
		let $f_s$ be some element of $C_s$.
		\item Apply the decoding algorithm of $L$ to $f=(f_s)_{s\in S}\in \prod_{s\in S}C_s$.
		Let $f'$ be the output.
		\item The ensemble $f'=(f'_s)_{s\in S}\in L$
		determines an element $g'\in C$. Output $g'$.
	\end{enumerate}	 
	We claim that this  algorithm has the required properties.
	
	The time complexity is clearly the one stated in the proposition. 
	
	Suppose now that the input $g$ of the algorithm satisfies  $\dist(g,C)<
	\frac{D_{\min}  }{D_{\max}k_{\max}}\eta$ and choose $g_0\in C$ such that $\dist(g,g_0)=\dist(g,C)$.
	We need to show that the output $g'$ of the algorithm is $g_0$.
	Let $f_0\in L$ correspond to $g_0$.
	By Lemma~\ref{LM:line-dist-vs-lifted-dist},
	\begin{align*}
	\dist(f,f_0)
	\leq
	\frac{D_{\max}k_{\max}}{D_{\min} }\dist(g,g_0)<\eta.
	\end{align*}
	Thus, applying the decoding algorithm of $L$ to $f$ returns $f_0$.
	Consequently $f'=f_0$ and the algorithm outputs $g_0$, as required.
\end{proof}

We now turn to consider the testability
of the line code $L$ by a $2$-query tester.
To that end, let $G$ be a graph equipped
with a labelling $\ell:G\to P([n])$
such that $\ell$ restricts to a bijection $G(0)\to S$
and $\ell(v)\supseteq \ell(e)$ for every $e\in G(1)$ and $v\in e(0)$.

\begin{example}\label{EX:intersection-graph}
	We can take $G$ to be the \emph{intersection graph}
	of $S$: The vertex set of $G$ is $S$
	and   $s,s'\in S=G(0)$ are connected by an edge precisely when $s\cap s'\neq\emptyset$.
	The labelling $\ell:G\to P([n])$ then maps every $s\in G(0)$ to itself
	and every edge $\{s,s'\}$ to $s\cap s'$. 
\end{example}

Having fixed a labelled graph $(G,\ell)$ as above,
we define a $2$-query tester $T_{G,\ell}$ for  
$L\subseteq \Sigma'^S$ as follows: Given $f=(f_s)_{s\in S}\in \Sigma'^S$,
choose an edge $e=\{u,v\}$ in $G$ uniformly at random, probe $f_{\ell(u)}$ and $f_{\ell(v)}$,
and accept $f$ if and only if
$f_{\ell(u)}|_{\ell(e)}=f_{\ell(v)}|_{\ell(e)}$.

\begin{remark}\label{RM:agreement-is-testability-for-line}
	Give $G$ the uniform weight function $w_{\uni}$.
	Then $(\{C_s\}_{s\in S},G,w_{\uni},\ell)$ is an agreement tester 
	(\S\ref{subsec:agreement})
	and it has soundness $\mu\geq 0$ if and only if  the tester $T_{G,\ell}$
	for the code $L\subseteq \Sigma'^S$ has soundness $\mu$.
	This continues to hold if we give $G$ any normalized weight function
	$w$ which is uniform on $G(0)$, provided that  in $T_{G,\ell}$ we choose $e\in G(1)$
	according to $w$ (rather than uniformly).
\end{remark}

We now show that  
if the tester $T_{G,\ell}$ of $L$   has   soundness $\mu$,
then  the natural tester of the lifted code $C=C(\{C_s\})_{s\in S}$ has soundness $\Omega(\mu)$. 
We will   apply this key observation to
some  particular lifted codes and their line codes later on.

\begin{thm}\label{TH:line-code-testability}
	Keep Notation~\ref{NT:line-codes}, let $(G,\ell)$
	be a labelled graph as above,
	and suppose that
	every vertex in $G$ belongs to at least $d_{\min}$   and at most $d_{\max}$
	edges.	
	If the tester $T_{G,\ell}$ for $L\subseteq\Sigma'^S$
	has soundness   $\mu$ ($\mu\geq 0$),
	then the natural tester   of $C$ has soundness  
	$\frac{k_{\min}D_{\min}}{k_{\max}D_{\max}}
	\cdot \frac{\mu  }{\mu  + 2d_{\max}d_{\min}^{-1} }$.
\end{thm}

%%be an expander  graph.

\begin{proof}
	We give $G$ the uniform weight function $w:=w_{\uni}$ and identify
	$G(0)$ with $S$ via $\ell$.
	
	Let $g\in \Sigma^n$ and $S_g=\{s\in S\suchthat g|_s\in C_s\}$.
	The probability that the natural tester of $C$ rejects $g$
	is $\frac{|S_g|}{|S|}$, so we need to show
	that
	\begin{align}\label{EQ:line-to-lifted-conclusion}
	\frac{|S_g|}{|S|}
	\geq
	\frac{k_{\min}D_{\min}}{k_{\max}D_{\max}}
	\cdot \frac{\mu  }{\mu  + 2d_{\max}d_{\min}^{-1} }\cdot\dist(g,C).
	\end{align}
	
	Define $f$ as in Lemma~\ref{LM:line-dist-vs-lifted-dist},
	choose $f_0\in L$ such that $\dist(f,f_0)=\dist(f,L)$
	and let $g_0\in C$ be the codeword corresponding to $f_0$.
	By Lemma~\ref{LM:line-dist-vs-lifted-dist}, we have
	\begin{align}\label{EQ:bound-on-d-g-g-zero}
	\dist(g,g_0)\leq \frac{D_{\max} k_{\max}}{D_{\min}k_{\min}}\Circs{\frac{|S_g|}{|S|}+\dist(f,f_0)}
	=
	\frac{D_{\max} k_{\max}}{D_{\min}k_{\min}}\Circs{\frac{|S_g|}{|S|}+\dist(f,L)} .
	\end{align}
	Next, observe that the probability that $T_{G,\ell}$ rejects $f$ is 
	at most 
	\[
	w(\bigcup_{s\in S_g} G(1)_s) 
	\leq \sum_{s\in S_g} \frac{|G(1)_s|}{|G(1)|}
	\leq \sum_{s\in S_g} \frac{  d_{\max}}{\frac{1}{2}d_{\min}|S|} =
	\frac{2 d_{\max}}{d_{\min}} \frac{|S_g|}{|S|}.
	\]
	(Recall that we identified $G(0)$ with $S$ and that $G(1)_s$ is
	the set of edges having $s$ as a vertex.)
	Since $T_{G,\ell}$ has soundness $\mu$, it follows that
	\[
	\dist(f,L)\leq \frac{2   d_{\max}}{\mu d_{\min}} \frac{|S_g|}{|S|}.
	\]
	Plugging this into \eqref{EQ:bound-on-d-g-g-zero} gives
	\[
	\dist(g,C)\leq \dist(g,g_0)
	\leq 
	\frac{D_{\max} k_{\max}}{D_{\min}k_{\min}}\Circs{1+ \frac{2   d_{\max}}{\mu d_{\min}}}\frac{|S_g|}{|S|}.
	\]
	Rearranging gives the desired conclusion \eqref{EQ:line-to-lifted-conclusion}.
\end{proof}

We finish this section with giving a   converse
to Theorem~\ref{TH:line-code-testability} when
$G$ is the intersection graph of $S$,
as well as a   partial converse to Proposition~\ref{PR:line-code-eff-decoding}. 
This will not be needed in the sequel.
We do not know if
there is a converse to Theorem~\ref{TH:line-code-testability} which holds for a general
graph $G$.

\begin{thm}\label{TH:lifted-to-line-test}
	Keep Notation~\ref{NT:line-codes}, let $(G,\ell)$
	be the intersection graph of the set $S$
	(Example~\ref{EX:intersection-graph}) 
	and let $d_{\max}$ (resp.\ $d_{\min}$) denote the maximal 
	(resp.\ minimal) degree of a vertex in $G$. 
	Suppose further that the natural tester of the lifted
	code $C$ has a soundness $\mu\geq 0$.
	Then: 
	\begin{enumerate}[label=(\roman*)]
		\item The tester $T_{G,\ell}$ for $L\subseteq \Sigma'^S$
		has soundness
		$\frac{D_{\min}k_{\min}\mu }{d_{\max} D^2_{\max}k^2_{\max}  (\mu+D_{\min}^{-1} D_{\max}   k_{\max})} $.
		
		\item If $C$ has a decoding algorithm for words
		that are $\eta$-close to $C$ ($\eta\in [0,\frac{1}{2}\delta(C)]$),
		then $L$ has a decoding algorithm for words that
		are $\eta'$-close to $L$,
		where
		\[
		\eta'=
		\min\left\{
			\frac{ d_{\min} D_{\min} k_{\min} \mu }{2d_{\max}^2 D^2_{\max}k^2_{\max}  }\cdot \eta,
			\left(\frac{2d_{\max}^2 D^2_{\max}k^2_{\max}  (\mu+D_{\min}^{-1} D_{\max}  k_{\max})}{
	d_{\min} D_{\min}k_{\min}\mu }+1\right)^{-1} \delta(L) \right\}.\] 
		Provided that $k_{\max}=O(1)$, its time complexity
		is $O(|G(1)|+|S|+n)$ plus the time complexity of the decoding algorithm of $C$.
	\end{enumerate}
\end{thm}

\begin{proof}
	Again, we identify $G(0)$ with $S$ via the labelling $\ell$.
	We also write $E=G(1)$.
	
	(i) Let $f\in \prod_{s\in S}C_s$  and let $E_f=\{e=\{s,s'\}\in E\suchthat f_s|_{s\cap s'}=f_{s'}|_{s\cap s'}\}$.
	We need to show
	that 
	\[\frac{|E_f|}{|E|}\geq \frac{D_{\min} k_{\min} \mu }{ d_{\max} D^2_{\max} k^2_{\max} (\mu+D_{\min}^{-1} D_{\max} k_{\max})}   \dist(f,L).\]
	
	Let $S_f=\bigcup_{e\in E_f}e(0)$ ($e(0)$ is the set of vertices of $e$). Then 
	\[
	\frac{|S_f|}{|S|}\leq \frac{2|E_f|}{|S|}=
	\frac{2|E|}{|S|}\cdot \frac{ |E_f|}{|E|}  \leq 
	d_{\max} \frac{|E_f|}{|E|}.
	\]
	Next, put $M=\bigcup_{s\in S_f} s$ (so that $M\subseteq [n]$). Then
	\[
	\frac{|M|}{n} 
	\leq \frac{k_{\max} |S_f|}{n}
	=k_{\max} \frac{|S_f|}{|S|}\frac{|S|}{n}
	\leq 
	\frac{ d_{\max} D_{\max}k_{\max}}{k_{\min}}
	\frac{|E_f|}{|E|},
	\]
	where in the second inequality we used Lemma~\ref{LM:size-of-S}.
	
	Let $i\in [n]-M$  and suppose that $s,s'\in S$
	satisfy $i\in s\cap s'$. Then $\{s,s'\}\in E-E_f$ (here we need
	$G$ to be the intersection graph of $S$), and thus $(f_s)_i=(f_{s'})_i$.
	This allows us to define $g\in \Sigma^n$ as follows:
	For $i\in [n]-M$, choose some $s\in S$ with $i\in s$
	and define $g_i=(f_s)_i$; this is independent of $s$ by what we just showed.
	For $i\in M$, choose $g_i\in \Sigma$ arbitrarily.
	
	Let $S_g=\{s\in S\suchthat g|_s\notin C_s\}$. Observe that
	every $s\in S$ with $s\cap M=\emptyset$ satisfies
	$s\notin S_g$, because $g|_s=f_s\in C_s$. Otherwise
	stated, $S_g \subseteq \{s\in S\suchthat s\cap M\neq\emptyset\}$.
	Thus, 
	\begin{equation}\label{EQ:S-g-to-S}
	\frac{|S_g|}{|S|}\leq \frac{D_{\max}|M|}{|S|}
	=D_{\max}\frac{|M|}{n}\frac{n}{|S|}
	\leq
	\frac{d_{\max} D^2_{\max}k^2_{\max} }{D_{\min} k_{\min}}
	\frac{|E_f|}{|E|}  
	\end{equation}
	(we used Lemma~\ref{LM:size-of-S} again).
	The number $\frac{|S_g|}{|S|}$ is   also the probability
	that the natural tester of $C$ rejects $g$. Thus,
	there exists $g_0\in C$ such that
	\[
	\dist(g,g_0)\leq \mu^{-1}\cdot \frac{|S_g|}{|S|}
	\leq 
	\frac{d_{\max} D^2_{\max}k^2_{\max} }{ D_{\min} k_{\min} \mu}
	\frac{|E_f|}{|E|}.  
	\]
	Let $f_0\in L$ denote the codeword corresponding to $g_0$.
	
	Define $f'\in\Sigma'^S$ by letting $f'_s=g|_s$ if $g|_s\in C_s$,
	and choosing $f'_s $ arbitrarily otherwise.
	By Lemma~\ref{LM:line-dist-vs-lifted-dist},
	\[
	\dist(f',f_0)\leq \frac{D_{\max} k_{\max}}{D_{\min}  }\dist(g,g_0)
	\leq \frac{D_{\max} k_{\max}}{\mu D_{\min}  } 
	\frac{ d_{\max} D^2_{\max}k^2_{\max}}{ D_{\min} k_{\min}}
	\frac{|E_f|}{|E|}.
	\]
	Note also that if  $s\in S$ satisfies
	$s\cap M=\emptyset$, then
	$g|_s=f_s\in C_s$, so $f_s=f'_s$.
	This means that $\{s\in S\suchthat f_s\neq f'_s\}\subseteq\{s\in S\suchthat
	s\cap M\neq\emptyset\}$,
	and together with \eqref{EQ:S-g-to-S}, we get
	\[
	\dist(f,f')\leq    \frac{D_{\max}|M|}{|S|}
	\leq \frac{d_{\max} D^2_{\max}k^2_{\max}  }{D_{\min} k_{\min}}
	\frac{|E_f|}{|E|}.
	\]
	It follows that
	\[
	\dist(f,L)\leq \dist(f,f_0)\leq \dist(f,f')+\dist(f',f_0)
	\leq 
	\frac{ d_{\max} D^2_{\max} k^2_{\max} }{ D_{\min} k_{\min}}
	\Circs{1+\frac{D_{\max} k_{\max}}{\mu D_{\min}   } }
	\frac{|E_f|}{|E|}.
	\]
	Rearranging gives the desired conclusion.
	
	(ii) Consider the following algorithm taking $f\in \prod_{s\in S}C_s$
	and outputing $f_0\in L$:
	\begin{enumerate}[label=(\arabic*)]
		\item Define $E_f$, $M$, $g$ as in the proof of (i).
		\item Apply the decoding algorithm of $C$ to $g$; let $g_0$ denote the output.
		\item Return the codeword $f_0\in L$ corresponding to $g_0$.
	\end{enumerate}
	It clearly has the time complexity claimed in the theorem.
	It remains to show that it decodes $f$ if $\dist(f,L)<\eta'$.
	
	Let $f_1\in L$ such that $\dist(f,f_1)=\dist(f,L)<\eta'$.
	Let $T=\{s\in S\suchthat f_s\neq f_{1,s}\}$. Then $|T|<\eta'|S|$.
	Note that any $e\in \{s,s'\}\in E$ with $s,s'\notin T$ does not
	lie in $E_f$ because $f_s|_{s\cap s'}=f_{1,s}|_{s\cap s'}
	=f_{1,s'}|_{s\cap s'}=f_{s'}|_{s\cap s'}$. It follows that
	every $e\in E_f$ has a vertex in $T$.
	Thus,  
	\[
	\frac{|E_f|}{|E|}
	\leq
	\frac{|T|d_{\max}}{|E|}
	=d_{\max} \frac{|S|}{|E|} \frac{|T|}{|S|}
	<\frac{2d_{\max}\eta'}{d_{\min}}.
	\]
	As shown in the proof of (i), there is $g'\in C$ with 
	\[\dist(g,g')\leq 
	\frac{d_{\max} D^2_{\max}k^2_{\max} }{ D_{\min} k_{\min} \mu}
	\frac{|E_f|}{|E|}
	<
	\frac{2d_{\max}^2 D^2_{\max}k^2_{\max} \eta'}{ d_{\min} D_{\min} k_{\min} \mu}
	\leq \eta.
	\]
	This means that the decoding algorithm of $C$ will work for $g$
	and returns $g'$, i.e., $g_0=g'$.
	We also observed in the proof of (i) that
	\begin{align*}
	\dist(f,f_0)
	&\leq 
	\frac{d_{\max} D^2_{\max}k^2_{\max}  (\mu+D_{\min}^{-1} D_{\max}  k_{\max})}{D_{\min}k_{\min}\mu }
	\frac{|E_f|}{|E|}
	\\
	&<
	\frac{2d_{\max}^2 D^2_{\max}k^2_{\max}  (\mu+D_{\min}^{-1} D_{\max}  k_{\max})}{
	d_{\min} D_{\min}k_{\min}\mu }\eta'\leq \delta(L)-\eta'.
	\end{align*}
	Thus, $\dist(f_0,f_1)<\eta'+\delta(L)-\eta' = \delta(L)$,
	so $f_0=f_1$, and the algorithm returns $f_1$.
\end{proof}

Suppose that $D_{\max}$ and $k_{\max}$   remain
bounded as $n$ grows.
By putting together 
Remark~\ref{RM:agreement-is-testability-for-line},
Theorem~\ref{TH:line-code-testability} and
Theorem~\ref{TH:lifted-to-line-test},
one sees that  
a lifted code $C:=(\{C_s\}_{s\in S})\subseteq\Sigma^n$
is locally testable w.r.t.\ its natural tester if and only
if it can be enriched into a uniformly weighted 
agreement tester $(\{C_s\}_{s\in S},G,w_{\uni},\ell)$
in  \emph{some}  way.
That is, as long as $D_{\max}$ and $k_{\max}$ are $O(1)$,
\emph{agreement testability and local testability
for a lifted code are essentially the same}. A precise formulation
of this is the following corollary.

Call an agreement tester $(\{C_s\}_{s\in S},G,w ,\ell)$ \emph{non-redundant}
if $G$ is a simple graph and $\ell(e)\neq \emptyset$ for all $e\in G(1)$.

\begin{cor}\label{CR:agreement-equiv-loc-test}
	Fix $D,k\in\N$ and, with Notation~\ref{NT:line-codes},
	suppose that $D_{\max}\leq D$ and $k_{\max}\leq k$.
	Let $(G',\ell')$ be the intersection graph of $S$
	as in Example~\ref{EX:intersection-graph}
	and
	let $\mu\in [0,1]$. Then:
	\begin{enumerate}[label=(\roman*)]
		\item If $C=C(\{C_s\}_{s\in S})$  
		is a part of a non-redundant agreement tester
		$(\{C_s\}_{s\in S}, G,w_{\uni},\ell)$ having soundness
		$\mu$,
		then the natural tester
		of $C=C(\{C_s\}_{s\in S})$ has soundness $\frac{\mu}{kD(1+2kD)}$.
		\item If the natural tester
		of $C=C(\{C_s\}_{s\in S})$ has soundness $\mu$, then
		the non-redundant agreement tester $(\{C_s\}_{s\in S}, G',w_{\uni},\ell')$  
		has soundness $\frac{\mu}{k^3D^3(1+kD)}$.
	\end{enumerate}
\end{cor}

\begin{proof}
	(i)
	By Remark~\ref{RM:agreement-is-testability-for-line}
	(and Remark~\ref{RM:generalized-codes}),
	the tester $T_{G,\ell}$ for the line code $L$
	has soundness at least $\mu$.
	Our assumption that $(\{C_s\}_{s\in S}, G,w_{\uni},\ell)$ is non-redundant
	implies that the maximal vertex degree of $G$ is
	at most $k_{\max} D_{\max}\leq kD$.
	Thus, by 
	Theorem~\ref{TH:line-code-testability}, the natural
	tester of $C$ has soundness $\frac{\mu}{kD(1+2kD)}$ (because $\mu\leq 1$).
	
	(ii)
	The maximal vertex degree of $G'$
	is at most $kD$.
	Thus, by Theorem~\ref{TH:lifted-to-line-test},
	the tester $T_{G',\ell'}$ for the line code $L$
	has soundness $\frac{\mu}{k^3D^3(1+kD)}$.
	By Remark~\ref{RM:agreement-is-testability-for-line},
	this is equivalent to saying that the agreement tester
	$(\{C_s\}_{s\in S}, G',w_{\uni},\ell')$ has soundness
	$\frac{\mu}{k^3D^3(1+kD)}$.
\end{proof}

\section{Graded Partially Ordered Sets}
\label{sec:posets}

Recall that a partially ordered set,
or a \emph{poset} for short, is a set $P$ equipped with a transitive
anti-reflexive relation $<$. We then write $a\leq b$ to denote that $a< b$ or $a=b$,
and $a\lhd b$ to denote that $a<b$ and there is no $c\in X$ with $a<c<b$.
Every subset of a poset $X$ will also be viewed as a poset by giving it the partial
order inherited from $X$. 
\emph{If not indicated otherwise, all posets are finite.}

\subsection{Graded Posets}
\label{subsec:graded-posets}

\begin{dfn}[Graded Poset]
	A \emph{graded poset} is a poset $X$ together with a 
	dimension function\footnote{
		Also called  a rank function.	
	} $\dim=  \dim_X :X\to \Z$
	such that $x\leq y$ implies
	$\dim x\leq \dim y$
	and  $x\lhd y$ implies $\dim  x+1=\dim y$   for all $x,y\in X$.\footnote{
		Some texts impose additional assumptions, e.g.,
		the requirement that $X$ admits an element $\emptyset_X$
		(necessarily unique)
		satisfying $\dim \emptyset_X =-1$ and $\emptyset_X\leq x$ for every
		$x\in X$. This forces $\dim x\geq 0$ for every $x\in X-\{\emptyset_X\}$.
	}
	In this case, we write 
	\[X(i)=\{x\in X \suchthat \dim x=i\}\] 
	for all $i\in\Z$ and define the
	dimension of  $X$ to be $\dim X:=\sup\{i\in\Z\suchthat X(i)\neq\emptyset\}$.
\end{dfn}

Beware that a subset of a graded poset is not a graded poset in general.

Let $X$ be a graded poset.
Motivated by examples of geometric nature,
we call the elements of $X(i)$ the \emph{$i$-faces} of $X$. Given a face
$x\in X$, a \emph{subface} of $x$ is a $y\in X$
satisfying $y\leq x$. We further write
\[x(i):=\{y\in X(i)\suchthat y\leq x\}\]
and call elements of $x(i)$ $i$-faces of $x$. 
The set of faces   $y\in X$ having $x$
as a face is denoted
$
X_x:=\{y\in X\suchthat y\geq x\}$. More generally, for every
$A\subseteq X$, we write
\[
A_x=\{a\in A\suchthat a\geq x\}.
\]
In particular, $X(i)_x$ is the set of $i$-faces of $X$ having $x$ as a subface.
Finally, we write $X({\leq } i)$ for the graded subposet $\bigcup_{j\leq i}X(j)$.

\begin{example}\label{EX:graded-poset}
	(i) Finite simplicial complexes and cube complexes are naturally graded posets.
	Their dimension function assigns every   face its geometric dimension  with the convention
	that the empty face has dimension $-1$.
	
	(ii) 
	Generalizing (i), the (closed) faces of a \emph{regular cell complex} 
	(also called a \emph{regular CW complex}) form 
	a graded poset w.r.t.\ inclusion of faces; 
	see \cite[Apx.~A.2]{Abramenko_2008_Buildings} for the definition.
	We follow the convention 
	that    a cell complex must include  a unique empty face of dimension $-1$.
	The posets of regular cell complexes can be characterized combinatorially
	\cite[Prop.~3.1]{Bjorner_1984_posets_regular_comps}, so henceforth,  a regular cell complex
	will mean the poset of faces of a regular cell complex (including the empty face).
	 	
	(iii) Let $\F$ be a finite field and $n,d\in\N$.
	Let $\AG_{d,n}(\F)$ denote all affine subspaces of $\F^n$ of dimension $d$ or less together
	with the set $\emptyset$.
	Then $\AG_{d,n}(\F)$ together with the containment relation is a poset known
	as the \emph{affine Grassmannian} of $d$-spaces in $\F^n$.
	It can be made into a graded poset by setting the dimension of $V\in \AG_{d,n}(\F)$
	to be its ordinary $\F$-dimension if $V\neq \emptyset$
	and $-1$ otherwise. 
\end{example}

\begin{example}[Viewing Hypergraphs as Graded Posets]\label{EX:hypergraph-as-graded-poset}
	A  (finite) hypergraph $X$ (possibly with double hyperedges) 
	is nothing but a graded poset
	$X$ concentrated in degrees $0$ and $1$, i.e.,
	a poset such that $X(i)=\emptyset$ for all $i\neq 0,1$.
	Indeed, think of the $0$-faces are the vertices of $X$, the $1$-faces
	as the hyperedges of $X$ and the relation $<$ as the incidence relation between
	vertices and hyperedges. In particular, we shall
	freely view graphs are graded posets concentrated in degrees
	$0$ and $1$.\footnote{
		According to our conventions, simple graphs and    simplicial 
		complexes of dimension $1$ or less are not exactly the same thing, the difference being
		that a   simplicial complex must include an empty
		face of dimension $-1$ while a graph cannot include such a face.
	}
\end{example}

\begin{example}[Opposite Graded Poset]
	\label{EX:op-poset}
	Let $X$ be a graded poset. The opposite graded poset of $X$
	is the set $X^{\op}=\{x^{\op}\where x\in X\}$ endowed with the relation
	$x^{\op}<y^{\op}$ $\iff $ $y<x$ and the dimension function $\dim(x^\op)=-\dim x$.
\end{example}

\begin{dfn}[Pure Graded Poset, $d$-Poset]
	Let $d\in \N\cup\{0\}$. A graded poset $X$ is said to be \emph{pure of dimension $d$}
	if it is nonempty and every face of $x$ is a subface of a $d$-face; 
	it is said to be \emph{pure}
	if it is pure for some $d\in\N\cup\{0\}$.
	We say that $X$ is a  \emph{$d$-poset} if it is pure
	of dimension $d$ and in addition, there is an element
	$\emptyset_X\in X$ satisfying $\dim \emptyset_X=-1$ and $\emptyset_X\leq x$
	for all $x\in X$.
\end{dfn}

When $X$ is a $d$-poset, the face $\emptyset_X$ is unique.
We call it the \emph{empty face} of $X$
and denote it by $\emptyset$ when $X$ is clear from the context.

The posets in Example~\ref{EX:graded-poset} are  $d$-posets when they
are pure.
A graph $G$ is pure in the sense of
\S\ref{subsec:graphs} if and only if it is  a pure poset of dimension $1$
(but it is never a $1$-poset because it has no face of dimension $-1$). 
 
\medskip 
 
If $X$ is a   $d$-poset, then every    subset $A\subseteq
X$ has a lower bound. Let $L$ be the set of lower bounds
of $A$. As usual, an \emph{infimum} of $A$ is a maximal member $L$. 
The set of infima of $A$ is denoted 
\[\Inf A.\]
This set is often a singleton, e.g., when $X$ is a simplicial complex.

\subsection{Weighted Posets}
\label{subsec:weights}

\begin{dfn}[Weighted Poset]
	A \emph{weighted poset} is a pair $(X,w)$ 
	where $X$ is a poset and  
	$w :X\to \R_+$. In this case, for any $A\subseteq X$,
	we let $w(A)=\sum_{a\in A} w(a)$.
	We say that $w$ or $(X,w)$ is \emph{normalized} if $w(X(i))=1$
	for all $i\in\Z$ with $X(i)\neq\emptyset$.
\end{dfn}

\begin{dfn}[Properly Weighted Poset]
	A \emph{properly weighted poset} is a weighted graded poset
	$(X,w)$ such that
	\begin{enumerate}[label=(\arabic*)]
		\item $X$ is pure of dimension $d$ for some (necessarily unique) $d\geq 0$, 
		\item $w(X(d))=1$, and
		\item $w(x)=\sum_{y\in X(d):y\geq x} \frac{w(y)}{|y(i)|}$
		for all $i\in \Z$ and $x\in X(i)$.
	\end{enumerate}
	In this case, we also say that $w$ is a \emph{proper} weight function on $X$.
\end{dfn}

It follows readily from the definition that if
$X$ is a properly weighted  poset of dimension $d$
and $x$ is an $i$-face of $X$,
then $w(x)$ is the probability of getting $x$
by   choosing a $d$-face $y$ of $X$  randomly
according to   $w|_{X(d)}$ and then
choose an $i$-face of $y$ uniformly at random.
Thus, a properly weighted poset is also normalized.

Following Example~\ref{EX:hypergraph-as-graded-poset}, 
a  (properly) \emph{weighted hypergraph}  means
a (properly) weighted graded poset $(X,w)$ concentrated in degrees $0$ and $1$.
In the case of graphs, this agrees with the notion of a properly weighted
graph from \S\ref{subsec:graphs}.

\begin{example}
	\label{EX:natural-weight}
	(i) Let $X$ be a pure poset of dimension $d$. The \emph{natural weight function} of $X$
	is the weight function $w_{\nat} :X\to \R_+$ defined by
	\[
	w_{\nat} (x)=\frac{1}{|X(d)|}\sum_{y\in X(d):y\geq x}\frac{w(y)}{|y(i)|}.
	\]
	The natural weight function is always proper. 
	
	(ii) Let $X$ be an poset. The \emph{uniform weight function} of $X$
	is the weight function $w_{\uni} :X\to \R_+$ defined by
	\[
	w_{\uni}(x)=\frac{1}{|X(i)|}.
	\]
	The  uniform weight function is normalized, but not always proper.
\end{example}

\subsection{Links}
\label{subsec:links}

\begin{dfn}[Link in Graded Poset]
	Let $X$ be a graded poset and let $z\in X$. The \emph{link}
	of $X$ at $z$ is 
	\[X_z=\{x\in X\suchthat x\geq z\},\] 
	viewed as a subposet of $X$,
	and endowed with the dimension function
	given by $\dim_{X_z}(x)=\dim_X x - \dim_X z - 1$.
\end{dfn}

We will abbreviate $\dim_{X_z}$ to $\dim_z$ when there is no risk of confusion.

\begin{example}
	Let $X$ be a simplicial complex and let $z\in X$.
	The link of $X$ in $z$ is usually defined
	to be the poset $X'_z:=\{y\in X\suchthat \text{$y\cup z\in X$ and $y\cap z=\emptyset$}\}$
	which is also a simplicial complex;
	see \cite[Dfn.~A.19]{Abramenko_2008_Buildings}, for instance.
	While our $X_z$ is different from $X'_z$ in general,
	we have a graded poset isomorphism $X'_z\to X_z$ given
	by $y\mapsto y\cup z$, so $X_z$ is isomorphic to the usual link of $X$ at $z$.
\end{example}

When the graded poset $X$ is pure of dimension $d$ and $z$ is an $i$-face
of $X$, the link $X_z$ is a graded $(d-\dim z-1)$-poset with
$\emptyset_{X_z}=z$. 
Moreover, every proper weight function $w:X\to \R_+$  
induces a proper weight function on $w_z:X_z\to \R_+$ defined by\label{symdef:wz}
\[
	w_z(x)=\frac{1}{w(X(d)_z)}\sum_{y\in X(d)_z} \frac{w(y)}{\#\{x'\in X(\dim_X x)_z\where x'\leq y\}}.
\]
If $w$ is the natural weight function of $X$,
then $w_z$ is the natural weight function of $X_z$.
For details about the ratio between
$w$ and $w_z$, see Lemma~\ref{LM:weight-link-vs-all} below.

\medskip

When $X$  is a $d$-poset, a    \emph{proper link} of $X$
means a link $X_z$ with $z \neq \emptyset_X$.
For a proper link $X_z$, we have $\dim X_z<\dim X$, whereas
$X_{\emptyset_X}=X$.

\subsection{Subface Counting Constants and Lower-Regular Posets}
\label{subsec:lower-regular}

Throughout, let $X$ be a graded poset. 

Given integers $  i\leq j\leq k $, we let
$\Fmax_{i,j,k}(X)$ (resp.\ $\Fmin_{i,j,k}(X)$) denote the maximal 
(resp.\ minimal) possible number of $j$-faces living between
an $i$-face and a $k$-face that are incident in $X$. Formally, if there
exist $x\in X(i)$ and $z\in X(k)$ with $x\leq z$, define\label{symref:Fmax}
\begin{align*}
\Fmax_{i,j,k}(X) &= \max\{\#\{y\in X(j)\suchthat x\leq y\leq z\}\where
x\in X(i),\, z\in X(k),\, x\leq z\}, \\
\Fmin_{i,j,k}(X) &= \min\{\#\{y\in X(j)\suchthat x\leq y\leq z\}\where
x\in X(i),\, z\in X(k),\, x\leq z \}.
\end{align*}
Otherwise, set $\Fmax_{i,j,k}=\Fmin_{i,j,k}=0$.
When $X$ is a $d$-poset and
$i=-1$, the number $\Fmax_{i,j,k}(X)$ (resp.\ $\Fmin_{i,j,k}(X)$) is the maximal
(resp.\ minimal)
possible number of $j$-faces contained in a $k$-face of $X$, and we abbreviate
\[
\Fmax_{j,k}(X)=\Fmax_{-1,j,k}(X)
\qquad\text{and}\qquad
\Fmin_{j,k}(X)=\Fmin_{-1,j,k}(X).
\]
Once $X$ is clear from the context, we will drop it from the notation,
writing just $\Fmax_{i,j,k}$ and $\Fmin_{i,j,k}$.

\begin{lem}\label{LM:relation-between-face-numbers}
	Let $X$ be a graded poset.
	Suppose that $ i\leq j\leq k\leq \ell $ are integers.
	Then
	\[\Fmin_{i,j,\ell} \Fmin_{j,k,\ell} \leq  \Fmax_{i,k,\ell} \Fmax_{i,j,k}
	\qquad \text{and} \qquad
	\Fmax_{i,j,\ell}\Fmax_{j,k,\ell}\geq \Fmin_{i,k,\ell}\Fmin_{i,j,k}.\]
\end{lem}

\begin{proof} 
	Let $u\in X(i)$ and $z\in X(\ell)$ be incident;
	if there are no such $u$ and $z$ then both sides of both inequalities evaluate to $0$. 
	Write $[u,z](j)$
	for   the set of $j$-faces of $X$ lying between $u$ and $z$. Then 
	\begin{align*}
		& \Fmin_{i,j,\ell} \Fmin_{j,k,\ell} \leq  \sum_{x\in [u,z](j)} \sum_{y\in [x,z](k)} 1
		= \sum_{y\in [u,z](k)} \sum_{x\in [u,y](j)} 1 \leq  \Fmax_{i,k,\ell} \Fmax_{i,j,k}.
	\end{align*}
	This proves the first inequality. The second inequality is shown similarly.
\end{proof}

\begin{dfn}[Lower-Regular Graded Poset]
	A graded poset $X$ is called \emph{lower-regular} if for all integers
	$  i\leq j\leq k $, we have
	$\Fmax_{i,j,k}=\Fmin_{i,j,k}$. In this case, we write $F_{i,j,k}$
	for both quantities (and $F_{j,k}=F_{-1,j,k}$).\footnote{
		This condition is stronger then the lower regularity
		considered in \cite{Kaufman_2023_Garland_for_Posets}.	
	}
\end{dfn}

Lower-regular graded posets are both common and better behaved than general graded posets.
For such posets, the inequalities
of Lemma~\ref{LM:relation-between-face-numbers} become 
an equality: $F_{i,j,\ell} F_{j,k,\ell} =  F_{i,k,\ell} F_{i,j,k}$

\begin{example}
	Simplicial complexes, cube complexes and the affine Grassmannian $\AG_{d,n}(\F)$
	are lower-regular graded posets. 
	For a simplicial complex of dimension $d$, we have $F_{i,j,k}={k-i \choose j-i}$
	if $-1\leq i\leq j\leq k\leq d$ and $F_{i,j,k}=0$ otherwise.
\end{example}

\begin{dfn}[Lower Irregularity of a $d$-Poset]
	Let $X$ be a $d$-poset
	and let $-1\leq i\leq j\leq k\leq d$ be integers.
	The the $(i,j,k)$-lower irregularity   of $X$
	is\label{symref:Lijk}
	\[
	L_{i,j,k}=L_{i,j,k}(X)=\frac{\Fmax_{i,j,k}(X)}{\Fmin_{i,j,k}(X)}
	\]
	and we abbreviate $L_{-1,j,k}$ to $L_{j,k}$.
	The   lower irregularity   
	of $X$ is
	\[
	L(X)=\max_{-1\leq i\leq j\leq k\leq d} L_{i,j,k}(X).
	\]
\end{dfn}

Note that $L_{i,j,k}(X)$
is well-defined
because
the assumption that $X$ is a $d$-poset
guarantees that $\Fmin_{i,j,k}\geq 1$.
The reason is that every face of $X$ contains   $\emptyset_X$ (of dimension $-1$) 
and is contained in some $d$-face, so every 
pair of incident faces in $X$ is a part of a
chain of faces $\emptyset=x_{-1} < x_0 < x_1<\dots<x_d$
with $\dim x_\ell=\ell$ for all $\ell$.

The lower irregularity of   $X$ measures how far $X$ is from being lower-regular.
We always have $L(X)\geq 1$ and equality holds if and only if $X$ is lower-regular.

\medskip

We now consider properly weighted   $d$-posets $(X,w)$. The following 
fundamental    lemmas
use the constants $\Fmax_{i,j,k}$ and $\Fmin_{i,j,k}$ to relate the weights of subsets
of $X$ and its links. They will be used repeatedly in the sequel.
Note that the inequalities in the lemmas become equalities when $X$ is lower-regular.

\begin{lem}\label{LM:weight-of-j-faces-cont-z}
	Let $(X,w)$ a properly weighted
	$d$-poset, let $-1\leq i\leq j\leq d$ and let $z\in X(i)$. Then
	\[
	\ContZwZLo{i}{j}{d}
	\leq \frac{w(X(j)_z)}{w(z)} \leq
	\ContZwZUp{i}{j}{d}.
	\]
\end{lem}

\begin{proof}
	We have
	\begin{align*}
		w(X(j)_z)
		&= \sum_{x\in X(j)_z} w(x) 
		= \sum_{x\in X(j)_z} \sum_{y\in X(d)_x} \frac{w(y)}{|y(j)|}
		= \sum_{y\in X(d)_z} \sum_{x\in X(j):z\leq x\leq y}
		\frac{w(y)}{|y(j)|}\\
		&\leq
		\sum_{y\in X(d)_z}\frac{\Fmax_{i,j,d}w(y)}{ |y(j)|}
		\leq 
		\sum_{y\in X(d)_z}\frac{\Fmax_{i,j,d}\Fmax_{i,d}}{ \Fmin_{j,d}}\cdot\frac{w(y)}{|y(i)|}
		=\frac{\Fmax_{i,j,d} \Fmax_{i,d}}{\Fmin_{j,d}}w(z).
	\end{align*}
	This gives the right inequality. The left inequality is shown similarly.
\end{proof}

\begin{lem}\label{LM:weight-sum-over-i-faces}
	Let $(X,w)$ a   weighted
	$d$-poset,
	let  $-1\leq i\leq j\leq d$ and let $\emptyset\neq A\subseteq X(j)$. Then
	\[
	\Fmin_{i,j}\leq \frac{\sum_{z\in X(i)} w(A_z)}{w(A)}
	\leq \Fmax_{i,j}.
	\]
\end{lem}

\begin{proof}
We have
\begin{align*}
\sum_{z\in X(i)} w(A_z)
&=
\sum_{z\in X(i)} \sum_{x\in A_z} w(x)
=\sum_{x\in A}\sum_{z\in x(i)} w(x)
\leq \sum_{x\in A} \Fmax_{i,j}w(x)=
\Fmax_{i,j} w(A).
\end{align*}
This gives the right inequality. The left inequality is shown similarly.
\end{proof}

\begin{lem}\label{LM:weight-link-vs-all}
	Let $(X,w)$ a properly weighted
	$d$-poset, let $-1\leq i\leq j\leq d$, let $z\in X(i)$ and let $x\in X(j)_z$.
	Then
	\[
	w(X(d)_z)^{-1} \cdot
	\LinkAllLo{i}{j}{d}	
	\leq \frac{w_z(x)}{w(x)}
	\leq
	w(X(d)_z)^{-1} \cdot
	\LinkAllUp{i}{j}{d}.
	\]
\end{lem}

\begin{proof}
	We have
	\begin{align*}
	w_z(x)
	&=
	\sum_{y\in X(d)_x} \frac{w(y)}{w(X(d)_z)|y(j)_z|}
	\leq
	\sum_{y\in X(d)_x} \frac{w(y)}{w(X(d)_z)\Fmin_{i,j,d}}\cdot \frac{\Fmax_{j,d}}{|y(j)|}
	\\
	&=
	w(X(d)_z)^{-1}\cdot\frac{\Fmax_{j,d}}{\Fmin_{i,j,d}}
	\sum_{y\in X(d)_x}\frac{w(y)}{ |y(j)|}=
	w(X(d)_z)^{-1}\cdot\frac{\Fmax_{j,d}}{\Fmin_{i,j,d}} w(x).
	\end{align*}
	This gives the right inequality. The left inequality is shown similarly.
\end{proof}

\begin{cor}\label{CR:weight-of-face-bound}
	Let $(X,w)$ a properly weighted
	$d$-poset, let $-1\leq i \leq d$ and let $z\in X(i)$. 
	Then
	\[
	\frac{w(X(d)_z)}{\Fmax_{i,d}} 
	\leq w(z)\leq 
	\frac{w(X(d)_z)}{\Fmin_{i,d}} 
	\]
\end{cor}

\begin{proof}
	Apply Lemma~\ref{LM:weight-link-vs-all} with $x=z$
	and observe that $w_z(z)=1$
	and $\Fmin_{i,i,d}=\Fmax_{i,i,d}=1$.
\end{proof}

\begin{cor}\label{CR:weights-in-lower-regular-posets}
	Let $(X,w)$ be a properly weighted lower-regular $d$-poset
	and let $0\leq j\leq d$. 
	Then $(X({\leq} j),w|_{W({\leq} j)})$ is a properly weighted lower-regular $j$-poset. In particular,
	for every $-1\leq i\leq j$ and $x\in X(i)$, we have
	\[
	w(x)=\sum_{y\in X(j)_x}\frac{w(y)}{|y(i)|}. 
	\]
\end{cor}

\begin{proof}
	Let $i\in\{-1,\dots,j\}$ and $x\in X(i)$.
	By Lemma~\ref{LM:weight-of-j-faces-cont-z} and the lower regularity
	of $X$, we have
	$w(x)=\frac{F_{j,d}}{F_{i,j,d}F_{i,d}}w(X(j)_x)$.
	By Lemma~\ref{LM:relation-between-face-numbers},
	the right hand size
	equals
	$\frac{w(X(j)_x)}{F_{i,j}}=  \sum_{y\in X(j)_x}\frac{w(y)}{|y(i)|}$,
	so we proved the equality in the corollary.
	All other assertions are now straightforward.
\end{proof}

\begin{remark}
	Corollary~\ref{CR:weights-in-lower-regular-posets}
	implies that if $(X,w)$ is a properly weighted lower-regular
	$d$-poset, then $w$ is a \emph{standard weight function} in the sense of \cite{Kaufman_2023_Garland_for_Posets}.
\end{remark}

\subsection{Degree and Upper-Regular Posets}

Again, let $X$ be a graded poset. 
Given integers $i\leq j$ such that $X$ has an $i$-face incident
to a $j$-face, the maximal  (resp.\ minimal) \emph{$(i,j)$-degree}
of $X$ is largest (resp.\ smallest)
possible number of $j$-faces containing an $i$-face
in $X$. The maximal     $(i,j)$-degree and minimal $(i,j)$-degree
of $X$ are denoted\label{symref:Dmax}
\[
D^{\max}_{i,j}(X)\qquad\text{and}\qquad
D^{\min}_{i,j}(X),
\]
respectively.
If no $i$-face of $X$ is incident to a $j$-face, we set $D^{\max}_{i,j}(X)=D^{\min}_{i,j}(X)=0$.
When $X$ is clear from the context, we shall simply write $D^{\max}_{i,j}$
and $D^{\min}_{i,j}$.

The (total) \emph{degree} of a $d$-poset $X$ is\label{symref:D} 
\[
D(X):=\max_{v\in X(0)} |X_v|.
\]
One has $D^{\max}_{0,d}\leq D(X)\leq D^{\max}_{0,d}+\dots+D^{\max}_{d,d}$.

\begin{dfn}[Upper-Regular Graded Poset]
	A graded poset $X$ is called \emph{upper-regular} if
	$D^{\max}_{i,j}(X)=D^{\min}_{i,j}(X)$ for all integers $i\leq j$.
\end{dfn}

\begin{dfn}[Upper Irregularity of a $d$-Poset]
	Let $X$ be a $d$-poset. For integers $-1\leq i\leq j\leq d$,
	the  $(i,j)$-upper irregularity of $X$ is\label{symref:Uij}
	\[
	U_{i,j}=U_{i,j}(X)=\frac{D_{i,j}^{\max}}{D_{i,j}^{\min}}.
	\]
	The upper irregularity of $X$ is 
	\[
	U(X)=\max_{-1\leq i\leq j\leq d}U_{i,j}(X).
	\]
\end{dfn} 

As with lower regularity, the upper irregularity of a $d$-poset
$X$ measures how far it is from being upper-regular --- we 
always have $U(X)\geq 1$ and equality
holds and if and only $X$ is upper-regular.
Unfortunately, upper-regular $d$-posets are not so common for $d\geq 2$.

\begin{example}
	(i) A graph (viewed as a poset) is upper-regular if and only if it is
	a regular graph in the usual sense, i.e., there is $k\in\N$
	such that every vertex belongs to exactly $k$ edges.
	
	(ii) The explicit
	Ramanujan complexes of \cite{Lubotzky_2005_explicit_constructions_of_Ramanujan_complexes}
	(see also \cite{Li_2004_ramanujan_hypergraphs})
	are famous examples of simplicial complexes
	that are \emph{high dimensional expanders}.
	They are upper-regular in dimension $2$ but are not upper-regular in dimensions $3$ and above.
	
	(iii) The double Cayley complex $\mathrm{Cay}(A,G,B)$ associated to a group 
	$G$ and two symmetric generating sets $A,B\subseteq G$ --- see \S\ref{subsec:intro-ltc}
	or \S\ref{subsec:ltc-poset} --- is a square complexes that is upper-regular if and only if
	$|A|=|B|$. In general, its upper irregularity is $\max\{\frac{|A|}{|B|},\frac{|B|}{|A|}\}$.
\end{example}

\begin{prp}\label{PR:face-weight-ratio-via-irreg}
Let $X$ be a $d$-poset, let $w$ be the natural weight function
of $X$ and let $-1\leq i\leq d$.
Then for every $x,x'\in X(i)$, we have
$w(x)\leq U_{i,d}L_{i,d} w(x')$.
\end{prp}

\begin{proof}
We have 
\[w(x)=\sum_{y\in X(d)_x}\frac{w(y)}{|y(i)|}
=\frac{1}{|X(d)|}\sum_{y\in X(d)_x}\frac{1}{|y(i)|}
\leq \frac{1}{|X(d)|}\frac{D^{\max}_{i,d}}{\Fmin_{i,d}}
=\frac{1}{|X(d)|}\frac{D^{\min}_{i,d}}{\Fmax_{i,d}}U_{i,d} L_{i,d}.\]
Similarly,
$w(x')\geq \frac{1}{|X(d)|}\frac{D^{\min}_{i,d}}{\Fmax_{i,d}}$
and the proposition follows. 
\end{proof}

\subsection{Orientation}
\label{subsec:orientation}

Recall that, given a poset $X$, we write $x\lhd y$ to denote that $x<y$
and there is no $z\in X$ with $x<z<y$.
Recall also that  rings are assumed to be commutative and
$\units{R}$ denotes the group of invertible elements in a ring $R$.

\begin{dfn}[Oriented Poset]\label{DF:orientation}
	Let $X$ be a graded poset and
	let $R$ be a   commutative ring, e.g.\ $\Z$.
	An \emph{$R$-orientation} on $X$ is a function
	\[
	(y,x)\mapsto [y:x]: \{ (y,x)\in X\times X\suchthat 
	x\lhd y \} \to \units{R}
	\]
	such that whenever $x,z\in X$ satisfy $x\leq z$ and $\dim z=\dim x+2$,
	we have
	\[
	\sum_{y:x<y<z}[z:y][y:x]=0
	\]
	in $R$.
	An \emph{$R$-oriented poset} is a graded poset $X$ endowed with an $R$-orientation
	$[ : ]$. 
\end{dfn}

We will often be agnostic about which $R$-orientation is chosen and only
care that an $R$-orientation exists. In this case,
we will say that our poset is \emph{$R$-orientable}. 
If $X$ admits a $\Z$-orientation $[:]$, then
$X$ admits an $R$-orientation  for any commutative ring $R$ defined
by $(y,x)\mapsto [y:x]1_R$.

\begin{example}[Regular Cell Complexes are $\Z$-Orientable]
	\label{EX:orientation-of-cell-complexes}
	Every regular cell complex $X$ admits a $\Z$-orientation, and therefore an $R$-orientation
	for every commutative ring $R$, such that $[v:\emptyset]=1$ for every $v\in X(0)$.
	In particular, simplicial complexes and cube complexes are $R$-orientable.
	
	In more detail, let $\calX$ be a topological realization of
	$X$. Then every $x\in X$ with $\dim x\geq 0$ corresponds to a topological
	embedding $j_x:D^n\to \calX$ of an $n$-dimensional disc in $\calX$.
	Choose an orientation for every cell $j_x:D^n\to \calX$  
	(i.e., a generator of $ \pi_{n }(D^n,\partial D^n)\cong \Z$). 
	Then, given nonempty faces $x,y$ with $x\lhd y$, take $[y:x]$ be $1$ if the orientations of 
	the discs of
	$y$ and $x$ agree and $-1$ otherwise. When $x$ is empty, just set $[y:x]=1$.
	
	In practice, choosing an orientation for the faces of $X$ means choosing a sign
	($+$ or $-$) for every vertex $v\in X(0)$, a direction for every edge $e\in X(1)$ (i.e.\ 
	labelling one its vertices with a $+$ and the other with a $-$),
	a spin for every $2$-dimensional $x\in X(2)$ (i.e.\ a direction for every edge of $x$
	such that two edges sharing a vertex give opposite signs to that vertex),
	and so on. In general, choosing an orientation for an $i$-face $x$ ($i\geq 0$)
	amounts to choosing  an orientation for    every  $(i-1)$-face of $x$ such that
	every two $(i-1)$-faces which share an $(i-2)$-face restrict to opposite orientations on
	that face.
\end{example}

\begin{example}
	(i) If a graded poset $X$ admits an $\F_2$-orientation, then
	this orientation is unique and is given by $[y:x]=1$, because $\units{\F_2}=\{1\}$.
	The poset $X$ is $\F_2$-orientable if and only if for every
	$x,z\in X$ with $\dim z=\dim x+2$, the number
	of $y\in X$ with $x<y<z$ is even.

	(ii) Let $\F_q$ be a finite field with $q$ elements  and
	let $X=\AG_{d,n}(\F_q)  $ (Example~\ref{EX:graded-poset}(iii)).
	Then $X$ admits a $\Z/(q+1)\Z$-orientation
	given by $[y:x]=1$ for every $x,y\in X$ with $x\lhd y$.
	This is an orientation because for every $x,z\in X$
	with $x\leq z$ and $\dim z=\dim x+2$, we have
	$q+1\mid \#\{y\in X\suchthat x<y<z\}$, so 
	$\sum_{y:x<y<z}[z:y][y:x]=
	\sum_{y:x<y<z}1=0$ in $\Z/(q+1)\Z$.
	On the other hand, parity considerations as in (i)
	show that $\AG_{d,n}(\F_q)$  	 has no $\F_2$-orientation
	when $q$ is even, and hence no $\Z$-orientation.
	
	(iii) A \emph{linear} graded poset with at least $3$ elements 
	has no $R$-orientation for every nonzero commutative ring $R$.
\end{example}

\begin{example}\label{EX:orientation-of-link}
	Let $X$ be a graded poset and $[:] $ an  $R$-orientation on $X$.
	Let $z\in X$. Then the restriction of $[:] $ to
	the link $X_z$ is an orientation of $X_z$. \emph{We will always give $X_z$
	the orientation it inherits from from $X$.}
\end{example}

\section{Sheaves on Partially Ordered Sets}
\label{sec:sheaves}

Sheaves   on   cell complexes, also called
\emph{cellular sheaves} were first considered by
Shepard  \cite{Shepard_1985_cellular_descrip_of_der_cat_PhD}.
The  theory was further developed by Curry \cite{Curry_2014_sheaves_cosheaves_PhD},
who also considered the dual notion of cellular \emph{cosheaves}. 
A more concise treatment appears in 
\cite{Hansen_2019_spectral_thy_of_sheaves} (for regular cell complexes).
The definition of sheaves on cell complexes extends naturally to general posets;
this is   briefly
considered in \cite[\S4.2.2]{Curry_2014_sheaves_cosheaves_PhD}
and   \cite{Panteleev_2022_good_quantum_codes}.
We recall it here,  
and then define sheaf cohomology when the underlying poset is graded and oriented.

Recall our standing assumption that rings are commutative and all modules
are left modules. \emph{Throughout, $R$ is a ring.}

\subsection{Sheaves on Posets}

\begin{dfn}[Sheaf on a Poset]
	Let $R$ be a ring, e.g., $\Z$ or a field $\F$, and let $X$ be a poset.
	An \emph{$R$-sheaf} $\calF$ on $X$ consists of:
	\begin{itemize}
		\item an $R$-module $\calF(x)$ for every $x\in X$;
		\item an $R$-linear map $\res^{\calF}_{y\from x}:\calF(x)\to \calF(y)$ for every $x,y\in X$ with
		$x<y$;
	\end{itemize}
	such that whenever $x<y<z$, we have
	\begin{equation}\label{EQ:sheaf-cond}
	\res^{\calF}_{z\from y}\circ \res^{\calF}_{y\from x} = \res^{\calF}_{z\from x}.
	\end{equation}
	In this case, we also define $\res^{\calF}_{x\from x}=\id_{\calF(x)}$, so 
	that \eqref{EQ:sheaf-cond} also holds when $x\leq y\leq z$.
\end{dfn}

One can similarly define sheaves of abelian groups, but they
are   the same thing as $\Z$-sheaves.

The maps $\res^{\calF}_{y\from x}$ are called the \emph{restriction maps} of $\calF$.
We will write $\res^{\calF}_{y\from x}(f)$ as $\res_{y\from x}(f)$ 
when there there is no risk of confusion.
One can similarly define \emph{cosheaves} by reversing the direction of the restriction maps. 
We spell out the definition explicitly.

\begin{dfn}[Cosheaf on a Poset]
	Let $R$ be a ring  and let $X$ be a poset.
	An $R$-cosheaf $\calG$ on $X$ consists of:
	\begin{itemize}
		\item an $R$-module $\calG(x)$ for every $x\in X$;
		\item an $R$-linear map $\cores^{\calG}_{x\from y}:\calG(y)\to \calG(x)$ for every $x,y\in X$ with
		$x<y$;
	\end{itemize}
	such that whenever $x<y<z$, we have
	\begin{equation}\label{EQ:cosheaf-cond}
	\res^{\calG}_{x\from y}\circ \res^{\calG}_{y\from z} = \res^{\calG}_{x\from z}.
	\end{equation}
	In this case, we also define $\cores^{\calG}_{x\from x}=\id_{\calF(x)}$, so 
	that \eqref{EQ:cosheaf-cond} also holds when $x\leq y\leq z$.
\end{dfn}

The maps $\cores^{\calG}_{x\from y}$ are the \emph{corestriction maps} of $\calG$.
A cosheaf on $X$ is essentially the same thing as a sheaf on the opposite poset $X^\op$
(Example~\ref{EX:op-poset}).
It is beneficial to differ between these two notions
because often  the posets $X$
and $X^{\op}$ behave differently.

\begin{example}
	Let $M$ be an $R$-module.
	The \emph{constant $R$-sheaf}  on $X$  associated to $M$
	is defined by setting $\calF(x)=M$ for every $M$
	and $\res^{\calF}_{y\from x}=\id_M$ for every $x\leq y$.
	This sheaf is denoted $M_X$.
	One can similarly define
	a constant cosheaf associated to $M$.
\end{example}

More sophisticated examples of sheaves will be considered later.

\subsection{Sheaf Cohomology}
\label{subsec:sheaf-coh}

Suppose that $X$ is an $R$-oriented graded poset,
where $R$ is a commutative ring.
In this case, one can associate
cohomology groups to $R$-sheaves on $X$  and homology groups to $R$-cosheaves
on $X$ as follows.

Let $\calF$ be an $R$-sheaf on $X$. For every $i\in\Z$, define
\[
C^i=C^i(X,\calF):=\prod_{x\in X(i)}\calF(x).
\]
We call $C^i$ the space of \emph{$i$-cochains} on $X$ with coefficients in $\calF$,
and given an $i$-cochain $f\in C^i=\prod_{x\in X(i)}\calF(x)$, we write the $x$-component
of $f$ as $f(x)$.
The
\emph{$i$-th coboundary map} $d_i:C^i\to C^{i+1}$ is determined by
\[
(d_if)(y)=\sum_{x\in y(i)} [y:x]\res_{y\from x} f(x) 
\]
for all $f\in C^i$ and $y\in X(i+1)$. 
The subscript $i$ in $d_if$ will often be clear from the context,
so we shall sometimes just write $df$.
It is a recommended   standard exercise
to check   that $d_{i+1}\circ d_i=0$. Thus, we get a cochain complex of $R$-modules 
\[
C^\bullet=C^\bullet(X,\calF):=[\cdots \to C_{-1}\xrightarrow{d_{-1}} C_0\xrightarrow{\,d_0\,}
C_1 \xrightarrow{\,d_1\,} C_2 \to \cdots].
\]
As usual, its $R$-modules of \emph{$i$-coboundaries} and \emph{$i$-cocycles} are
\[
B^i=B^i(X,\calF):=\im d_{i-1}
\qquad\text{and}\qquad
Z^i=Z^i(X,\calF):=\ker d_i.
\]
Clearly, $B^i\subseteq Z^i$. The quotient module $Z^i/B^i$ is denoted
$\HH^i(X,\calF)$ and called the \emph{$i$-th cohomology} $R$-module of $X$ with coefficients in $\calF$, but this will
not be needed in the sequel. 

\begin{remark}
Beware that at this level of generality, $\HH^i(X,\calF)$ may be nonzero for negative values of $i$.
(To experts, we also caution that $\{\HH^i(X,-)\}_{i\geq 0}$   may not be the \emph{right
derived functors} of $\HH^0(X,-)$, even when $X(i)=\emptyset$ for all $i<0$.)  
In addition, the isomorphism classes of 
$Z^i$, $B^i$ and $\HH^i(X,\calF)$
may depend on
the $R$-orientation of $X$.
However, when $X$ is a regular cell complex and $\calF(\emptyset)=0$, 
everything behaves as expected: $\HH^i(X,\calF)=0$ for $i<0$ 
and, as we shall see in \S\ref{subsec:indep-of-orient}, 
changing the $R$-orientation has no effect on the isomorphism class of $C^\bullet$. 
(Moreover, $\{\HH^i(X,-)\}_{i\geq 0}$ are   indeed the right
derived functors  of $\HH^0(X,-)$, but that will not be needed here.)
\end{remark}

\medskip

The homology groups of a cosheaf $\calG$ on $X$ are defined similarly, but with the following differences.
One replaces $d_i:C^i\to C^{i+1}$ with the \emph{$i$-th boundary map} $\partial_i:C^i\to C^{i-1}$
defined by
\[
(\partial_i f)(y)=\sum_{x\in X(i)_y} [x:y]\cores_{y\from x} f(x),
\]
and we get a chain complex:
\[
\cdots \from C_{-1}\xleftarrow{\,\partial_0\,} C_0\xleftarrow{\,\partial_1\,}
C_1 \xleftarrow{\,\partial_2\,} C_2 \from \cdots
\]
The \emph{$i$-boundaries} and \emph{$i$-cycles} are $B_i=B_i(X,\calF):=\im \partial_{i+1}$
and $Z_i=Z_i(X,\calF):=\ker\partial_i$ and the $i$-th homology is
$\HH_i(X,\calF)=Z_i/B_i$.\footnote{
	If $X$ is allowed to be infinite, then one should also replace
	$C^i=\prod_{x\in X(i)}\calF(x)$ with the $R$-module
	of \emph{$i$-chains} $C_i=C_i(X,\calF)=\bigoplus_{x\in X(i)}\calF(x)$.
	Otherwise, the summation in the definition of $\partial_i$ is not always  well-defined.
}

\begin{remark}\label{RM:opposite-sheaf}
	Let $\calG$ be a cosheaf on graded oriented poset $X$. Define a  sheaf
	$\calG^\op$ on $X^\op$ (Example~\ref{EX:op-poset}) by
	setting $\calG^{\op}(x^\op)=\calG(x)$ and $\cores^{\calG^{\op}}_{y^\op\from x^\op}=
	\res^{\calG}_{x\from y}$ ($x<y$).  We call $\calG^\op$ the \emph{opposite  sheaf} of 
	the cosheaf $\calG$. Give $X^\op$ the orientation determined
	by $[x^\op:y^\op]=[y:x]$.
	Then   $C^{i}(X,\calG)=C^{-i}(X^\op,\calG^\op)=$ and $\partial_{ i}^\calG=d_{-i}^{\calG^\op}$.
	Thus, cosheaf homology may be realized as sheaf cohomology of the opposite sheaf. 
\end{remark}

\begin{example}\label{EX:sheaf-on-hypergraph}
	Let $\calF$ be an $R$-sheaf on an $R$-oriented
	hypergraph $X$ (viewed as poset; Example~\ref{EX:hypergraph-as-graded-poset}).
	Then $C^0=C^0(X,\calF)=\prod_{v\in X(0)}\calF(v)$,
	and after unfolding the definitions, one finds that $Z^0$ is the set of $f=(f(v))_{v\in X(0)}\in\prod_{v\in X(0)}\calF(v)$ such that for every hyperedge $e\in X(1)$,
	\[
	\sum_{v\in e(0)} [e:v] \res_{e\from v}f(v) = 0.
	\]
\end{example}

\subsection{Restricting Sheaves to The Links}
\label{subsec:sheaf-at-link}

\begin{dfn}[Sheaf Restricted to a Link]
	Let $\calF$ be an $R$-sheaf on a graded poset $X$ 
	and let $z\in X $. The restriction of $\calF$ to $X_z$ is the $R$-sheaf
	$\calF_z$ obtained by restricting $\calF$ to $X_z$. That is,
	$\calF_z(x)=\calF(x)$ and $\res^{\calF_z}_{y\from x} =\res^{\calF}_{y\from x}$
	for all $x,y\in X_z$ with $x\leq y$.
\end{dfn}

Suppose now that $X$ is an $R$-oriented graded poset and $\calF$ is an $R$-sheaf on $X$.
Let $z\in X$ and recall that we give $X_z$
the $R$-orientation inherited from $X$ (Example~\ref{EX:orientation-of-link}).
Given  $i\in\Z$ and $f\in C^i(X,\calF)$,
we define $f_z\in C^{i-\dim z-1}(X_z,\calF_z)$
by\label{symdef:g-low-u} 
\begin{align*}
 f_z(x')&=f(x') & & \forall x'\in X(i)_z,
\end{align*}
and given $f'\in C^{i-\dim z-1}(X_z,\calF_z)$, we define 
$f'^z\in C^i(X,\calF)$ by\label{symdef:g-up-u}
\begin{align*}
f'^z(x )& =\left\{
\begin{array}{ll}
f'(x) & x\in X(i)_z \\
0 & x'\notin X(i)_z
\end{array} 
\right.
\end{align*}
That is, $f_z$ is the restriction of $f$ to $X(i)_z=X_z(i-\dim z-1) $,
and $f'^z$ is obtained by extending $f'$ from $X(i)_z$ to $X(i)$
by setting it to be $0$ on  the $i$-faces not in $X_z$.

A straightforward computation (and a recommended exercise)
gives the following lemma.

\begin{lem}\label{LM:extension-from-link-and-d}
	Let $\calF$ be an $R$-sheaf on an $R$-oriented graded poset $X$, let $z\in X$
	and let $i\in\Z$.
	Write $d_z$ for the coboundary map of $\calF_z$.
	Then, for every $f\in C^{i-\dim z-1}(X_z,\calF_z)$,
	\[
	(d_z f)^z = d (f^z). 
	\]	
\end{lem}

\subsection{Independence of The Orientation for Regular Cell Complexes}
\label{subsec:indep-of-orient}

Let $X$ be an $R$-oriented graded poset and let $\calF$
be an $R$-sheaf on $X$.
We noted earlier that the definition of the cochain complex
$C^\bullet(X,\calF)$ depends on the $R$-orientation of $X$.
We now show that when $X$ is a regular cell complex
(e.g.\ a simplicial complex or a cube complex) this choice has essentially no effect.
To that end, we first prove two lemmas.
Given $x\in X$, an \emph{$x$-flag}
is a sequence $f=(x_{-1},x_0,x_1,\dots,x_i)$ of faces
in $X$ such that $\emptyset=x_{-1} \lhd x_0 \lhd \dots\lhd x_i=x$.

\begin{lem}\label{LM:flag-equivalence}
	Let $X$ be a regular cell complex and let $x\in X $.
	Then, for any two $x$-flags $f$ and $f'$, there is a sequence
	of $x$-flags $f=f_0,f_1,\dots,f_n=f'$ in which
	every two consecutive $x$-flags  
	differ by at most one term.
\end{lem}

\begin{proof} 
	This is well-known, but 
	we sketch the  proof for the sake of completeness.
	Given $x$-flags $f$ and $f'$, we write $f\sim f'$
	to denote that there is a sequence of $x$-flags as in the lemma.
	The proof is by induction on $i:=\dim x$. The cases $i=-1$ and $i=0$
	are clear, so assume   that $i>0$ and write  $f:=(x_{-1},x_0 ,\dots,x_i)$
	and $f'=(x'_{-1},x'_0,\dots,x'_i)$.
	Suppose that $x_{i-1}$ and $x'_{i-1}$ share an $(i-2)$-face
	$y$, and choose a $y$-flag $ (y_{-1},\dots,y_{i-2})$.
	Then, by the induction hypothesis,
	$(x_{-1},\dots,x_{i-2},x_{i-1})\sim (y_{-1},\dots,y_{i-2},x_{i-1})$ 
	and $(y_{-1},\dots,y_{i-2},x'_{i-1})\sim (x'_{-1},\dots,x'_{i-1})$. This means
	that $f\sim (y_{-1},\dots,y_{i-2},x_{i-1},x_i)$ 
	and $f'\sim (y_{-1},\dots,y_{i-2},x'_{i-1},x_i)$,
	so $f\sim f'$.
	In general, since the topological realization of $X_{<x}=\{y\in X:y<x\}$
	is a sphere, we can find a sequence  $x_{i-1}=z^{(0)},\dots,z^{(\ell)}=x'_{i-1}$
	of $(i-1)$-faces of $X$ such that $z^{(k-1)}$ and $z^{(k)}$
	share an $(i-2)$ face for every $k$. By choosing a flag
	$f^{(k)}=(\cdots,z^{(k)},x)$ for every $0<k<\ell$ 
	and using what we have shown,
	we see that $f \sim f^{(1)}\sim\dots\sim f^{(\ell-1)}\sim f'$.
\end{proof}

\begin{lem}\label{LM:orientation-ratio}
	Let $X$ be a  regular cell complex, let $[:]$ and $(:)$
	be two $R$-orientations on $X$  and let $x\in X$.
	Choose an $x$-flag $f=(x_{-1},\dots,x_i)$
	and define $c_x=\prod_{n=0}^{i} [x_i:x_{i-1}]^{-1} (x_i:x_{i-1})\in\units{R}$.
	Then $c_x$ does not depend on the choice of $f$.
\end{lem}

\begin{proof}
	Thanks to Lemma~\ref{LM:flag-equivalence}, it is enough
	to show that $c_x$ does not change if we replace
	$x_j$ ($-1<j<i$) with a different face $x'_j$
	lying between $x_{j-1}$ and $x_{j+1}$.
	Indeed, since $X$ is a regular cell complex,
	$x_j$ and $x'_j$ are the only faces between $x_{j-1}$
	and $x_{j+1}$, which means that $[x_{j+1}:x_j][x_j:x_{j-1}]+[x_{j+1}:x'_j][x'_j:x_{j-1}]=0$,
	or rather, 
	$[x_{j+1}:x_j][x_j:x_{j-1}]= -[x_{j+1}:x'_j][x'_j:x_{j-1}]$.
	Similarly, $(x_{j+1}:x_j)(x_j:x_{j-1})= -(x_{j+1}:x'_j)(x'_j:x_{j-1})$,
	and it follows that
	\[[x_{j+1}:x_j]^{-1}(x_{j+1}:x_j)[x_j:x_{j-1}]^{-1}(x_j:x_{j-1})=
	[x_{j+1}:x'_j]^{-1}(x_{j+1}:x'_j)[x'_j:x_{j-1}]^{-1}(x'_j:x_{j-1}) .\]
	As a result, $c_x$ does not change when we replace $x_j$ with $x'_j$. 
\end{proof}

\begin{prp}\label{PR:change-of-orientation}
	Let $X$ be a regular cell complex, let $\calF$ be an $R$-sheaf
	on $X$ and let $[:]$ and $(:)$ be two $R$-orientations on $X$.
	Denote by $C^\bullet$ and $C'^\bullet$ the cochain complexes
	associated to $X$ and $\calF$ using the orientations $[:]$
	and $(:)$, respectively.  The $i$-th coboundary map of $C'^\bullet$
	is denoted $d'_i$. 
	For every $x\in X$, define $c_x\in\units{R}$ as in Lemma~\ref{LM:orientation-ratio},
	and for every $i\in\Z$,
	define $T_i=\prod_{x\in X(i)} c_x\id_{\calF(x)}:C^i\to C'^i$,
	that is, $T_i(f)=(c_x f(x))_{x\in X(i)}$.
	Then $T=(T_i)_{i\in \Z}$, defines an isomorphism of cochain complexes
	from $C^\bullet$ to $C'^\bullet$, i.e., each $T_i$ is an $R$-module
	isomorphism and the following diagram commutes.
	\[
	\xymatrix{
	C^{-1} \ar[r]^{d_{-1}} \ar[d]^{T_{-1}} & C_0 \ar[r]^{d_0} \ar[d]^{T_0} 
	& C_1 \ar[r]^{d_1} \ar[d]^{T_1} 
	& \cdots 
	\\
	C'^{-1} \ar[r]^{d'_{-1}}   & C'_0 \ar[r]^{d'_0}  
	& C'_1 \ar[r]^{d'_1}   
	& \cdots
	}
	\]	
	In particular, the map $T_i$ induces
	isomorphisms $Z^i \to Z'^i $ and $B^i\to B'^i$,
	where $Z'_i$ and $B'_i$ denote the $i$-cocycles
	and $i$-cochains of $C'^\bullet$.
\end{prp}

\begin{proof}
	It is clear that $T_i$ is bijective. It remains to check
	that $T_{i+1}\circ d_{i}=d'_{i}\circ T_i$
	for every $i\geq -1$.
	Let $f\in C^i=\prod_{x\in X(i)} \calF(x)$.
	Then for every $y\in X(i+1)$,
	\begin{align*}
		(T_{i+1} d_i f)(y) &=
		c_y\sum_{x\in y(i)} [y:x]\res_{y\from x} f(x)
		=\sum_{x\in y(i)} (y:x)\res_{y\from x}( c_x f(x))
		=(d'_i T_i f)(y),
	\end{align*}
	where the second equality holds because $c_y[y:x]=(x:y)c_x$.
\end{proof}

An analogue of Proposition~\ref{PR:change-of-orientation} holds for cosheaves. We omit the details.

\subsection{Aside: Constraint Systems}
\label{subsec:constraint-system}

	Our discussion of sheaf cohomology and   our main results
	(Theorem~\ref{TH:main-very-technical} and its simpler versions in Section~\ref{sec:main-simple})
	actually apply to a slightly more general structure which we
	call a constraint system.
	
	Let $R$ be a commutative ring.
	An \emph{$R$-constraint system} $\calS$ on a graded poset $X$ consists of
	\begin{itemize}
		\item an $R$-module $\calS(x)$ for every $x\in X$ and
		\item an $R$-homomorphism $r^{\calS}_{y\from x}:\calS(x)\to \calS(y)$
		for every $x,y\in X$ with $x\lhd y$
	\end{itemize}
	such that for every $x,z\in X$ with $\dim z=\dim x+2$, we have
	\[
	\sum_{y:x<y<z} r^\calS_{z\from y}\circ r^{\calS}_{y\from x}=0.
	\]
	Note that, unlike the case of sheaves, 
	$r^{\calS}_{y\from x}$ is only defined when $x\lhd y$. Also, it need not be the case
	that $r^{\calS}_{z\from y}\circ r^{\calS}_{y\from x}=r^{\calS}_{z\from x}$.
	One can dually define   \emph{$R$-constraint cosystems} on $X$  by
	reversing the direction of the maps $r^{\calS}_{y\from x}$.

	Given an $R$-constraint system on $\calS$ on $X$, we let
	$C^i=C^i(X,\calS):=\prod_{x\in X(i)}\calS(x)$ and define  
	a coboundary map $d^i:C^i\to C^{i+1}$ exactly as we did for sheaves, but
	with $r_{y\from x}$ instead of $\res_{y\from x}$
	and
	without the invoking the factor $[y:x]$. This gives rise to a cochain complex
	and a cohomology theory
	of $R$-constraint systems, and the proof of our main result 
	applies verbatim in this context. Note, however,
	that $X$ is not required  
	to have an $R$-orientation as in the case of sheaves.
	
	Examples of constraint systems
	include the \emph{codes modeled over   $2$-layer
	systems} of \cite{Kaufman_2021_amplified_local_testability_preprint};
	they 
	can be realized as constraint systems on  $2$-posets.
	In addition, the \emph{based chain complexes} 
	of \cite{Panteleev_2022_good_quantum_codes} can be realized as the chain
	complexes of constraint cosystems on graded posets.  
	
	Every $R$-sheaf $\calF$ on an $R$-oriented poset $X$  gives
	rise to an $R$-constraint system $\calS$ on $X$ by setting $\calS(x)=\calF(x)$
	and $r^{\calS}_{y\from x}=[y:x]\res^{\calF}_{y\from x}$ when $x\lhd y$.
	In this case, $C^i(X,\calF)=C^i(X,\calS)$ and $d_i^\calF=d_i^\calS$,
	so $\calF$ and $\calS$ give the same modules of cocycles and coboundaries. 
	By constant, not all $R$-constraint systems are induced from sheaves and
	orientations in this manner. (Indeed, $X$ may not be $R$-orientable.)
	
	We use sheaves and not constraint systems in this work because the theory
	of sheaves is more developed and more intuitive, and all
	the   examples of constraint systems that are of interest to us
	arise naturally from sheaves.

\section{Locally Testable Codes from Sheaves and Cosheaves}
\label{sec:cse}

In this section, we explain in detail how sheaves on graded posets give
rise to error correcting  codes. These codes come equipped with a natural tester, the soundness
of which is governed by the cosystolic expansion of the sheaf at hand.
The entire discussion dualizes to cosheaves thanks to Remark~\ref{RM:opposite-sheaf}.

\medskip

Throughout,  $R$ is a (commutative) ring, e.g.\ a finite field, and $X$ is an $R$-oriented graded poset.

\subsection{Cocycle Codes}
\label{subsec:cocycle-codes}

Let $\calF$ be an $R$-sheaf on $X$ and let $i\in \Z$ be an integer such that $X(i)$ 
is  nonempty.
Suppose further that there is an $R$-module $\Sigma$ such that
$\calF(x)=\Sigma$ for every $x\in X(i)$.
Then $C^i=C^i(X,\calF)=\Sigma^{X(i)}$ and so we may view
$Z^i=Z^i(X,\calF)$ as a code inside $C^i=\Sigma^{X(i)}$;
it is
called the \emph{$i$-cocycle code} of $(X,\calF)$.
The constraints defining $Z^i$ inside $C^i=\Sigma^{X(i)}$ give rise
to a natural tester for this code: Given
$f\in \Sigma^{X(i)}$, choose $y\in X(i+1)$ uniformly at random, read
$f(x)$ for every $x\in y(i)$, and accept $f$ if and only if $d f(y)=\sum_{x\in y(i)}[y:x] \res_{y\from x} f(x)=0$. This is a $q$-query tester for $q=\Fmax_{i,i+1}(X)$.

Let $w_{\uni}:X\to\R_+$ denote the uniform weight function (Example~\ref{EX:natural-weight}(ii)).
Given $f,g\in C^j$,   let  $\supp f = \{x\in X(j)\suchthat \text{$f(x)\neq 0$ in $\calF(x)$}\}$
and  define\label{symdef:dist-norm-I}
\[
\|f\|_{\mathrm{uni}}  = w_{\mathrm{uni}}(\supp f)
\qquad\text{and}\qquad
\dist_{\mathrm{uni}}(f,g) =\|f-g\|_{\mathrm{uni}}.
\]
For $j=i$, 
these 
are just the normalized Hamming norm and distance in $C^i=\Sigma^{X(i)}$.
Now,  by definition,  the natural tester of $Z^i\subseteq\Sigma^{X(i)}$ has soundness $\geq \mu$ if and only
if 
\begin{equation}\label{EQ:soundness-of-cocycle-code}
\|d f\|_{\uni}\geq \mu \dist_{\uni}(f,Z^i)
\qquad\forall f\in C^i
\end{equation}
and the distance of $Z^i\subseteq\Sigma^{X(i)}$ is at least $\delta$ if and only if
\begin{equation}\label{EQ:distance-of-cocycle-code}
\|f\|_{\mathrm{uni}}\geq \delta \qquad
\forall f\in Z^i-\{0\}.
\end{equation}
If $d_{i-1}\neq 0$, then $B^i=\im(d_{i-1})$ usually 
contains short vectors (the coboundaries of $(i-1)$-cochains of small support), so we can expect the distance
of $Z^i$ to be large only if $B^i=0$.

Conditions \eqref{EQ:soundness-of-cocycle-code} and~\eqref{EQ:distance-of-cocycle-code}
may be seen as conditions on the expansion of the $i$-th coboundary map
$d_i:C^i\to C^{i+1}$, and indeed,
upon changing the weight function $w_{\uni}$ and
replacing $Z^i-\{0\}$ with $Z^i-B^i$,
we  recover the definition of a  \emph{$(\mu,\delta)$-cosystolic expander} in dimension $i$,
which we now discuss in detail.

\subsection{Cosystolic Expansion}
\label{subsec:cse}

Suppose henceforth that 
$w:X\to \R_+$ be a normalized weight function; typically, we would
like $w$ to be proper.
Given an $R$-sheaf $\calF$ on $X$ and $f,g\in C^i(X,\calF)$, set
\label{symdef:normw}$\|f\|_w=w(\supp(f))$
and \label{symdef:distw}$\dist_w(f,g)=\|f-g\|_w$. We will drop the subscript $w$ when it is clear from the context.

\begin{dfn}[Cosystolic Expansion of Sheaves]
	\label{DF:cse}
	Let $(X,w)$ and $\calF$ be as above and let $i\in\Z$.
	The \emph{$i$-cosystolic expansion} of $\calF$  w.r.t.\ $w$,
	denoted $\cse_i(X,w,\calF)$ 
	is the supremum of the set of  $\mu \in [0,\infty)$
	such that
	\[
	\|df\|_w\geq \mu \dist_w(f,Z^i) \qquad\forall f\in C^i.
	\]
	The \emph{$i$-cocycle distance} of $\calF$  w.r.t.\ $w$  is
	\[
	\ccd_i(X,w,\calF) :=\inf\{\|f\|_w\where f\in Z^i-B^i\}.
	\]
	Given $\mu,\delta>0$, we say that $(X,w,\calF)$   is an 
	$(\mu,\delta)$-cosystolic
	expander in dimension $i$ if
	$\cse_i(\calF)\geq \mu$ and $\ccd_i(\calF)\geq \delta$.
	
	When $X$ is a $d$-poset,
	the   $i$-cosystolic expansion of $\calF$ and $i$-cocycle
	distance of $\calF$ are defined to be
	$\cse_i(\calF):=\cse_i(X,w_{\nat},\calF)$
	and $\ccd_i(\calF):=\ccd(X,w_{\nat},\calF)$, where $w_{\nat}$
	is the natural weight function of $X$ (Example~\ref{EX:natural-weight}).
\end{dfn}

Suppose that there is a finite $R$-module $\Sigma$
such that $\calF(x)=\Sigma$ for all $x\in X(i)$.
Then $\cse_i(X,w_{\uni},\calF)$ is precisely the maximal soundness of the natural
tester of the cocycle code $Z^i=Z^i(X,\calF)\subseteq C^i(X,\calF)=\Sigma^{X(i)}$, 
and provided that $B^i(X,\calF)=0$
(e.g., if $\calF(y)=0$ for all $y\in X(i-1)$), $\ccd_i(X,w_{\uni},\calF)=\delta(Z^i)$.

Unfortunately, we cannot use these observations directly because our results
about cosystolic expansion and cocycle distance will only apply when $w$ is proper,
and that is often not the case for $w_{\uni}$.
Nevertheless, when our weight function $w:X\to \R_+$ is
not too far from being uniform, we can effectively relate
$\cse_i(X,w,\calF)$ and $\ccd_i(X,w,\calF)$ to the soundness and distance of
$Z^i\subseteq\Sigma^{X(i)}$ by means of the following lemma.

\begin{lem}\label{LM:cse-to-ltc}
	Let $\calF$, $X$, $i$ and $\Sigma$ be as in \S\ref{subsec:cocycle-codes}
	and let $w:X\to \R_+$ be a normalized weight function.
	Assume further that there are  $M,M'\in[1,\infty)$
	such that $w(x)\leq Mw(y)$ for every $x,y\in X(i)$
	and $w(x')\leq M'w(y')$ for every $x',y'\in X(i+1)$.
	Then:
	\begin{enumerate}[label=(\roman*)]
		\item If $B^i=0$, then the relative distance of $Z^i\subseteq \Sigma^{X(i)}$
		is at least $\frac{1}{M}\ccd_i(X,w,\calF)$.
		\item The soundness of the natural tester
		of $Z^i\subseteq \Sigma^{X(i)}$ is at least
		$\frac{1}{MM'}\cse_i(X,w,\calF)$.
	\end{enumerate}
	Moreover, if we would modify the natural tester to choose an $(i+1)$-face according
	to the distribution $w|_{X(i+1)}$, then its soundness would
	be at least $\frac{1}{M}\cse_i(\calF)$.
\end{lem}

\begin{proof}
	We first observe that $M^{-1} w_{\mathrm{uni}}  (x) \leq 
	w(x)\leq M w_{\mathrm{uni}}  (x)$ for every $x\in X(i)$,
	and thus $M^{-1}\|f\|_{\mathrm{uni}}\leq \|f\|_w\leq M\|f\|_{\mathrm{uni}}$
	for all $f\in C^i=C^i(X,\calF)$.
	To see the right inequality, observe that
	\[
	|X(i)|w(x) =\sum_{y\in X(i)}w(x)\leq \sum_{y\in X(i)}M w(y)=M=M|X(i)|w_{\mathrm{uni}}(x).
	\]
	The left inequality is shown similarly. In the same way,
	$M'^{-1}\|g\|_{\mathrm{uni}}\leq \|g\|_w\leq M'\|g\|_{\mathrm{uni}}$
	for all $g\in C^{i+1}$. We now prove (i) and (ii).
	
	(i) Let $0\neq f\in Z^i$. Then $f\in Z^i-B^i$ because $B^i=0$,
	and
	thus,	
	$
	\|f\|_{\mathrm{uni}}\geq M^{-1} \|f\|\geq M^{-1}\ccd_i(X,w,\calF)$.
	
	(ii) For every $f\in C^i$, we have
	\[
	\|d_if\|_{\mathrm{uni}}\geq M'^{-1}\|d_i f\|_w
	\geq M'^{-1}\cse_i(X,w,\calF) \dist_w(f,Z^i)\geq
	\frac{\cse_i(X,w,\calF)}{MM'}\dist_{\mathrm{uni}}(f,Z^i).
	\]
	A similar computation shows that $\|d_if\|_w\geq \frac{1}{M}\cse_i(\calF)
	\dist_{\mathrm{uni}}(f,Z^i)$
	and this gives the last assertion of the lemma.
\end{proof}

\subsection{Coboundary Expansion}

We will also need to consider a stronger variant
of cosystolic expansion, called \emph{coboundary expansion}.

\begin{dfn}[Coboundary Expansion of Sheaves]\label{DF:cbe}
	With notation  as in Definition~\ref{DF:cse}, the \emph{$i$-coboundary expansion}
	of $\calF$  w.r.t.\ $w$, denoted $\cbe_i(X,w,\calF)$ 
	is the smallest $\veps\in [0,\infty)$ such that
	\[
	\|df\|_w\geq\veps \dist_w(f,B^i(X,\calF))
	\qquad\forall f\in C^i(X,\calF).
	\]
	We say that  $(X,w,\calF)$ is an $\veps$-cosystolic expander
	in dimension $i$ if $\cbe_i(X,w,\calF)\geq \veps$.
	
	When $X$ is a $d$-poset, the   $i$-coboundary
	expansion of $\calF$ is $\cbe_i(\calF):=\cbe_i(X,w_{\nat},\calF)$,
	where $w_{\nat}$ is the natural weight function of $X$.
\end{dfn}

If $\cbe_i(X,w,\calF)>0$, then we must have $Z^i=B^i$, or rather, $\HH^i(X,\calF)=0$.
Also, if $\HH^i(X,\calF)=0$, then $\cbe_i(X,w,\calF)=\cse_i(X,w,\calF)$.
In the context of \S\ref{subsec:cocycle-codes},
the $i$-coboundary expansion may be thought of as measuring the soundness
of the natural tester of $Z^i\subseteq \Sigma^{X(i)}$, but for the code
$B^i\subseteq \Sigma^{X(i)}$ (which usually has poor distance if $B^i\neq 0$).

\medskip

Coboundary expansion of $d$-posets in dimensions $-1$ and $0$   appears in some of our main
results, so it is worthwhile to unfold the definition  in these cases. This can be
informally summarized as follows:
\begin{itemize}
	\item Coboundary expansion in dimension $-1$ is similar to the relative distance of a code.
	\item Coboundary expansion in dimension $0$ is similar to agreement testability
	(\S\ref{subsec:agreement}).
\end{itemize}

\begin{example}[Coboundary Expansion in Dimension $-1$]
	\label{EX:cbe-dim-minus-one}
	Let $(X,w)$ be a normalized weighted $R$-oriented $d$-poset ($d\geq 0$) and let $\calF$ be 
	an $R$-sheaf on $X$. Then $C^{-1}=C^{-1}(X,\calF)=\calF(\emptyset)$
	and $d_{-1}:\calF(\emptyset)\to C^0=\prod_{v\in X(0)}\calF(v)$
	is given by $d f=([v:\emptyset]\res_{v\from \emptyset}(f))_{v\in X(0)}$.
	Moreover, we have   $\dist(f,B^{-1})=1$ for every $f\in C^{-1}-B^{-1}=\calF(\emptyset)-\{0\}$.
	This means that
	\[
	\cbe_{-1}(X,w,\calF)
	=\left\{
	\begin{array}{ll}
	0 & \ker d_{-1}\neq 0 \\
	\inf\{\|g\|_w\where g\in B^0-\{0\}\} & \ker d_{-1}=0 .
	\end{array}
	\right.
	\]
	Since $\|([v:\emptyset]\res_{v\from \emptyset}(f))_{v\in X(0)}\|=\|( \res_{v\from \emptyset}(f))_{v\in X(0)}\|$,
	we can define $\tilde{d}_{-1}: \calF(\emptyset)\to \prod_{v\in X(0)}\calF(v)$
	by $\tilde{d}_{-1} f= ( \res_{v\from \emptyset}(f))_{v\in X(0)}$
	and get that
	\[
	\cbe_{-1}(X,w,\calF)
	=\left\{
	\begin{array}{ll}
	0 & \ker \tilde{d}_{-1}\neq 0 \\
	\min\{\|g\|_w\where g\in (\im \tilde{d}_{-1})-\{0\}\} & \ker \tilde{d}_{-1}=0 .
	\end{array}
	\right. .
	\]
	(This is independent of the $R$-orientation.)
	In particular, if $\tilde{d}_{-1}$ is injective,
	$w$ is the uniform weight function 
	and 
	there is an $R$-module $\Sigma$ such that $\calF(v)=\Sigma$
	for every $v\in X(0)$, then $\cbe_{-1}(X,w,\calF)$
	is the relative distance of the code
	$\im(\tilde{d}_{-1})$ inside $\prod_{v\in X(0)}\calF(v)=\Sigma^{X(0)}$.
\end{example}

\begin{example}[Coboundary Expansion in Dimension $0$]
	\label{EX:cbe-dim-zero}
	Let $X$ be a regular cell complex and let $w:X\to \R_+$ be a normalized weight function. 
	We choose a $\Z$-orientation on $X$ such that $[v:\emptyset]=1$
	for every $v\in X(0)$; see Example~\ref{EX:orientation-of-cell-complexes}. This implies that
	every edge $e\in X(1)$ has a unique vertex $u$ with $[e:u]=1$
	and a unique vertex $v$ with $[e:v]=-1$; denote the former by $e^+$
	and the latter by $e^-$.
	
	Let $\calF$ be an $R$-sheaf on $X$.
	Then   for every $f\in C^0(X,\calF)=\prod_{v\in X(0)} \calF(v)$,
	we have $df = (\res_{e\from e^+} f(e^+)-\res_{e\from e^-} f(e^-))_{e\in X(1)}$.
	This means that $\cbe_0(\calF)$ is the smallest $\kappa\geq 0$ such that
	\[
	w(\{e\in X(1)\suchthat \res_{e\from e^+} f(e^+)\neq \res_{e\from e^-} f(e^-)\})
	\geq \kappa \dist_w(f,B^0)
	\qquad\forall f\in C^0.
	\]
	This is reminiscent of the soundness of an agreement tester (\S\ref{subsec:agreement}),
	and in  fact,   agreement testability may be realized
	as the  $0$-coboundary of a sheaf.
	
	Indeed, let $(\{C_s\}_{s\in S},G,w,\ell)$ be an agreement tester for a lifted
	code $C=C(\{C_s\}_{s\in S})\subseteq \Sigma^n$. Suppose moreover
	that $\Sigma$ is an abelian group and every $C_s$ is a subgroup of $\Sigma^s$.
	Let $X$ be the regular cell complex obtained from $G$
	by adding a single face of dimension $-1$. We extend $w$ from $G$ to $X$
	by setting $w(\emptyset_X)=1$ and endow $X$ with a $\Z$-orientation as above.
	Define a $\Z$-sheaf $\calF$ on $X$ by setting
	\begin{itemize}
		\item $\calF(\emptyset)=C:=C(\{C_s\}_{s\in S})\subseteq \Sigma^n$,
		\item $\calF(v) = C_{\ell(s)}$ for all $v\in X(0)$,
		\item $\calF(e) \subseteq \Sigma^{\ell(e)}$ for all $e\in X(1)$,
		\item $\res_{v\from \emptyset}:C\to C_{\ell(v)}$ is given by $f\mapsto f|_{\ell(v)}$ for all $v\in X(0)$,
		\item $\res_{e\from v}:C_{\ell(v)}\to \Sigma^{\ell(e)}$ 
		is given by $f\mapsto f|_{\ell(e)}$ for all $e\in X(1)$ and $v\in e(0)$.
	\end{itemize}
	Since $B^0(X,\calF)=C$,
	our earlier observations imply readily that
	$(\{C_s\}_{s\in S},G,w,\ell)$ has soundness
	$\kappa$ if and only if $\cbe_0(X,w,\calF)\geq \kappa$.	
\end{example}

%%	and $\calF$ is 
%%	an $\veps$-coboundary expander
%%	in dimension $k$, then $\veps\leq 
%%	
%%	$.

\subsection{Independence of The Orientation}

We continue to use the notation of \S\ref{subsec:cse}.
Recall that  $Z^i=Z^i(X,\calF)$ and $B^i=B^i(X,\calF)$
depend on the implicit $R$-orientation we gave $X$,
so {\it a priori},
$\cse_i(X,w,\calF)$, $\ccd_i(X,w,\calF)$ and $\cbe_i(X,w,\calF)$ depend on that choice as well.
However, when $X$ is a regular cell complex, Proposition~\ref{PR:change-of-orientation}
shows that the effect of changing the $R$-orientation   on $Z^i$
and $B^i$ is
just a   coordinate-dependent scaling, which does not effect the norm $\|\cdot\|_w$
on $C^i$. As a result, in this special case, $\cse_i(X,w,\calF)$, $\ccd_i(X,w,\calF)$ 
and $\cbe_i(X,w,\calF)$ do not depend on the $R$-orientation of $X$.
We record this observation:

\begin{prp}\label{PR:cse-not-dep-on-orientation}
	Let $(X,w)$ be a normalized weighted regular cell complex and let $\calF$
	be an $R$-sheaf on $X$.
	Then $\cse_i(X,w,\calF)$, $\ccd_i(X,w,\calF)$ 
	and $\cbe_i(X,w,\calF)$ 
	do not depend on the $R$-orientation chosen for $X$.
\end{prp}

\subsection{An Example: Sheafy Expander Codes}
\label{subsec:expander-codes}

We give a concrete example of a good $0$-cocycle code.

Fix some $m,k\in \N$ with $\frac{k}{2}<m\leq k$,
let $X$ be a $k$-regular graph  endowed with the uniform weight function  and let $\F$ be
a finite field. 
Given $v\in X(0)$, write $E(v)$ for set of edges having $v$ as a vertex 
and choose    
an injective $\F$-linear map $T_v:\F^m\to \F^{E(v)}\cong \F^k$.
Then  each $C_v:=\im T_v$  is a linear code inside $\F^{E(v)}\cong \F^k$ with alphabet $\F$ and rate $\frac{m}{k}$.

Define a sheaf $\calF$ on $X$ as follows:
\begin{itemize}
\item $\calF(v)=\F^m$ for every $v\in X(0)$,
\item $\calF(e)=\F$ for every $e\in X(1)$,
\item $\res^{\calF}_{e\from v}=\mathrm{pr}_e \circ T_v:\F^m\to \F$,
where $\mathrm{pr}_e:\F^{E(v)}\to \F$ is projection onto the $e$-coordinate.
\end{itemize}
We give $X$ an $\F$-orientation by choosing a vertex $e^+\in e(0)$
for every $e\in X(1)$, naming its other vertex of   as $e^-$,
and setting $[e:e^+]=1$ and $[e:e^-]=-1$.
Since $\calF(v)=\F^m$ for every $v\in X(0)$, we may form the $0$-cocycle
code $Z^0 = Z^0(X,\calF)\subseteq \Sigma^{X(0)}$ with alphabet $\Sigma=\F^m$.
One readily checks that
\[Z^0=\{(f_v)_{v\in X(0)}\in \Sigma^{X(0)}\suchthat 
\text{$(T_u(f_u))_e=(T_v(f_v))_e$ for every $e=\{u,v\}\in X(1)$}\}.\]
Given $f\in \Sigma^{X(0)}$, 
the natural tester of $Z^0$   chooses an edge $e=\{u,v\}\in X(1)$
uniformly at random, probes $f_u$ and $f_v$, and accepts $f$ if and only
if $(T_u(f_u))_e=(T_v(f_v))_e$.

A straightforward counting of linear constraints shows that $r(Z^0)\geq 1-\frac{k}{2m}>0$.
We will see later (Example~\ref{EX:sheafy-expander-codes-distance}) that if each of the codes $C_v\subseteq\F^{E(v)}$
satisfies $\delta(C_v)\geq \delta_0$ and $X$ is a $\lambda$-expander ($\lambda\in [-1,1]$),
then $\delta(Z^0)\geq \frac{\delta_0-\lambda}{1-\lambda}$.
In particular, when $\delta_0>\lambda$, the family
of $0$-cocycle codes $Z^0(X,\calF)\subseteq \Sigma^{X(0)}$ with $X$
a $k$-regular $\lambda$-expander graph 
forms  a good code.

More concretely, it is known that for every $k\geq 3$, there exist    infinite families of
$k$-regular   $\frac{2\sqrt{k-1}}{k}$-expander graphs
(\cite{Marcus_2013_interlacing_families}; see also \cite{Lubotzky_1988_Ramanujan_graphs},
\cite{Margulis_1988_expander_graphs}).
We can let $X$ range over such a family.
To get the required linear codes $C_v\subseteq \F^{E(v)}$, one can choose
$\F$ with $|\F|\geq k$
and take $C_v$ to be an $m$-dimensional \emph{Reed--Solomon code} inside $\F^k$,
so that  $\delta(C_v)=1-\frac{m-1}{k}$.
These choices give a good $0$-cocycle code with alphabet $\F^m$ 
for every $m$
and $k$ such that $\frac{k}{2}<m\leq k$
and $1-\frac{m-1}{k}>\frac{2\sqrt{k-1}}{k}$,
or equivalently, $\frac{k}{2}<m<k+1-2\sqrt{k-1}$.
A  suitable $m$ exists for $k=13$ and every $k\geq 15$.

\begin{remark}
	(i) This example is  related to Sipser and Spielman's
	famous \emph{expander codes} \cite{Sipser_1996_expander_codes}.
	Recall that the expander code associated to the codes $\{C_v\}_{v\in X(0)}$
	and the graph $X$
	is the lifted code $C=C(\{C_v\subseteq \F^{E(v)}\}_{v\in X(0)})\subseteq \F^{X(1)}$ with alphabet $\F$.
	The code $Z^0(X,\calF)\subseteq \Sigma^{X(0)}$ is nothing but  the line code of
	$C$.
	
	Suppose that $X$ is a $\lambda$-expander and $\delta(C_v)\geq \delta_0>\lambda$ for all $v\in X(0)$.
	Combining our claims about $\delta(Z^0)$ and $r(Z^0)$
	with Proposition~\ref{PR:line-code-rate-and-dist} (with $D^{\max}=D^{\min}=2$ , $k^{\max}=k$ and $\tilde{k}=\delta_0 k$)
	shows that $\delta(C)\geq \delta_0\cdot\frac{\delta_0-\lambda}{1-\lambda}$
	and $r(C)\geq \frac{2m}{k}-1=2r(C_v)-1$ (all the $C_v$ have the same rate). 
	For comparison,  
	\cite[Lem.~15]{Sipser_1996_expander_codes} establishes the same lower bound on $r(C)$,
	but gives a slightly different bound on $\delta(C)$,
	namely,   $\delta(C)\geq (\frac{\delta_0-\lambda}{1-\lambda})^2$.
	The lower bound $\delta_0\cdot\frac{\delta_0-\lambda}{1-\lambda}$ is   bigger
	than $(\frac{\delta_0-\lambda}{1-\lambda})^2$ when $\delta_0$ is close enough to $\lambda$.

	(ii) The  code $Z^0(X,\calF)$ 
	is usually \emph{not} locally testable. Indeed, since $m>\frac{k}{2}$,
removing one of the   constraints $(T_u(f_u))_e=(T_v(f_v))_e$ defining $Z^0$
(equivalently, removing the edge $e=\{u,v\}$ from $X$) will typically
enlarge $Z^0$ into a bigger code $Z'$. 
The same considerations showing that $\delta(Z_0)$ is large
also imply that $\delta(Z')$ is   large.
Thus, any $f\in Z'-Z^0$ would be   far from $Z^0$
and at the same time
satisfy 
all of the constraints defining $Z^0$ except for one.
\end{remark}

\subsection{Dual Notions for Cosheaves}

Let $(X,w)$ be a normalized weighted $R$-oriented graded poset
and
let $\calG$ be an $R$-cosheaf on $X$ such that $\calG(x)=\Sigma$
for every $x\in X(i)$. Here, $\Sigma$ is some $R$-module
and   $X(i)\neq \emptyset$. 
Then we may consider $Z_i=Z_i(X,\calG)$
as a code inside $C_i=\Sigma^{X(i)}$; such codes
are called \emph{$i$-cycle codes}. As with cocycle codes, the code $Z_i\subseteq  \Sigma^{X(i)}$
has a natural tester: Given $f\in C^i$, choose $y\in X(i-1)$ uniformly
at random, read $f(x)$ for every $x\in X(i)_{ y}$
and accept $f$ if $\partial_i f(y)=0$. This is a $D_{i-1,i}(X)$-query tester.

For a general cosheaf $\calG$,
the \emph{$i$-systolic expansion} and \emph{$i$-cycle distance}
of $(X,w,\calG)$,
denoted
$\se_i(X,w,\calG)$ and $\cd_i(X,w,\calG)$, respectively,
are defined   exactly as their   counterparts in \S\ref{subsec:cse}
by replacing
$d_i$, $Z^i$, $B^i$ with $\partial_i$, $Z_i$, $B_i$.
In the setting of the last paragraph, an analogue of
Lemma~\ref{LM:cse-to-ltc} holds and may be used to relate $\cd_i(X,w,\calG)$
and $\se_i(X,w,\calG)$ to the distance of $Z^i\subseteq\Sigma^{X(i)}$
and the soundness of its natural tester. This may also be derived directly from
our discussion of cocycle codes using Remark~\ref{RM:opposite-sheaf}.

The \emph{expander codes} of Sipser and Spielman \cite{Sipser_1996_expander_codes}
may be realized as $1$-cycle codes of a cosheaves on   graphs
\cite{Meshulam_2018_graph_codes_arxiv_version}.

\section{No-Intersection Hypergraphs, Skeleton Expansion, Intersection Profiles}
\label{sec:non-intersect}

\subsection{No-Intersection Hypergraphs}
\label{subsec:no-intersect}

Recall that a weighted hypergraph  means a weighted
poset $(X,w)$ such that $X$ 
is concentrated in degrees $0$ and $1$ (Example~\ref{EX:hypergraph-as-graded-poset}). 
We define the underlying hypergraph of a graded poset $X$ to 
be the subposet
$\ugr(X):=X(0)\cup X(1)$\label{symdef:ugrXw}. For example, when $X$ is regular cell complex,
$\ugr(X)$ is the   underlying graph of $X$ in the usual sense. The weighted
hypergraph underlying a weighted graded poset is
\[
\ugr(X,w):=(X(0)\cup X(1),w|_{X(0)\cup X(1)}).
\]

\begin{dfn}[Related   Hypergraph]
	A \emph{related   
	hypergraph} is a pair $(X, \sim)$ consisting of
	a   hypergraph $X$ and a binary relation $\sim $ on the set of vertices $X(0)$,
	subject to the requirement that $u\sim v$ implies that $u,v\in e(0)$
	for some hyperedge   $e\in X(1)$.
	A related weighted hypergraph is a triple $(X,w,\sim)$
	such that $(X,w)$ is a weighted hypergraph and $(X,\sim)$
	is a related hypergraph.
\end{dfn}

\begin{example}\label{EX:standard-rel}
	Let $G$ be a hypergraph.
	The \emph{standard relation} on $G(0)$ is 
	defined by setting $u\sim v$
	if and only if $u\neq v$ and some hyperedge includes both $u$ and $v$
	as vertices, i.e., $X(1)_u\cap X(1)_v\neq \emptyset$.
	We will use the standard relation to view any (weighted) hypergraph
	as a related (weighted)  hypergraph.
\end{example}

%%By default, every weighted hypergraph $X$ will be viewed as a related
%%weighted hypergraph by setting $\sim$ to be the full relation $X(0)\times X(0)$.

Our main example of a related weighted hypergraph is the following:

\begin{dfn}[No-Intersection Related Hypergraph of a Poset]\label{DF:nih}
	Let $(X,w)$ be a normalized weighted $d$-poset  and let  $i,j,k\in\{0,\dots,d\}$ with $i,j< k$.
	The   \emph{$(i,j,k)$-no-intersection hypergraph}   of $(X,w)$ is 
	a related weighted hypergraph $  \NI^{i,j,k}(X,w)=(\NI^{i,j,k}(X),w_{\NI^{i,j,k}(X)},\sim)$
	defined as follows:
	\begin{itemize}
		\item the vertices of  $  \NI^{i,j,k}(X )$ are $X(i)\cup X(j)$ (if $i=j$, the vertex set is just $X(i)$);
		\item the hyperedges of $  \NI^{i,j,k}(X )$ are $X(k)$;
		\item the vertices of a hyperedge $z\in X(k)$ are the  
		$x\in X(i)\cup X(j)$ with $x\leq z$.
	\end{itemize}
	We give $  \NI^{i,j,k}(X )$ the weight function $w_{\NI^{i,j,k}(X )}:=w_{i,j,k}$ determined as follows:
	\begin{itemize}
		\item for a hyperedge $z\in  X(k)$, set $w_{i,j,k}(x)=w(x)$;
		\item for a vertex $x\in  X(i)\cup X(j)$, set $w_{i,j,k}(x)=w (x)$ if
		$i=j$ and $w_{i,j,k}(x)=\frac{1}{2}w (x)$ if $i\neq j$.
	\end{itemize}
	Finally, we endow the vertices of $\NI^{i,j,k}(X ) $ with a binary relation $\sim$:
	\begin{itemize}
		\item for   $x,y\in \NI^{i,j,k}(X )(0)$, let
		$x\sim  y$ if 
		(1) there is $z\in X(k)$ with $x,y\leq z$,
		(2) $\Inf\{x,y\}=\{\emptyset_X\}$ in $X$
		and (3) $(\dim x,\dim y)\in\{(i,j),(j,i)\}$.
	\end{itemize} 
\end{dfn}

Since $(X,w)$ is normalized, $ \NI^{i,j,k}(X,w) $ is normalized
as well. 
If $(X,w)$ is properly weighted, $X$ is lower-regular and $i=j$,
then $ \NI^{i,j,k}(X,w) $ is also properly weighted (Corollary~\ref{CR:weights-in-lower-regular-posets}),
but this is false in general.

We will also consider the following simpler variation of  
the no-intersection hypergraph. 
Recall that graphs are allowed to have multiple edges, but no loops.

\begin{dfn}[No-Intersection Graph of a Poset]
	\label{DF:nig}
	Let $(X,w)$ be a normalized weighted  $d$-poset  and let $i,j,k\in\{0,\dots,d\}$ with $i,j< k$.
	The   \emph{$(i,j,k)$-no-intersection  graph} of $X$ is a weighted graph
	$\NNI^{i,j,k}(X,w)=(\NNI^{i,j,k}(X),w_{\NNI^{i,j,k}(X)})$ 
	defined as follows:
	\begin{itemize}
		\item the vertices of $\NNI^{i,j,k}(X)$
		are $X(i)\cup X(j)$ (if $i=j$, the vertex set is just $X(i)$);
		\item the edges of $\NNI^{i,j,k}(X)$
		are pairs $(z,\{x,y\})$ such that 
		$z\in X(k)$, $x\in X(i)$, $y\in X(j)$,
		$x\neq y$, $\Inf\{x,y\}=\{\emptyset_X\}$ and $x,y\leq z$;
		\item the vertices of the edge $(z,\{x,y\})$ are $x$ and $y$.
	\end{itemize}
	We endow $\NNI^{i,j,k}(X)$ with the weight function $w_{\NNI^{i,j,k}(X)}:=w_{i,j,k}$ determined
	as follows:  
	\begin{itemize}
		\item for an edge $(z,\{x,y\})\in G(1)$, set $w_{i,j,k}(z,\{x,y\})=w (z)$;
		\item for a vertex $x\in G(0)=X(i)\cup X(j)$, set
		$w_{i,j,k}(x)=w (x)$ if $i=j$ and $w_{i,j,k}(x)=\frac{1}{2}w (x)$ if $i\neq j$.
	\end{itemize} 
\end{dfn}

Recall that we also view  $\NNI^{i,j,k}(X,w)$
as a related hypergraph
by setting $u\sim v$ if  there is $e\in \NNI^{i,j,k}(X)(1)$ with $e(0)=\{u,v\}$. 

Note that while the total weight of the vertices in $\NNI^{i,j,k}(X)$
is $1$, the total weight of the edges may be different from $1$. Thus,
in contrast  to $\NI^{i,j,k}(X,w)$,  the weighted graph
$\NNI^{i,j,k}(X,w)$ may not be normalized.

The no-intersection graph and the no-intersection hypergraph
are related by surjective map
\[\vphi_X:\NNI^{i,j,k}(X)\to \NI^{i,j,k}(X)\]
given by the identity on  $0$-faces
and by $(z,\{x,y\})\mapsto z$ on  $1$-faces. It  preserves
the poset relation and the face weights, and also respects
the relation $\sim$ on the vertices.

\begin{example}\label{EX:underlying-graph-as-no-intersec}
	Let $(X,w)$ be a normalized weighted $d$-poset.
	Then $\NI^{0,0,1}(X,w)=\ugr(X,w)$.
	Moreover precisely, $\NI^{0,0,1}(X,w)$ is   the underlying weighted hypergraph
	of $(X,w)$ together with the standard
	relation (Example~\ref{EX:standard-rel}).
	
	Suppose further that $X$ is a regular cell complex.
	Then every $e\in X(1)$ has exactly two $0$-faces.
	This implies that  $\vphi_X:\NNI^{0,0,1}(X)\to \NI^{0,0,1}(X)$ is an isomorphism
	of related weighted  hypergraphs, meaning in particular
	that $\NI^{0,0,1}(X)$ is a graph and
	$\NNI^{0,0,1}(X)=\ugr(X)$.
\end{example}
	
\begin{example}\label{EX:intersec-of-simp}
	Let $X$ be a pure simplicial complex of dimension $d$ 
	and   let $i,j,k\in\{0,\dots,d\}$
	with $i,j<k$. The nature of $H:=\NI^{i,j,k}(X)$ and $G:=\NNI^{i,j,k}(X)$ 
	depends on whether $i+j=k-1$, $i+j>k-1$ or $i+j<k-1$.
	
	When $i+j=k-1$, the graph $G$ is simple.
	Furthermore, the map $\vphi_X:\NNI^{i,j,k}(X)\to \NI^{i,j,k}(X)$
	is bijective on vertices and is $\frac{1}{2}{k+1 \choose i+1}$-to-$1$ on
	hyperedges. It further induces a bijection from
	$\sim_G$ to $\sim_H$.
	One may therefore think of $H$ as the related hypergraph obtained
	from the   graph $G$ by gluing, for every $z\in X(k)$, 
	the edges of the form $(z,\{x,y\})$ to each other.

	When $i+j> k-1$, the graph   $G $ 
	has no edges, because an $i$-face and a $j$-face of a given $k$-face
	must have nonempty intersection. On the other hand, $H$
	has hyperedges, but the relation $\sim $ on its vertices is the empty relation. 
	
	Finally, if $i+j<k-1$, then  neighboring vertices
	in    $G$ will usually be connected by many edges,
	because   the union of an $i$-face and a $j$-face which are disjoint 
	may be contained in many $k$-faces.
	In $H$, a hyperedge is still uniquely determined by its vertices,
	but the relation $\sim_H$ becomes complicated.

	We also note that unless $(i,j)= (0,0)$, 
	the relation $\sim_H$ is \emph{not} the standard relation on $H(0)$.
\end{example}

\subsection{Skeleton Expansion}

Skeleton expansion was considered for properly weighted graphs in \cite{Evra_2016_cosystolic_expanders},
\cite{Kaufman_2018_cosystolic_expanders} 
and similar sources. We now   generalize this concept to 
related weighted  hypergraphs.

\begin{dfn}[Skeleton Expansion]
	Let $(X,w,\sim)$ be a related 
	weighted hypergraph and let $\alpha,\beta\in [0,\infty)$. 
	Given $A\subseteq X(0)$, define
	\[
	E_2(A)=\{e\in X(1)\suchthat \text{there are distinct $u,v\in A$ s.t.\ $u< e$,
	$v< e$ and $u\sim v$}\}.
	\]
	We say that $(X,w,\sim)$ is an $(\alpha,\beta)$-skeleton expander   if for every
	$A\subseteq X(0)$, we have
	\[
	w(E_2(A))\leq \alpha w(A)+\beta w(A)^2.
	\]
	A (non-related) weighted hypergraph $(X,w)$ is said to be
	an $(\alpha,\beta)$-skeleton expander if this holds once we give
	$X$ the standard relation (Example~\ref{EX:standard-rel}).
\end{dfn}

Loosely speaking, the smaller $\alpha$ is, the more skeleton-expanding
$(X,w,\sim)$ is considered.

Note that if $X$ is just a graph given the standard relation, then
$E_2(A)$ is just the set of edges having both of their vertices in $A$,  
denoted   $E(A)$.

In  \cite{Evra_2016_cosystolic_expanders},
\cite{Kaufman_2018_cosystolic_expanders} 
and related sources,
a   \emph{properly weighted} graph $(X,w)$    was called an $\alpha$-skeleton expander
if $w(E(A))\leq \alpha w(A)+ w(A)^2$ for every $A\subseteq X(0)$. This is equivalent to $X$
being an $(\alpha,1)$-skeleton expander in our sense.
We introduced the additional constant $\beta$ to better accommodate improper, or even
non-normalized, weight functions and non-connected
graphs.

\begin{example}\label{EX:skeleton-exp}
	(i) Let $(X,w)$ be a  properly  weighted graph  and
	let $\lambda$ be the second largest eigenvalue of the normalized
	adjacency operator of $(X,w)$
	(\S\ref{subsec:graphs}).
	By Proposition~\ref{PR:eml-special-case},
	$(X,w)$ is a 	$(\lambda,1-\lambda)$-skeleton expander.

	(ii) Let $(X,w,\sim)$ be a weighted related hypergraph.
	If $X$ has no hyperedges, or $\sim$ is the empty relation,
	then   $(X,w,\sim)$ is a $(0,0)$-skeleton expander. 
	Thus,   if
	$X$ is a weighted pure simplicial complex of dimension $d$
	and $i,j,k\in\{0,\dots,d\}$ satisfy $  i,j<k$ and $i+j>k-1$,
	then  both $\NI^{i,j,k}(X)$ and $\NNI^{i,j,k}(X)$ are $(0,0)$-skeleton expanders
	(cf.\ Example~\ref{EX:intersec-of-simp}).
	
	(iii) Every properly weighted hypergraph $(X,w)$ is
	a  $(\frac{1}{2}\Fmax_{0,1}(X),0)$-skeleton expander.
	Indeed, by Lemma~\ref{LM:weight-of-j-faces-cont-z},
	for every $A\subseteq X(0)$,
	we have $w(E_2(A))\leq \frac{1}{2}\sum_{v\in A} w(X(1)_v)
	\leq \frac{1}{2}\Fmax_{0,1} w(A)$.
	
	(iv) Let $G$ be a pure graph
	and let $G_n$ ($n>1$) be a graph consisting of $n$ disjoint copies of $G$.
	Give $G$ and $G_n$ their natural weight functions.
	Since $G_n$ is not connected, the second largest eigenvalue
	of its weighted adjacency operator is $1$, and so  Proposition~\ref{PR:eml-special-case}
	only tells us that $G_n$ is a $(1,0)$-skeleton expander --- same as (iii).
	However, if $G$ is an $(\alpha,\beta)$-skeleton expander,
	then $G_n$ is an $(\alpha,n\beta)$-skeleton expander. Indeed,
	given $A\subseteq G_n(0)$, let $A_i$ be the intersection
	of $A$ with the $i$-th copy of $G$ in $G_n$.
	Then
	\begin{align*}
	w_{G_n}(E(A))&=
	\sum_{i=1}^n w_{G_n}(E(A_i))
	=\frac{1}{n}\sum_{i=1}^n w_{G }(E(A_i))
	\leq \frac{1}{n}\sum_{i=1}^n [\alpha w_G(A_i)+\beta w_G(A_i)^2]
	\\
	& \leq \alpha w_{G_n}(A)+\frac{\beta}{n}(\sum_{i=1}^n w_G(A_i))^2
	=\alpha w_{G_n}(A)+n\beta w_{G_n}(A)^2.
	\end{align*}
\end{example}

Let $(X,w)$ be a normalized weighted  $d$-poset.
We will ultimately be interested in the skeleton expansion
of related weighted hypergraphs of the form $\NI^{i,j,k}(X,w)$, but very little is known 
about it.
The following lemma tells us that it is enough to bound the skeleton expansion  
of the graph $\NNI^{i,j,k}(X)$, which is much more manageable. In fact,
this  is the   reason why we   need to    consider  $\NNI^{i,j,k}(X)$ in this work.
We expect that in general 
the skeleton expansion of $\NI^{i,j,k}(X)$ will not  be much
smaller (in both parameters) than  that of
$\NNI^{i,j,k}(X)$.

\begin{lem}\label{LM:nig-to-nih}
	Let $(X,w)$ be a normalized weighted $d$-poset and let $i,j,k\in\{0,\dots,d\}$
	with $i,j< k$. 
	If $\NNI^{i,j,k}(X,w)$ is an $(\alpha,\beta)$-skeleton expander,
	then $\NI^{i,j,k}(X,w)$ is an $(\alpha,\beta)$-skeleton expander.
\end{lem}

\begin{proof}
	Let $A$ be a set of vertices in $H:=\NI^{i,j,k}(X)$. We may also view
	$A$ as a set of vertices in $G:=\NNI^{i,j,k}(X)$.
	One readily checks that $\vphi_X(E_G(A))= E_2(A)$,
	and hence $w_H(E_2(A))\leq w_G(E_G(A))\leq \alpha w_G(A)+\beta w_G(A)^2
	= \alpha w_H(A)+\beta w_H(A)^2$.
\end{proof}

\subsection{Intersection Profiles}
\label{subsec:profile}

\begin{notation}\label{NT:non-intersec-of-link}
	Let $(X,w)$ be a properly weighted $d$-poset, let $z\in X$
	and let $i,j,k\in \{0,\dots,d\}$ with $b:=\dim z< i,j<k$.
	We let
	\[
	\NI^{i,j,k}_z(X,w)=\NI^{i-b-1,j-b-1,k-b-1}(X_z,w_z) 
	\]
	and define $\NNI^{i,j,k}_z(X,w)$ similarly.
	Here, $w_z$ is the proper weight function induced
	by $w$ on $X_z$ (\S\ref{subsec:links}). 
\end{notation}

In our main result, we would need to consider
the skeleton expansion of related weighted hypergraphs
of the form $\NI^{i,j,k}_z(X,w)$ for various values of $i,j,k$ and $b=\dim z$.
For a particular $X$, it will often be enough
to consider a subset of the  legal   quadruples $(k,i,j,b)$. In order
to keep track of the quadruples which are needed, we introduce the notion
of an \emph{intersection profile}.

Let $k\geq 0$ be an integer.
Broadly speaking, a   $k$-intersection profile for a $d$-poset $X$
is  encodes the dimensions of some quadruples of faces
$(z,x,y,u)$ such that $x,y$ lie between $z$ and $u$ 
and $u\in \Inf\{x,y\}$. 
The axioms guarantee that all the quadruples 
$(z,x,y,u)$ with $\dim z=k+1$ and $\dim x=\dim y=k$ must be considered,
and also that quadruples obtained by taking infima of faces that were previously encountered
are recorded as well. 
The formal definition is as follows.

\begin{dfn}[Abstract Intersection Profile]\label{DF:intersection-profile-abs}
	Let $k\in\N\cup \{0\}$.
	An \emph{abstract $k$-intersection profile}
	$\calP$
	consists of a set of integer quadruples
	$(t,\ell,r,b)$\footnote{
		The letters $t,\ell,r,b$ allude to the words top, left, right and bottom.	
	}  with $k+1\geq t>\ell\geq r>b\geq -1$, also denoted $\calP$,
	and a set of pairs of integers $(i,j)$ with $k+1\geq i>j\geq -1$,
	denoted $\operatorname{Ad}(\calP)$, such that the following conditions are met:
	\begin{enumerate}[label=(\arabic*)]
		\item for every $(t,\ell,r,b)\in\calP$, we have
		$(t,\ell),(t,r),(\ell,b),(r,b)\in \operatorname{Ad}(\calP)$;
		\item if $t>\ell\geq r>-1$ are integers such 
		that  $(t,r),(t,\ell)\in\operatorname{Ad}(\calP)$,
		then   $\calP$
		contains a quadruple of the form $(t,\ell,r,*)$;
		\item if $t>\ell> r\geq -1$ are integers such 
		that  $(t,\ell),(t,r)\in \operatorname{Ad}(\calP)$, then
		$(\ell,r)\in\operatorname{Ad}(\calP)$;\footnote{
			The rationale behind this requirement is that we would
			have liked the illegal quadruple $(t,\ell,r,r)$ to be in $\calP$.
			Since we are not allowed to include it, we compensate for that
			by requiring that   $(\ell,r)$ is in $\operatorname{Ad}(\calP)$.
		}
		\item $(k+1,k)$ and $(k+1,-1),(k,-1),\dots,(0,-1)$
		are in $\operatorname{Ad}(\calP)$.
	\end{enumerate}
	We call $\operatorname{Ad}(\calP)$ the set of $\calP$-admissible pairs.
\end{dfn}

Thanks to assumptions (2) and (4),  
$\operatorname{Ad}(\calP)$ is the set of integer pairs $(i,j)$ with $k+1\geq i>j\geq -1$
that  either occur in some $(t,\ell,r,b)\in \calP$ as in (1), or satisfy $j=-1$.
Thus, $\operatorname{Ad}(\calP)$ is uniquely determined
by $\calP$. We will therefore often not specify
$\operatorname{Ad}(\calP)$ explicitly.

%%	In this case, given
%%	$\rho=(t,\ell,r,b)\in\calP$, we write $\hgt \rho=b$ and call it the height
%%	of  $\rho$. 
%%	We   also write $t(\rho)$, $\ell(\rho)$, $r(\rho)$, $b(\rho)$
%%	for $t$, $\ell$, $r$, $b$, respectively.
%%	

\begin{dfn}[Intersection Profile]\label{DF:intersection-profile}
	Let $X$ be a   $d$-poset, let $k\in\N\cup\{0\}$ and
	let $\calP$ be an abstract $k$-intersection profile.
	A pair of faces $(x,y)\in X\times X$ is said to be $\calP$-admissible
	if $x>y$ and $(\dim x,\dim y)\in\operatorname{Ad}(\calP)$.
	We say that $\calP$ is
	a \emph{$k$-intersection profile} if for every  $x,y,z\in X$
	such that $(x,y)$ and $(x,z)$ are $\calP$-admissible with $\dim y\geq \dim z$, 
	either $x>y\geq z$,
	or  $(\dim x,\dim y, \dim z,\dim u)\in\calP$ for every  $u\in \Inf\{y,z\}$.
\end{dfn}

\begin{example}
	Let $d, k\geq 0$ be integers. Every $d$-poset $X$ has a unique minimal
	$k$-intersection profile $\calP$ which may be constructed iteratively
	as follows: 
	\begin{itemize}
		\item Start with $\calP=\emptyset$ and $\operatorname{Ad}(\calP)=\{(k+1,k)\}\cup\{(k+1,-1),\dots,(0,-1)\}$.
		\item For every integer $i$ running down from $k+1$ to $0$, perform: 
		\begin{itemize}
			\item For every $x\in X(i)$ and every $y,z< x$ with 
			$(\dim x,\dim y),(\dim x,\dim z)\in\operatorname{Ad}(\calP)$ 
			and $\dim y\geq \dim z$, perform:
			\begin{itemize}
				\item If $x>y>z$, add $(\dim y,\dim z)$ to $\operatorname{Ad}(\calP)$.
				\item Otherwise, for every $u\in \Inf\{y,z\}$, add
				$(\dim y,\dim u)$, $(\dim z,\dim u)$ to $\operatorname{Ad}(\calP)$
				and
				$(\dim x,\dim y,\dim z,\dim u)$ to $\calP$.
			\end{itemize}
		\end{itemize}
	\end{itemize}
	At the end of this process, 
	$\calP$ and $\operatorname{Ad}(\calP)$ will determine a $k$-intersection profile
	for $X$.
	(Then name intersection profile comes from the repeated use of  intersections ($\Inf$)
	in this construction.)
\end{example}

\begin{example}[Intersection Profile for Simplicial Complexes]
	\label{EX:simp-intersection-prof}
	The abstract $k$-intersection profile
	\[
	\calP^{(k)}_{\triangle}:=\{(i+1,i,i,i-1)\where i\in\{0,\dots,k\}\}
	\]
	is a $k$-intersection profile for any pure simplicial complex $X$.
	Indeed, that $\calP^{(k)}_{\triangle}$
	is an abstract $k$-intersection profile follows directly from the definition.
	Furthermore,  if     $x,y,z\in X$ are distinct faces such that $(x,y)$ and $(x,z)$
	are $\calP^{(k)}_{\triangle}$-admissible,
	then  
	$-1<\dim y=\dim z=\dim x-1$. Since $X$ is simplicial,
	$\inf \{y,z\}= y\cap z $ and indeed $(\dim x,\dim y,\dim z,\dim (y\cap z))=
	(\dim x,\dim x-1,\dim x-1,\dim x-2)\in\calP^{(k)}_{\triangle}$. 
	
	Given $(t,\ell,r,b)\in \calP^{(k)}_{\triangle}$
	and $z\in X(b)$, we must have $t=b+2$ and $\ell=r=b+1$,
	so   $\NI^{t,\ell,r}_z(X)=\NI^{0,0,1}(X_z)$
	is just the underlying graph of $X_z$.
\end{example}

\begin{example}[Intersection Profile for Cube Complexes]
	\label{EX:cube-intersection-prof}
	The abstract $k$-intersection profile
	\[
	\calP^{(k)}_{\square}:=\calP^{(k)}_{\triangle} \cup \{(i+1,i,i,-1)\where i\in\{1,\dots,k\}\}
	\]
	is a $k$-intersection profile for any pure cube complex $X$.
	To see this, let  $x,y,z\in X$ be distinct nonempty faces 
	such that $(x,y)$ and $(x,z)$ are $\calP^{(k)}_{\square}$-admissible.
	Then, writing $i:=\dim x-1$,   we must have $\dim y=\dim z=i$. Since $X$ is a cube complex,  
	the faces $y$ and $z$ must be distinct $i$-faces of the $(i+1)$-dimensional cube $x$,
	so  their intersection is either an $(i-1)$-dimensional cube, or the empty face.
	In both cases, $(\dim x,\dim y,\dim z,\dim (y\cap z))\in\calP^{(k)}_{\square}$.
	
	As in the last example, given $z\in X$ of the dimension $i $,
	the hypergraph $\NI^{i+1,i+1,i+2}_z(X)$ (corresponding to $(i+2,i+1,i+1,i)\in \calP^{(k)}_{\square}$)
	is just $\ugr(X_z)$. On the other hand,  
	$\NNI^{i,i,i+1}(X)$ (corresponding to 
	$(i+1,i,i,-1)\in\calP^{(k)}_{\square}$) is the graph obtained by taking the $i$-cubes
	in $X$ as vertices and adding one edge between a pair of $i$-cubes $x,y$ for every $(i+1)$-cube having
	$x$ and $y$ as opposite sides. 
\end{example}

\begin{example}\label{EX:intersection-profiles}
	(i) $\calP^{(0)}:=\{(1,0,0,-1)\}$
	is a  $0$-intersection profile for every poset $X$.
	In fact, $\calP^{(0)}$ is the only abstract $0$-intersection profile.

	(ii) $\calP^{(1)}:=\{(2,1,1,0),(2,1,1,-1),(1,0,0,-1)\}$
	is a $1$-intersection profile for every poset $X$.
	It  coincides with $\calP^{(1)}_{\square}$.
	
	(iii) Let $\calP^{(k)}_{\max}$
	denote the set of all quadruples of integers
	$(t,\ell,r,b)$ with $k+1\geq t>\ell\geq r>b\geq -1$.
	Then $\calP^{(k)}_{\max}$ is a $k$-intersection profile for any $d$-poset.
	It is also the largest possible abstract $k$-intersection profile.
	In fact, the slightly smaller abstract $k$-intersection profile
	\[\calP^{(k)}_{\mathrm{univ}}:=\calP^{(k-1)}_{\max}\cup\{(k+1,k,k,i)\where i\in\{-1,\dots,k-1\}\}\]
	(with the convention 
	$\calP^{(-1)}_{\max}=\emptyset$)
	is also a $k$-intersection profile for every $d$-poset $X$, and is the smallest one having this property.
\end{example}

\section{Main Result: Simple Versions}
\label{sec:main-simple}

Our main result is a criterion for establishing cosystolic expansion
of a sheaf on a poset. 
The most general form of this criterion ---
Theorem~\ref{TH:main-very-technical} --- is   technical,
and so we find it instructive to first give in 
this section several  special cases 
which are simpler and easier to apply; they will be derived from the
general main result
in Section~\ref{sec:tech-versions}.
In fact, these special cases
suffice for the applications considered in this paper, and likely 
for other potential applications.

\begin{thm}\label{TH:lgp-simple-version}
	Let  $B\in \R_+$, $F\in \N$, $L\in [1,\infty)$
	and
	$k\in\{0\}\cup\N$.
	Then there exist     (small) constants $K,K' \in (0,1]$ 
	such that the following hold:
	Let $R$ be a commutative ring, let
	$d\geq k+2$,
	let $(X,w)$ be a properly weighted $R$-oriented $d$-poset
	of lower irregularity at most $L$
	and such that $ \Fmax_{i,j}(X)\leq F$
	for all $ -1\leq i\leq j\leq k+2$,
	let $\calP$  
	be a $k$-intersection 
	profile for $X$ and
	let $\calP'$ be a $(k+1)$-intersection profile for $X$. 
	Let
	$\calF$ be an $R$-sheaf on $X$, let
	$\veps\in (0,1]$ and $B'\in [1,\infty)$
	and suppose that:
	\begin{enumerate}[label=(\arabic*)]
		\item[(1a)] 
		$\cbe_{k-\dim u-1}(X_u,w_u,\calF_u)\geq \veps$  for every
		$u\in X(0)\cup\dots \cup X(k)$;
		\item[(1b)]  $\cbe_{k-\dim u}(X_u,w_u,\calF_u)\geq \veps$ for every
		$u\in X(0)\cup\dots \cup X(k+1)$;
		\item[(2a)] for every $(t,\ell,r,b)\in \calP\cup\calP'$ with $b>-1$
		and $u\in X(b)$, the related weighted hypergraph
		$\NI^{\ell,r,t}_u(X )$ (Notation~\ref{NT:non-intersec-of-link})
		is a  
		$((K\veps)^{2^{k+1-r}},B(K\veps)^{-2^{k+1-r}(2^{(r-1)-b}-1)})$-skeleton expander.
		\item[(2b)] for every $(t,\ell,r,b)\in \calP\cup\calP'$ with $b=-1$,
		the related weighted hypergraph
		$\NI^{\ell,r,t}(X )$ 
		is a  
		$((K\veps)^{2^{k+1-r}},B'B(K\veps)^{-2^{k+1-r}(2^{r}-1)})$-skeleton expander.
	\end{enumerate}
	Then
	\[
		\cse_k(X,w,\calF)\geq  B'^{-1} K' (K\veps)^{2^{k+2}-1}
		\qquad\text{and}\qquad
		\ccd_k(X,w,\calF) \geq B'^{-1} K' (K\veps)^{2^{k+1}-1}.
	\]
	Moreover, if $f\in C^k=C^k(X,\calF)$
	satisfies $\dist_w(f,Z^k)<B'^{-1} K' (K\veps)^{2^{k+2}-1}$,
	then applying   Algorithm~\ref{AL:correction-algorithm-simplified} below
	to $f$ with the parameter  $q=K(K\veps)^{2^{k+1}-1}$
	returns $f'\in Z^k$ such that $\dist_w(f,f')\leq K'^{-1}(K\veps)^{-2^{k+1}}\dist_w(f,Z^k)$.
\end{thm}

\begin{alg}[Correction Algorithm]
	\label{AL:correction-algorithm-simplified}
	Let $(X,w),\calF,d,k$ be as in Theorem~\ref{TH:lgp-simple-version}. The input to the algorithm
	is some $f\in C^k(X,\calF)$ and a real parameter $q\geq 0$.
	The algorithm outputs  another $k$-cocycle $f'\in C^k(X,\calF)$, computed as follows:
	\begin{enumerate}[label=(\arabic*)]
		\item Set $f_0=f$ and $i=0$.
		\item Do:
		\begin{enumerate}[label=(\alph*)]
			\item Look for $u\in X(0)\cup\dots\cup X(k )$
			and $g \in C^{k-\dim u-1}(X_u,\calF_u)$ such that 
			$\|df_i+d (g^u)\|< \|df_i\|-  q \cdot w(u)$.
			(For the definition of $g^u$, see \S\ref{subsec:sheaf-at-link}.)
			\item If no  such $u$ and $g$ exist, return $f_i$.
			Otherwise, set $f_{i+1}=f_i+g^u $ and increase $i$ by $1$.
		\end{enumerate}
	\end{enumerate}
\end{alg}

Theorem~\ref{TH:lgp-intro} from the Overview
is obtained from Theorem~\ref{TH:lgp-simple-version} by
setting $B'=1$ and replacing $\NI_u^{\ell,r,t}(X)$ with
$\NNI_u^{\ell,r,t}(X)$, which is allowed by Lemma~\ref{LM:nig-to-nih}.

A few remarks are now in order.

\begin{remark}
	Concerning Theorem~\ref{TH:lgp-simple-version}:

	(i) The smaller the reduced parts 
	of intersection profiles $\calP$ and $\calP'$ are, the more lax conditions
	(2a) and (2b) are.
	Thus, one should choose the intersection profiles to be as small as possible.
	In addition,
	using smaller intersection profiles  usually  improves the constants $K$ and $K'$
	by making them larger. We also note that often one has $\calP'\supseteq \calP$.
	
	(ii)
	In applications of Theorem~\ref{TH:lgp-simple-version}, 
	$F$, $L$, $k$ and $\veps$   typically remain constant
as the degree of $X$ grows, whereas the skeleton expansion of
$\NI^{\ell,r,t}_u(X )$ 
behaves like $(o(1),O(1))$   if $u\neq \emptyset$
and $(o(1),*)$ if $u=\emptyset$ (as  functions of $D(X)$).
Thus, once the degree of $X$ is large enough (but constant),
$B$ is chosen to be large enough in advance and $B'$ is chosen to suitably 
depend on $D(X)$,
conditions (1a), (1b), (2a) and (2b) will hold,
and the theorem could be applied.

	(iii) Assumptions (1a), (1b) and (2a) of Theorem~\ref{TH:lgp-simple-version} are \emph{local}
	in the sense that they care only about the structure of $X_z$ and $\calF_z$
	for $\emptyset\neq z\in X$ and not about the global structure of $X$
	and $\calF$. Condition (2b) is not local.
	However, 
	as we shall see shortly, this non-local requirement can be replaced with a
	local condition when $X$ is a simplicial complex.

	(iv) 
	In the special case where $X$ is a simplicial complex
	and $\calF$ is a constant sheaf, criteria for cosystolic
	expansion appeared in 
	\cite{Kaufman_2016_isoperimetic_inequalities} ($\dim X\leq 3$, $\calF=(\F_2)_X$),
	\cite{Evra_2016_cosystolic_expanders_arxiv_version}
	($\calF=(\F_2)_X$),
	\cite{Kaufman_Mass_2021_cosystolic_non_abelean},
	\cite{Dikstein_2023_cbe_cse_without_dep_on_dim_deg}.
	In addition, the result \cite[Thm.~1.17]{Kaufman_2021_amplified_local_testability_preprint}
	can be interpreted in our language as a criterion for $0$-cosystolic expansion
	of $2$-posets equipped with a special kind of a \emph{constraint system} (\S\ref{subsec:constraint-system}).
	We surveyed these and other related works in detail and compared them
	to Theorem~\ref{TH:lgp-simple-version} in \S\ref{subsec:intro-lgp}.

	(v) The $R$-module $\calF(\emptyset)$ plays no role in Theorem~\ref{TH:lgp-simple-version}.
	Setting it to be $0$ would not affect anything. 
\end{remark}

\begin{remark}\label{RM:alg-complexity}
Concerning Algorithm~\ref{AL:correction-algorithm-simplified}, if one fixes
some $M \geq 1$ and requires that   $w(x)\leq Mw(y)$ 
and $|\calF(x)|\leq M $ for all $x,y\in X(0)\cup\dots\cup X(k)$ 
and $D_{i,k}(X)\leq M$ for all $i\in\{0,\dots,k\}$,
then the time complexity of Algorithm~\ref{AL:correction-algorithm-simplified} 
is \emph{linear} in $|X(k)|$
(the constant depends on $F$, $M$ and $q^{-1}$); see Proposition~\ref{PR:alg-complexity}
below. To do the search in (a) in constant time on average, one has to keep a set of the possible
faces $u$ which may satisfy $\|df_i+d (g^u)\|< \|df_i\|-  q \cdot w(u)$ 
and update it during the search and every time (b) is performed. For details,
see Appendix~\ref{sec:correction-efficient}.
\end{remark}

As an immediate corollary of Theorem~\ref{TH:lgp-simple-version},
we get the following criterion for showing
that a cocycle code is locally testable and has linear distance.
Note that if $w$ is the natural weight function of $X$,
then we can take $M=L_{k,d} U_{k,d}  $ and $M'=L_{k+1,d}  U_{k+1,d}  $
by Proposition~\ref{PR:face-weight-ratio-via-irreg}.

\begin{cor}\label{CR:ltcs-from-sheaves}
	With notation as in Theorem~\ref{TH:lgp-simple-version},	
	suppose further that:
	\begin{itemize}
		\item $B^k=B^k(X,\calF)=0$,
		\item there is an $R$-module $\Sigma$ such that $\calF(x)=\Sigma$
		for every $x\in X(k)$ and
		\item there are $M,M'\in \R_+$ such that $w(x)\leq Mw(y)$
	for all $x,y\in X(k)$ and $w(x')\leq M'w(y')$ 
	for all $x',y'\in X(k+1)$.
	\end{itemize} 
	Then the $k$-cocycle code
	\[
	Z^k=Z^k(X,\calF)\subseteq \Sigma^{X(k)}
	\]
	satisfies $\delta(Z^k)\geq \frac{1}{M}B'^{-1} K' (K\veps)^{2^{k+1}-1}$
	and  its natural tester has soundness
	$\frac{1}{MM'}B'^{-1} K' (K\veps)^{2^{k+2}-1}$.
	Moreover, Algorithm~\ref{AL:correction-algorithm-simplified}
	with $q=K(K\veps)^{2^{k+1}-1}$
	is a decoding algorithm for $Z^k$ that works
	for words that are $\frac{B'^{-1}K'^2(K\veps)^{2^{k+2}-1}}{M(1+K'(K\veps)^{2^{k+1}})}$-close to $Z^k$.
\end{cor}

\begin{proof}
	By Theorem~\ref{TH:lgp-simple-version},
	$\ccd_k(X,\calF)\geq B'^{-1} K'(K\veps)^{2^{k+1}-1}$
	and $\cse_k(X,\calF)\geq B'^{-1}K'(K\veps)^{2^{k+2}-1}$,
	so the claims about the relative distance and
	soundness of
	the tester follow from Lemma~\ref{LM:cse-to-ltc},
	and 
	the claim about the decoding follows from the next lemma.
\end{proof}

\begin{lem}\label{LM:decoding-of-cocycle-code}
	Let $(X,w)$ be an $R$-oriented properly weighted $d$-poset, let $\calF$ be an $R$-sheaf on $X$,
	let $k\in\Z$ and let $M\in [1,\infty)$. Suppose that $X(k)\neq\emptyset$,
	$B^k=B^k(X,\calF)=0$,
	and for every $x,y\in X(k)$,
	we have   $\calF(x)=\Sigma$ and $w(x)\leq Mw(y)$.
	Assume further that there are $A,B\in\R_+$ and an algorithm which takes
	$f\in C^k=C^k(X,\calF)$ with $\dist_w(f,Z^k)<A$ and returns
	$f'\in Z^k$ such that $\dist(f,f')\leq B\dist_w(f,Z^k)$.
	Then this algorithm is also a decoding algorithm for the $k$-cocycle
	code $Z^k\subseteq \Sigma^{X(0)}$ which can decode
	words that are $\frac{1}{M}\min\{\frac{\ccd_k(X,\calF)}{B+1},A\}$-close to
	$Z^k$.
\end{lem}

\begin{proof}
	Write $\eta = \frac{1}{M}\min\{\frac{\ccd_k(X,\calF)}{B+1},A\}$,
	and let $f\in C^k$ be $\eta$-close to $Z_k$. 
	Let $\dist_{\uni}$ be the distance function on $C^k$ induced by
	the uniform weight function $w_{\uni}:X\to\R_+$.	
	Then there
	is $f_0$ such that $\dist_{\mathrm{uni}}(f,f_0)<\eta$.
	As in   the proof of Lemma~\ref{LM:cse-to-ltc},
	we have $\dist_w(f,f_0)\leq M \dist_{\mathrm{uni}}(f,f_0)< M\eta\leq A$.
	Thus, applying the algorithm to $f$ yields $f'\in Z^k$ with
	$\dist_w(f,f')\leq B\dist_w(f,Z^k)\leq B\dist_w(f,f_0)<BM\eta$.
	Thus, $\dist(f_0,f')\leq \dist(f_0,f)+\dist(f,f')
	<M\eta+BM\eta\leq \ccd_k(X,\calF)$. 
	Since $B^k=0$, this means that $f'=f_0$, so we have shown
	that   the algorithm   decoded $f$.
\end{proof}

We proceed with specializing Theorem~\ref{TH:lgp-simple-version}
to simplicial complexes and cube complexes.
In the former case, it simplifies into the following
theorem.

\begin{thm}[Cosystolic Expansion for Simplicial Complexes]
	\label{TH:lgp-simplicial-case}
	Let   $k\in \N$.
	Then there exist     constants $K,K' \in (0,1]$ 
	such that the following hold:
	Let $R$ be a commutative ring, 
	let $(X,w)$ be a properly weighted pure $d$-dimensional simplicial complex
	with $d\geq k+2$, let $\calF$ be an $R$-sheaf
	on $X$   and let $\veps\in (0,1]$. Fix some $R$-orientation on $X$,\footnote{
		This choice is inconsequential by Proposition~\ref{PR:cse-not-dep-on-orientation}.	
	} and suppose that:
	\begin{enumerate}[label=(\arabic*)]
		\item[(1a)] $\cbe_{k-\dim u-1}(X_u,w_u,\calF_u)\geq \veps$   for every
		$u\in X(0)\cup\dots \cup X(k)$;
		\item[(1b)] $\cbe_{k-\dim u }(X_u,w_u,\calF_u)\geq \veps$ for every
		$u\in X(0)\cup\dots \cup X(k+1)$;
		\item[(2${}_\triangle$)] 
		for every $u\in X(-1)\cup\dots\cup X(k)$,
		the graph $\ugr(X_u,w_u)$
		is  an $((K\veps)^{2^{k-\dim u}},1)$-skeleton expander.
	\end{enumerate}
	Then
	\[
		\cse_k(\calF)\geq  K' (K\veps)^{2^{k+2}-1}
		\qquad\text{and}\qquad
		\ccd_k(\calF) \geq K' (K\veps)^{2^{k+1}-1}.
	\]
	Moreover, if $f\in C^k=C^k(X,\calF)$
	satisfies $\dist_w(f,Z^k)<K' (K\veps)^{2^{k+2}-1}$,
	then applying   Algorithm~\ref{AL:correction-algorithm-simplified}  
	to $f$ with the parameter $q=K(K\veps)^{2^{k+1}-1}$
	returns $f'\in Z'$ such that $\dist_w(f,f')\leq K'^{-1}(K\veps)^{-2^{k+1}}\dist_w(f,Z^k)$.
\end{thm}

\begin{proof}
	This follows
	from 
	Theorem~\ref{TH:lgp-simple-version} by taking
	$B=B'=1$, $L=1$, $F={k+2 \choose \ceil{(k+2)/2}}$,
	$\calP=\calP^{(k)}_{\triangle}$
	and $\calP'=\calP^{(k+1)}_{\triangle}$ (notation
	as in Example~\ref{EX:simp-intersection-prof}).
	Indeed,   $X$ is always $R$-oriented and lower-regular (i.e.\ $L(X)=1$),
	and  
	$F_{i,j}(X)\leq {k+2 \choose \ceil{(k+2)/2}}=:F$
	for all $-1\leq i\leq j\leq k+2$.
\end{proof}

\begin{remark}
	The only non-local assumption in Theorem~\ref{TH:lgp-simplicial-case}
	is that $\ugr(X,w)$ is a $((K\veps)^{2^{k+1}},1)$-skeleton expander
	(take $u=\emptyset_X$ in ($2_{\triangle}$)). This assumption can be replaced
	by a stronger local condition thanks to Oppenheim's Trickling Down Theorem
	\cite[Thm.~1.4]{Oppenheim_2018_local_spectral_expansion_I}.
	Specifically, we can fix some
	$0\leq \ell\leq d-2$
	and then replace ($2_{\triangle}$) in the case $u=\emptyset$ 
	with requiring that $\ugr(X_z,w_z)$
	is a $\lambda$-expander (see \S\ref{subsec:graphs}) every $z\in X(\ell)$,
	where   $\lambda$ is small enough so that $\frac{\lambda}{1-(\ell+1)\lambda}\leq (K\veps)^{2^{k+1}}$.
	Indeed, by Oppenheim's Theorem, this would guarantee that $\ugr(X,w)$ is a
	$ (K\veps)^{2^{k+1}}$-expander and hence a $( (K\veps)^{2^{k+1}},1)$-skeleton expander
	(Proposition~\ref{PR:eml-special-case}).
\end{remark}

For cube complexes, Theorem~\ref{TH:lgp-simple-version}
similarly specializes to the following theorem.

\begin{thm}[Cosystolic Expansion for Cube Complexes]
	\label{TH:lgp-cube-case}
	Let   $k\in \N$.
	Then there exist     constants $K,K' \in (0,1]$ 
	such that the following hold:
	Let $R$ be a commutative ring, 
	let $(X,w)$ be a properly weighted pure $d$-dimensional cube complex
	with $d\geq k+2$, let $\calF$ be an $R$-sheaf
	on $X$, let $\veps\in (0,1]$ and let $B'\in[1,\infty)$. 
	Fix an $R$-orientation on $X$. Suppose that conditions (1a)--(2${}_\triangle$) of
	Theorem~\ref{TH:lgp-simplicial-case} hold and in addition,
	\begin{enumerate}[label=(\arabic*)]
%%		the properly weighted graph $(X_u({\leq} 1),w_{X_u}|_{X_u({\leq} 1)})$   
		\item[(2${}_\square$)] 
		for every $i\in\{1,\dots,k+1\}$,
		$\NI^{i,i,i+1}(X,w)$
		is  an 
		$((K\veps)^{2^{k+1-i}},B'(K \veps)^{-2^{k+2-i}(2^{i}-1)})$-skeleton expander.
	\end{enumerate}
	Then
	\[
		\cse_k(\calF)\geq  B'^{-1} K' (K\veps)^{2^{k+2}-1}
		\qquad\text{and}\qquad
		\ccd_k(\calF) \geq B'^{-1} K' (K\veps)^{2^{k+1}-1}.
	\]
	Moreover, if $f\in C^k=C^k(X,\calF)$
	satisfies $\dist_w(f,Z^k)<B'^{-1} K' (K\veps)^{2^{k+2}-1}$,
	then applying   Algorithm~\ref{AL:correction-algorithm-simplified}  
	to $f$ with the parameter $q=K(K\veps)^{2^{k+1}-1}$
	returns $f'\in Z'$ such that $\dist_w(f,f')\leq K'^{-1}(K\veps)^{-2^{k+1}}\dist_w(f,Z^k)$.	
\end{thm}

By Lemma~\ref{LM:nig-to-nih}, the theorem also holds if we replace the hypergraph
$\NI^{i,i,i+1}(X)$ with the graph
$\NNI^{i,i,i+1}(X)$.

\begin{proof}
	Let $F_{i,j}$ be the number of $i$-faces a $j$-dimensional
	cube has.
	The theorem follows by applying Theorem~\ref{TH:lgp-simple-version}
	with $F=\max\{F_{0,k+2},\dots,F_{k,k+2}\}$, $L=1$, $B=1$,
	$\calP=\calP^{(k)}_{\square}$
	and $\calP'=\calP^{(k+1)}_{\square}$ (notation as in Example~\ref{EX:cube-intersection-prof}).
\end{proof}

We now restrict our attention to $0$-cocycle
codes. In this special case, the general form of our criterion
for cosystolic expansion (Theorem~\ref{TH:main-very-technical})
simplifies into the following result.

\begin{thm}[Criterion for $0$-Cosystolic Expansion]
	\label{TH:lgp-zero-detailed}
	For every $F\in\N$ and $L\in [1,\infty)$, there are (small) real  constants
	$E,E',E'',E''',D,D',D''>0$ such that the following hold:
	Let $R$ be a ring, let $(X,w)$ be a properly weighted
	$R$-oriented $d$-poset ($d\geq 2$) such that 
	$L(X)\leq L$ and $\Fmax_{0,2}(X), \Fmax_{1,2}(X)\leq F$, 
	and let $\calF$ be an $R$-sheaf on $X$.
	Let $\veps,\veps'\in\R_+$, $\alpha_0,\beta_0,
	\alpha_ {-1},\beta_{-1},\alpha_{||},\beta_{||}\in[0,\infty)$
	and suppose that:
	\begin{enumerate}[label=(\arabic*)]
		\item[(1a)]
		$\cbe_{-1}(X_v,w_v,\calF_v)\geq \veps$
		for every $v\in X(0)$;
		\item[(1b)] $\cbe_{-1}(X_e,w_e,\calF_e)\geq \veps'$ for every 
		$e\in X(1)$;
		\item[(1c)] $\cbe_0(X_v,w_v,\calF_v)\geq \veps'$ 
		for every $v\in X(0)$;
		\item[(2a)] 
		$ \ugr(X_v,w_v)$ is an $(\alpha_0,\beta_0)$-skeleton expander for all $v\in X(0)$;
		\item[(2b)] $\ugr(X,w) $ is an 	$(\alpha_{-1},\beta_{-1})$-skeleton expander;	
		\item[(2c)] $\NI^{1,1,2}(X,w)$ 
		is an $(\alpha_{||},\beta_{||})$-skeleton expander.	
	\end{enumerate}
	Suppose further that 
	\[
	\alpha_{-1} < E\veps
	\]
	and one can find $h_{-1},h_0,h_{||}\in (0,1]$ satisfying the   inequality 
	\begin{align*}
		(\alpha_0+\beta_0 h_0) + (\alpha_{||}+\beta_{||}h_{||})+\frac{\alpha_{-1}+\beta_{-1}h_{-1}}{h_0}
		\leq E' \veps'.
	\end{align*}
	Then 
	\[
	\qquad	
	\cse_0(X,w,\calF)\geq  \frac{E''}{h_0^{-1}h_1^{-1}+h_{||}^{-1}}
	\qquad
	\text{and}
	\qquad
	\ccd_0(X,w,\calF)\geq \frac{E'''(E\veps-\alpha_{-1})}{\beta_{-1}}.
	\]
	Moreover, 
	if $f\in C^0=C^0(X,\calF)$
	satisfies $\dist(f,Z^0)<\frac{D}{h_0^{-1}h_1^{-1}+h_{||}^{-1}}$,
	then applying   Algorithm~\ref{AL:correction-algorithm-simplified}  
	to $f$ with the parameter 
	$q=D'h_0$
	returns $f'\in Z'$ such that $\dist(f,f')\leq \frac{1}{D''h_0}\dist(f,Z^0)$.
	
	Explicit values  of $E,E',E'',E''',D,D',D''$ to which this applies are listed in Table~\ref{TB:lgp-zero-constants},
	both in general, and under some assumptions $X$.
\end{thm}

\begin{table}[h]
\centering
%%%%%%%%%%%%%% increase row strech in table
\bgroup
\def\arraystretch{1.2}
%%%%%%%%%%%%%%%
\begin{tabular}{|p{0.2\textwidth}|c|c|c|c|c|c|c|}
\hline
Assumption on $X$ & $E$ & $E'$ & $E''$ & $E'''$ & $D$ & $D'$ & $D''$   \\
\hline 
None & $L^{-8}$ & $\frac{1}{4}L^{-18}F^{-5}$ & $\frac{1}{2}L^{-13}F^{-2}$ & $L^{-3}$ 
	 & $L^{-15}F^{-1}$ & $\frac{1}{4}L^{-12}F^{-2}$ & $\frac{1}{4}L^{-16}F^{-3}$ \\
\hline
Lower-regular & $1$ & $\frac{1}{4}F^{-2}$ & $\frac{1}{2}F^{-1}$ & $1$ & 
           $F^{-1}$ & $\frac{1}{4}F^{-1}$ & $\frac{1}{4}F^{-2}$ \\
\hline
Simplicial complex & $1$ & $\frac{1}{12}$ & $\frac{1}{2}$ & $1$ &
     $\frac{1}{2}$ & $\frac{1}{4}$ & $\frac{1}{8}$ \\
\hline
Cube complex & $1$ & $\frac{1}{16}$ & $\frac{1}{2}$ & $1$ &
       $\frac{1}{2}$ & $\frac{1}{4}$ & $\frac{1}{8}$ \\
\hline 
$X({\leq} 2)$ is an $m$-gon complex &
$L^{-8}$ & $\frac{1}{4m}L^{-18}$ & $\frac{1}{2}L^{-13}$ & $L^{-3}$ &
$\frac{1}{2}L^{-15}$ & $\frac{1}{4}L^{-12}$ & $\frac{1}{8}L^{-16}$ 
\\
\hline
\end{tabular}
%%%%%%%%%%%%%%%%%%%%%%%%% cancel row strech
\egroup
%%%%%%%%%%%%%%%%%%%%%%%%
\medskip
\caption{Values for the constants of Theorem~\ref{TH:lgp-zero-detailed}.} 
\label{TB:lgp-zero-constants}
\end{table}

%%The implicit base
%%scheme $S$ is suppressed from the notation .}

\begin{remark}
Applying Theorem~\ref{TH:lgp-simple-version} with $k=0$
and the intersection profiles $\calP^{(0)}$ and $\calP^{(1)}$
of Example~\ref{EX:intersection-profiles} gives a similar, but less flexible
result.
%%Theorem~\ref{TH:lgp-zero-detailed} is more flexible than
\end{remark}

\begin{cor}\label{CR:ltc-for-sheaves-two-posets}
	Continuing Theorem~\ref{TH:lgp-zero-detailed},	
	suppose further that $\calF(\emptyset_X)=0$,
	that there is an $R$-module $\Sigma$ such that $\calF(v)=\Sigma$
	for every $v\in X(0)$ 
	and that there are $M,M'\in [1,\infty)$ such that $w(v)\leq Mw(v')$
	for all $v,v'\in X(0)$ and $w(e)\leq M'w(e')$ 
	for all $e,e'\in X(1)$. 
	Then the $0$-cocycle code
	$Z^0=Z^0(X,\calF)\subseteq \Sigma^{X(0)}
	$
	satisfies $\delta(Z^0)\geq  \frac{E'''(E\veps-\alpha_{-1})}{M\beta_{-1}}$
	and  its natural tester has soundness
	$ \frac{E''}{MM'(h_0^{-1}h_1^{-1}+h_{||}^{-1})}$.
	Moreover, Algorithm~\ref{AL:correction-algorithm-simplified}
	with $q=D'h_0$
	is a decoding algorithm  for words that are $
	\min\{\frac{D}{M(h_0^{-1}h_{-1}^{-1}+h_{||}^{-1})}, 
	\frac{D''E'''(E\veps-\alpha_{-1})}{M\beta_{-1}( h_0^{-1}+D'')})\}$-close to $Z^0$.
\end{cor}

Again, when $w$ is the natural weight function of $X$, we can take
$M=L_{0,d}U_{0,d}$ and $M'=L_{1,d}U_{1,d}$, by Proposition~\ref{PR:face-weight-ratio-via-irreg}.

\begin{proof}
	This follows from Theorem~\ref{TH:lgp-zero-detailed},
	Lemma~\ref{LM:cse-to-ltc} and Lemma~\ref{LM:decoding-of-cocycle-code}.
\end{proof}

In the course of proving   the general
version of our main result (Theorem~\ref{TH:main-very-technical}),
we also give  a criterion for bounding
the $k$-cocycle distance from below. In its most general form,
this is    Corollary~\ref{CR:ccd-lower-bound}
appearing later in the text.
For $0$-cocycles it simplifies into the following theorem, which we prove in \S\ref{subsec:reduction}.

\begin{thm}\label{TH:ccd-lower-bound}
	For every $F\in\N$ and $L\in[1,\infty)$, there are real  constants
	$E, E''' >0$ (the same constants as in Table~\ref{TB:lgp-zero-constants})
	such that the following hold:
	Let $R$ be a ring, let $(X,w)$ be a  properly weighted
	$R$-oriented\footnote{
		Actually, it is enough to require
		that only the subposet  $X(0)\cup X(1)$ is given an $R$-orientation --- see
		Remark~\ref{RM:partial-orientation-exp-loc-min} below. 
	}
	$d$-poset ($d\geq 1$) such that 
	$L(X)\leq L$ and
	every $1$-face of $X$ contains  at most $F$
	$0$-faces.
	Let $\calF$ be an $R$-sheaf on $X$ and let $\veps,  \alpha ,\beta \in[0,\infty)$. Suppose that 
	\begin{enumerate}[label=(\arabic*)]
		\item[(1)]
		$\cbe_{-1}(\calF_v)\geq \veps$
		for every $v\in X(0)$ and
		\item[(2)] $\ugr(X,w) $ is an $(\alpha,\beta )$-skeleton expander.
	\end{enumerate}
	Then $\ccd_0(X,w,\calF)\geq \frac{E'''(E\veps -\alpha)}{\beta}$.
\end{thm}

\begin{example}\label{EX:sheafy-expander-codes-distance}
	We apply Theorem~\ref{TH:ccd-lower-bound} to bound the distance of the $0$-cocycle
	code $Z^0(X,\calF)$ defined in \S\ref{subsec:expander-codes}.
	Recall that $X$ is a $k$-regular graph, $\F$ is a finite field, $m\in\Z\cap (\frac{k}{2},k]$
	and for every $v\in X(0)$,  we have an injective $\F$-linear map $T_v:\F^m\to \F^{X(1)_v}\cong \F^k$
	with image $C_v\subseteq \F^k$. The sheaf $\calF$ is defined by setting $\calF(v)=\F^m$
	and $\calF(e)=\F$
	for every $v\in X(0)$ and $e\in X(1)$, and $\res_{e\from v}(f)=(T_v(f))_e$ whenever
	$v< e$ and $f\in \F^m$.
	The $0$-cocycle code $Z^0=Z^0(X,\calF)$ has  alphabet $\Sigma=\F^m$.
	We think of $X$ as   $1$-poset by adding it  a $-1$-face $\emptyset$  and extend
	to that face $\calF$ by   setting $\calF(\emptyset)=0$.
	We give $X$ its natural weight function $w=w_{\nat}$, which coincides with the uniform
	weight function because $X$ is a regular graph.
	
	Suppose that there are $\delta_0>\lambda>0$
	such that $(X,w)$ is a $\lambda$-expander and $\delta(C_v)\geq \delta_0$ for all $v\in X(0)$.
	By Example~\ref{EX:cbe-dim-minus-one}, $\cbe_{-1}(X_v,w_{v},\calF_v)\geq \delta_0$
	for every $v\in X(0)$, and by 
	Example~\ref{EX:skeleton-exp}(i),	
	$(X,w)$ is a $(\lambda,1-\lambda)$-skeleton expander.
	Thus, by Theorem~\ref{TH:ccd-lower-bound}, $\ccd_0(X,w,\calF)\geq \frac{\delta_0-\lambda}{1-\lambda}$.
	Since $w=w_{\uni}$, it follows  
	that $\delta(Z^0)\geq \frac{\delta_0-\lambda}{1-\lambda}$. 
\end{example}

\section{$2$-Query LTCs from Sheaves on Square Complexes}
\label{sec:ltc-example}

In this section we apply
Theorem~\ref{TH:lgp-zero-detailed}  to certain sheaves on double Cayley complexes
in order to construct good $2$-query LTCs. Our LTCs turn out to be
the line codes of the good LTCs constructed in \cite{Dinur_2022_ltcs_const_rate},
so we can recover the     properties of the latter (with slightly different constants)
using the relations between a lifted code and its line
code established in Section~\ref{sec:lifted-vs-line}.
This offers a new perspective on the good LTCs of \cite{Dinur_2022_ltcs_const_rate},
showing how they can be neatly derived from our criterion for $0$-cocystolic
expansion.

\subsection{The Poset}
\label{subsec:ltc-poset}

The poset which we will use is a double Cayley complex --- a special kind of square
complex constructed as follows:
Let $G$ be a finite group and let $A,B\subseteq G-\{1_G\}$ be two generating sets
for $G$   such that $A=A^{-1}$,
$B=B^{-1}$
and  
\begin{equation}\label{EQ:tnc}
gag^{-1}\neq b\qquad\forall a\in A,b\in B,g\in G.
\end{equation}
The double Cayley complex of $G$, $A$, $B$ is the square
complex $X=\mathrm{Cay}(A,G,B)$ constructed as follows:
\begin{itemize}
	\item $X(0)=\{\{g\}\where g\in G\}$,
	\item $X(1)=\{\{g,ag\}\where g\in G,\, a\in A\}\cup \{\{g,gb\}\where g\in G,\, b\in B\}$,
	\item $X(2)=\{\{g,ag,gb,agb\}\where g\in G,\,a\in A,\,b\in B\}$.
\end{itemize}
We also set $X(-1)=\{\emptyset\}$ and endow $X$ with the containment relation.
Condition \eqref{EQ:tnc} guarantees that this is
indeed a square complex.
Moreover, it implies that if $e\in X(1)$ and $\{g\}$ ($g\in G$) is a vertex of $e$,
then either $e=\{g,ag\}$ with unique $a\in A$,
or $e=\{g,gb\}$ with unique $b\in B$ (but not both).
Furthermore, if $e=\{g,ag\}$ ($g\in G$, $a\in A$) 
is an edge contained in a square
$s$, then there is a unique $b\in B$ such
that $s=\{g,ag,gb,agb\}$. Likewise,
if  $e=\{g,gb\}\leq s\in X(2)$, then there is  a unique $a\in A$
such that $s=\{g,ag,gb,agb\}$.

We give $X$ its natural weight function (Example~\ref{EX:natural-weight}) and denote it
by $w$.
Explicitly, we have $w(\{g,ag,gb,agb\})=\frac{1}{|X(2)|}=\frac{4}{|A||B||G|}$,
$w(\{g,ag\})=\frac{1}{|A||G|}$, $w(\{g,gb\})=\frac{1}{|B||G|}$ and
$w(\{g\})=\frac{1}{|G|}$ for all $g\in G$, $a\in A$, $b\in B$.

Since $X$ is a square complex, and in particular a regular cell complex,
there is a $\Z$-orientation $[:]$ on $X$ such that $[v:\emptyset]=1$
for every $v\in X(0)$; see Example~\ref{EX:orientation-of-cell-complexes}.
We fix such a $\Z$-orientation on $X$; it induces an $R$-orientation on 
$X$ for every commutative ring $R$.

\subsection{The Sheaf}
\label{subsec:ltc-sheaf}

Let $G$, $A$, $B$, $X=\mathrm{Cay}(A,G,B)$ be as above.
Fix a finite field $\F$ and   let $C_A\subseteq \F^A$
and $C_B\subseteq \F^B$ be \emph{linear} codes with alphabet $\F$.

It will be convenient to view $\F^A\otimes \F^B$ (all tensor products are taken over $\F$)
as the space of matrices with rows indexed by $A$ and columns index by $B$,
denoted $\nMat{\F}{A\times B}$. (Explicitly, for $u\in \F^A$ and $v\in \F^B$, 
the tensor $u\otimes v$ corresponds to $(u_av_b)_{a\in A,b\in B}\in \nMat{\F}{A\times B}$.)
With this interpretation,
the subspace $C_A\otimes C_B$ of $\F^A\otimes \F^B$ is 
\[
\{m\in \nMat{\F}{A\times B}\suchthat 
\text{$r_a(m)\in C_B$ and $c_b(m)\in C_A$ for all $a\in A$, $b\in B$}\},
\]
where, as before, $r_a(\cdot )$ and $c_b(\cdot )$ mean taking the $a$-th row and $b$-th column, respectively.
It is well-known that the code $C_A\otimes C_B\subseteq \nMat{\F}{A\times B}$
satisfies $r(C_A\otimes C_B)=r(C_A)r(C_B)$ and $\delta(C_A\otimes C_B)=\delta(C_A)\delta(C_B)$.

We   define an $\F$-sheaf $\calF$ on $X$ as follows.
For every $g\in G$, $a\in A$, $b\in B$, set
\begin{itemize}
	\item $\calF(\emptyset)=0$,
	\item $\calF(\{g\})=C_A\otimes  C_B$,
	\item $\calF(\{g,ag\})=C_B$,
	\item $\calF(\{g,gb\})=C_A$,
	\item $\calF(\{g,ag,gb,agb\})=\F$,
	\item $\res_{\{g,ag\}\from \{g\}}=r_a :C_A\otimes  C_B\to C_B$,
	\item $\res_{\{g,gb\}\from \{g\}}=c_b:C_A\otimes  C_B\to  C_A$,
	\item $\res_{\{g,ag,gb,agb\}\from \{g,ag\}} :C_B\to \F$ sends $v$ to $v_b$,
	\item $\res_{\{g,ag,gb,agb\}\from \{g,gb\}}:C_A\to \F$ sends $u$ to $u_a$,
	\item $\res_{\{g,ag,gb,agb\}\from\{g\}}:C_A\otimes C_B\to \F$
	sends $m$ (an $A\times B$-matrix) to $m_{a,b}$.
\end{itemize}

\begin{lem}
	With notation as above,
	$\calF$ is a well-defined sheaf on $X$.
\end{lem}

\begin{proof}
	Recall that condition~\eqref{EQ:tnc}
	guarantees that that for every $v=\{g\}\in X(0)$ and $e\in X(1)$,
	either $e=\{g,ag\}$ for a unique $a$,
	or $e=\{gb,g\}$ for a unique $b$.
	This shows that the restriction map $\res^{\calF}_{e\from v}$
	is well-defined. Similarly, all the restriction maps are well-defined.
	Since $\calF(\emptyset)=0$, it remains to check
	that if $s\in X(2)$, $e\in s(1)$ and $v\in e(0)$, then 
	$\res_{s\from v}=\res_{s\from e}\circ\res_{e\from v}$.
	Writing $g=\{g\}$ with $g\in G$, suppose that
	$e=\{g,ag\}$ for $a\in A$. Then there is $b\in B$
	such that $s=\{g,ag,gb,agb\}$. Now, for every $m\in C_A\otimes C_B$,
	we have
	$\res_{s\from e}\res_{e\from v}(m)=(r_a(m))_b=m_{a,b}=\res_{s\from v}(m)$,
	as required. The case $e=\{g,bg\}$ with $b\in B$ is handled similarly.
\end{proof}

\subsection{The Code and Its Tester}
\label{subsec:ltc-code}

Keeping the previous notation,
put $\Sigma=C_A\otimes C_B$ and observe that $\calF(v)=\Sigma$
for every $v\in X(0)$.
We may therefore form the $0$-cocycle code 
\[ Z^0=Z^0(X,\calF)\subseteq \Sigma^{X(0)}=\Sigma^G.\]
Recall that we chose the $\F$-orientation on $X$ in such a way that
$[v:\emptyset]=1$ for every $v\in X(0)$. This implies that
for every edge $e\in X(1)$ with vertices $u$ and $v$,
we have $[e:v]=-[e:u]$.
As a result,   $Z^0$ consists of the words $f=(f(g))_{g\in G}\in \Sigma^G=(C_A\otimes C_B)^G$
which satisfy
\[
\res_{\{g,ag\}\from \{g\}} f(g)=\res_{\{g,ag\}\from \{ag\}} f(ag)
\qquad \text{and}
\qquad
\res_{\{g,gb\}\from \{g\}} f(g)=\res_{\{g, gb\}\from \{gb\}} f(gb)
\]
for all $a\in A$, $b\in B$.
Since $\{g,ag\}=\{ag,a^{-1}(ag)\}$
and $\{g,gb\}=\{gb,gb(b^{-1})\}$,
this is equivalent to
\[
r_a(f(g))=r_{a^{-1}}(f(ag))\qquad\text{and}\qquad
c_b(f(g))=c_{b^{-1}}(f(gb)),
\]
respectively. 
This can be further restated as saying that $f(g)_{a,b}=f(ag)_{a^{-1},b}$
and $f(g)_{a,b}=f(gb)_{a,b^{-1}}$ for all $a\in A$, $b\in B$, $g\in G$.
We conclude that $Z^0$ may be viewed as the space of ensembles of matrices $(m_g)_{g\in G}\in \Sigma^G$
such that
\begin{equation}\label{EQ:defining-constraints-of-code}
(m_g)_{a,b}=(m_{ag})_{a^{-1},b}=(m_{gb})_{a,b^{-1}}\qquad\forall a\in A,b\in B,g\in G.
\end{equation}

From the above description, we see that the natural tester of $Z^0$ operates as follows:
Given $f=\{f(g)\}_{g\in G}\in \Sigma^G\subseteq \nMat{\F}{A\times B}^G$,
choose $g\in G$ and  $x\in A\sqcup B$ uniformly at random. When
$x\in A$, accept $f$ if and only if
$r_x(f(g))=r_{x^{-1}}(f(xg))$ and when $x\in B$,
accept $f$ if and only if $c_x(f(g))=c_{x^{-1}}(f(gx))$.

\medskip

Now that we have described the code $Z^0\subseteq \Sigma^G$,
we bound its rate from below.

\begin{prp}\label{PR:rate}
	Let $G,A,B,X,C_A,C_B,\calF$ be as above  and write $r_A=r(C_A)$
	and $r_B=r(C_B)$.
	Then $r(Z^0(X,\calF))\geq \frac{4r_A r_B-3}{4r_A r_B}$.
\end{prp}

\begin{proof}
	By replacing  $g,a,b$ with $agb,a^{-1},b^{-1}$ in \eqref{EQ:defining-constraints-of-code},
	we see that $Z^0 $ is the subspace of $\Sigma^G$
	defined by the constraints
	\begin{equation}\label{EQ:defining-constraints-of-code-II}
	(m_g)_{a,b}=(m_{ag})_{a^{-1},b}=(m_{gb})_{a,b^{-1}}=(m_{agb})_{a^{-1},b^{-1}}.
	\end{equation}
	for all $g\in G$, $a\in A$, $b\in B$.
	One readily checks that \eqref{EQ:defining-constraints-of-code-II} depends
	only on the square $\{g,ag,gb,agb\}$ and not on $g,a,b$.
	Thus, $Z^0 $ is defined by $3|X(2)|=\frac{3|G||A||B|}{4}$
	linear constraints inside $\Sigma^G$.
	As $\dim \Sigma^G = r_Ar_B|A||B||G|$,
	it follows that $\dim Z^0 \geq |A||B||G|(r_A r_B-\frac{3}{4})=\dim \Sigma^G\cdot
	(1-\frac{3}{4r_A r_B})$
	and the proposition follows.
\end{proof}

\subsection{Verifying The Expansion Conditions}
\label{subsec:interpretation}

Our goal now is to apply
Corollary~\ref{CR:ltc-for-sheaves-two-posets} to $X=\mathrm{Cay}(A,G,B)$
and the sheaf $\calF$ we constructed.
To that end, we   show that
for our $(X,\calF)$,   conditions (1a)--(2c) of 
Theorem~\ref{TH:lgp-zero-detailed}
can be derived from properties of the codes $C_A$, $C_B$ and
the one-sided Cayley graphs $\mathrm{Cay}(A,G)$ and $\mathrm{Cay}(G,B)$.

We begin with   (1a) and (1b). Recall that $\cbe_i(\calF)$ means $\cbe_i(X,w_{\nat},\calF)$.

\begin{lem}\label{LM:one-a-alt}
	With notation as in \S\ref{subsec:ltc-sheaf}, 
	let $g\in G$, $a\in A$ and $b\in B$. Then
	\begin{enumerate}[label=(\roman*)]
		\item $\cbe_{-1}(\calF_{\{g,ag\}})=\delta(C_B)$,
		\item $\cbe_{-1}(\calF_{\{g,gb\}})=\delta(C_A)$,
		\item $\cbe_{-1}(\calF_{\{g\}})= \frac{1}{2}(\delta(C_A)+\delta(C_B))$.
	\end{enumerate}
\end{lem}

\begin{proof}
	(i) 
	Write $e=\{g,ag\}$.
	Every  $2$-face in $X$ containing $e$
	is of the form  $s_b:= \{g,ag,gb,agb\}$ for a unique $b\in B$.
	Recall that $\calF(\{g,ag\})=C_B$,
	and for every $b\in B$, 
	$\calF(s_b)=\F$  and $\res_{s_b\to e} : C_B\to \F$ is the projection onto the $b$-component.
	This means that the map $f\mapsto (\res_{s_b\from e} (f))_{b\in B}:
	\calF(e)\to 
	\prod_{b\in B} \calF(s_b)$ is just the inclusion map $C_B\to \F^B$. 
	The natural weight function on the $0$-poset $X_e$
	coincides with its uniform weight function,
	so, as noted in Example~\ref{EX:cbe-dim-minus-one}, $\cbe_{-1}(\calF_{\{g,ag\}})=\delta(C_B)$.

	(ii) This is similar to (i).
	
	(iii) Write $v=\{g\}$ and $w_v$ for the natural weight function
	of $X_v$. There are $|A|+|B|$ edges
	containing $v$, namely, $\{e_a:=\{g,ag\}\}_{a\in A}$
	and $\{e_b:=\{g,gb\}\}_{b\in B}$.
	One readily checks that $w_v(e_a)=\frac{1}{2|A|}$ 
	and $w_v(e_b)=\frac{1}{2|B|}$ for all $a\in A$
	and $b\in B$.
	Let $f\in \calF(v)-\{0\}=C_A\otimes C_B - \{0\}$,
	and let $A_0=\{a\in A\suchthat \res_{e_a\from v}f\neq 0\}$
	and $B_0=\{b\in B\suchthat \res_{e_b\from v}f\neq 0\}$.
	By Example~\ref{EX:cbe-dim-minus-one},
	we need to show that $w_v(A_0\cup B_0)\geq \frac{1}{2}(\delta(C_A)+\delta(C_B))$, and that equality
	is attained for some choice of $f$.
	
	Recall that we view $f$ as a matrix in $\nMat{\F}{A\times B}$
	with $r_a(f)\in C_B$ and $c_b(f)\in C_A$ for all $a\in A$, $b\in B$.
	Then  $A_0=\{a\in A\suchthat r_a(f)\neq 0\}$
	and $B_0=\{b\in B\suchthat c_b(f)\neq 0\}$.
	Since $f\neq 0$, there are $a_0\in A$ and $b_0\in B$
	such that $f_{a_0,b_0}\neq 0$. This means that
	$r_{a_0}(f)\in C_B-\{0\}$, so at least $\delta(C_B)|B|$
	entries in the $a_0$-th row of $f$ are nonzero. As a result,
	$|B_0|\geq \delta(C_B)|B|$.
	Similarly, $|A_0|\geq \delta(C_A)|A|$
	and it follows that
	$w_v(A_0\cup B_0)=w_v(A_0)+w_v(B_0)\geq \frac{1}{2}(\delta(C_B)+\delta(C_A))$.
	To see that equality can be attained, choose
	$f_A\in C_A$, $f_B\in C_B$
	with  $\|f_A\|=\delta(C_A)$
	and $f_B\in C_B$ with $\|f_B\|=\delta(C_B)$,  and take $f=f_A\otimes f_B$
	(i.e.,  the matrix $(f_A(a)f_B(b))_{a\in A, b\in B}$). 
\end{proof}

We proceed with (1c).

\begin{lem}
	With notation as above, let $g\in G$ and $\kappa\in [0,\infty)$.
	Then $\cbe_0(\calF_{\{g\}})\geq \kappa$
	if and only if $C_A\otimes C_B$ is $\kappa$-agreement testable
	(Example~\ref{EX:agreement-test-for-tensor-codes}).
\end{lem}

\begin{proof}
	By Proposition~\ref{PR:cse-not-dep-on-orientation},
	changing the $\F$-orientation of $X_{\{g\}}$ is harmless.
	We therefore choose a $\Z$-orientation on $X_{\{g\}}$
	such that $[v:\{g\}]=1$ for every $v\in X_{\{g\}}(0)$;
	this is possible by Example~\ref{EX:orientation-of-cell-complexes}.
	
	Observe that $\ugr(X_{\{g\}})$ may be   identified with the complete bipartite graph
	on $A$ and $B$ --- simply   map the vertices $\{g,ga\}$ 
	and $\{g,bg\}$ to $a$ and $b$, respectively,
	and the edge $\{g,ag,gb,agb\}$ to the dege $\{a,b\}$ ($a\in A$, $b\in B$).
	It is now routine to check 
	using  Examples~\ref{EX:cbe-dim-zero}
	and~\ref{EX:agreement-test-for-tensor-codes} that 
	$\cbe_0(\calF_{\{g\}})\geq \kappa$
	if and only if $C_A\otimes C_B$ is $\kappa$-agreement testable.
\end{proof}

Condition (2a) of Theorem~\ref{TH:lgp-zero-detailed} holds automatically for $X$.

\begin{lem}
	With notation as above, for every $g\in X$, the graph $\ugr(X_{\{g\}},w_{\{g\}})$ 
	is a $(0,1)$-skeleton expander.
\end{lem}

\begin{proof}
	The graph $X_{\{g\}}$ is the complete bipartite graph
	on the sets $\{\{g,ag\}\where a\in A\}$
	and $\{\{g,gb\}\where b\in B\}$ and $w_{\{g\}}$
	is its natural weight function. It is well-known that
	such a graph is a $0$-expander, so the lemma
	follows from Proposition~\ref{PR:eml-special-case}. 
\end{proof}

In order  to secure (2b) and (2c), we 
need to require that the Cayley graphs $\mathrm{Cay}(A,G)$
and $\mathrm{Cay}(G,B)$
are $\lambda$-expanders (\S\ref{subsec:graphs}).

\begin{lem}\label{LM:twobc-alt}
	With notation as above, suppose 
	that both $\mathrm{Cay}(A,G)$
and $\mathrm{Cay}(G,B)$
are $\lambda$-expanders ($\lambda\geq 0$).
	Then:
	\begin{enumerate}[label=(\roman*)]
		\item $\ugr(X,w)$ is a $\lambda$-expander
		and a  $(\lambda,1)$-skeleton expander.
		\item  $\NNI^{1,1,2}(X,w)$ and $\NI^{1,1,2}(X,w)$
		are $(2\lambda ,4\max\{|A|,|B|\})$--skeleton expanders.
	\end{enumerate}
\end{lem}

\begin{proof}
	(i) Let  $\calA_A$ 
	and $\calA_B$ denote the weighted
	adjacency operators of
	$\mathrm{Cay}(A,G)$
and $\mathrm{Cay}(G,B)$ endowed with their natural weight functions,
	and let $\calA$ denote the weighted adjacency operator of $\ugr(X,w)$ (\S\ref{subsec:graphs}).
	One readily checks that $\calA_A\calA_B=\calA_B\calA_A$
	and $\calA=\frac{1}{2}(\calA_A+\calA_B)$.
	The former means that $\calA_A$ and $\calA_B$
	can be simultaneously diagonalized.
	Thus, every eigenvalue $\mu$ of $\calA$ on $C^0_{\circ}(X ,\R)$
	is of the form $\frac{1}{2}(\mu_A+\mu_B)$
	where $\mu_A$, $\mu_B$
	are eigenvalues of $\calA_A$, $\calA_B$ on the same space.
	As   $\mu_A,\mu_B\in [-1,\lambda]$,
	we conclude that $\mu\leq \lambda$.
	This shows that $\ugr(X,w)$ is a $\lambda$-expander, and so it is also
	a $(\lambda,1)$-skeleton expander by Proposition~\ref{PR:eml-special-case}.
	
	\smallskip
	
	(ii) 
	By Lemma~\ref{LM:nig-to-nih}, 
	it is enough to prove the statement about
	$H:= \NNI^{1,1,2}(X)$. Recall (Definition~\ref{DF:nig}) that $w$
	induces a weight function on $H$; we denote it by $w_H$.
	In fact,   $H$ is the disjoint
	union of two subgraphs,
	$H_A$ and $H_B$, having 
	vertex sets  
	$\{\{g,ag\}\where a\in A,g\in G\}$
	and  
	$\{\{g,gb\}\where b\in B,g\in G\}$, respectively. We write $w_A$  
	for the natural weight function  of $H_A$ and similarly for $H_B$.
	We will show that $H$ is a $(2\lambda ,4\max\{|A|,|B|\})$--skeleton expander in two steps.
	
	\smallskip
	
	\Step{1} We claim that $(H_A,w_A)$ is
	a $(\lambda,|A|)$-skeleton expander.
	Likewise, $(H_B,w_B)$
	is a $(\lambda,|B|)$-skeleton expander.	
	
	To see this, let $H'_A$ be the graph
	with vertex set $\{(g,ag)\where g\in G,a\in A\}$
	and edges 
	\[\{\{(g,ag),(gb,agb)\}\where g\in G,a\in A,b\in B\} \]
	and let $w'_A$ be its natural weight function.
	The map $p:H'_A\to H_A$ sending
	$(g,ag)$ to $\{g,ag\}$
	and
	$\{(g,ag),(gb,agb)\}$ to $\{g,ag,gb,agb\}$
	is a double covering of $H_A$.
	This means that $w'_A(p^{-1}(U))=w_A(U)$ for every
	$U\subseteq H_A$ and $p^{-1}(E_{H_A}(V))=E_{H'_A}(p^{-1}(V))$
	for every $V\subseteq H_A(0)$.	Consequently,
	if $(H'_A,w'_A)$ is an $(\alpha,\beta)$-skeleton expander, then so is 
	$(H_A,w_A)$.
	However, the graph $H'_A$ is the disjoint union of $|A|$
	copies of $\mathrm{Cay}(G,B)$, which is
	a $(\lambda,1)$-skeleton expander (Proposition~\ref{PR:eml-special-case}).
	By Example~\ref{EX:skeleton-exp}(iv),
	this means that $(H'_A,w'_A)$ is a $(\lambda,|A|)$-skeleton expander, hence our claim.

	\smallskip
	
	\Step{2} We now prove (ii).
	The graph $ H_A $ is $|B|$-regular, so
	for every $e\in H_A(0)$
	and $s\in H_A(1)$, we have
	$w_A(e)=\frac{1}{|H_A(0)|}=\frac{2}{|A||G|}$
	and $w_A(s)=\frac{1}{|H_A(1)|}=\frac{4}{|A||B||G|}$. On the other hand, by unfolding the definition of $w_H$,
	one finds that
	$w_H(e)=\frac{1}{|A||G|}=\frac{1}{2}w_A(e)$ and $w_H(s)=\frac{4}{|A||B||G|}=w_A(s)$.
	Now, by 
	Step~1, for every $U\subseteq H_A(0)$
	with $w_H(U)=\alpha$, we have
	\[
	w_H(E_H(U))=w_A(E_{H_A}(U))
	\leq |A| w_A(U)^2+\lambda w_A(U)=
	4|A| w_H(U)^2+2\lambda w_H(U)=4|A| \alpha^2+2\lambda\alpha.
	\]
	Similarly, for every $V\subseteq H_B(0)$ with $w_H(V)=\beta$,
	we have
	\[
	w_H(E_H(V))\leq 4|B|\beta^2+2\lambda\beta.
	\]
	
	Finally, let $Z\subseteq H(0)$ with $w_H(Z)=\gamma$, 
	and put $U=Z\cap H_A(0)$ and $V=Z\cap H_B(0)$.
	Then, with $\alpha$ and $\beta$ as before,
	we have
	\begin{align*}
	w_H(E_H(Z))&=w_H(E_{H}(U)\cup E_{H}(V))
	\leq w_H(E_{H}(U))+ w_H(E_{H}(V))
	\\
	&\leq 4(|A|\alpha^2+|B|\beta^2)+2\lambda(\alpha+\beta)
	\leq 4\max\{|A|,|B|\}\gamma^2 + 2\lambda\gamma,
	\end{align*}
	which is what we want.
\end{proof}

\subsection{Constructing $2$-Query LTCs}

We finally plug   our previous observations
into Corollary~\ref{CR:ltc-for-sheaves-two-posets} to get the following theorem,
which implies
Theorem~\ref{TH:ltc-intro} from the introduction.

\begin{thm}\label{TH:ltc-square}
	Let $G,A,B,X,\F,C_A,C_B,\calF$
	be as in \S\S\ref{subsec:ltc-poset}--\ref{subsec:ltc-sheaf} 
	and let $\veps\in (0,1]$.
	Suppose that    the following conditions are met:
	\begin{enumerate}
		\item[(1a$'$)] $\delta(C_A)\geq \veps$,
		\item[(1b$'$)] $\delta(C_B)\geq \veps$,
		\item[(1c$'$)] $C_A\otimes C_B$ is $\veps$-agreement testable,
		\item[(2\,$'$)] the Cayley graphs $\mathrm{Cay}(A,G)$ and $\mathrm{Cay}(G,B)$
		are $\frac{\veps^2}{2500}$-expanders w.r.t.\ their natural weight functions.
	\end{enumerate}
	View $Z^0=Z^0(X,\calF)$ as a code inside $C^0(X,\calF)= \Sigma^G$,
	where $\Sigma=C_A\otimes C_B$.
	Then
	\[
	\delta(Z^0)\geq 0.9996\veps,
	\qquad
	r(Z^0)\geq \frac{4r(C_A)r(C_B)-3}{4r(C_A)r(C_B)},
	\]	
	and the natural $2$-query tester of $Z^0$ has soundness
	\[
	\frac{1}{\max\{\frac{|A|}{|B|},\frac{|B|}{|A|}\}[
	250000\veps^{-3} + 4800\max\{|A|,|B|\}\veps^{-1}]}
	.
	\]
	Moreover, setting
	$\eta=\frac{1}{250000\veps^{-3}+4800\max\{|A|,|B|\}\veps^{-1}}$, 
	Algorithm~\ref{AL:correction-algorithm-simplified}
	with $q=\frac{\veps}{100}$ is a correction algorithm for words 
	that are $\eta$-close to $Z_0$.
\end{thm}

When the distance and soundness of $Z^0$,
as well as  the   required expansion
of $\mathrm{Cay}(A,G)$ and $\mathrm{Cay}(G,B)$, are viewed
as   functions of $\veps$, $|A|$, $|B|$, their order of magnitude
is the best we can get
by using Corollary~\ref{CR:ltc-for-sheaves-two-posets}.
However, we did not attempt to optimize the constants.
We also remark that as $|X(0)|=|G|$ grows, $|A|$ and $|B|$ must be $\Omega(\veps^{-4})$ for   (2$'$)
to hold,
because by the Alon--Boppana Theorem
\cite{Nilli_1991_Alon_Boppana_thm}, the second-largest normalized eigenvalue
of $\mathrm{Cay}(A,G)$
is 
$\frac{2\sqrt{|A|-1}}{|A|}-o(1)$,
and likewise for $\mathrm{Cay}(G,B)$.

\begin{proof}
	The claim about the rate is Proposition~\ref{PR:rate}.
	Let $\lambda=\frac{\veps^2}{2500}$.
	Assumptions (1a$'$)--(2$'$) and the lemmas in \S\ref{subsec:interpretation}
	imply that assumptions (1a)--(1c) of Theorem~\ref{TH:lgp-zero-detailed}
	hold, and in addition,
	\begin{itemize}
		\item $\ugr(X_v,w_v)$ is an $(0,1)$-skeleton expander for all $v\in X(0)$,
		\item $\ugr(X,w)$ is a $(\lambda,1)$-skeleton expander,
		\item $\NI^{1,1,2}(X)$ is a $(2\lambda,4\max\{|A|,|B|\})$-skeleton expander. 
	\end{itemize}
	Observe also that $w(u)=w(v)$ for every $u,v\in X(0)$
	and $w(e)\leq \max\{\frac{|A|}{|B|},\frac{|B|}{|A|}\}w(e')$
	for every $e,e'\in X(1)$.
	We may therefore apply Corollary~\ref{CR:ltc-for-sheaves-two-posets} (with the constants
	$E,E',\dots$ taken from the fourth row of Table~\ref{TB:lgp-zero-constants},
	$M=1$ and $M'=\max\{\frac{|A|}{|B|},\frac{|B|}{|A|}\}$)
	for any $h_0,h_{-1},h_{||}\in \R_+$ such that
	\[
	h_0+2\lambda+4\max\{|A|,|B|\}h_{||} + \frac{\lambda+h_{-1}}{h_0}\leq E'\veps=\frac{\veps}{16}.
	\]
	(Note that the requirement $\lambda<E\veps=\veps$ holds automatically.)
	Our theorem is obtained by choosing
	$h_0=\sqrt{\lambda}=\frac{\veps}{50}$, $h_{-1}=\lambda=\frac{\veps^2}{2500}$ and $h_{||}=\frac{\veps}{2400\max\{|A|,|B|\}}$. This works because
	$h_0+2\lambda+4\max\{|A|,|B|\}h_{||} + \frac{\lambda+h_{-1}}{h_0}
	=3\sqrt{\lambda}+2\lambda+\frac{\veps}{600}\leq 3\cdot\frac{\veps}{50}+2\cdot\frac{\veps}{2500} + \frac{\veps}{600}<\frac{\veps}{16}$.
\end{proof}

In order to get a good $2$-query LTC from Theorem~\ref{TH:ltc-square}, 
it remains to show that it
can be applied to an
infinite family of $G,A,B,C_A,C_B$ satisfying assumptions (1a$'$)--(2$'$). 
The existence of a suitable family
has been shown in \cite[\S5--6]{Dinur_2022_ltcs_const_rate}, but we recall some details
for the sake of completeness.
Specifically, we will show the following.

\begin{thm}\label{TH:existence-of-good-codes}
	For every $r>0$ and finite field $\F$, there are 
	$m\in\N$ and $\veps>0$ for which there
	exist:
	\begin{enumerate}[label=(\roman*)]
		\item a sequence of groups $\{G_i\}_{i\in\N}$ with $|G_i|\to \infty$,
		\item symmetric generating subsets $A_i,B_i\subseteq G_i-\{1\}$ (for every $i\in\N$)
		satisfying \eqref{EQ:tnc},
		$|A_i|=|B_i|=m$, 
		and such that $\mathrm{Cay}(A_i,G_i)$ and $\mathrm{Cay}(G_i,B_i)$
		are $\frac{\veps^2}{2500}$-expanders,
		\item a linear code $C_0\subseteq \F^m $ such that $r(C_0)\geq r$, $\delta(C_0)\geq \veps$
		and the tensor code $C_0\otimes C_0$ is   $\veps$-agreement testable.
	\end{enumerate}
\end{thm}

Thus, if $r>\frac{3}{4}$ and $\calF_i$ is the sheaf
on $X_i:=\mathrm{Cay}(A_i,G_i,B_i)$ constructed in 
\S\ref{subsec:ltc-sheaf} with $A_i,B_i,G_i$ in place of $A,B,G$	
and $C_{A_i}=C_{B_i}=C_0$,  then 
Theorem~\ref{TH:ltc-square} tells us that the family	
\[\{Z_i(X_i,\calF_i)\subseteq (C_0\otimes C_0)^{G_i}\}_{i\geq 0}  \]
is a $2$-query LTC with alphabet $\Sigma=C_0\otimes C_0$. Moreover, there is
$\eta>0$ such that this family admits a linear-time decoding algorithm
for words that are $\eta$-close to the code.

\medskip

Theorem~\ref{TH:existence-of-good-codes} is shown
by combining two results. 

\begin{lem}[{\cite[Lem~.5.1]{Dinur_2022_ltcs_const_rate}}]
	\label{LM:existence-of-good-base-code}
	For every $0<r<1$ and finite field $\F$, there are $\delta_0,\kappa_0>0$
	and $d_0\in\N$ such that for any $D\in\N$ divisible
	by $d_0$, there exists a linear code $C_0\subseteq\F^D$
	with $r(C_0)\geq r$, $\delta(C_0)\geq \delta_0$
	and such that $C_0\otimes C_0$ is 
	$\kappa_0$-agreement testable.
\end{lem}

The proof in \cite{Dinur_2022_ltcs_const_rate} consists of showing that a random \emph{LDPC code} will satisfy all the requirements
(for suitable $\delta_0$, $\kappa_0$)
with positive probability as the block length of the code grows.
It   is written under the assumption that $\F=\F_2$, but works for every finite field $\F$.

\begin{lem}[{\cite[Lem~.5.2]{Dinur_2022_ltcs_const_rate}}]
	\label{LM:existence-of-good-group}
	Let $d_0\in\N$, let $q$ be an odd prime number with $q\geq d_0^2$,
	and
	let $D=d_0\floor{\frac{q+1}{d_0}}$.
	Then  for every $i\in\N$, there is an explicit group $G_i$ of size
	$\Theta(q^{3i})$ admitting two  
	symmetric generating sets $A_i,B_i\subseteq G_i-\{1\}$ of size $D$
	which satisfy   \eqref{EQ:tnc}
	and such that $\mathrm{Cay}(A_i,G_i)$ and $\mathrm{Cay}(G_i,B_i)$
	are $4 D^{-1/2}$-expanders.  
\end{lem}

This is shown using known constructions of \emph{Ramanujan graphs}.

\begin{proof}[Proof of Theorem~\ref{TH:existence-of-good-codes}]
	Recall
	that we are given $0<r<1$ and a finite field $\F$. 
	Let $\delta_0,\kappa_0$ and $d_0$ be as in Lemma~\ref{LM:existence-of-good-base-code},
	and put $\veps=\min\{\delta_0,\kappa_0\}$.
	Choose a prime number $q$ sufficiently large so that $q \geq d_0^2$
	and
	$4D^{-1/2}\leq \frac{\veps^2}{2500}$, where $D$ is as in Lemma~\ref{LM:existence-of-good-group}.
	Having fixed such a $q$, let $G_i,A_i,B_i$
	be the family promised by that lemma.
	Take $m=D$.
	Since $m$ is divisible by $d_0$, Lemma~\ref{LM:existence-of-good-base-code} 
	supplies us with a code $C_0\subseteq \F^m$ such that $r(C_0)\geq r$, $\delta(C_0)\geq \veps$
	and $C_0\otimes C_0$ is $\veps$-agreement testable. This is exactly what we want.
\end{proof}

\subsection{Realization as a Line Code}

\label{subsec:square-line-code}

Let $G,A,B,X,\F,C_A,C_B,\calF,\Sigma=C_A\otimes C_B$ be as in \S\S\ref{subsec:ltc-poset}--\ref{subsec:ltc-code}.
We finish this section by
showing that   $Z^0(X,\calF)\subseteq \Sigma^G$ is in fact
the line code of a linear lifted code $C(A,G,B)\subseteq \F^{X(2)}$
that was constructed in \cite{Dinur_2022_ltcs_const_rate} (in the case
$\F=\F_2$).
The main result of \cite{Dinur_2022_ltcs_const_rate} states that under conditions resembling those
of Theorem~\ref{TH:ltc-square}, $C(A,G,B)\subseteq \F^{X(2)}$ is a good LTC. 
We shall recover this result  (with slightly
different parameters)
by applying  the results of Section~\ref{sec:lifted-vs-line} to the lifted  
code of $C(A,G,B)\subseteq \F^{X(2)}$ and its line code $Z^0(X,\calF)\subseteq \Sigma^G$.

\medskip

The code  $C(A,G,B)\subseteq \F^{X(2)}$ is constructed as follows:
For every $\{g\}\in X(0)$, there is a bijection  
$\vphi_g:A\times B\to X(2)_{\{g\}}$   given by $\vphi_g(a,b)= \{g,ag,gb,agb\}$.
We use this bijection to identify
$\F^{X(2)_{\{g\}}}$ with $\F^{A\times B}=\nMat{\F}{A\times B}$ 
and let $C_g$ be the subspace of $\F^{X(2)_{\{g\}}}$ corresponding
to $C_A\otimes C_B$ under this identification; formally,
once viewing every $f\in C_A\otimes C_B$ as a function $f:A\times B\to \F$,
we have
\[
C_g=\{f\circ \vphi_g^{-1} \where f\in C_A\otimes C_B\}.
\]
The code $C=C(A,G,B)\subseteq \F^{X(2)}$ is the lifted code 
determined by the 
small codes $\{C_g\subseteq \F^{X(2)_{\{g\}}}\}_{g\in G}$.
That is, 
\[
C(A,G,B)=\{f:X(2)\to \F\suchthat \text{$f|_{X(2)_{\{g\}}}\in C_g$
for all $g\in G$}\}.
\]
Since every small code $C_g$ is canonically
identified with $\Sigma =C_A\otimes C_B$,
we may form the   line code $L=L(\{C_g\}_{g\in G})\subseteq \Sigma^G$ of $C(A,G,B)$
(Section~\ref{sec:lifted-vs-line}). As we now show, this code
is precisely $Z^0(X,\calF)$.

\begin{lem}\label{LM:line-code-of-CAGB}
	With notation as above, the line code of $C(A,G,B)=C(\{C_g\}_{g\in G})\subseteq \F^{X(2)}$
	is $Z^0(X,\calF)\subseteq \Sigma^G$.
\end{lem}

\begin{proof}
	The elements of the line  code $L$ of $C(A,G,B)$ are the
	ensembles $f=(f_g)_{g\in G}\in (C_A\otimes C_B)^G$
	which satisfy the following condition for all $g,h\in G$:
	\begin{enumerate}
		\item[($\star$)] $f_g\circ \vphi_g^{-1}  $ agrees with $f_h\circ \vphi_h^{-1}$  on 
	$X(2)_{\{g\}}\cap X(2)_{\{h\}}$.
	\end{enumerate}	
	Let $g,h\in G$.
	If $X(2)_{\{g\}}\cap X(2)_{\{h\}}=\emptyset$ or $g=h$, then condition ($\star$)
	holds. Otherwise, there are  $a_0\in A$ and $b_0\in B$
	such that $h\in\{a_0g,gb_0,a_0gb_0\}$.
	Suppose that $h=a_0g$. Then $X(2)_{\{g\}}\cap X(2)_{\{h\}}=\{\{g,a_0g,gb,a_0gb\}\where b\in B\}$. 
	Since $\vphi_g(a_0,b)=\{g,a_0g,gb,a_0gb\}=\{h,a_0^{-1} h,hb,a_0^{-1} hb\}=\vphi_h(a_0^{-1},b)$,
	condition ($\star$) is equivalent to having
	\begin{equation}\label{EQ:CAGB-line-cond-row}
	r_{a_0}(f_g)=r_{a_0^{-1}}(f_{a_0g}).
	\end{equation}
	Likewise, when $h=gb_0$, condition ($\star$) is equivalent to
	\begin{equation}\label{EQ:CAGB-line-cond-col}
	c_{b_0}(f_g)=c_{b_0^{-1}}(f_{gb_0}).
	\end{equation}
	Finally, if $g=a_0gb_0$, then $X(2)_{\{g\}}\cap X(2)_{\{h\}}=\{\{a_0g,gb_0,a_0gb_0\}\}$
	and condition ($\star$) becomes
	\[
	(f_g)_{a_0,b_0}=(f_{a_0g b_0})_{a_0^{-1},b_0^{-1}}.
	\]
	However, this already follows from \eqref{EQ:CAGB-line-cond-row} and \eqref{EQ:CAGB-line-cond-col}
	(for all $g\in G$),
	because they imply that $(f_g)_{a_0,b_0}=(f_{a_0 g})_{a_0^{-1},b_0}=(f_{a_0gb_0})_{a_0^{-1} g b_0^{-1}}$.
	
	By comparing \eqref{EQ:CAGB-line-cond-row} and~\eqref{EQ:CAGB-line-cond-col} with
	the description of $Z^0(X,\calF)$ in \S\ref{subsec:ltc-code},
	we see that $L=Z^0(X,\calF)$.
\end{proof}

\begin{cor}\label{CR:CAGB-are-good-ltcs}
	Let $G,A,B,X,\F,C_A,C_B $
	be as in \S\S\ref{subsec:ltc-poset}--\ref{subsec:ltc-sheaf}, 
	and let $C=C(A,G,B)\subseteq \F^{X(2)}$
	be the lifted code constructed above.
	Let $\veps\in (0,1]$, and 
	suppose that conditions (1a$'$)--(2$'$) of Theorem~\ref{TH:ltc-square}
	hold. Then 
	\[
	\delta(C)\geq 0.9996\veps^3, 
	\qquad 
	r(C) \geq 4r(C_A)r(C_B)-3,
	\]
	and the natural tester of $C$ has soundness
	\[
	\frac{1}{2\max\{\frac{|A|}{|B|},\frac{|B|}{|A|}\}[
	250000\veps^{-3} + 4800\max\{|A|,|B|\}\veps^{-1}]+1}
	.
	\]
	Moreover, provided that $|A|,|B|,|\F|$ are $O(1)$,
	$C$ has a linear-time decoding algorithm able to correct words
	that are $\frac{1}{|A||B|}\cdot \frac{1}{
	250000\veps^{-5} + 4800\max\{|A|,|B|\}\veps^{-3}}$-close
	to $C$.
\end{cor}

\begin{proof}
	Write $L=Z^0(X,\calF)$; this is the line code of $C$ by Lemma~\ref{LM:line-code-of-CAGB}.
	Theorem~\ref{TH:ltc-square} provides us with lower bounds
	on $\delta(L)$, $r(L)$ and the soundness $\mu$ of the natural $2$-query tester
	of $L$, as well as a decoding algorithm.
	The lower bounds on $\delta(C)$ and $r(C)$ are now obtained
	by applying Proposition~\ref{PR:line-code-rate-and-dist}; in our case $D_{\min}=D_{\max}=4$,
	$k_{\min}=k_{\max}=|A||B|$ and $\tilde{k}=\delta(C_A\otimes C_B)|A||B|\geq \veps^2|A||B|$.
	Next, we  apply  Theorem~\ref{TH:line-code-testability} using the graph $\ugr(X)$
	and the labelling $\ell$ mapping a face $x\in \ugr(X)$ to $X(2)_x$, namely, the set of squares in $X$ containing
	$x$. 
	It implies that the   natural tester of $C$ has soundness $\frac{\mu}{2+\mu}=\frac{1}{2\mu^{-1}+1}$
	(in our case $d_{\min}=d_{\max}=|A|+|B|$).
	Finally, the claim about the decoding algorithm follows from Proposition~\ref{PR:line-code-eff-decoding}.
\end{proof}

By applying 
Corollary~\ref{CR:CAGB-are-good-ltcs} to the family $\{G_i,A_i,B_i\}_{i\in\N}$
from Theorem~\ref{TH:existence-of-good-codes} (with $\F=\F_2$), we recover the good locally testable 
codes of
\cite{Dinur_2022_ltcs_const_rate}. Our bounds on the rate, distance, soundness
and decoding
are slightly different, though.  

\begin{remark}
	The proof in \cite{Dinur_2022_ltcs_const_rate}   
	that $C(A,G,B)\subseteq \F^{X(2)}$ is locally testable
	under assumptions (1a$'$)--(2$'$) of Theorem~\ref{TH:ltc-square} 
	implicitly  establishes  the local testability
	of the line code $Z^0(X,\calF)\subseteq (C_A\otimes C_B)^G$ 
	and then deduces from it the local testability of $C(A,G,B)$.
	Specifically, observe
	that  \cite[Algorithm~1]{Dinur_2022_ltcs_const_rate}
	is  essentially our Algorithm~\ref{AL:correction-algorithm-simplified}
	restricted to the $X$ and $\calF$ constructed from  $G$, $A$, $B$.
	Our proof of Corollary~\ref{CR:CAGB-are-good-ltcs} illuminates
	that aspect of the proof as well as the hidden role of cosystolic expansion of sheaves.
\end{remark}

\begin{remark}
	In \cite[\S4.1]{Dinur_2022_ltcs_const_rate}, the authors
	give   lower bounds on the rate of $C(A,G,B)\subseteq \F^{X(2)}$ that are slightly better
	than those of Corollary~\ref{CR:CAGB-are-good-ltcs}.
	Using them together with Proposition~\ref{PR:line-code-rate-and-dist}(i)
	gives   lower bounds on the rate of $Z^0(X,\calF)\subseteq \Sigma^G$ that
	are better than those given in Lemma~\ref{PR:rate}.
\end{remark}

\begin{remark}
	The argument which we use  to deduce
	Corollary~\ref{CR:CAGB-are-good-ltcs} from
	Theorem~\ref{TH:ltc-square}
	can be reversed.
	More precisely, by
	Theorem~\ref{TH:lifted-to-line-test} and other results in Section~\ref{sec:lifted-vs-line},
	if   $C(A,G,B)\subseteq \F^{X(2)}$ is a good LTC 
	admitting a   linear-time decoding algorithm,
	then   $Z^0(X,\calF)\subseteq \Sigma^G$
	admits a $2$-query tester making it into
	a good LTC with a linear-time decoding algorithm.
	However,   the alluded $2$-query tester   is \emph{not}
	the natural tester of $Z^0(X ,\calF )$.
	Rather, it is the tester defined in Section~\ref{sec:lifted-vs-line} 
	corresponding to the intersection
	graph of the sets $\{X (2)_{\{g\}}\where g\in G \}$ (Example~\ref{EX:intersection-graph}).	
	Explicitly, given $f\in Z^0(X ,\calF )$,
	this tester chooses uniformly at random a pair of vertices $\{g\},\{h\}\in X (0)$ that are contained
	in a common square and checks whether $f(g)\in \nMat{\F}{A\times B}\cong \F^{X (2)_{\{g\}}}$
	agrees with $f(h)\in \nMat{\F}{A\times B}\cong \F^{X (2)_{\{h\}}}$
	on $X(2)_{\{g\}}\cap X (2)_{\{h\}}$. This is different from the natural tester
	of $Z^0(X ,\calF )$ because $\{g\}$ and $\{h\}$ may be the opposite vertices of the a square in $X$.
	If we use \cite{Dinur_2022_ltcs_const_rate} as a black box to show that
	$C(A,G,B)$ is an LTC,
	then the soundness that Theorem~\ref{TH:lifted-to-line-test}  guarantees for this tester
	would be much small than the soundness of the natural tester
	promised by Theorem~\ref{TH:ltc-square}.	
\end{remark}

\section{Local Testability of $2$-Layer Lifted Codes}
\label{sec:two-layer}

In this section, we apply Theorem~\ref{TH:lgp-zero-detailed} to give a local
criterion for a lifted code to be locally testable w.r.t.\ to its natural tester.
This requires the lifted code to have some auxiliary extra structure. In particular, the local
codes
forming our lifted codes should be  lifted codes themselves.

\subsection{$2$-Layer Lifted Codes}
\label{subsec:two-layer}

Recall that a lifted code $C\subseteq \Sigma^n$ is determined by small codes
$\{C_s\subseteq \Sigma^s\}_{s\in S}$, where $S\subseteq P([n])$ and covers $[n]$.
We would like to consider lifted codes in which each small code $C_s\subseteq \Sigma^s$
is itself
a lifted code. This structure can be  neatly encoded using a   $1$-poset labelled
by subsets of $[n]$.

\begin{dfn}[$2$-Layer Lifted Code]\label{DF:two-layer-lifted-code}
	Let $\Sigma$ be a finite alphabet and let $n\in\N$. A $2$-layer
	lifted code inside $\Sigma^n$ consists of a triple $(X,\ell,\{C_e\}_{e\in X(1)})$,
	where
	\begin{itemize}
		\item $X$ is a $1$-poset,
		\item $\ell:X\to P([n])$ is a   function assigning every face of $X$ a nonempty subset of $[n]$,
		\item $C_e$ is a code inside $\Sigma^{\ell(e)}$ for every  edge $e\in X(1)$,
	\end{itemize}
	and such that the following conditions hold:
	\begin{enumerate}[label=(\arabic*)]
		\item $\ell(x)=\bigcup_{y\in X:y>x}\ell(y)$ for all $x \in X(0)\cup X(-1)$;
		\item $\ell(\emptyset)=[n]$.
	\end{enumerate}
	In this case, for every $x\in X(0)\cup X(-1)$, we   assign a lifted code $C_x\subseteq \Sigma^{\ell(x)}$
	defined by
	\[
	C_x=C(\{C_e\}_{e\in X(1)_x})=\{f\in \Sigma^{\ell (x)}
	\suchthat
	\text{$f|_{\ell(e)}\in C_e$ for all $e\in X(1)$}\}.
	\]
	The code $C_\emptyset\subseteq \Sigma^n$ will also be denoted as $C(X,\ell,\{C_e\}_{e\in X(1)})$.
	It can can also be realized as a lifted code w.r.t.\ the ``bigger'' small codes
	$\{C_v\}_{v\in X(0)}$. 
\end{dfn}

\begin{dfn}[Natural Tester of a $2$-Layer Lifted Code]
	Let $C=C(X,\ell,\{C_e\}_{e\in X(1)})\subseteq\Sigma^n$
	be a $2$-layer lifted code. The natural tester of $C$ is its natural tester
	when realized as a lifted code w.r.t.\ the small codes $\{C_v \}_{v\in X(0)}$.
	Explicitly, given $f\in \Sigma^n$, the natural tester picks $v\in X(0)$ uniformly at
	random, probes $f_i$ for every $i\in \ell(v)$, 
	and accepts $f$ if and only if $f|_{\ell(v)}\in C_v$.
\end{dfn}

\subsection{Subset-Labelled $d$-Posets}
\label{subsec:subset-labelling}

The notion of a $1$-poset labelled by subsets of $[n]$ extends naturally
to $d$-posets. 

\begin{dfn}[$S$-Subset Labelled $d$-Poset]
	Let $S$ be a finite set and let $X$ be a $d$-poset.
	An $S$-subset labelling on $X$ is a function $\ell:X\to P(S)-\{\emptyset\}$
	such that 
	\begin{enumerate}[label=(\arabic*)]
	\item $\ell(x)=\bigcup_{y\in X:y>x}\ell(y)$ for all $x\in X$ with $\dim x<d$ and
	\item $\ell(\emptyset)=S$.
	\end{enumerate}
	In this case, we call $(X,\ell)$ an \emph{$S$-subset labelled $d$-poset}.
\end{dfn}

\begin{example}\label{EX:labelled-posets}
	(i) Let $n\in\N$, let $V$ be a collection of nonempty subsets of $[n]$ which covers $[n]$
	and let $X$ be a   $d$-poset     with $X(0)=V$.
	Define a function $\ell:X\to P([n])$ by sending $\emptyset$ to $[n]$
	and every other $x$ to $\bigcap_{v\in x(0)}v$.
	If $\bigcap_{v\in x(0)}v\neq \emptyset$
	for every $x\in X-\{\emptyset\}$ and $X$ has enough faces
	so   that for  every $x\in X$ with $\dim x<d$,
	the sets $\{\bigcap_{v\in y(0)} v\where y\in X(\dim x+1)_x\}$
	cover $\bigcap_{v\in x(0)} v$, then $\ell:X\to P([n])$
	is an $[n]$-subset labelling of $X$.
	This generalizes the setting of  \S\ref{subsec:two-later-Tanner-intro}, which is essentially
	the case where $X$ is a pure $2$-dimensional simplicial complex.
	
	(ii) Let $0\leq d\leq d'$ be integers, let $X'$ be a $d'$-poset
	and set $S=X'(d')$.  
	Put $X=X'({\leq }d)$
	and define $\ell:X \to P([n])$ by $\ell(x)=Y(d)_x$.
	Then $\ell$ is an $S$-subset labelled $d$-poset.
\end{example}

An $S$-subset labelling on a $d$-poset $X$ induces a normalized weight function
$w_\ell:X\to \R_+$ given by $w_\ell(x)=\frac{1}{n}\sum_{j\in\ell(x)}\frac{1}{\# \{y\in X(\dim x)\suchthat j\in\ell(y)\}}$. The number $w_\ell(x)$ is also the probability
of getting   $x$ by choosing $j\in S$ uniformly
at random and then choosing a face $y\in X(\dim x) $  
with $j\in\ell(y)$ uniformly at random.

Given integers $-1\leq i\leq j\leq d$, we define the \emph{$(i,j)$-lower regularity} and \emph{$i$-upper irregularity}
of the $S$-subset labelling $\ell$ to be
\[
L_{i,j}(\ell):=
\frac{\max_{s\in S, x\in X(i):s\in\ell(x)}\#\{y\in X(j)_x\suchthat s\in \ell(y)\}}{
\min_{s\in S, x\in X(i):s\in\ell(x)}\#\{y\in X(j)_x\suchthat s\in \ell(y)\}}
\quad
\text{and}
\quad
U_i(\ell):= 
\frac{\max\{\#\ell (x)\where x\in X(i)\}}{\min \{\# \ell (x)\where
x\in X(i)\}},
\]
respectively. We say that $\ell$ is lower regular
if $L_{i,j}(\ell)=1$ for all $i,j$.
The $i$-th degree of $\ell$ is 
\[
D_i(\ell)=\max_{x\in X(i)}\#\ell(x).
\]

\begin{example}
	In the setting of Example~\ref{EX:graded-poset}(ii), $w_\ell:X\to \R_+$
	is just the restriction of the natural weight function of $X'$ to $X$.
	Moreover, $L_{i,j}(\ell)=L_{i,j,d'}(X')$, $U_i(\ell)=U_{i,d'}(X')$ and
	$D_i(\ell)=D_{i,d'}(X')$.
\end{example}

\begin{remark}\label{RM:w-ell}
	If $X$ is lower-regular and $\ell$ is lower-regular, then $w_\ell$ is a proper weight function on $X$. This follows from  
	Corollary~\ref{CR:weights-in-lower-regular-posets} and the following
	observation.  
\end{remark}

An $S$-subset labelling $\ell$ on a $d$-poset $X$
may be used to extend $X$ into a $(d+1)$-poset $Y:=X\sqcup S$,
where the elements of $S$ are viewed as $(d+1)$-faces
and for $x\in X$ and $s\in S$, we have $x<s$ if and only if $s\in \ell(x)$.
(Here we need $\ell(x)\neq \emptyset$ for all $x\in X$.)
The weight function $w_\ell:X\to \R_+$ is then   the restriction of
the natural weight function of $Y$ to $X$. Moreover,
$L_{i,j}(\ell)$, $U_i(\ell)$ and $D_i(\ell)$
are just $L_{i,j,d+1}(Y)$, $U_{i,d+1}(Y)$ and $D_{i,d+1}(Y)$, respectively.

\begin{example}
	Keep the setting of Example~\ref{EX:labelled-posets}(ii) and suppose further that
	$d'=d+1$. Then   then poset $X\sqcup S$ associated to the $S$-subset labelling
	$\ell$ coincides with $X'$.
\end{example}

We finally note that if $(X,\ell)$ is an $S$-subset labelled
$d$-poset and $z\in X$ is of   dimension $i$,
then the pair $(X_z,\ell|_{X_z})$
is an $\ell(z)$-subset labelled $(d-i-1)$-poset. We shall
abbreviate $\ell|_{X_z}$ to $\ell_z$. If we let $Y=X\sqcup S$
as above, then $w_{\ell_z}$ coincides with the natural weight function of $Y_z$,
because this poset is just $X_z\sqcup \ell(z)$.

\begin{lem}\label{LM:w-ell-z}
	With this notation, suppose that $X$ and $\ell$ are lower-regular (equiv.\ $Y=X\sqcup S$ is lower-regular).
	Then $w_{\ell_z}=(w_\ell)_z$ for all $z\in X$. (See \S\ref{subsec:links} for the definition
	of $(w_\ell)_z$.)
\end{lem}

\begin{proof}
	Let $x\in X$ be of dimension $j$. We write $F_{i,j,k}=F_{i,j,k}(Y)$ and
	$w$ for the natural weight function of $Y$.
	We observed in Remark~\ref{RM:w-ell} that $w_{\ell}$ is proper.
	Thus, by Lemmas~\ref{LM:weight-link-vs-all}, \ref{LM:weight-of-j-faces-cont-z}
	and~\ref{LM:relation-between-face-numbers},
	\[
	(w_\ell)_z(x)=w_\ell(X(2)_z)^{-1} w_\ell(x)\frac{F_{j,2}}{F_{i,j,2}}
	= \frac{F_{2,2}}{F_{i,2,2}F_{i,2} w_\ell(z)}w_\ell(x)\frac{F_{j,2}}{F_{i,j,2}}=
	\frac{w_\ell(x)}{w_{\ell}(z)}\frac{F_{j,2}}{F_{i,2}F_{i,j,2}}
	=\frac{w_\ell(x)}{F_{i,j}w_{\ell}(z)}.\]
	The same reasoning with $(Y,w)$ in place of $(X,w_{\ell})$
	(and $3$ instead of $2$)
	shows that $w_z(x)=\frac{w(x)}{F_{i,j} w(z)}$.
	Since $w_\ell=w|_X$ and $w_{\ell_z}=w_z|_{X_z}$, 
	it follows that $(w_\ell)_z(x)=(w_{\ell_z})(x)$.
\end{proof}

\subsection{A Criterion for a $2$-Layer Lifted Code to be Locally Testable}

Let $C=C(X,\ell,\{C_e\}_{e\in X(1)})\subseteq\Sigma^n$ be a two layer
lifted code. 
The following theorem
gives a criterion for $C$ to be locally
testable when   $(X,\ell)$ is the $[n]$-subset labeled $1$-poset underlying a 
$[n]$-subset labelled
pure $2$-dimensional regular cell complex.

\begin{thm}\label{TH:two-layer-general}
	Let $F\in\N$ and $L\in[1,\infty)$. Then there exist constants $S,S',S'',T_1,\dots,T_5>0$
	(all are   inverse-polynomial in $F$ and $L$)
	such that the following hold:\footnote{
		We encourage the reader to think of $F$ and $L$ (and thus of $S,S',T_1,\dots,T_5$)
		as being constant or $\Theta(1)$ as this is usually what happens in practice.	
	}
	Let $n\in\N$ and let $(X,\ell)$ be an $[n]$-subset labelled 
	pure $2$-dimensional regular cell complex (\S\ref{subsec:subset-labelling}) 
	such that:
	\begin{enumerate}[label=(0\alph*)]
		\item 
		$\Fmax_{i,2}(X) \leq F$ for all $i\in\{0,1\}$;	
		\item $L_{i,j}(\ell)\leq L$ for all integers $-1\leq i<j \leq 2$;
		\item for all $u,v\in X(0)$ and $j\in \ell(u)\cap \ell(v) $,
		there is a path of edges from $u$ to $v$
		such that $j\in\ell(e)$ for every edge $e$ along the path.
	\end{enumerate}
	Let $R$ be a commutative ring, let $\Sigma$ be an
	$R$-module, and
	for every $e\in X(1)$, let $C_e\subseteq \Sigma^{\ell(e)}$ be a  code which is also
	an $R$-submodule.
	Then $(X({\leq }1),\ell,\{C_e\}_{e\in X(1)})$ is a $2$-layer lifted code. 
	Let $\alpha_0,\beta_0,\alpha_{-1},\beta_{-1},\alpha_{||},\beta_{||}\in [0,\infty)$ and $\veps\in (0,1]$
	and suppose   that:
	\begin{enumerate}
		\item[(1a)] $\delta(C_e)\geq \veps$ for all $e\in X(1)$;
		\item[(1b)] for every $v\in X(0)$,
		the quartet $(\{C_e\}_{e\in X(1)_v},\ugr(X_v),\ell|_{\ugr(X_v)},{w}_{\ell_v})$
		is an $\veps$-agreement tester;\footnote{
			Note that $\ugr(X_v)$ is a graph by our assumption that $X$ is a $2$-dimensional regular cell complex. The weight function $w_{\ell_v}:\ugr(X_v)\to \R_+$ is given explicitly
			by 	$w_{\ell_v}(x)=
		\frac{1}{\#\ell(v)}\sum_{i\in\ell(x)}\frac{1}{
		\#\{y\in X(\dim x)_v\suchthat i\in \ell(y)\}}$.
		}
		\item[(2a)] 
		$ \ugr(X_v, w_{\ell_v})$ is an $(\alpha_0,\beta_0)$-skeleton expander for all $v\in X(0)$;
		\item[(2b)] $\ugr(X,w_{\ell})$ is an 	$(\alpha_{-1},\beta_{-1})$-skeleton expander;	
		\item[(2c)] $\NI^{1,1,2}(X,w_{\ell})$ (see \S\ref{subsec:no-intersect})
		is an $(\alpha_{||},\beta_{||})$-skeleton expander.	
	\end{enumerate}
	Suppose further that 
	\begin{align}\label{EQ:two-layer-general:ineq1}
	p:=S\veps- \alpha_{-1}\beta_0 - S'\alpha_0>0
	\end{align}
	and one can find $h_{-1},h_0,h_{||}\in (0,1]$ satisfying the following inequality
	\begin{align}\label{EQ:two-layer-general:ineq2}
		(\alpha_0+\beta_0 h_0) + (\alpha_{||}+\beta_{||}h_{||})+\frac{\alpha_{-1}+\beta_{-1}h_{-1}}{h_0}
		\leq S'' \veps.
	\end{align}
	Then the   $2$-layer lifted code $C=C(X({\leq} 1),\ell,\{C_e\}_{e\in X(1)})\subseteq\Sigma^n$ satisfies
	\[
	\delta(C)>\frac{T_1}{U_0(\ell)^2}\cdot \frac{\beta_0^{-1}p^2 +\alpha_{-1}p}{\beta_0\beta_{-1}}
	\]
	and its natural tester has soundness
	\[
	\frac{1}{U_0(\ell)^2 U_1(\ell)}\cdot\frac{T_2}{T_3^{-1}+h_0^{-1} h_1^{-1}+h_{||}^{-1}}.
	\]
	Moreover,  $C\subseteq \Sigma^n$
	has a linear-time decoding algorithm
	(the constant depends on
	$D_0(\ell)$, $|\Sigma|$, $F$, $L$, $h_0$)
	able to correct words
	that are $\eta$-close to $C$,
	where
	\[
	\eta = 
	\frac{1}{U_0(\ell)D_0(\ell)} \min\left\{\frac{T_4p}{\beta_0 \beta_{-1} h_0^{-1} },
	\frac{T_5}{h_0^{-1}h_1^{-1}+h_{||}^{-1}}\right\}.
	\]
\end{thm}

We  prove Theorem~\ref{TH:two-layer-general} in the next subsection.

%%	(i) The weight function $w_{\ell_v}:X_v\to \R_+$ in (1b) is the uniform weight
%%	function if $|\ell(x)|=|\ell(y)|$ whenever  $x,y\in X(0)$ or $x,y\in X(1)$,
%%	or equivalently, if $U_0=U_1=1$.
%%	
%%	(ii) 

\begin{example}
	Similarly to  Example~\ref{EX:labelled-posets}(ii),
	let $X'$ be a $d'$-poset ($d'\geq 2$)
	with $X'(d)=[n]$, and define
	$\ell:X:=X'(\leq 2)\to P([n])$ by $\ell(x)=X'(d)_x$.
	Suppose that $X =X'({\leq} 2)$ is a regular cell complex
	and that for every $y\in X'(d')$, the graph underlying $\{z\in X\suchthat z\leq y\}$
	is connected (this holds if $X'$ is itself a regular cell complex).
	Then $X$ and $\ell|_X$
	satisfy conditions (0a)--(0b) of Theorem~\ref{TH:two-layer-general}
	with $L=L(X')$ and $F=\max\{\Fmax_{0,2}(X'),\Fmax_{1,2}(X')\}$,
	and the connectivity assumption implies that (0c) holds as well.
	Moreover, the weight functions
	$w_{\ell}$ and $w_{\ell_v}$ ($v\in X(0)$)
	coincide with the natural weight functions of $X'$ and $X'_v$,
	respectively.
	We can now choose an $R$-module $\Sigma$ 
	and $R$-submodules  
	$C_e\subseteq\Sigma^{\ell(e)}$ for every $e\in X(1)$
	and attempt to apply Theorem~\ref{TH:two-layer-general}.
\end{example}

\begin{example}[Instantiation of Theorem~\ref{TH:two-layer-general}]
	We can recover the assertions about distance and testability
	in Corollary~\ref{CR:CAGB-are-good-ltcs} using
	Theorem~\ref{TH:two-layer-general}, though under slightly different
	assumptions. Specifically, using the notation of \S\ref{subsec:square-line-code},
	take $X=\mathrm{Cay}(A,G,B)$ and give it the $X(2)$-subset labelling $\ell$
	given by $\ell(x)=X(2)_x$. Given an edge $e=\{g,ag\}$ with $a\in A$, $g\in G$,
	we identify $\ell(e)=X(2)_e$ with $B$ via $b\mapsto \{g,ag,gb,agb\}$ and, 
	using this identification, set $C_e=C_B\subseteq \F^B\cong \F^{\ell(e)}$.
	Similarly, for $e=\{g,gb\}$ ($g\in G$, $b\in B$),
	set $C_e=C_A\subseteq \F^A\cong \F^{\ell(e)}$.
	One can now check using the lemmas of \S\ref{subsec:interpretation} that
	there is a constant $K>0$ such that if $\delta(C_A)\geq \veps$,
	$\delta(C_B)\geq \veps$, $C_A\otimes C_B$ is $\veps$-agreement-testable
	and $\mathrm{Cay}(A,G)$ and $\mathrm{Cay}(G,B)$
	are $K\veps^2$-expanders, then all the assumptions of
	Theorem~\ref{TH:two-layer-general} are satisfied.
	This allows us to deduce that $C(A,G,B)\subseteq\F^{X(2)}$ is an LTC
	with linear distance.
\end{example}

When the regular cell complex $X$ from Theorem~\ref{TH:two-layer-general}
is a simplicial complex and the labelling $\ell$ is lower-regular, we can use Theorem~\ref{TH:two-layer-general}
together with Oppenheim's Trickling Down Theorem \cite[Thm.~4.1]{Oppenheim_2018_local_spectral_expansion_I} to get
the following \emph{local} criterion for showing that
a $2$-layer  lifted code is locally testable.

\begin{thm}\label{TH:two-layer-local}
	There are   constants $K,K'>0$ such that the following hold:
	Let $X$ be a pure $2$-dimensional simplicial   complex,
	let $n\in\N$
	and let $\ell:X\to P([n])$ be a lower-regular $[n]$-subset labeling satisfying condition 
	(0c) of Theorem~\ref{TH:two-layer-general}.
	Let $(X({\leq }1),\ell,\{C_e\}_{e\in X(1)})$ be a $2$-layer lifted code with alphabet
	$\Sigma$ as
	in Theorem~\ref{TH:two-layer-general},
	let $\veps\in (0,1]$ and suppose that:
	\begin{enumerate}
		\item[(1a)] $\delta(C_{e})\geq \veps$ for every $e\in X(1)$;
		\item[(1b)] for every $v\in X(0)$,
		the quartet $(\{C_e\}_{e\in X(1)_v},\ugr(X_v),\ell|_{\ugr(X_v)},{w}_{\ell_v})$
		is an $\veps$-agreement tester;
		\item[(2a)] $\ugr(X_v,w_{\ell_v})$ is a 
		$K\veps^2$-spectral expander.
	\end{enumerate}
	Then the natural tester of the 
	$2$-layer lifted code $C=C(X({\leq} 1),\ell,\{C_e\}_{e\in X(1)})\subseteq\Sigma^n$  
	has soundness $\frac{K' \veps^3}{U_0(\ell)^2U_1(\ell)}$. Moreover, 
	$\delta(C)\geq \frac{K'\veps^2}{U_0(\ell)^2}$ and there is a linear-time decoding algorithm
	for words in $\Sigma^n$ that are $\frac{K'\veps^3}{D_0(\ell)U_0(\ell)}$-far from $C$.
\end{thm}

\begin{proof}
	Suppose for time being that $K$ and $K'$ were specified and $K\leq \frac{1}{2}$.
	The actual values will be given later on.

	Since $X$ is simplicial, $\Fmax_{i,2}(X)=3$ for   $i\in\{0,1\}$.
	Moreover, the relation on the vertices of
	$\NI^{1,1,2}(X)$ is the empty relation, so it is a $(0,0)$-skeleton expander
	(Example~\ref{EX:skeleton-exp}(ii)).
	Since both $X$ and $\ell$ are lower-regular, $w_\ell$ is  proper
	(Remark~\ref{RM:w-ell}) and $(w_\ell)_v=w_{\ell_v}$ for
	all $v\in X(0)$ (Lemma~\ref{LM:w-ell-z}). In particular, 
	$w_{\ell_v}$ is proper. 
	Now, by (2a), $(X_v({\leq} 1),w_{\ell_v})$ is a $(K\veps^2,1)$-skeleton expander.
	Since $w_{\ell_v}=(w_\ell)_v$, we may use
	Oppenheim's Trickling Down Theorem \cite[Thm.~4.1]{Oppenheim_2018_local_spectral_expansion_I}
	and get that
	$(X,w_\ell)$ is a  $\frac{K\veps^2}{1-K\veps^2} $-spectral expander, and hence
	a $(2K\veps^2 ,1)$-skeleton expander provided that $K\leq \frac{1}{2}$.
	
	Let $S,S',S'',T_1,\dots,T_5$ be the constants guaranteed by Theorem~\ref{TH:two-layer-general} when
	$F=3$ and $L=1$. We choose   $K$ to be   small enough to satisfy    
	$(2+S')K\leq \frac{S}{2}$ and $5\sqrt{K}\leq S''$.
	
	We claim that we may   apply Theorem~\ref{TH:two-layer-general} to   $(X,\ell)$ and
	$\{C_e\}_{e\in X(1)}$
	with $(\alpha_0,\beta_0)=(K\veps^2,1)$, 
	$(\alpha_{-1},\beta_{-1})=(2K\veps^2,1)$, $(\alpha_{||},\beta_{||})=(0,0)$,
	$h_0=\sqrt{K} \veps$, $h_{-1}=K\veps^2$ and $h_{||}=1$.
	Note that conditions (0a)--(2c) of Theorem~\ref{TH:two-layer-general} hold
	by the last paragraph and our assumptions, so it remains to check the inequalities \eqref{EQ:two-layer-general:ineq1}
	and \eqref{EQ:two-layer-general:ineq2}.
	Indeed, by our choice of $K$,
	\begin{align*}
	p:=S\veps-\alpha_{-1}\beta_0-S'\alpha_0=S\veps-2K\veps^2-S'K\veps^2\geq 
	\veps(S-2K-S'K)\geq	
	\frac{S}{2}\veps>0
	\end{align*}
	and
	\begin{align*}
	(\alpha_0+\beta_0 h_0) + (\alpha_{||}+\beta_{||}h_{||})+\frac{\alpha_{-1}+\beta_{-1}h_{-1}}{h_0}
	=K\veps^2+\sqrt{K}\veps + \frac{2K\veps^2+K\veps^2}{\sqrt{K}\veps}\leq 5\sqrt{K}\veps \leq S''\veps.
	\end{align*}

	Now,  Theorem~\ref{TH:two-layer-general} tells us that natural tester of   $C\subseteq\Sigma^n$
	has soundness
	$
	\frac{1}{U_0(\ell)^2U_1(\ell)}\frac{T_2}{T_3^{-1}+K^{-1.5}\veps^{-3}+1}
	$. Moreover,
	$\delta(C)>\frac{T_1}{U_0(\ell)^2}\cdot p^2 \geq \frac{T_1S^2}{4 U_0(\ell)^2}\veps^2$
	and $C$ has a linear-time decoding algorithm able to correct words that are $\eta$-far
	from $C$ with 
	$\eta = \frac{1}{U_0(\ell)D_0(\ell)}\min\{\frac{T_4S\veps}{2K^{-0.5}\veps^{-1}},
	\frac{T_5}{K^{-1.5}\veps^{-3}+1}\}$
	From this, one   sees that there is a constant $K'>0$ for which the assertions of the theorem hold.
\end{proof}

\subsection{Proof of Theorem~\ref{TH:two-layer-general}}

We  prove Theorem~\ref{TH:two-layer-general} by realizing the line
code of $C=C(\{C_v\}_{v\in X(0)})$ as the $0$-cocycle code of a sheaf on $X$,
applying Theorem~\ref{TH:lgp-zero-detailed} to that sheaf,
and then deducing the good properties of $C$ using the results of Section~\ref{sec:lifted-vs-line}.
This will be done a series of lemmas.
A byproduct of this approach is that   the line code of $C=C(\{C_v\}_{v\in X(0)})$
is also locally testable and has linear distance; this is Lemma~\ref{LM:twolayer-zero-cocycle-ltc}.

\medskip

Throughout, we will use the following general notation:
Let $n\in\N$ and let $(X,\ell)$ be an $[n]$-subset labelled
pure $2$-dimensional regular cell complex.
Let $R$ be a ring and
let $\Sigma$ be an $R$-module. For every $e\in X(1)$,
let $C_e\subseteq\Sigma^{\ell(e)}$ be an $R$-submodule
and for  $v\in X(0)$, let $C_v=C(\{C_e\}_{e\in X(1)_v})\subseteq \Sigma^{\ell(v)}$.
Recall that $C:=C(X,\ell,\{C_e\}_{e\in X(1)})$ also equals $C(\{C_v\}_{v\in X(0)})$.

As in \S\ref{subsec:subset-labelling}, let
$Y=X\sqcup [n]$ be the $3$-poset associated to $(X,\ell)$.
We denote the natural weight function of $Y$ by $w$.
Recall that $w_\ell= w|_X$ and $w_{\ell_v}=w_v|_{X_v}$ for all $v\in X(0)$.
In addition,
$\Fmax_{i,2}(Y)=\Fmax_{i,2}(X)$ for all $i\in\{0,1\}$
and condition (0b) of Theorem~\ref{TH:two-layer-general} is equivalent to saying
that $L_{i,j,3}(Y)\leq L$ for all $-1\leq i\leq j\leq 3$.
As $L_{i,j,k}(Y)\leq \Fmax_{i,j,k}(Y)$ for all $i\leq j\leq k$,
it follows that if conditions (0a) and (0b) of Theorem~\ref{TH:two-layer-general}
hold, then $L(Y)\leq \max\{L,F\}$.

%%In addition to the notation just set, we will use the following general notation.

Since $X=Y({\leq} 2)$ is a regular cell complex, it admits a $\Z$-orientation
with $[v:\emptyset]=1$ for every
$v\in X(0)$ (Example~\ref{EX:orientation-of-cell-complexes}).
We fix such an orientation once and for all.
Note, however, that  this orientation may not  extend  to $Y$.
(Admittedly, $Y$ 
is needed only
for the weights.)

Define a sheaf $\calF$ on $Y $ as follows:
\begin{itemize}
\item $\calF(v)=C_v\subseteq \Sigma^{\ell(v)}$ for all $v\in Y(0)$;
\item $\calF(e)=C_e\subseteq \Sigma^{\ell(e)}$ for all $e\in Y(1)$;
\item $\calF(x)=\Sigma^{\ell(x)}$ for all $x\in Y(2)$;
\item $\calF(y)=0$ for every other face  $y\in Y(-1)\cup Y(3)$;
\item $\res_{y\from x}:\calF(x)\to \calF(y)$ is given by
$\res_{y\from x}(f)=f|_{\ell(y)}$ whenever $0\leq \dim x<\dim y\leq 2$
(recall that $f\in \Sigma^{\ell(x)}$ and thus  $f|_{\ell(y)}\in \Sigma^{\ell(y)}$);
\item $\res_{y\from x}$ is the zero map in all other cases.
\end{itemize}
We will think of $Z^0=Z^0(Y,\calF)$ as a code with varying alphabet
inside $C^0=C^0(Y,\calF)=\prod_{v\in Y(0)}C_v$; see Remark~\ref{RM:generalized-codes}.

\begin{lem}\label{LM:twolayer-line-code}
	Assume that condition (0c) of Theorem~\ref{TH:two-layer-general}
	holds. Then $Z^0(Y,\calF)$ is the line code of the lifted
	code $C:=C(\{C_v\}_{v\in Y(0)})\subseteq\Sigma^n$.
\end{lem}

\begin{proof} 
	Let $f=(f_v)_{v\in X(0)}\in C^0=\prod_{v\in Y(0)}C_v $.
	Since    $[v:\emptyset]=1$ for all $v\in X(0)$,
	every edge $e\in X(1)$ admits exactly one vertex
	$v$ with $[e:v]=1$ and the other vertex $u$ satisfies
	$[e:u]=-1$.
	Consequently, the condition $(d f)(e)=0$
	is equivalent to having $\res_{e\from u} f_u=\res_{e\from v} f_v$.
	As a result, $Z^0(X,\calF)$ consists of the set of
	$f=(f_v)_{v\in X(0)}\in 
	\prod_{v\in X(0)}C_v$
	such that $f_u|_{\ell(e)}=f_v|_{\ell(e)}$
	for every $e=\{u,v\}\in X(1)$.
	This means that $Z^0(X,\calF)$ contains the line code $L$
	of $C$.
	It remains to show that $L\supseteq Z^0(X,\calF)$. 
	
	Suppose that $f\in Z^0(X,\calF)$. In order to show that
	$f\in L$,
	we need to show that for every
	$u,v\in X(0)$ with $\ell(u)\cap \ell(v)\neq\emptyset$,
	we have $f_u|_{\ell(u)\cap \ell(v)} = f_v|_{\ell(u)\cap \ell(v)}$.
	Let $i\in \ell(u)\cap \ell(v)$.
	By condition (0c) of Theorem~\ref{TH:two-layer-general},
	there is a path of edges $e_1,\dots,e_r\in X(1)$
	from $u$ to $v$ such that $i\in \ell(e_j)$ for every $j\in \{1,\dots,r\}$.
	Write $e_j=\{u_{j-1},u_j\}$ so that $u=u_0$ and $v=u_r$.
	Our assumption that $f\in Z^0(X,\calF)$
	implies that $f_{u_{j-1}}|_{\ell(e_j)}=f_{u_j}|_{\ell(e_j)}$,
	and in particular $f_{u_{j-1},i}=f_{u_j,i}$.
	Thus, $f_{u,i}=f_{u_0,i}=\dots=f_{u_r,i}=f_{v,i}$.
	As this holds for all $i\in\ell(u)\cap \ell(v)$,
	we conclude that  $f_u|_{\ell(u)\cap \ell(v)} = f_v|_{\ell(u)\cap \ell(v)}$.
\end{proof}

The following lemma is the reason we give $X$ the weight function $w_\ell$
in Theorem~\ref{TH:two-layer-general}.

\begin{lem}\label{LM:twolayer-cbe-edges}
	With notation as above, let   $e\in Y(1)$. Then $\cbe_{-1}(Y_e,w_e,\calF_e)\geq  \delta(C_e) $.
\end{lem}

\begin{proof}
	Let   $f\in\calF(e)-\{0\}=C_e-\{0\}$  and let
	$A=\{x\in Y(2)_e\suchthat \res_{x\from e}f\neq 0\}$.
	The $\|d f\|=w_e(A)$ and $\|f\|=1$.
	We therefore need to show that $w_e(A)\geq  \delta(C_e) $.
	In what follows, $x$ ranges over $Y(2)_e$ and $j$ ranges over $Y(3)_e=\ell(e)$.
	We have
	\begin{align*}
	w_e(A)& = \sum_{x: f|_{\ell(x)}\neq 0} w_e(x)
	=\sum_{x: f|_{\ell(x)}\neq 0}\, \sum_{j: j>x} \frac{1}{|\ell(e)||\{y\in Y(2)_e: y<j\}|}
	\\
	&\geq \sum_{j: f_j\neq 0} \,\sum_{x: e<x<j}  \frac{1}{|\ell(e)||\{y\in Y(2)_e: y<j\}|}
	=\frac{|\{j\in \ell(e)\suchthat f_j\neq0\}|}{|\ell(e)|}\geq \delta(C_e).
	\qedhere
	\end{align*}
\end{proof}

\begin{lem}\label{LM:twolayer-cbe-vertices}
	With notation as above, let
	$v\in Y(0)$, let $\alpha,\beta\geq 0$, $\veps>0$ and let $F\in\N$.
	Suppose that $\ugr(Y_v,w_v)$ is an $(\alpha,\beta)$-skeleton expander,
	$\delta(C_e)\geq \veps$ for every $e\in Y(1)_v$ 
	and $\Fmax_{0,1}(Y_v)\leq F$.
	Then there are constants $Q,Q'>0$, depending only on $F$,
	such that    
	\[\cbe_{-1}(Y_v,w_v,\calF_v)\geq \frac{Q'(Q\veps-\alpha )}{\beta } .\]
	The constants $Q$ and $Q'$
	are  $E$ and $E'''$ from Table~\ref{TB:lgp-zero-constants} once taking $X=Y_v$. 
	In particular,
	we can take $Q=Q'=1$ when $Y$ is lower-regular.
\end{lem}

\begin{proof}
	Let $\calF'_v$ be the sheaf on $Y_v$ obtained from $\calF_v$
	by changing $\calF_v(v)=C_v$ to $0$.
	Then for every $e\in Y(1)_v$,
	we have $\cbe_{-1}(\calF'_e)=\cbe_{-1}(\calF_e)\geq \veps$,
	where the inequality is
	By Lemma~\ref{LM:twolayer-cbe-edges}.
	Applying Theorem~\ref{TH:ccd-lower-bound} to $(Y_v,w_v,\calF'_v)$
	now tells us that   $\ccd_{0}(\calF'_v)\geq \frac{E'''(E\veps-\alpha )}{\beta }$.
	It is therefore enough to show that
	$\cbe_{-1}(\calF_v)\geq \ccd_0(\calF'_v)$.
	Indeed, let $f\in C^{-1}(Y_v,\calF_v)-\{0\}=C_v-\{0\}$.
	Since $\ell(v)=\bigcup_{e\in Y(1)_v}\ell(e)$,
	there is some $e\in X(1)_v$ such that $\res_{e\from v} f=f|_{\ell(e)}\neq 0$.
	As $df(e)=\pm \res_{e\from v} f$, it follows that $df\in Z^0(Y_v,\calF_v)-\{0\}=Z^0(Y_v,\calF'_v)-B^0(Y_v,\calF'_v)$ 
	and so $\|df\|\geq  \ccd_0(\calF'_v)=\ccd_0(\calF'_v)\|f\|$, as claimed.
\end{proof}

\begin{lem}\label{LM:twolayer-cse}
	With notation as above, let $F\in\N$ and $L\in[1,\infty)$.
	Let $E,E',E'',E'''$, $D,D',D''$  
	be as in 
	Table~\ref{TB:lgp-zero-constants} when $X$ is taken to be our $Y$,
	and let $Q,Q'$ be as in
	Lemma~\ref{LM:twolayer-cbe-vertices}
	(recall that
	these constants depend only on $F$ and $L$).
	Suppose that 
	$\Fmax_{i,2}(X)\leq F$ for   $i\in\{0,1 \}$,
	that $L(Y)\leq L$ and
	that
	assumptions (1a)--(2c)
	of Theorem~\ref{TH:two-layer-general} hold
	for some $\veps>0$
	and $\alpha_0,\beta_0,\alpha_{-1},\beta_{-1},\alpha_{||},\beta_{||}\geq 0$.
	Suppose further 
	that
	\[
	\alpha_{-1}\beta_0 + EQ'\alpha_0 < EQQ'\veps 
	\]
	and   there are $h_0,h_{-1},h_{||}\in (0,1]$ such that
	\begin{align*}
		(\alpha_0+\beta_0 h_0) + (\alpha_{||}+\beta_{||}h_{||})+\frac{\alpha_{-1}+\beta_{-1}h_{-1}}{h_0}
		\leq E' \veps.
	\end{align*}
	Then
	\[
	\cse_0(Y,w,\calF)\geq  \frac{E''}{h_0^{-1}h_1^{-1}+h_{||}^{-1}}
	\qquad
	\text{and}
	\qquad
	\ccd_0(Y,w,\calF)\geq \frac{E'''(EQQ'\veps-EQ'\alpha_0-\alpha_{-1}\beta_0)}{\beta_0 \beta_{-1}}.
	\]
	Moreover, 
	if $f\in C^0=C^0(Y,\calF)$
	satisfies $\dist(f,Z^0)<\frac{D}{h_0^{-1}h_1^{-1}+h_{||}^{-1}}$,
	then applying   Algorithm~\ref{AL:correction-algorithm-simplified}  
	to $f$ with the parameter $q= D' h_0 $
	returns $f'\in Z^0$ such that $\dist(f,f')\leq \frac{1}{D''h_0}\dist(f,Z^0)$.
\end{lem}

\begin{proof}
	We show this by applying Theorem~\ref{TH:lgp-zero-detailed}
	to $(Y,w,\calF)$. (Note that $Y$ may not admit an $R$-orientation,
	but  $X=Y(\leq 2)$ has one and  that is enough by Remark~\ref{RM:partial-orientation} below.)
	Condition  (1b) of Theorem~\ref{TH:lgp-zero-detailed}
	holds with $\veps'=\veps$ by Lemma~\ref{LM:twolayer-cbe-edges},
	condition (1c) of Theorem~\ref{TH:lgp-zero-detailed}
	holds holds with $\veps'=\veps$ 
	by condition  (1b) of Theorem~\ref{TH:two-layer-general}
	and Example~\ref{EX:cbe-dim-zero},
	and condition (1a) of Theorem~\ref{TH:lgp-zero-detailed}
	holds with $\veps=\frac{Q'(Q\veps-\alpha_0)}{\beta_0}$
	by Lemma~\ref{LM:twolayer-cbe-vertices}.
	Our assumption that $\alpha_{-1}\beta_0 + EQ'\alpha_0 < EQQ'\veps$
	implies readily that $\alpha_{-1}< E\cdot \frac{Q'(Q\veps-\alpha_0)}{\beta_0}$.
	As  all other assumptions of Theorem~\ref{TH:lgp-zero-detailed}
	clearly hold, we may apply it 
	(with $\frac{Q'(Q\veps-\alpha_0)}{\beta_0}$ and $\veps$
	in place of $\veps$ and $\veps'$)
	and derive all the assertions of the lemma.
\end{proof}

%%	Define $B_i$ ($i\in\{0,1\}$) as in Theorem~\ref{TH:two-layer-general}.
%%	We have $w(x)=\frac{1}{n}\sum_{j\in \ell(x)}\frac{1}{|j(i)|}$.
%%	The right hand side is bigger than
%%	$\frac{\min\{\#\ell(z)\where z\in Y(i)\}}{n \Fmax_{i,3}}$
%%	and smaller than   
%%	$\frac{\max\{\#\ell(z)\where z\in Y(i)\}}{n \Fmin_{i,3}}$.
%%	Since the same holds for $w(y)$, the lemma follows.

The following lemma says that $Z^0(Y,\calF)\subseteq \prod_{v\in X(0)}C_v$ is locally testable and has
linear distance under the assumptions of Theorem~\ref{TH:two-layer-general}.

\begin{lem}\label{LM:twolayer-zero-cocycle-ltc}
	Keep the notation and  assumptions of Lemma~\ref{LM:twolayer-cse}.
	Then the code $Z^0(Y,\calF)\subseteq \prod_{v\in X(0)}C_v$ satisfies
	\[
	\delta(Z^0(Y,\calF))\geq \frac{1}{U_0(\ell)L_0(\ell)}\cdot \frac{E'''(EQQ'-EQ'\alpha_0-\alpha_{-1}\beta_0)}{  \beta_0\beta_{-1}}:=\delta_0
	\]
	and its natural tester has soundness
	\[
	\mu_0:=\frac{1}{U_0(\ell)U_1(\ell)L_0(\ell)L_1(\ell)}\cdot\frac{E''}{h_0^{-1} h_1^{-1} +h_{||}^{-1}}.
	\]
	Moreover,  Algorithm~\ref{AL:correction-algorithm-simplified}  
	with the parameter $q= D' h_0 $ is a decoding algorithm
	for $Z^0(Y,\calF)\subseteq\Sigma^n$ which can fix words that
	are $\eta_0$-close to $Z^0(Y,\calF)$, where
	\[
	\eta_0 = \frac{1}{U_0(\ell)L_0(\ell)}\min\left\{
		\frac{E'''(EQQ'\veps-EQ'\alpha_0-\alpha_{-1}\beta_0)}{ 
			\beta_0\beta_{-1}(1+D''^{-1} h_0^{-1})},
		\frac{D}{h_0^{-1}h_1^{-1} +h_{||}^{-1}}\right\}.
	\]
\end{lem}

\begin{proof}
	By Proposition~\ref{PR:face-weight-ratio-via-irreg},
	for every $i\in\{0,1\}$ and
	$x,y\in X(i)$, we have
	$w(x)\leq U_{i,0}(Y)L_{i,0}(Y) \cdot w(y)=U_i(\ell)L_i(\ell)\cdot w(y)$.
	The assertion about the relative distance and soundness is therefore
	a consequence of Lemma~\ref{LM:twolayer-cse}
	and Lemma~\ref{LM:cse-to-ltc} (with $M=U_0(\ell)L_0(\ell)$ and $M'=U_1(\ell)L_1(\ell)$). Moreover,
	Lemma~\ref{LM:decoding-of-cocycle-code} (with $M=U_0(\ell)L_0(\ell)$) tells us that
	Algorithm~\ref{AL:correction-algorithm-simplified}  
	with the parameter $q= D' h_0 $ can decode words which
	are $\eta_0$-close to $Z^0(X,\calF)$.
\end{proof}

We finally prove Theorem~\ref{TH:two-layer-general}.

\begin{proof}[Proof of Theorem~\ref{TH:two-layer-general}]
	In short, this follows from Lemma~\ref{LM:twolayer-line-code},
	Lemma~\ref{LM:twolayer-zero-cocycle-ltc}	
	and the results in Section~\ref{sec:lifted-vs-line}.
	All that remains is checking the claims about the constants.

	We use the notation of the Theorem~\ref{TH:two-layer-general} and 
	construct 
	$Y$ and $\calF$ above using   $X$, $R$, $\Sigma$, $\{C_e\}_{e\in X(1)}$ given in the theorem.
	We write $U_i:=U_i(\ell)=U_{i,3}(Y)$, $L_i:=L_{-1,i}(\ell)=L_{i,3}(Y)$
	and $D_i=D_i(\ell)$.
	Recall that assumptions (0a)
	and (0b) imply that $L(Y)\leq\max\{L,F\}$.

	Let $E,E',\dots,\delta_0,\mu_0,\eta_0$ be as in 
	Lemma~\ref{LM:twolayer-zero-cocycle-ltc}  applied with $\max\{L,F\}$ in place of $L$.
	We take $S=  EQQ'$, $S'=EQ'$  and $S''= E'$.
	The assumptions of Theorem~\ref{TH:two-layer-general} now imply
	that we may apply
	Lemma~\ref{LM:twolayer-zero-cocycle-ltc} with with the same $(X,\ell)$ 
	and $\{C_e\}_{e\in X(1)}$. 
	Thus,  $\delta(Z^0(Y,\calF))\geq \delta_0$,
	the   natural tester of $Z^0(Y,\calF)$
	has soundness $\mu_0$
	and there is a   
	decoding algorithm for words that are $\eta_0$-close
	to $Z^0(Y,\calF)$ with time complexity $O(|Y(0)|)=O(n)$. 
	Here,  
	the constant in the $O(n)$ depends on 
	$D_0$, $F$, $L$, $\max\{\#C_v\where v\in Y(0)\}$
	and $h_0$ (see Remark~\ref{RM:alg-complexity}
	and Proposition~\ref{PR:face-weight-ratio-via-irreg}).
	However, $\max\{\#C_v\where v\in Y(0)\}\leq |\Sigma|^{D_0 }$.
	
	Let us   realize $C:=C(X({\leq} 1),\ell,\{C_e\}_{e\in Y(1)})$
	as the lifted code $C(\{C_v\}_{v\in Y(0)})\subseteq\Sigma^n$.
	By Lemma~\ref{LM:twolayer-line-code}, the line code
	of $C$ is
	$Z^0(Y,\calF)\subseteq \prod_{v\in X(0)} C_v$. Moreover, by Lemma~\ref{LM:twolayer-cbe-vertices}, 
	and our choice of $S$ and $S'$, we
	have
	\[
	\delta(C_v)\geq \frac{Q'(Q\veps-\alpha_{0})}{\beta_0}=\frac{p}{\beta_0 E} + \frac{\alpha_{-1}}{E}.
	\]
	($p$ is as in \eqref{EQ:two-layer-general:ineq1})
	for all $v\in X(0)$, so the distance of each $C_v$
	is at least $\frac{p\beta_0^{-1}+\alpha_{-1}}{E}\cdot\min\{\#\ell(v)\where v\in X(0)\}=\frac{(p\beta_0^{-1}+\alpha_{-1})D_0}{EU_0}$.
	Now, by Proposition~\ref{PR:line-code-rate-and-dist} 
	(with $D_{\max}=\Fmax_{0,3}(Y)$, $D_{\min}=\Fmin_{0,3}(Y)$,
	$k_{\max}=D_0(\ell)$ and $\tilde{k}\geq \frac{(p\beta_0^{-1}+\alpha_{-1})D_0}{EU_0}$),
	\[
	\delta(C)\geq \frac{\Fmin_{0,3}(Y)\cdot \frac{(p\beta_0^{-1}+\alpha_{-1})D_0}{EU_0}}{\Fmax_{0,3}(Y) D_0 }\delta_0
	=
	\frac{p\beta_0^{-1}+\alpha_{-1}}{E L_0 U_0}
	\cdot \frac{E'''p}{U_0 L_0 \beta_0 \beta_{-1}}
	\geq \frac{1}{U_0^2}
	\cdot \frac{E'''p(p\beta_0^{-1}+\alpha_{-1})}{E L^2 \beta_0 \beta_{-1}}.\]
	Furthermore, by Proposition~\ref{PR:line-code-eff-decoding},
	$C$ has a linear-time decoding algorithms that can
	decode  words that are $\eta$-close to $C$
	for
	\begin{align*}
	\eta
	= \frac{\Fmin_{0,3}}{\Fmax_{0,3}D_0 }\eta_0
	& = 
	\frac{p\beta_0^{-1}+\alpha_{-1}}{ L_0 D_0 }\frac{1}{U_0L_0}\min\left\{
		\frac{E'''p}{ 
			\beta_0\beta_{-1}(1+D''^{-1} h_0^{-1})},
		\frac{D}{h_0^{-1}h_1^{-1} +h_{||}^{-1}}\right\} \\
	& \geq
	\frac{1}{U_0D_0 }\min\left\{
		\frac{E'''p}{ 
			L^2(1+D''^{-1}) \beta_0 \beta_{-1}  h_0^{-1} },
		\frac{D}{L^2(h_0^{-1}h_1^{-1} +h_{||}^{-1})}\right\} .
	\end{align*}
	Finally, by applying Theorem~\ref{TH:line-code-testability} with $(G,\ell)=(\ugr(X),\ell|_{\ugr(X)})$ 
	(in our setting
	$\frac{d_{\max}}{d_{\min}}=\frac{\Fmax_{0,1}}{\Fmin_{0,1}}\leq F$, $\frac{k_{\min}}{k_{\max}}=U_0^{-1}$
	and $\frac{D_{\min}}{D_{\max}}=L_0^{-1}$)
	the natural tester of $C$ has soundness
	\begin{align*}
	\frac{1}{U_0 L_0}\cdot \frac{\mu_0}{\mu_0+2F}
	&= \frac{1}{U_0 L_0 } \cdot \frac{1}{1+2F\mu_0^{-1}}
	=\frac{1}{U_0 L_0} \cdot \frac{1}{1+2FE''^{-1}U_0U_1L_0L_1(h_0^{-1}h_1^{-1}+h_{||}^{-1})}
	\\
	&\geq 
	\frac{1}{U_0 L_0} \cdot \frac{1}{U_0U_1+2 U_0U_1 F E''^{-1}L_0L_1(h_0^{-1}h_1^{-1}+h_{||}^{-1})}
	\\
	&\geq
	\frac{1}{U_0^2 U_1}\frac{\frac{1}{2} F^{-1}E'' L^{-3}}{
	\frac{1}{2} F^{-1}E'' L^{-2}+h_0^{-1}h_1^{-1}+h_{||}^{-1}}.
	\end{align*}
	One can now read from these assertions the  values for
	$T_1,\dots,T_5$.
\end{proof}

\section{Main Result: Technical Version}
\label{sec:tech-versions}

The remainder of this paper is dedicated to proving Theorems~\ref{TH:lgp-simple-version}
and~\ref{TH:lgp-zero-detailed}.
We shall derive both theorems from a single, more general theorem, which we state in this section
and prove in the following sections.

\subsection{Notation}
\label{subsec:polynomials}

Let $X$ be $d$-poset, let $k\in \{0,\dots,d-1\}$
and let $\calP$ be a  $k$-intersection profile for $X$  (\S\ref{subsec:profile}).
Given a quadruple $\rho=(t,\ell,r,b)\in\calP$,
we  write $t(\rho)=t$, $\ell(\rho)=\ell$ and so on.

We shall associate with $X,k,\calP$ a constant $U_\calP$ and Laurent polynomials
$T_k,\dots,T_{-1},S_{\alpha,\beta}$ in the variables $\{x_\rho\}_{\rho\in\calP}$,
making repeated use of the subface counting constants $\Fmin_{i,j,\ell}$ and $\Fmax_{i,j,\ell}$
defined in \S\ref{subsec:lower-regular}.

\medskip

The constant $U_{\calP}\in\N$ is defined as follows:
Recall (Definition~\ref{DF:intersection-profile}) that a pair of faces $(x,y)\in X\times X$
is said to be $\calP$-admissible if $x>y$ and $(\dim x,\dim y)\in\operatorname{Ad}(\calP)$.
Given $x\in X(k+1)$, let $\calV(x)$
denote the set of subsets
$A$ of $\{y\in X\suchthat y\leq x\}$ with $x\in A$
and  the property 
that  for every $a\in A$,
there are faces $x=a_0\geq \dots \geq a_t =a$ in $A$
such that $(a_{i-1},a_i)$ is $\calP$-admissible for all $i\in\{1,\dots,t\}$.
We call an element $a\in A$ terminal if there is no $a'\in A$ such that
$(a,a')$ is $\calP$-admissible, and denote the set of terminal elements in $A$
by $T(A)$.\footnote{
	Beware that a terminal element in $A$ is not necessarily minimal in $A$ with respect
	to inclusion.
}
Now set
\[
U(x)=\max_{A\in\calV(x)} |T(A)|.
\]
and define 
\[
U_{\calP} = U_\calP(X) =\max_{x\in X(k+1)} U(x).
\]

Next, let 
$\R[x_{\rho}^{\pm 1}\where \rho\in\calP]$
be the ring of Laurent polynomials with real coefficients in the variables
$\{x_\rho\}_{\rho\in \calP}$. We define Laurent polynomials
$T_k,T_{k-1},\dots,T_{-1}\in \R[x_\rho^{\pm 1} \where \rho\in \calP]$ 
(depending on $\calP$ and $X$)
inductively by the formula
\[
T_i = \left\{
\begin{array}{ll}
1 & i=k ,\\
\sum_{\rho\in \calP:b(\rho)=i}x_\rho^{-1} \Squares{
c_\rho T_{\ell(\rho)} + c'_\rho T_{r(\rho)}}
& i=k-1,k-2,\dots,-1,
\end{array}\right.
\]
where 
\begin{align*}
c_{(t,\ell,r,i)} &= 
	\frac{\frac{\Fmax_{\ell,d}\Fmax_{i,\ell}}{\Fmin_{i,\ell,d}}
	}{
		\LinkAllLo{i}{\ell}{d}  
		\ContZwZLo{i}{\ell}{d}				
		+
		\LinkAllLo{i}{r}{d}  
		\ContZwZLo{i}{r}{d}
	}
\qquad\text{and}\qquad
c'_{(t,\ell,r,i)}  = c_{(t,r,\ell,i)}.
\end{align*}
Suppose further that   we are given vectors
$\alpha=\{\alpha_\rho\}_{\rho\in\calP}
$
and
$\beta
=\{\beta_{\rho}\}_{\rho\in\calP}$ in $\R^{\calP}$. We then define 
$S_{\alpha,\beta}\in \R[x_\rho^{\pm1}\where \rho\in\calP]$ by
\[
S_{\alpha,\beta}
=\sum_{\rho=(t,\ell,r,i)\in\calP}
\!\!
\Sci
\Squares{
\Scii{\ell}
T_\ell+
\Scii{r}
T_r
}
(\alpha_\rho +\beta_\rho x_\rho).
\]

\begin{example}
	(i) We will see later in the proof
	of Lemma~\ref{LM:constants-bound}(i)  
	that 
	if $X$ is lower-regular, then
	$c_{\rho}=c'_{\rho}=\frac{1}{2}$ for every $\rho\in \calP$.
	This can also be seen directly using Lemma~\ref{LM:relation-between-face-numbers}.
	
	(ii) Suppose  that $X$ is a pure $d$-dimensional simplicial complex
	and $\calP=\calP^{(k)}_{\triangle}$ as in Example~\ref{EX:simp-intersection-prof}.
	Then 
	$\calP=\{(i+1,i,i,i-1)\where i\in\{0,\dots,k\}\}$
	and $\Fmin_{i,j,\ell}=\Fmax_{i,j,\ell}={\ell- i \choose \ell-j}$.
	Abbreviating $x_{(i+1,i,i,i-1)}$ to $x_{i-1}$, it is straightforward to   check 
	that 
	\[
	T_i(x_{-1}, \dots,x_{k-1})= \frac{1}{x_{i } x_{i+1}\cdots x_{k-1}}.
	\]
	
	(iii) Suppose that $X$ is a pure $d$-dimensional cube complex
	and $\calP=\calP^{(k)}_{\square}$ as in Example~\ref{EX:cube-intersection-prof}.
	We use the same abbreviations as in (ii) and also write
	$y_i=x_{(i+1,i,i,-1)}$ for $i\in\{0,\dots,k\}$ (so $y_0=x_{-1}$). One now readily checks
	that $T_i(x_{0},\dots,x_{k-1},y_0,\dots,y_k)=\frac{1}{x_{i } x_{i+1}\cdots x_{k-1}}$
	for $i\in\{0,\dots,k\}$ while
	\[
	T_{-1}(x_{0},\dots,x_{k-1},y_0,y_1,\dots,y_k)=
	\frac{1}{y_{0 } x_{0}\cdots x_{k-1}}+
	\frac{1}{y_{1 } x_{1}\cdots x_{k-1}}+
	\dots+
	\frac{1}{y_{k } }.	
	\]
\end{example}

More bounds on the coefficients of the $T_i$ and $S_{\alpha,\beta}$,
as well as   explicit computations for small $k$,
will be given in \S\ref{subsec:bounds} below.

\subsection{Main Theorem}
\label{subsec:second-tech}

	Let $R$ be a commutative ring,
	let $(X,w)$ be a properly weighed $R$-oriented $d$-poset,
	let $k\in\{0,\dots,d-2\}$, 
	let $\calP$ be a $k$-intersection
	profile for $X$ and let $\calP'$ be a $(k+1)$-intersection profile for $X$.
	
	Let $\{\veps_i\}_{i=0}^k$, $\{\veps'_i\}_{i=0}^{k+1}$, $\alpha= \{\alpha_\rho\}_{\rho\in\calP}$,
	$\beta=\{\beta_\rho\}_{\rho\in\calP}$,
	$\alpha'= \{\alpha'_\rho\}_{\rho\in\calP'}$,
	$\beta'=\{\beta'_\rho\}_{\rho\in\calP'}$ 
	be lists of non-negative real numbers.
	Let $U_{\calP}$ and $T_k,\dots,T_{-1},S_{\alpha,\beta}\in\R[x_\rho^{\pm1}\where
	\rho\in \calP]$
	be as in \S\ref{subsec:polynomials}.
	We   
	define $U_{\calP'}$
	and $T'_{k+1},\dots,T'_{-1},S'_{\alpha',\beta'}\in\R[x_{\rho'}^{\pm1}\where
	\rho'\in \calP']$ similarly by replacing $\calP$
	and $k$ with $\calP'$ and $k+1$,
	respectively. 
	Finally,
	let
	\begin{align*}
	\tilde{\veps}&=\min\left\{
		\Duc{i}\veps_{i}  
		\where i\in\{0,\dots,k\}
	\right\}, \\
	\tilde{\veps}'&=\min\left\{
		\Ducprime{i}\veps'_i
		\where i\in\{0,\dots,k+1\}
	\right\}
	\end{align*}
	and let
	\[
	C=\CorR.
	\]

\begin{thm}\label{TH:main-very-technical}
	With notation as above,  
	let $\calF$ be an $R$-sheaf on $X$ such that:
	\begin{enumerate}[label=(\arabic*)]
		\item[(1a)] $\cbe_{k-\dim u-1}(X_u,w_u,\calF_u)\geq \veps_{\dim u}$
		for every $u\in X(0)\cup\dots \cup X(k)$;
		\item[(1b)] $\cbe_{k-\dim u}(X_u,w_u,\calF_u)\geq  \veps'_{\dim u}$
		for every $u\in X(0)\cup\dots \cup X(k+1)$;
		\item[(2a)] $\NI^{\ell,r,t}_u(X)$ is an $(\alpha_\rho,\beta_\rho)$-skeleton
		expander for every $\rho=(t,\ell,r,b)\in \calP$ and $u\in X(b)$;
		\item[(2b)] $\NI^{\ell,r,t}_u(X)$ is an $(\alpha'_\rho,\beta'_\rho)$-skeleton
		expander for every $\rho=(t,\ell,r,b)\in \calP'$ and $u\in X(b)$.\footnote{
			One often has $\calP'\supseteq \calP$, so this condition
			may subsume (2a).		
		}
	\end{enumerate}
	Suppose that  there exist    $h=\{h_\rho\}_{\rho\in\calP}
	\in (0,1]^{\calP}$,
	$h'=\{h'_\rho\}_{\rho\in\calP'}
	\in (0,1]^{\calP'}$ 
	and $q\in \R_+$ such that
	\begin{align}
	\label{EQ:main-very-technical:p}
	p & :=\tilde{\veps}-U_{\calP}S_{ {\alpha}, {\beta}}(h)
	> 0\qquad\text{and}
	\\
	\label{EQ:main-very-technical:p-prime}
	p' & :=\tilde{\veps}'-U_{\calP'}S'_{ {\alpha}', {\beta}'}(h')-
	q  \sum_{i=0}^{k}
	\qc \Ducprime{i} \veps'_i	
	T'_i(h') > 0.
	\end{align}
	Then 
	\[
	\cse_k(X,w,\calF)\geq \min\left\{T'_{-1}(h')^{-1}, \frac{q}{C} \right\}
	\qquad
	\text{and}
	\qquad
	\ccd_k(X,w,\calF)\geq T_{-1}(h)^{-1}.
	\]	
	Moreover, if $f\in C^k= C^k(X,\calF)$
	satisfies $\dist_w(f,Z^k )< \wZContZLo{k}{k+1}{d} T'_{-1}(h')^{-1}$,
	then  applying Algorithm~\ref{AL:correction-algorithm-simplified}
	to $f$ and $q$   returns
	$f'\in Z^k $ such that  $\dist_w(f,f')\leq 
	\frac{C}{q} 
	\ContZwZUp{k}{k+1}{d}	
	\dist_w(f,Z^k)$.
\end{thm}

\begin{remark}\label{RM:partial-orientation}
	The assumption that $X$ is $R$-oriented in Theorem~\ref{TH:main-very-technical}
	can be relaxed to assuming that there is an $R$-orientation on the subposet
	$X(k-1)\cup X(k)\cup X(k+1)\cup X(k+2)$. When $\calF(x)=0$ for every
	$x\in X(k-1)$, it is even enough to assume that  $X(k)\cup X(k+1)\cup X(k+2)$
	is $R$-oriented. Indeed, all that is needed for the proof is that
	the cochain complex $C^{k-1}\to C^k\to C^{k+1}\to C^{k+2}$ is well-defined
	for our $X$ and $\calF$.
\end{remark}

The remainder of this section is dedicated to deriving Theorems~\ref{TH:lgp-simple-version}
and~\ref{TH:lgp-zero-detailed} from Theorem~\ref{TH:main-very-technical}.
Theorem~\ref{TH:main-very-technical} itself
will be proved in the next two sections.

\subsection{Bounds on Constants}
\label{subsec:bounds}

We first prove some lemmas which bound the constants
and the coefficients of the Laurent polynomials
defined in \S\ref{subsec:polynomials}
and \S\ref{subsec:second-tech}
in terms of
the lower irregularity of $X$ 
and the constants $\Fmax_{i,j}$  (\S\ref{subsec:lower-regular}). 

\begin{lem}\label{LM:constants-bound}
	With notation as in \S\ref{subsec:polynomials}, 
	write $F=\max\{\Fmax_{i,j}\where -1\leq i\leq j\leq k+1\}$ and $L=L(X)$.
	Then for every $\rho=(t,\ell,r,i)\in\calP$,
	we have:
	\begin{enumerate}[label=(\roman*)]
		\item \label{item:constants-bound:c}
		$\displaystyle
		c_{(t,\ell,r,i)}\leq 
			\frac{
		L_{\ell,d}L_{i,\ell}L_{i,\ell,d} L_{i,d}
		}{
		L_{\ell,d}^{-1}L_{i,\ell,d}^{-1}  
		+
		L_{r,d}^{-1}L_{i,r,d}^{-1}  
		}
			\leq \frac{1}{2}L^6 
		$,
		\item \label{item:constants-bound:Sci}
		$\displaystyle\Sci   \leq 
		L_{t,k+1,d}L_{k+1,d}
		L_{i,t,d} L_{t,d}^2
		\frac{\Fmax_{t,k+1} \Fmax_{i,t}}{2\Fmin_{i,d}}
		\leq \frac{1}{2}L^5F^2$
		\item \label{item:constants-bound:Scii-l}
		$\Scii{\ell} \leq
		L_{i,\ell,d} L_{\ell,d}
		\frac{(\Fmax_{i,\ell})^2}{\Fmin_{i,d}} \leq L^2 F^2
		$.
	\end{enumerate}
	In addition, for every $-1\leq i\leq k$, we have:
	\begin{enumerate}[resume, label=(\roman*)]
		\item \label{item:constants-bound:Duc}
		$
		\displaystyle
		\frac{1}{L^5F}
		\leq
		\frac{1}{L_{i,d}L_{i,k+1,d}L_{k,d}L_{i,k,d}L_{k+1,d}}\frac{\Fmin_{i,k+1}}{\Fmax_{i,k}}
		\leq		
		 \Duc{i} \leq L_{i,d}\frac{\Fmax_{i,k+1}}{\Fmin_{i,k}}
		\leq LF$,
		\item \label{item:constants-bound:CorR}
		$\displaystyle 
		1\leq \Fmin_{i,k}\leq \frac{\Fmax_{i,k,d}\Fmax_{i,d}}{\Fmin_{k,d}} \leq
		L_{i,k,d}L_{i,d}L_{k,d}\Fmax_{i,k}\leq L^3F$.
	\end{enumerate}
\end{lem}

\begin{proof}
	By Lemma~\ref{LM:relation-between-face-numbers}, whenever
	$0\leq i\leq j\leq d$, we have
	\[\frac{\Fmin_{i,j}\Fmin_{j,d}}{\Fmax_{i,d}\Fmax_{i,j,d}}\leq1
	\qquad\text{and}\qquad
	\frac{\Fmin_{i,d}\Fmin_{i,j,d}}{\Fmax_{i,j}\Fmax_{j,d}}\leq1.\]
	Making repeated use of these inequalities
	and   the definition of $L_{i,j,k}$,
	we now prove  each of (i)--(v) in turn.
	\begin{align*}
	c_{(t,\ell,r,i)} &= 
	\frac{\frac{\Fmax_{\ell,d}\Fmax_{i,\ell}}{\Fmin_{i,\ell,d}}
	}{
		\LinkAllLo{i}{\ell}{d}  
		\ContZwZLo{i}{\ell}{d}				
		+
		\LinkAllLo{i}{r}{d}  
		\ContZwZLo{i}{r}{d}
	}
	=
	\frac{
		L_{\ell,d}L_{i,\ell}L_{i,\ell,d} 
		\frac{\Fmin_{\ell,d}\Fmin_{i,\ell}}{\Fmax_{i,\ell,d}\Fmax_{i,d}}
		\Fmax_{i,d}	
	}{
		L_{\ell,d}^{-1}L_{i,\ell,d}^{-1}\Fmin_{i,d} 
		+
		L_{r,d}^{-1}L_{i,r,d}^{-1}\Fmin_{i,d} 
	}
	\\
	&\leq 
	\frac{
		L_{\ell,d}L_{i,\ell}L_{i,\ell,d} L_{i,d}
	}{
		L_{\ell,d}^{-1}L_{i,\ell,d}^{-1}  
		+
		L_{r,d}^{-1}L_{i,r,d}^{-1}  
	}
	\end{align*}
	This proves (i).
	\begin{align*}
	\Sci  & =
	\frac{
		\Fmax_{t,k+1,d} \Fmax_{i,t,d} 
	}{
		2\Fmin_{k+1,d}
	}
	L_{t,d}
	\\
	&=
	\frac{
		\Fmax_{t,k+1,d}\Fmin_{t,d} 	
	}{
		\Fmin_{k+1,d}\Fmax_{t,k+1}
	}
	\Fmax_{t,k+1}
	\frac{
		\Fmax_{i,t,d}\Fmin_{i,d}
	}{
		\Fmin_{t,d}\Fmax_{i,t}
	}
	\frac{\Fmax_{i,t}}{2\Fmin_{i,d}}
	L_{t,d}
	\\ 
	&=
	\frac{
		\Fmin_{t,k+1,d}\Fmin_{t,d} 	
	}{
		\Fmax_{k+1,d}\Fmax_{t,k+1}
	}
	L_{t,k+1,d}L_{k+1,d}
	\frac{
		\Fmin_{i,t,d}\Fmin_{i,d}
	}{
		\Fmax_{t,d}\Fmax_{i,t}
	}
	L_{i,t,d} L_{t,d}
	\frac{\Fmax_{t,k+1} \Fmax_{i,t}}{2\Fmin_{i,d}}
	L_{t,d} \\
	&\leq
	L_{t,k+1,d}L_{k+1,d}
	L_{i,t,d} L_{t,d}^2
	\frac{\Fmax_{t,k+1} \Fmax_{i,t}}{2\Fmin_{i,d}}
	\end{align*}
	This proves (ii).
	\begin{align*}
	\Scii{\ell} & =
		\frac{
			\Fmin_{i,\ell,d}\Fmin_{i,d}		
		}{
			\Fmax_{\ell,d}\Fmax_{i,\ell}
		}
		L_{i,\ell,d} L_{\ell,d}
		\frac{(\Fmax_{i,\ell})^2}{\Fmin_{i,d}}
		\leq
		L_{i,\ell,d} L_{\ell,d}
		\frac{(\Fmax_{i,\ell})^2}{\Fmin_{i,d}}
	\end{align*}
	This proves (iii). Next,
	\begin{align*}
	\Duc{i}
	&= 
		\frac{
			\Fmin_{k,d}\Fmin_{i,k}
		}{
			\Fmax_{i,k,d}\Fmax_{i,d}
		}
		\frac{
			\Fmin_{i,k+1,d}\Fmin_{i,d}
		}{
			\Fmax_{k+1,d}\Fmax_{i,k+1}
		}
		\frac{\Fmax_{i,d}\Fmax_{i,k+1}}{\Fmin_{i,k}\Fmin_{i,d}}
		\leq
		L_{i,d}\frac{ \Fmax_{i,k+1}}{\Fmin_{i,k} }.
	\end{align*}
	On the other hand, 
	\begin{align*}
	\Duc{i}
	&= 
	\frac{\Fmax_{i,k+1,d}\Fmax_{k,d}}{\Fmin_{i,k,d}\Fmin_{k+1,d}}
	\frac{1}{L_{i,k+1,d}L_{k,d}L_{i,k,d}L_{k+1,d}}
	\\
	&=
	\frac{
			\Fmax_{k,d}\Fmax_{i,k}
		}{
			\Fmin_{i,k,d}\Fmin_{i,d}
		}
		\frac{
			\Fmax_{i,k+1,d}\Fmax_{i,d}
		}{
			\Fmin_{k+1,d}\Fmin_{i,k+1}
		}
		\frac{\Fmin_{i,d}\Fmin_{i,k+1}}{\Fmax_{i,k}\Fmax_{i,d}}
		\frac{1}{L_{i,k+1,d}L_{k,d}L_{i,k,d}L_{k+1,d}}
	\\
	&\geq \frac{1}{L_{i,d}L_{i,k+1,d}L_{k,d}L_{i,k,d}L_{k+1,d}}\frac{\Fmin_{i,k+1}}{\Fmax_{i,k}}.
	\end{align*}
	This proves (iv).
	Finally, for (v), note that
	\begin{align*}
	\Fmin_{i,k} &\leq \frac{\Fmax_{i,k,d}\Fmax_{i,d}}{\Fmin_{k,d}}
	=\frac{\Fmin_{i,k,d}\Fmin_{i,d}}{\Fmax_{i,k}\Fmax_{k,d}}
	L_{i,k,d}L_{i,d}L_{k,d}\Fmax_{i,k}
	\leq L_{i,k,d}L_{i,d}L_{k,d}\Fmax_{i,k}.
	\end{align*}
	This completes the proof.
\end{proof}

\begin{lem}\label{LM:UP-bound}
	With notation as in \S\ref{subsec:polynomials},
	$U_\calP\leq \sum_{i=0}^k\Fmax_{i,k+1}$.
\end{lem}

When $k=0$, we   have $U_\calP=\Fmax_{0,1}$,
but when $k>0$,
$\calU_P$ is usually smaller than $\sum_{i=0}^k\Fmax_{i,k+1}$.

\begin{proof}
	Let $x\in X(k+1)$
	and let $A\in\calV(x)$ (see \S\ref{subsec:polynomials}).
	We need to show that $T(A)$, the set of terminal elements in $A$,
	contains at most $\sum_{i=0}^k\Fmax_{i,k+1}$ elements.
	If $x\in T(A)$, then we must have $A=\{x\}$ and lemma holds.
	If $\emptyset\in T(A)$, then we must have $T(A)=\{\emptyset\}$
	(because $(k+1,-1),\dots,(0,-1)\in\operatorname{Ad}(\calP)$),
	and again the lemma  holds.
	When both $x$ and $\emptyset$
	are not in $T(A)$, we have $T(A)\subseteq \bigcup_{i=0}^k x(i)$
	and thus $|T(A)|\leq \sum_{i=0}^k\Fmax_{i,k+1}$.
\end{proof}

\begin{lem}\label{LM:substitutions-bound}
	With notation as in \S\ref{subsec:polynomials}, 
	let $F=\max\{\Fmax_{i,j}\where -1\leq i\leq j\leq k+1\}$ and $L=L(X)$.
	Let $u\in (0,1]$, $A,B\in \R_+$, $B'\in [1,\infty)$ 
	and for every $\rho=(t,\ell,r,b)\in \calP$, let 
	\[h_\rho=\left\{\begin{array}{ll}
		u^{2^{k-b}-2^{k-r}} & b>-1 \\
		B'^{-1} u^{2^{k+1}-2^{k-r}} & b=-1
	\end{array}\right.,
	\qquad
	\alpha_\rho=A u^{2^{k-r}},
	\qquad
	\beta_\rho=\left\{\begin{array}{ll}
	B u^{2^{k+1-r}-2^{k-b}} & b>-1 \\
	B'B u^{2^{k+1-r}-2^{k+1}} & b=-1
	\end{array}
	\right. .\]
	Then:
	\begin{enumerate}[label=(\roman*)]
		\item For every $i\in\{-1,\dots,k\}$,  there is a real constant $Q_i>0$
		depending only on $k$ such that
		$T_i(h)\leq Q_i L^{6(k-i)} u^{1-2^{k-i}}$ if $i>-1$
		and $T_{-1}(h)\leq B'  Q_i L^{6(k+1)} u^{1-2^{k+1}}$.
		\item There is a real constant $P>0$ depending only on 
		$k$ such that 
		$S_{\alpha,\beta}(h)\leq  P L^{6k+7}F^4 (A+B)u$.
		\item 
		$\displaystyle \sum_{i=0}^{k-1}
		\Duc{i} T_i(h)\leq (\sum_{i=0}^{k-1} Q_i) L^{6k+1} F u^{1-2^k} $
		with $Q_0,\dots,Q_{k-1}$ as in (i).
	\end{enumerate}
\end{lem}

\begin{proof}
	If $\calP'$ is a $k$-intersection profile for $X$ which contains $\calP$, then replacing
	$\calP$ by $\calP'$ cannot decrease the left hand sides of the inequalities
	in (i), (ii) and (iii). We may therefore replace $\calP$
	with the maximal $k$-intersection profile of Example~\ref{EX:intersection-profiles}(iii).
	This ensures that the constants $Q_i$ and $P$ that we shall define
	depend only on $k$ and
	not on $\calP$. (However,
	a smaller $\calP$ allows for   smaller  constants $Q_i$ and $P$.) 

	(i) Define the $Q_i$ inductively for $i=k,k-1,\dots,-1$
	by setting $Q_k=1$ and 
	\[
	Q_i=\frac{1}{2}\sum_{\rho\in \calP:b(\rho)=i} (Q_{\ell(\rho)}+Q_{r(\rho)})
	\]
	for $i<k$. The desired inequality now follows by decreasing induction on $i$.
	Indeed,
	the case $i=k$ is clear since $T_k(h)=1$.
	Suppose now that $j<k$ and  inequality was verified for all $i\in\{k,k-1,\dots,j+1\}$.
	When $j>-1$,
	Lemma~\ref{LM:constants-bound}(i) tells us that
	\begin{align*}
		T_j(h) 
		&\leq  
			\sum_{\rho\in \calP:b(\rho)=j}h_\rho^{-1} 
			\frac{L^6}{2}(T_{\ell(\rho)}(h) + T_{r(\rho)}(h))
		\\
		&\leq
			\frac{L^6}{2}
			\sum_{\rho\in \calP:b(\rho)=j} 
			u^{2^{k-r(\rho)}-2^{k-j}}
			(Q_{\ell(\rho)}L^{6(k-\ell(\rho))}
			u^{1-2^{k-\ell(\rho)}} +
			Q_{r(\rho)}L^{6(k-r(\rho))}
			u^{1-2^{k-r(\rho)}})
		\\
		&\leq
			\frac{L^6}{2}
			\sum_{\rho\in \calP:b(\rho)=j}
			(Q_{\ell(\rho)} + Q_{r(\rho)})
			L^{6(k-j-1)}
			u^{2^{k-r(\rho)}-2^{k-j}+1-2^{k-r(\rho)}}
		\\
		&=Q_j L^{6(k-j)} u^{1-2^{k-j}}.
	\end{align*}
	When $j=-1$, a similar computation shows that $T_{-1}(h)\leq B'  Q_{-1} L^{6(k-j)} u^{1-2^{k-j}}$.
	
	(ii) Take $P=\frac{1}{2}\sum_{\rho\in\calP}(Q_{\ell(\rho)}+Q_{r(\rho)})$. 
	Observe that $\beta_\rho h_\rho = Bu^{2^{k-r(\rho)}}$ for all $\rho\in\calP$.
	Now,
	by (i) and Lemma~\ref{LM:constants-bound}(ii)--(iii), 
	\begin{align*}
		S_{\alpha,\beta} (h) 
		&\leq 
			\sum_{\rho=(t,\ell,r,i)\in\calP}  
			\frac{1}{2}L^5F^2
			\Circs{L^2F^2 T_{\ell}(h) +
			L^2F^2 T_{r}(h)}(\alpha_\rho +\beta_\rho h_\rho)
		\\
		&\leq
			\sum_{\rho=(t,\ell,r,i)\in\calP}  
			\frac{1}{2}L^7F^4 
			\Circs{  Q_\ell L^{6(k-\ell)} u^{1-2^{k-\ell}} +
			Q_r L^{6(k-r)} u^{1-2^{k-r}}}
		\cdot
			\Circs{
				Au^{2^{k-r}}
				+Bu^{2^{k-r}}
			}
		\\
		&\leq  
			\sum_{\rho=(t,\ell,r,i)\in\calP}  
			\frac{1}{2}L^{6k+7}F^4
			\Circs{Q_\ell   +
			Q_r}
			(A+B)u^{1-2^{k-r}+2^{k-r}}
		\\
		& =P L^{6k+7}F^4 (A+B)u.
	\end{align*}
	
	(iii) By (i) and   Lemma~\ref{LM:constants-bound}(iv),
	\begin{align*}
		\displaystyle \sum_{i=0}^{k-1}
		\Duc{i} T_i(h)
		& \leq
		\sum_{i=0}^{k-1} LF \cdot Q_i L^{6(k-i)} u^{1-2^{k-i}}
		\leq (\sum_{i=0}^{k-1} Q_i) L^{6k+1} F u^{1-2^k} .
		\qedhere
	\end{align*}
\end{proof}

\begin{lem}\label{LM:ST-k-zero}
	With notation as in \S\ref{subsec:polynomials}, 
	suppose that $k=0$ and $\calP=\{(1,0,0,-1)\}$
	(cf.\ Example~\ref{EX:intersection-profiles}(i)).
	We abbreviate the variable $x_{(1,0,0,-1)}$ to $x_{-1}$ and similarly
	for other variables.
	Then, for every $\alpha_{-1},\beta_{-1}\in\R$, we have:
	\[
	T_{-1}(x_{-1})=\frac{L_{0,d}^3}{x_{-1}}
	\qquad\text{and}
	\qquad
	S_{\alpha,\beta}(x_{-1})=L_{1,d}^2 L_{0,d}(\alpha_{-1}+\beta_{-1} x_{-1}).
	\]
\end{lem}

\begin{proof}
	By direct computation.
\end{proof}

\begin{lem}\label{LM:ST-k-one}
	With notation as in \S\ref{subsec:polynomials}, 
	suppose that $k=1$ and $\calP=\{(2,1,1,0)$, $(2,1,1,-1)$, $(1,0,0,-1)\}$
	(cf.\ Example~\ref{EX:intersection-profiles}(ii)).
	We abbreviate the variables $x_{(2,1,1,0)},
	x_{(1,0,0,-1)},x_{(2,1,1,-1)}$ to $x_0,x_{-1},x_{||}$, respectively,
	and similarly
	for other variables indexed by $\rho\in\calP$.
	Suppose that $h\in \R_+^{\calP}$, $\alpha,\beta\in[0,\infty)^{\calP}$ and let $L=L(X)$.
	Then:
	\begin{enumerate}
		\item[(i)] $T_0(h)\leq \frac{L^6}{h_0}$,
		\item[(ii)] $T_{-1}(h)\leq \frac{L^{12}}{h_0h_{-1}}+\frac{L^6}{h_{||}}$,
		\item[(iii)]
		$S_{\alpha,\beta}(h)\leq L^6\frac{\Fmax_{0,2}(\Fmax_{0,1})^2}{(\Fmin_{0,d})^2}(\alpha_0+\beta_0 h_0)
		+L^6(\alpha_{||}+\beta_{||}h_{-1})+L^{13}\Fmax_{1,2}\frac{\alpha_{-1}+\beta_{-1}h_{-1}}{h_0}$.
	\end{enumerate}
\end{lem}

\begin{proof}
	For (i), note that $T_0(h)=\frac{c_{(2,1,1,0)}}{h_{0}}$ and $c_{(2,1,1,0)}\leq L^6$
	by Lemma~\ref{LM:constants-bound}(i).
	To see (ii), we use (i) and Lemma~\ref{LM:constants-bound}(i)   to get
	\[
	T_{-1}(h)=\frac{c_{(1,0,0,-1)}T_0(h)}{h_{-1}} +
	\frac{c_{(2,1,1,-1)}T_1(h)}{h_{||}}\leq
	\frac{L^{12}}{h_0h_{-1}}+\frac{L^6}{h_{||}}.
	\]
	Finally, by the definition of $S_{\alpha,\beta}$, parts (i) and (ii),
	and Lemma~\ref{LM:constants-bound}(ii)--(iii), we have
	\begin{align*}
	S_{\alpha,\beta}(h)
	&=L^4L_{2,2,d} \frac{\Fmax_{2,2}\Fmax_{0,2}}{\Fmin_{0,d}}\cdot L^2\frac{(\Fmax_{0,1})^2}{\Fmin_{0,d}}
		\cdot (\alpha_0+\beta_0 h_0)T_{1}(h)
	\\
	&\phantom{=}
	+L^4L_{2,2,d} \frac{\Fmax_{2,2}\Fmax_{-1,2}}{\Fmin_{-1,d}}\cdot L^2\frac{(\Fmax_{-1,1})^2}{\Fmin_{-1,d}}
		\cdot (\alpha_{||}+\beta_{||} h_{||})T_{1}(h)
	\\
	&\phantom{=}
	+L^5 \frac{\Fmax_{1,2}\Fmax_{-1,1}}{\Fmin_{-1,d}}\cdot L^2\frac{(\Fmax_{-1,0})^2}{\Fmin_{-1,d}}
		\cdot (\alpha_{-1}+\beta_{-1} h_{-1})T_{0}(h)
	\\
	&\leq  L^6\frac{\Fmax_{0,2}(\Fmax_{0,1})^2}{(\Fmin_{0,d})^2}(\alpha_0+\beta_0 h_0)
		+L^6(\alpha_{||}+\beta_{||}h_{-1})+L^{13}\Fmax_{1,2}\frac{\alpha_{-1}+\beta_{-1}h_{-1}}{h_0}.
	\end{align*}
	This proves (iii).
\end{proof}

\subsection{Proofs of Theorems~\ref{TH:lgp-simple-version} and~\ref{TH:lgp-zero-detailed}}

We now use Theorem~\ref{TH:main-very-technical} to prove
Theorems~\ref{TH:lgp-simple-version} and~\ref{TH:lgp-zero-detailed}.

\begin{proof}[Proof of Theorem~\ref{TH:lgp-simple-version} assuming Theorem~\ref{TH:main-very-technical}]
	We use the notation of Theorem~\ref{TH:lgp-simple-version}.	
	Recall that we are given $B\in\R_+$, $F\in\N$, $L\in [1,\infty)$
	and $k\in\{0\}\cup\N$.
	Recall also  that   $(X,w)$ is a properly weighted $R$-oriented $d$-poset ($d\geq k+2$)
	with $L(X)\leq L$ and $\Fmax_{i,j}(X)\leq F$
	for all $-1\leq i\leq j\leq k+2$.
	In addition $\calP$  
	is a $k$-intersection  profile for $X$, $\calP'$
	is a $(k+1)$-intersection profile for $X$ and $\calF$
	is an $R$-sheaf on $X$.

	We will apply Lemma~\ref{LM:substitutions-bound} both to $k$, $\calP$ and   $k+1$, $\calP'$. The constants provided
	by the lemma in the second case will be denoted $Q'_{k+1},\dots,Q'_{-1}$
	and $P'$.
	
	Let $K >0$ be a constant  depending on $B,F,L,k$ to be specified later.
	Let $\veps>0$ and $B'\in [1,\infty)$.
	For every $i\in\{0,\dots,k\}$ and $j\in\{0,\dots,k+1\}$, 
	set
	\[
	\veps_i=\veps'_j=\veps.
	\]
	In addition, 
	for every $\rho=(t,\ell,r,b)\in\calP$ and $\rho'=(t',\ell',r',b')\in\calP'$,
	define 
	\begin{align*}
	\alpha_\rho &=  (K\veps)^{2^{k-r}},
	&
	\alpha'_{\rho'} &=  (K\veps)^{2^{k+1-r'}}, 
	\\
	\beta_\rho &= 
		\left\{\begin{array}{ll}
		B (K\veps)^{2^{k+1-r}-2^{k-b}} & b>-1 \\
		B'B(K\veps)^{2^{k+1-r}-2^{k+1}} & b=-1
		\end{array}
		\right. , 
	&
	\beta'_{\rho'} &= 
		\left\{\begin{array}{ll}
		B (K\veps)^{2^{k+2-r'}-2^{k+1-b'}} & b'>-1 \\
		B'B (K\veps)^{2^{k+2-r'}-2^{k+2}} & b'=-1
		\end{array}
		\right. , 
	\\
	h_{\rho} &= 
		\left\{\begin{array}{ll}
			(K\veps)^{2^{k-b}-2^{k-r}} & b>-1 \\
			B'^{-1} (K\veps)^{2^{k+1}-2^{k-r}} & b=-1
		\end{array}\right. ,
	&
	h'_{\rho'} &= 
		\left\{\begin{array}{ll}
			(K\veps)^{2^{k+1-b'}-2^{k+1-r'}} & b'>-1 \\
			B'^{-1}(K\veps)^{2^{k+2}-2^{k+1-r'}} & b'=-1
		\end{array}\right. .
	\end{align*}
	Finally, let 
	\[q=K(K\veps)^{2^{k+1}-1}.\]
	
	In order to prove Theorem~\ref{TH:lgp-simple-version}, it is enough
	to show that $K$   can chosen in such a way that
	that $K\veps\leq 1$
	and the  inequalities 
	\eqref{EQ:main-very-technical:p}
	and \eqref{EQ:main-very-technical:p-prime}
	of	
	Theorem~\ref{TH:main-very-technical}
	hold. Indeed, suppose
	that assumptions (1a)--(2b) of Theorem~\ref{TH:lgp-simple-version} hold with   $\veps$
	and $B'$ above.
	Then conditions (1a)--(2b) of Theorem~\ref{TH:main-very-technical}
	with our $\veps_i$, $\veps'_i$, $\alpha_\rho$ and $\beta_\rho$ hold as well.
	(Indeed, for $\rho=(t,\ell,r,i)\in\calP\cup \calP'$
	and $z\in X(i)$,
	the hypergraph $\NI^{\ell,r,t}_z(X,w)$ is an $(\alpha'_\rho,\beta_\rho)$-skeleton 
	expander, and hence also
	an 	$(\alpha'_\rho,\beta'_\rho)$- and
	an $(\alpha_\rho,\beta_\rho)$-skeleton expander, because $\alpha'_\rho\leq \alpha_\rho$
	and $\beta_\rho\leq \beta'_\rho$.)	
	Provided that \eqref{EQ:main-very-technical:p}
	and \eqref{EQ:main-very-technical:p-prime} also hold,
	all the conclusions of Theorem~\ref{TH:main-very-technical} are true.
	Then, by applying Lemma~\ref{LM:substitutions-bound}(i) with $u=K\veps\in (0,1]$
	(both for $k$, $\calP$ 
	and $k+1$, $\calP'$),
	we see that
	$T_{-1}(h)\leq B' Q_{-1}L^{6k+6}(K\veps)^{1-2^{k+1}}$
	and $T'_{-1}(h')\leq B'  Q'_{-1}L^{6k+12}(K\veps)^{1-2^{k+2}}$. In addition, 
	by Lemma~\ref{LM:constants-bound}(v), we have
	$C\leq L^3F$ and $L^{-3}F^{-1}\leq \frac{\Fmin_{k+1,d}}{\Fmax_{k,k+1,d}\Fmax_{k,d}}  $.
	Combining this with the conclusions of Theorem~\ref{TH:main-very-technical},
	we get that 
	\begin{align*}
	\cse_k(X,w,\calF)& \geq \min\left\{T'_{-1}(h')^{-1},\frac{q}{C}\right\}
	\geq \min\left\{\frac{(K\veps)^{2^{k+2}-1}}{B'Q'_{-1}L^{6k+12}},
	\frac{ (K\veps)^{2^{k+1}}}{L^3F}\right\}
	\\
	&\geq B'^{-1} \min\{(Q'_{-1})^{-1}L^{-6k-12}, L^{-3}F^{-1}\}(K\veps)^{2^{k+2}-1}
	\end{align*}
	and
	\[
	\ccd_k(X,w,\calF)\geq T_{-1}(h)^{-1}\geq \frac{(K\veps)^{2^{k+1}-1}}{B' Q_{-1} L^{6k+6} }.
	\]
	Furthermore, applying Algorithm~\ref{AL:correction-algorithm-simplified}
	with the parameter $q= (K\veps)^{2^{k+1}}$ to any
	$f\in C^k$ with 
	\[\dist(f,Z^k)<B'^{-1}  L^{-3}F^{-1}
	(Q'_{-1}L^{6k+12})^{-1}(K\veps)^{2^{k+2}-1} 
	\leq \wZContZLo{k}{k+1}{d} T'_{-1}(h')^{-1}\]
	results in $f'\in Z^k$ with
	\[\dist(f,f')<
	\frac{C}{q} \ContZwZUp{k}{k+1}{d}
	\dist(f,Z^k(X,\calF)) \leq (L^3F)^2  (K\veps)^{-2^{k+1}}\dist(f,Z^k(X,\calF)).\]
	From this one readily sees that there is $K'>0$, depending
	only on $F$, $L$, $B$, $k$, for which the   assertions
	of Theorem~\ref{TH:lgp-simple-version} about $(X,w,\calF)$ hold. Explicitly, any
	\begin{align*}
	K'\leq \min\left\{
	(Q'_{-1})^{-1}L^{-6k-12},
	L^{-3}F^{-1},
	(Q_{-1})^{-1}L^{-6k-6},
	(Q'_{-1})^{-1}L^{-6k-15}F^{-1}, 
	L^{-6}F^{-2}
	\right\},
	\end{align*}
	will work, e.g., $K'=\max\{Q_{-1},Q'_{-1}\}^{-1}L^{-6k-15}F^{-2}$.
	
	We now show the existence of the constant $K>0$.
	Note first that we can secure $K\veps\leq 1$ by choosing
	$K\leq 1$, because $\veps\leq 1$. 
	Next, observe that Lemma~\ref{LM:constants-bound}(iv) implies
	that  
	$
	\tilde{\veps}\geq \frac{1}{L^5F}\veps
	$	and likewise 
	$
	\tilde{\veps}'\geq \frac{1}{L^5F}\veps
	$.
	Consider the inequality
	\eqref{EQ:main-very-technical:p}. By Lemma~\ref{LM:substitutions-bound}
	(with $u=K\veps\leq 1$ and $A=1$)
	and Lemma~\ref{LM:UP-bound}, we have
	\begin{align*}
	p & =\tilde{\veps}-U_{\calP}S_{ {\alpha}, {\beta}}(h) 
	\geq \frac{\veps}{L^5F}- (k+1)F\cdot PL^{6k+7}F^4(1+B)K\veps
	\end{align*}
	so we can guarantee that $p>0$ by taking $K<P^{-1}(k+1)^{-1}L^{-6k-12}F^{-6}(1+B)^{-1}$.
	Next, consider \eqref{EQ:main-very-technical:p-prime}.
	By the same lemmas applied with $k+1$ and $\calP'$, we get
	\begin{align*}
	p' & =\tilde{\veps}'-U_{\calP'}S'_{ {\alpha}', {\beta}'}(h')-
	q  \sum_{i=0}^{k}
	\qc \Ducprime{i} \veps'_i	
	T'_i(h') 
	\\
	&\geq \frac{\veps}{L^5F}- (k+2)F\cdot PL^{6k+13}F^4(1+B)K\veps
	-K(K\veps)^{2^{k+1}-1}(\sum_{i=0}^{k} Q'_i)L^{6k+7}F\veps(K\veps)^{1-2^{k+1}}
	\\
	&
	\geq \frac{\veps}{L^5F}- \Squares{(k+2)F\cdot PL^{6k+13}F^4(1+B)
	 + L^{6k+7}F \sum_{i=0}^{k} Q'_i} K\veps ,
	\end{align*}
	and again, we can guarantee that $p'>0$
	by taking $K< \frac{1}{L^5F}[(k+2)  PL^{6k+13}F^5(1+B)
	+ L^{6k+7}F \sum_{i=0}^{k} Q'_i]^{-1}$.
	In particular, once $k$ is fixed,
	$K=\Omega(L^{-6k-18}F^{-6}(1+B)^{-1})$ works.
	This completes the proof.
\end{proof}

	\begin{proof}[Proof of Theorem~\ref{TH:lgp-zero-detailed} assuming Theorem~\ref{TH:main-very-technical}]
		Recall that we are given an $R$-oriented properly weighted $d$-poset $(X,w)$
		of lower irregularity at most $L$ and
		such that $\Fmax_{i,j}\leq F$ whenever $-1\leq i\leq j\leq 2$.
		We need to show that there are constants $E,E',E'',E''',D,D',D''>0$, depending only
		$L$ and $F$, 
		for which the conclusions   of Theorem~\ref{TH:lgp-zero-detailed}
		hold. 
		
		To that end, consider
		the $1$-intersection profile $\calP^{(1)}$ of Example~\ref{EX:intersection-profiles}(ii)
		and define
		\begin{align*}
		K&= 
		\max\left\{
		\frac{\Fmax_{0,1,d}\Fmax_{2,d}}{\Fmin_{0,2,d}\Fmin_{1,d}},
		\frac{\Fmax_{1,1,d}\Fmax_{2,d}}{\Fmin_{1,2,d}\Fmin_{1,d}}
		\right\} \leq 
		L^5\max\left\{\frac{\Fmax_{0,1}}{\Fmin_{0,2}},\frac{1}{\Fmin_{1,2}}\right\}
		\leq L^5F\\
		N&=  U_{\calP^{(1)}}\max\left\{ 
			L^6 \frac{\Fmax_{0,2}(\Fmax_{0,1})^2}{(\Fmin_{0,d})^2},
			L^{13}\Fmax_{1,2} \right\}
			\leq 2FL^6\max\{F^3,L^7F\}			
			\\
		V&= L^7 \frac{\Fmax_{0,2}}{\Fmin_{0,1}}\leq L^7F .
		\end{align*}
		The inequalities hold by Lemma~\ref{LM:constants-bound}(iv)
		and Lemma~\ref{LM:UP-bound}.
		Also let $C=\ContZwZUp{0}{0}{d}=L_{0,d}\leq L$ (this is the same as in Theorem~\ref{TH:main-very-technical}
		with $k=0$). 
		We will show that Theorem~\ref{TH:lgp-zero-detailed} holds for our $X$
		with any positive
		$E,E',E'',E''',D,D',D''$ satisfying 
		\begin{align*}
		D& \leq  L^{-12}\wZContZLo{0}{1}{d}
		& E & \leq  L_{0,d}^{-1} L_{0,1}^{-3}L^{-4} 
		\\
		D'& \leq \frac{1}{4VK}
			& E'&\leq \frac{1}{2NK} \\
		D''&\leq \wZContZLo{0}{ 1}{d} \frac{1}{4VCK}
			& E'' &\leq \min\left\{\frac{1}{L^{12}},\frac{1}{2VKC}\right\} \\
		& &  E''' &\leq L_{0,d}^{-3}
		\end{align*}
		Since $ \wZContZLo{0}{1}{d}\geq \frac{1}{L^3\Fmax_{0,1}}$
		(Lemma~\ref{LM:constants-bound}(v)), we can choose values
		for these constants
		which depend only on $L$ and $F$.

		Let  $\calF$ be an $R$-sheaf
		on $X$  such that conditions (1a)--(2c) of Theorem~\ref{TH:lgp-zero-detailed}  hold,
		$
		\alpha_{-1}<E\veps$, and there are
		and  $h_0,h_{-1},h_{||}\in (0,1]$ satisfying
		\begin{align}\label{EQ:assumption-for-p-prime}
		(\alpha_0+\beta_0 h_0) + (\alpha_{||}+\beta_{||}h_{||})+\frac{\alpha_{-1}+\beta_{-1}h_{-1}}{h_0}
		\leq  E' \veps'.
		\end{align}
		Let $k=0$, and let $\calP$ and $\calP'$ be the intersection profiles
		$\calP^{(0)} = \{(1,0,0,-1)\}
		$ and $\calP^{(1)} =\{(2,1,1,0),(2,1,1,-1),(1,0,0,-1)\}
		$  
		of Example~\ref{EX:intersection-profiles}.
		Now define $T_0$, $T_{-1}$, $T'_1$, $T'_0$, $T'_{-1}$, $S_{\alpha,\beta}$,
		$S'_{\alpha',\beta'}$
		as in \S\ref{subsec:second-tech}.
		We abbreviate the variables $x_{(2,1,1,0)}$,
		$x_{(1,0,0,-1)}$ and
		$x_{(2,1,1,-1)}$ used in the definition of $T_0$, $T_{-1}$, etc.\ to $x_{0}$, $x_{-1}$ and $x_{||}$,
		respectively. With this convention, we view 
		$\alpha':=(\alpha_0,\alpha_{-1},\alpha_{||})$,
		$\beta':=(\beta_0,\beta_{-1},\beta_{||})$
		and $h':=(h_0,h_{-1},h_{||})$ from Theorem~\ref{TH:lgp-zero-detailed}
		as vectors in 
		$ \R^{\calP'}$. 
		We also view the numbers $\alpha_{-1}$ and $\beta_{-1}$  as a vectors
		$\alpha,\beta \in \R^{\calP}$.	
		Choose
		some $\gamma\in [\frac{1}{2},1)$, 
		and set
		\[h:=\gamma\frac{E\veps-\alpha_{-1}}{\beta_{-1}}
		\qquad\text{and}
		\qquad q=\frac{\gamma h_0}{2VK}.\]
		We view $h$ as a vector in $\R^{\calP}$.
		Finally, set $\veps_{0}=\veps$ and $\veps'_{0}=\veps'_1=\veps'$,
		where $\veps$ and $\veps'$
		are those  given in Theorem~\ref{TH:lgp-zero-detailed}. 
		
		We will prove Theorem~\ref{TH:lgp-zero-detailed}
		by applying Theorem~\ref{TH:main-very-technical} to $(X,w,\calF)$
		and the parameters we chose. 
		Assumptions (1a)--(2b) of Theorem~\ref{TH:main-very-technical}
		are precisely assumptions (1a)--(2c) of Theorem~\ref{TH:lgp-zero-detailed},
		so it remains to verify the inequalities \eqref{EQ:main-very-technical:p}
		and~\eqref{EQ:main-very-technical:p-prime}.
		By Lemma~\ref{LM:constants-bound}(iv), 
		$\tilde{\veps}=\frac{\Fmin_{0,1,d}\Fmin_{0,d}}{\Fmax_{0,0,d}\Fmax_{1,d}}\veps
		\geq L^{-4} \Fmin_{0,1}\veps$ and by  Lemma~\ref{LM:UP-bound},
		$U_{\calP}\leq \Fmax_{0,1}$. Now, by Lemma~\ref{LM:ST-k-zero} 
		and our assumption $\alpha_{-1}< E\veps$, 
		\begin{align*}
		p & = \tilde{\veps}-U_{\calP}S_{\alpha,\beta}(h)
		\geq  L^{-4}\Fmin_{0,1}\veps
		- \Fmax_{0,1} L_{0,1}^2L_{0,d}(\alpha_{-1} +\beta_{-1} \gamma\frac{E\veps-\alpha_{-1}}{\beta_{-1}})
		\\
		&
		=  L^{-4}\Fmin_{0,1}\veps
		- E^{-1}L^{-4}\Fmin_{0,1}(\alpha_{-1} + \gamma(E\veps-\alpha_{-1}))
		\\
		&
		=  E^{-1}L^{-4}\Fmin_{0,1}(E\veps -\alpha_{-1}-\gamma(E\veps-\alpha_{-1}))
		\\
		&=  E^{-1}L^{-4}\Fmin_{0,1}(1-\gamma)(E\veps -\alpha_{-1})>0. 
		\end{align*}
		Next, in order to prove \eqref{EQ:main-very-technical:p-prime},
		note that  $K$ was chosen so that
		$\tilde{\veps}'\geq \frac{\veps'}{K}$.
		Now, by the definition of $N$, Lemma~\ref{LM:ST-k-one},
		Lemma~\ref{LM:constants-bound}(iv)  
		and \eqref{EQ:assumption-for-p-prime}, 
		we have
		\begin{align*}
		p'&= \tilde{\veps}'-U_{\calP'}S'_{ {\alpha}', {\beta}'}(h')-
			q  
			\qc 
			\frac{
				\Fmin_{0, 2,d}
				\Fmin_{ 1,d}
			}{
				\Fmax_{0, 1,d}
				\Fmax_{ 2,d}
			}  \veps'_0	
			T'_0(h')
		\\ 
		&\geq 
			\frac{\veps'}{K}
			-N\Circs{(\alpha_0+\beta_0 h_0) + (\alpha_{||}+\beta_{||}h_{||})+\frac{\alpha_{-1}+\beta_{-1}h_{-1}}{h_0}}
			-q  L_{0,d}\frac{\Fmax_{0,2}}{\Fmin_{0,1}} \frac{L^6}{h_0} \veps'
		\\
		&\geq
			\frac{\veps'}{K}-NE'\veps'-\frac{\gamma h_0}{2VK} L_{0,d}\frac{\Fmax_{0,2}}{\Fmin_{0,1}} \frac{L^6}{h_0} \veps'
		\\
		& \geq  \frac{\veps'}{K}- 
		\frac{\veps'}{2K}-
		\frac{\gamma}{2 K }  \veps'=
		(1-\gamma)\frac{\veps'}{2K}>0.
		\end{align*}
		This completes the verification of the assumptions
		of Theorem~\ref{TH:main-very-technical}.

		Therefore, the assertions of Theorem~\ref{TH:main-very-technical}
		hold for our $(X,w,\calF)$. Thanks to Lemmas~\ref{LM:ST-k-zero}
		and~\ref{LM:ST-k-one}, this means that
		\begin{align*}
		\cse_k(X,w,\calF)&\geq \min\left\{\frac{1}{T'_{-1}(h') },\frac{q}{C}\right\}
		\geq \min\left\{
			\frac{1}{ L^{12}(h_0^{-1}h_{-1}^{-1}+h_{||}^{-1})},\frac{\gamma h_0}{2VKC} 
			\right\}
		\\
		&\geq \min\left\{\frac{1}{L^{12}},\frac{\gamma}{2VKC}\right\}\frac{1}{ h_0^{-1}h_{-1}^{-1}+h_{||}^{-1}}
		\geq \frac{\gamma E''}{h_0^{-1}h_{-1}^{-1}+h_{||}^{-1}},
		\\
		\ccd_k(X,w,\calF) & \geq
			T_{-1}(h)^{-1}=\frac{h}{L_{0,d}^3}\geq E'''\gamma\frac{E\veps-\alpha_{-1}}{\beta_{-1}}.
		\end{align*}
		Moreover, 
		for every $f\in C^0$ with 
		$\dist(f,Z^0)< \frac{  D}{ (h_0^{-1}h_{-1}^{-1}+h_{||}^{-1})}
		\leq \wZContZLo{0}{ 1}{d} T'_{-1}(h')^{-1}$,
		applying Algorithm~\ref{AL:correction-algorithm-simplified} to 
		$f$ with the parameter being
		$D'h_0\leq \frac{\gamma h_0}{2VK }=q$
		results in $f'\in Z^k$ such
		that $\dist(f,f')\leq \frac{C}{q} 
		\ContZwZUp{0}{1}{d}		
		\dist(f,Z^k)\leq D''^{-1} h_0^{-1} \dist(f,Z^k)$.
		By Letting $\gamma$ approach $1$, we obtain the required bounds.
		
		We finish with explaining the values listed in Table~\ref{TB:lgp-zero-constants}.
		For the values in the first row of the table, we simply
		substitute $K=L^5F$, $N=2F^4L^{13}$, $V=L^7 F$ and $C=L$
		in the upper bounds for $D,D'D,'',E,E',E'',E'''$ 
		and replace $ \wZContZLo{0}{1}{d}$ and $L_{i,d}^{-1}$ with the smaller quantities $\frac{1}{L^3F}$.
		and $L^{-1}$, respectively.
		
		Suppose now that $X$ is lower-regular.
		Then $L=1$, $\Fmax_{0,1}\leq \Fmax_{0,2}=\Fmin_{0,2}$ and $\Fmax_{0,1}\leq \Fmax_{0,d}=\Fmin_{0,d}$.
		From this it follows that $K\leq 1$, $N\leq U_{\calP'}\max\{\Fmax_{0,2},\Fmax_{1,2}\}\leq 2F^2$.
		Substituting $K=1$, $N=2F^2$, $V=F$, $C=1$ and noting that 
		$ \wZContZLo{0}{1}{d}= \frac{1}{\Fmax_{0,1}}\geq \frac{1}{F}$ gives the second row of the table.

		Finally, assume that   $X({\leq}2)$ is an $m$-gon complex.
		Then 
		\begin{align*}
		K&\leq 
			L^5\max\left\{\frac{\Fmax_{0,1}}{\Fmin_{0,2}},\frac{1}{\Fmin_{1,2}}\right\}
			=L^5 \max\left\{\frac{2}{m},\frac{1}{m}\right\}=\frac{2L^5}{m}
		\\
		N&= U_{\calP^{(1)}}\max\left\{ 
			L^6 \frac{\Fmax_{0,2}(\Fmax_{0,1})^2}{(\Fmin_{0,d})^2},
			L^{13}\Fmax_{1,2}\right\}
		\leq 
			m \max\left\{ 
			L^6 \frac{m2^2}{m^2},
			L^{13}m \right\}=L^{13} m^2			
			\\
		V&= L^7 \frac{\Fmax_{0,2}}{\Fmin_{0,1}}=
		\frac{L^7 m}{2}.
		\end{align*}
		We get the last row of the table by
		substituting the right hand sides 
		in the upper bounds
		for  $D,D'D,'',E,E',E'',E'''$ as well as replacing $C$ with $L$
		and 
		$ \wZContZLo{0}{1}{d}$ with $\frac{1}{L^3 \Fmax_{0,1}}=\frac{1}{2L^3}$.		
		The third (resp.\ fourth) row then follow by taking
		$m=3$ (resp.\ $m=4$) and $L=1$.
	\end{proof}

\section{Proof of Theorem~\ref{TH:main-very-technical}}
\label{sec:proof-of-tech}

Throughout this section, $R$ is a commutative ring,
$(X,w)$ is a properly weighted $R$-oriented $d$-poset
and $\calF$ is an $R$-sheaf on $X$.
When there is no risk of confusion, we shall
write $C^k=C^k(X,\calF)$, $Z^k=Z^k(X,\calF)$ and $B^k=B^k(X,\calF)$.
We further write $\|\cdot\|$ for $\|\cdot\|_w$ and $\dist(\cdot,\cdot)$ for $\dist_w(\cdot,\cdot)$.

\subsection{Mock Locally Minimal Cochains}

Recall (\S\ref{subsec:sheaf-at-link}) that given $u\in X$ of dimension $i$ and
$g\in C^{k-i-1}(X_u,\calF_u)$, we write $g^u\in C^k(X,\calF)$
for the $k$-cochain defined by $g^u(x)=g(x)$ if $x\geq u$ and $g^u(x)=0$ otherwise.
Also, given $f\in C^k(X,\calF)$, we let $f_u\in C^{k-i-1}(X_u,\calF_u)$
be defined by $f_u(x)=f(x)$ for every $x\in X(k)_u$.

\begin{dfn}[Mock Locally $q$-Minimal Cochain]
	\label{DF:mock-loc-min}
	Let $q\in[0,\infty)$  
	and  
	let $k\in \{0,\dots,d\}$.
	Given $u\in X$ with $i:=\dim u\leq k$,
	a $k$-cochain $f\in C^k(X,\calF)$
	is called \emph{mock $q$-locally minimal  at   $u$}
	if  for all $b\in B^{k-i-1}(X_u,\calF_u)$, we have
	\[\|f\|\leq \|f+b^u\|+
	\qc
	q \cdot w(u).\]
	We say that $f$ is \emph{mock $q$-locally minimal} if it is mock $q$-locally
	minimal at every face $u$ with $0\leq\dim u\leq  k$.
	A  \emph{mock locally minimal} cochain is a mock $0$-locally minimal cochain.
\end{dfn}

\begin{remark}
	(i) Following \cite{Evra_2024_cosystolic_expanders_bdd_deg} and other sources,
	it is natural to call a  $k$-cochain $f\in C^k(X,\calF)$   $q$-\emph{locally minimal}
	at $u\in X(i)$ if
	\[
	\|f_u\|\leq \|f_u+b\|+q
	\]
	for every $b\in B^{k-i-1}(X_u,\calF_u)$. When 
	$X$ is lower-regular, 
	this is equivalent to $f$ being mock $q'$-locally minimal
	at $u$ for $q'=\frac{F_{i,k,d}F_{i,d}}{F_{k,d}} q$ (use Lemma~\ref{LM:weight-link-vs-all}
	and Corollary~\ref{CR:weight-of-face-bound}).
	In general, however, there is no relation between being $q$-locally minimal
	and being mock $q$-locally minimal; this is why we use the word ``mock'' in 
	Definition~\ref{DF:mock-loc-min}. It will be important to use
	\emph{mock} locally $q$-minimal cochains, rather than $q$-minimal cochains,
	in the proof of Proposition~\ref{PR:alg-complexity} below.
	
	(ii) Every $f\in C^k(X,\calF)$ is   mock ($0$-)locally minimal at every $u\in X(k)$,
	because   $B^{-1}(X_u,\calF_u)=0$. Consequently, every $0$-cochain is mock locally minimal.

	(iii) One can introduce a coefficient depending on $\dim u$ 
	before the factor $ q w(u)  $ in Definition~\ref{DF:mock-loc-min}.
	This has no effect beyond 
	scaling the   constant coefficients in Theorem~\ref{TH:main-very-technical} and its
	simplifications. 
	We did not attempt to look for coefficients
	that may  give a better result. 
\end{remark}

\begin{alg}[Algorithm for Making a Cochain Mock $q$-Locally Minimal]
	\label{AL:correction-algorithm}
	Let $k\in\{-1,\dots,d-1\}$. 
	The algorithm
	takes as input $h\in C^{k+1}(X,\calF)$ and some $q\in [0,\infty)$ 
	and outputs $g\in C^k(X)$ such that
	$h+dg$ is mock $q$-locally minimal.
	The procedure is as follows:
	\begin{enumerate}[label=(\arabic*)]
		\item Set $g_0=0\in C^k(X,\calF)$ and $i=0$.
		\item \label{item:AL:correction-algorithm:loop} While $h_i := h+dg_i$ is not mock $q$-locally minimal:
		\begin{enumerate}[label=(\alph*)]
			\item Choose some $u\in \bigcup_{j=0}^{k } X(j) $ 
			such that $h_i$ is not mock $q$-locally minimal at $u$.
			\item Find some $g'\in C^{k-1}(X_u,\calF_u)$ such that $\|h_i+d g'^u \|<\|h_i\|-q w(u)$.
			(It exists because $h_i$ is not mock $q$-locally minimal at $u$.)
			\item Set $g_{i+1}=g_i+g'^u$ and increase $i$ by $1$.
		\end{enumerate}
		\item Return $g_i$.
	\end{enumerate}
\end{alg}

\begin{remark}\label{RM:relation-between-alg}
	Algorithm~\ref{AL:correction-algorithm-simplified} may be seen
	as special case of 
	Algorithm~\ref{AL:correction-algorithm},
	where we take   $h=df$ some $f\in C^k$ and return $f':=f+g_i$ instead of $g_i$.
	
	Under mild assumptions on $X$ and $\calF$, the search for the face $u$ in (a) can be done
	efficiently on average, so that the time complexity of Algorithm~\ref{AL:correction-algorithm} is linear
	in $|X(k)|$; see Appendix~\ref{sec:correction-efficient}.
\end{remark}

\begin{prp}\label{PR:alg-complexity}
	Let $k\in\{-1,\dots,d-1\}$, $h\in C^{k+1}(X,\calF)$
	and $q\in [0,\infty)$. 
	Suppose that Algorithm~\ref{AL:correction-algorithm} is applied to $h$
	and $q$
	and let $g$ be its output (assuming it stops).
	Then:
	\begin{enumerate}[label=(\roman*)]
		\item The algorithm stops. 
		\item If $q>0$ 
		and there is $M\in\R_+$ such that 
		$w(x)\leq Mw(y)$ for all $x,y\in X(1)\cup\dots\cup X(k)$,
		then the 	  loop \ref{item:AL:correction-algorithm:loop}
		is executed at most $M q^{-1} |X(k)|\cdot \|h\|$ times.
		\item $h+dg$ is mock $q$-locally minimal and $\|h+dg\|\leq \|h\|$.
		\item $\|g\|\leq \CorR  q^{-1}\|h\|$.
	\end{enumerate}
\end{prp}

\begin{proof}
	Let $u_i$ and $g'_i$ denote the   $u$ and $g'$ chosen at the $i$-th iteration
	of the loop \ref{item:AL:correction-algorithm:loop}
	and let $r_i =\dim u_i$.
	
	(i) By the construction of the $h_i$,
	we have   $\|h_{i+1}\| < \|h_i\|-q w(u_i)$.
	In particular, $\|h_i\|>\|h_{i+1}\|$ for all $i$. 
	Since $X$ is a finite, 
	$\|\cdot\|:C^{k+1}(X,\calF)\to \R$ attains only finitely many values and so the algorithm
	must stop.
	
	(ii) 
	Suppose that $q>0$ and we are given $M\in\R_+$   as in (ii).	
	There is some $x\in X(k)$ with $w(x)\leq\frac{1}{|X(k)|}$.
	Thus, 
	\[\|h_{i+1}\| < \|h_i\|-q w(u_i) \leq \|h_i\|-qM^{-1} w(x)\leq \|h_i\|-qM^{-1} |X(k)|^{-1} .\]
	By iterating this,
	we see that $\|h_i\|< \|h\|-    \frac{qi}{|X(k)|M} $. 
	Since $\|h_i\|\geq 0$,
	this means that the loop \ref{item:AL:correction-algorithm:loop} is executed at most $Mq^{-1}|X(k)|\cdot\|h\|$ times.

	(iii) The first claim is immediate from the stopping condition of the loop \ref{item:AL:correction-algorithm:loop}. The second claim follows from
	our earlier observation that $\|h+dg_i\|=\|h_i\|> \|h_{i+1}\|=\|h+dg_{i+1}\|$ whenever both
	sides are defined.
	
	(iv) Let $n$ be the   value of $i$ when  the algorithm stops.
	Recall that
	$
	\|h_{i }\|-\|h_{i+1}\|\geq
	  q w(u_i)$.
	Using this, the definition of $g'_i$ and Lemma~\ref{LM:weight-of-j-faces-cont-z}, 
	we see that
	\begin{align*}
	\|g'^{u_i}_i\|
	&\leq
	w(X(k)_{u_i})
	\leq 
	\ContZwZUp{r_i}{k}{d} w(u_i)
	\leq \ContZwZUp{r_i}{k}{d} q^{-1}(\|h_{i }\|-\|h_{i+1}\|).
	\end{align*}
	This means that
	\begin{align*}
	\|g\| &\leq
	\sum_{i=1}^n\|g'^{u_i}_i\|
	\leq
	\sum_{i=1}^n \ContZwZUp{r_i}{k}{d}   q^{-1}(\|h_{i+1}\|-\|h_{i}\|)
	\\
	&\leq
	\max  \left\{
			\ContZwZUp{ i}{k}{d}  
			\where
			i\in\{0,\dots,k\}
	\right\}q^{-1}(\|h_0\|-\|h_n\|)
	\end{align*}
	and   (iv) follows because $h_0=h$.
\end{proof}

\begin{cor}\label{CR:cochain-cohomolog-to-mock-min}
	Let $k\in \{0,\dots,d-1\}$ and $h\in C^k(X,\calF)$. Then
	there exists $g\in C^{k-1}(X,\calF)$ such that $h+dg$ is mock locally minimal
	and $\|h+dg\|\leq \|h\|$.
\end{cor}

\begin{proof}
	Apply Algorithm~\ref{AL:correction-algorithm} (with $k-1$ in place of $k$)
	to $h$ with $q=0$. The algorithm
	stops by \ref{PR:alg-complexity}(i) and its output is the required $g$.
\end{proof}

\subsection{Reduction to Expansion of Small Locally Minimal Cochains}
\label{subsec:reduction}

Let $\ \gamma\in [0,1]$. We call a cochain $f\in C^k(X,\calF)$
\emph{$\gamma$-small} (w.r.t.\ $w:X\to \R_+$) 
if $\|f\|<\gamma$.
Given a subset
$S\subseteq C^k(X,\calF)$ and $\beta\in [0,\infty)$, we say that $\calF$ \emph{$\beta$-expands} $S$ 
if $\|df\|\geq \beta\|f\|$ for every $f\in S$.
It turns out that if $\calF$ expands small mock locally
minimal $k$-cochains and $(k+1)$-cochains, then $\calF$ has good cosystolic
expansion in dimension $k$:

%%\gap{In the following proposition is is enough to assume that a small loc.min.\ cochain is 0.}

\begin{prp}\label{PR:small-expansion-to-CSE}
	Let $X$, $w$, $\calF$, $d$ be as in the beginning of this section, let 
	$k\in \{0,\dots,d-2\}$
	and $C=\CorR$.
	Let $\beta, \beta'\in [0,\infty)$,
	$\gamma,\gamma',q\in [0,1]$,
	and suppose that
	\begin{enumerate}
		\item[(1)] $\calF$ $\beta$-expands $\gamma$-small mock locally minimal $k$-cochains, and
		\item[(2)] $\calF$ $\beta'$-expands $\gamma'$-small $q$-mock locally minimal $(k+1)$-cochains. 
	\end{enumerate}
	Then
	\[
	\ccd_k(X,\calF)\geq \gamma\qquad\text{and}\qquad
	\cse_k(X,\calF)\geq \min\{\gamma', \frac{q}{C} \}
	\]	
	The   assertion about $\ccd_k(X,\calF)$ holds even without assuming (2). 
	Moreover, if $f\in C^k=C^k(X,\calF)$
	satisfies $\dist(f,Z^k )<\wZContZLo{k}{k+1}{d} \gamma'$,
	then   Algorithm~\ref{AL:correction-algorithm-simplified},
	applied to
	to $ f$ and $q$   returns
	$f'\in Z^k  $ such that  
	$\dist(f,f') \leq \frac{C}{q} \ContZwZUp{k}{k+1}{d}
	\dist(f,Z^k(X,\calF))$.
\end{prp}

\begin{proof}
	Let $f\in Z^k   $ and  suppose that $\|f\|<\gamma$. By Corollary~\ref{CR:cochain-cohomolog-to-mock-min}, 
	there is $g\in B^{k-1}$ such that $f+dg$ is mock locally minimal and
	$\|f+dg\|\leq \|f\|<\gamma$. By assumption (1), $0=\|d(f+dg)\|\geq \beta\|f+dg\|$.
	Thus $f+dg=0$ and   $f=-dg \in B^k$.
	
	Next, let $f\in C^k$. We need to show that 
	$\|df\|\geq \min\{\gamma', \frac{q}{C} \}\dist(f,Z^k)$.
	If $\|df\|\geq \gamma'$, then this holds automatically,
	so assume $\|df\|< \gamma'$.
	By applying Proposition~\ref{PR:alg-complexity} with $h=df$
	and the parameter $q$, we see that there is $g\in C^k$ such that
	$df+dg$ is mock $q$-locally minimal,
	$\|df+dg\|\leq \|df\|<\gamma'$ and
	$\|g\|\leq C q^{-1}\|df\| $.
	By assumption (2), $0=\|d(df+dg)\|\geq \beta'\|df+dg\|$,
	so $df+dg=0$. This means that $f+g\in Z^k$.
	As a result, $\dist(f,Z^k)\leq \|g\|\leq C q^{-1}\|df\|$.
	By rearranging, we get $\|df\|\geq \frac{q}{C} \dist(f,Z^k)\geq \min\{\frac{q}{C},\gamma'\}\dist(f,Z^k)$.
	
	To finish, suppose that $f\in C^k$
	satisfies $\dist(f,Z^k)< 
	\wZContZLo{k}{k+1}{d}  \gamma'$. Choose some
	$g'\in Z^k$ which minimizes $\|f-g'\|$.
	By Lemma~\ref{LM:weight-of-j-faces-cont-z},
	\begin{align*}
		\|df\| &=\|d(f-g')\|\leq \sum_{x\in \supp(f-g')} w(X(k+1)_x)
		\leq \\
		&\leq 
		\ContZwZUp{k}{k+1}{d}		
		w(\supp(f-g'))
		=\ContZwZUp{k}{k+1}{d} \dist(f,Z^k)<\gamma'.
	\end{align*}	
	Recall  from Remark~\ref{RM:relation-between-alg}  
	that applying Algorithm~\ref{AL:correction-algorithm-simplified} to
	$f$
	is essentially the same as applying 	Algorithm~\ref{AL:correction-algorithm} with
	$h=df$ and setting $f'=f+g$ at the end, where $g$ is the output
	of Algorithm~\ref{AL:correction-algorithm}.
	Thus, as in the last paragraph, it follows that
	$f'=f+g\in Z^k$ and $\dist(f,f')=\|g\|\leq C  q^{-1}\|df\|\leq 
	\frac{C}{q} \ContZwZUp{k}{k+1}{d}
	\dist(f,Z^k)$.
\end{proof}

Using Proposition~\ref{PR:small-expansion-to-CSE}, we can reduce
Theorem~\ref{TH:main-very-technical} into proving the following theorem.
Recall that given a $k$-intersection profile $\calP$
for $X$ and lists $\alpha=\{\alpha_\rho\}_{\rho\in\calP}$,
$\beta=\{\beta_{\rho}\}_{\rho\in\calP}$ of non-negative real numbers,
we defined in \S\ref{subsec:polynomials} Laurent polynomials
\[
T_{k},T_{k-1},\dots,T_{-1},S_{ {\alpha}, {\beta}}\in \R[x_\rho^{\pm1}\where \rho\in\calP ]
\]
and a natural number $U_{\calP}\in\N$.

\begin{thm}\label{TH:expansion-of-loc-min-cochains}
	Let $R$ be a commutative ring,
	let $(X,w)$ be a properly weighted $R$-oriented $d$-poset,
	let $\calF$ be an $R$-sheaf
	on $X$, let $k\in\{0,\dots,d-1\}$, and let $\calP$
	be a $k$-intersection profile for $X$.
	Let $\{\veps_i\}_{i=0}^k$, $\alpha= \{\alpha_\rho\}_{\rho\in\calP}$,
	$\beta=\{\beta_\rho\}_{\rho\in\calP}$ 
	be lists of non-negative real numbers.
	Put
	\[
	\tilde{\veps}=\min\left\{
	\Duc{i} \veps_i
	\where i\in\{0,\dots,k\}
	\right\}
	\]	
	and suppose that exist   $h=\{h_\rho\}_{\rho\in\calP}
	\in (0,1]^{\calP}$  and $q\in [0,1]$ such that
	\[
	p:=\tilde{\veps}-U_{\calP}S_{ {\alpha}, {\beta}}(h)-
	q  \sum_{i=0}^{k-1} 
	\qc \Duc{i} \veps_i	
	T_i(h) > 0.
	\]
	Suppose further that the following conditions are met:
	\begin{enumerate}[label=(\arabic*)]
		\item $(X_u,w_u,\calF_u)$ is an $\veps_{\dim u}$-coboundary
		expander in dimension $k-\dim u-1$ for every $u\in X(0)\cup\dots \cup X(k)$;
		\item $\NI^{\ell,r,t}_u(X)$ is an $(\alpha_\rho,\beta_\rho)$-skeleton
		expander for every $\rho=(t,\ell,r,b)\in \calP$ and $u\in X(b)$. 
	\end{enumerate}
	Then $\calF$ $p$-expands $T_{-1}(h)^{-1}$-small 
	mock $q$-locally minimal $k$-cochains.
	That is, for every mock $q$-locally minimal $f\in C^k(X,\calF)$
	with $\|f\|<T_{-1}(h)^{-1}$, we have $\|df\|\geq p\|f\|$. 
\end{thm}

\begin{remark}\label{RM:partial-orientation-exp-loc-min}
	In
	Theorem~\ref{TH:expansion-of-loc-min-cochains}, there is no need
	to assume that all of $X$ is $R$-oriented.
	Rather,  only the subposet $X(k-1)\cup X(k)\cup X(k+1)$
	should be $R$-oriented. When $\calF(x)=0$ for all $x\in X(k-1)$, it even
	enough to assume that $X(k)\cup X(k+1)$ is $R$-oriented (such an $R$-orientation always exists).
	Indeed, we only need the $3$-term 
	cochain complex $C^{k-1}\to C^k\to C^{k+1}$ to be well-defined for the proof;
	cf.\ Remark~\ref{RM:partial-orientation}.
\end{remark}

\begin{proof}[Proof of Theorem~\ref{TH:main-very-technical} assuming Theorem~\ref{TH:expansion-of-loc-min-cochains}]
	We use the notation
	of Theorem~\ref{TH:main-very-technical}.
	By Proposition~\ref{PR:small-expansion-to-CSE}, it is enough
	to show that $\calF$ (a)  $p$-expands $T_{-1}(h)$-small mock locally minimal $k$-cochains
	and (b)  $p'$-expands $T'_{-1}(h')$-small mock $q$-locally minimal $(k+1)$-cochains.
	To that end, we apply Theorem~\ref{TH:expansion-of-loc-min-cochains} twice: once
	for $X$, $w$, $\calF$, $k$, $\calP$, $\alpha$, $\beta$, $h$ with $q$ being $0$, 
	and again for $X$, $w$, $\calF$, $k+1$, $\calP'$, $\alpha'$, $\beta'$, $h'$
	with $q$ from Theorem~\ref{TH:main-very-technical}. This gives (a) and (b), respectively.
\end{proof}

We will prove Theorem~\ref{TH:expansion-of-loc-min-cochains} in 
the next section.
Before turning to that,
we note that Theorem~\ref{TH:expansion-of-loc-min-cochains} also
gives  a criterion for bounding the cocycle distance from below:

\begin{cor}\label{CR:ccd-lower-bound}
	Keep the notation of Theorem~\ref{TH:expansion-of-loc-min-cochains}.
	Suppose that assumptions (1) and (2) of that theorem hold and
	there is $h=\{h_\rho\}_{\rho\in\calP}$ in $(0,1]^{\calP}$
	such that
	$
	\tilde{\veps} > \calU_P S_{\alpha,\beta}(h)
	$.
	Then 
	$\ccd_k(X,w,\calF)\geq T_{-1}(h)^{-1}$.
\end{cor}

\begin{proof}
	By Theorem~\ref{TH:expansion-of-loc-min-cochains},
	applied with $q=0$,
	the sheaf $\calF$ expands  $T_{-1}(h)^{-1}$-small 
	mock locally minimal
	$k$-cochains.
	The lower bound on $\ccd_k(X,\calF)$
	now follows from   Proposition~\ref{PR:small-expansion-to-CSE}.
\end{proof}

We use Corollary~\ref{CR:ccd-lower-bound} to prove Theorem~\ref{TH:ccd-lower-bound} from earlier.

\begin{proof}[Proof of Theorem~\ref{TH:ccd-lower-bound}]
	In short, this is just unfolding Corollary~\ref{CR:ccd-lower-bound}
	in the case where $k=0$ and $\calP$ is $\calP^{(0)}$ from Example~\ref{EX:intersection-profiles}(i).
	
	Recall that we are given an $R$-sheaf $\calF$ on a properly weighted $d$-poset $(X,w)$.
	It is further given that assumptions (1) and (2) of Theorem~\ref{TH:expansion-of-loc-min-cochains}
	hold for $k=0$, $\veps_1= \veps$, $\alpha_{(1,0,0,-1)}=\alpha$,
	$\beta_{(1,0,0,-1)}=\beta$. We   choose the constants $E$ and $E'''$ as in the
	proof of Theorem~\ref{TH:lgp-zero-detailed} (given in \S\ref{subsec:second-tech}), i.e.,
	\[
	E= L_{0,d}^{-1} L_{1,d}^{-3} L^{-4},
	\qquad
	E''' =  L_{0,d}^{-3}.
	\]
	Since the theorem holds automatically when $\alpha\geq E\veps$,
	we may assume that $\alpha<E\veps$. 
	Let $\gamma\in[\frac{1}{2},1)$ and
	$h=\gamma\frac{E\veps-\alpha}{\beta}$, and let $p$ be as in Theorem~\ref{TH:expansion-of-loc-min-cochains}
	(with $k=0$ and $\calP=\calP^{(0)}$).
	Then, as in the proof of Theorem~\ref{TH:lgp-zero-detailed}, we have $p>0$.
	By Lemma~\ref{LM:ST-k-zero}, $T_{-1}(h)=\frac{1}{E''' h}$.
	We may therefore apply Corollary~\ref{CR:ccd-lower-bound}.
	and assert that $\ccd_k(X,\calF)\geq E''' h=E'''\gamma \frac{E\veps-\alpha}{\beta}$.
	As this holds for all $\gamma\in [\frac{1}{2},1)$, we are done.
\end{proof}

\section{Proof of Theorem~\ref{TH:expansion-of-loc-min-cochains}}
\label{sec:proof-of-exp-small-coch}

Throughout, $R$ is a commutative ring, $(X,w)$ is a properly weighted $R$-oriented $d$-poset 
and $\calF$ is an $R$-sheaf on $X$. 
(Actually, milder assumptions on the orientation suffice --- see
Remark~\ref{RM:partial-orientation-exp-loc-min}.)
If not indicated otherwise, $k\in\{0,\dots,d-1\}$
and $\calP$ is a $k$-intersection profile for $X$.
We   write $\|\cdot\|$ for $\|\cdot\|_w$ and $\dist(\cdot,\cdot)$ for $\dist_w(\cdot,\cdot)$.
Recall that given $A\subseteq X$ and $z\in X$, we let $A_z$ stand for $\{x\in A\suchthat z\geq x\}$.

\subsection{Heavy Faces}

For this subsection,
we fix $k\in\{0,\dots,d-1\}$, a $k$-cochain $f\in C^k(X,\calF)$,
a $k$-intersection profile $\calP$ for $X$,
and choose   numbers $h\in (h_{\rho})_{\rho\in\calP }$
in the interval $(0,1]$.    

Following \cite{Evra_2016_cosystolic_expanders_arxiv_version}, for every $i\in\{-1,\dots,k\}$, we define
a set of   $i$-faces $A_i=A_i(f,h,\calP)$ by decreasing induction
on $i$ as follows:
\begin{itemize}
	\item $A_k=\supp(f)$. 
	\item Assuming $A_{i+1},\dots,A_k$ were defined, define
	$A_i$ to be the set of face $u\in X(i)$ such that for \emph{some}
	$\rho=(t,\ell,r,b)\in \calP$ with $b=i$, we have
	\begin{equation}\label{EQ:heavy-face-cond}
	w_z(A_{\ell,z })+w_z(A_{r,z })\geq 2h_\rho.
	\end{equation}
	(When $r=\ell$, this simplifies to $w_z(A_{\ell,z})\geq h_\rho$.)
\end{itemize}
Elements of $A_{-1}\cup\dots\cup A_k$ will be called $(f,h,\calP)$-\emph{heavy},
or just \emph{heavy} for short.
Informally, a  face is heavy if it is in $\supp(f)$, or it is contained in 
relatively
many heavy faces of larger dimension.

Recall from \S\ref{subsec:polynomials} that we defined Laurent polynomials
$T_k, \dots,T_{-1}\in \R[x_\rho^{\pm 1} \where \rho\in \calP]$
inductively by setting $T_k=1$ and
\[
T_i =\sum_{\rho\in \calP:b(\rho)=i}x_\rho^{-1} \Squares{
c_\rho T_{\ell(\rho)} + c'_\rho T_{r(\rho)}}
\]
if $i\in\{k-1,\dots,0,-1\}$,
where
\begin{align*}
c_{(t,\ell,r,i)} &= 
	\frac{\frac{\Fmax_{\ell,d}\Fmax_{i,\ell}}{\Fmin_{i,\ell,d}}
	}{
		\LinkAllLo{i}{\ell}{d}  
		\ContZwZLo{i}{\ell}{d}				
		+
		\LinkAllLo{i}{r}{d}  
		\ContZwZLo{i}{r}{d}
	}
\qquad\text{and}\qquad
c'_{(t,\ell,r,i)}  = c_{(t,r,\ell,i)}.
\end{align*}
We now show  that  
the weight of the $(f,h,\calP)$-heavy $i$-faces, $w(A_i)$, 
is at most proportional to $\|f\|$.

\begin{lem}\label{LM:bound-weight-heavy-i-faces}
	Let $h$ and $f$ be as above
	and suppose $-1\leq i\leq k$. Then
	\[ w(A_i)\leq T_i(h)\|f\|.\]
\end{lem}

\begin{proof}
	Let $z$ be a heavy $i$-face and let $\rho=(t,\ell,r,i)$ be
	a member of $\calP$
	for which \eqref{EQ:heavy-face-cond} holds.
	By Lemmas~\ref{LM:weight-link-vs-all}
	and~\ref{LM:weight-of-j-faces-cont-z},
	we have
	\begin{align*}
		2h_\rho &=
		h_\rho(w_z(X(\ell)_z)+w_z(X(r)_z))
		\geq 
		h_\rho w(X(d)_z)^{-1} \Squares{
				\LinkAllLo{i}{\ell}{d}
				w(X(\ell)_z) +
				\LinkAllLo{i}{r}{d}
				w(X(r)_z)
		} \\
		&\geq 
		h_\rho w(X(d)_z)^{-1} \Squares{
				\LinkAllLo{i}{\ell}{d} \cdot
				\ContZwZLo{i}{\ell}{d}				
				+
				\LinkAllLo{i}{r}{d} \cdot
				\ContZwZLo{i}{r}{d}	
		}w(z),
	\end{align*}
	whereas on the other hand,
	\begin{align*}
		2h_\rho &\leq 
		w_z(A_{\ell,z })+w_z(A_{r,z })
		\leq 
		w(X(d)_z)^{-1} 
		\Squares{
				\LinkAllUp{i}{\ell}{d}
				w(A_{\ell,z }) +
				\LinkAllUp{i}{r}{d}
				w(A_{r,z })
		}.
	\end{align*}
	Together, we get
	\begin{align*}
		w(z) &\leq h_\rho^{-1} \Squares{
			\frac{
				\LinkAllUp{i}{\ell}{d}
				w(A_{\ell,z })
			}{
				\LinkAllLo{i}{\ell}{d}  
				\ContZwZLo{i}{\ell}{d}				
				+
				\LinkAllLo{i}{r}{d}  
				\ContZwZLo{i}{r}{d}}
			+\frac{
				\LinkAllUp{i}{r}{d}
				w(A_{r,z })
			}{
				\LinkAllLo{i}{\ell}{d}  
				\ContZwZLo{i}{\ell}{d}				
				+
				\LinkAllLo{i}{r}{d}  
				\ContZwZLo{i}{r}{d}}
			}
		\\
		&=  h_\rho^{-1}\Squares{
		\frac{c_\rho}{\Fmax_{i,\ell}} w(A_{\ell,z}) +
		\frac{c'_\rho}{\Fmax_{i,r}} w(A_{r,z})}
		.
	\end{align*}
	
	We now prove the lemma by decreasing induction on $i$. The case $i=k$
	is clear because $w(A_k)=w(\supp f)=T_k(h)\|f\|$.
	Suppose now that $i<k$ and the lemma was established for larger values of $i$.
	Then by what we have shown, Lemma~\ref{LM:weight-sum-over-i-faces} and the induction hypothesis,
	\begin{align*}
		w(A_i) = \sum_{z\in A_i} w(z)
		& \leq \sum_{z\in A_i} \sum_{\rho\in\calP:b(\rho)=i} h_\rho^{-1}\Squares{
			\frac{c_\rho}{\Fmax_{i,\ell(\rho)}} w(A_{\ell(\rho),z}) +
			\frac{c'_\rho}{\Fmax_{i,r(\rho)}} w(A_{r(\rho),z})} \\
		& \leq \sum_{\rho\in\calP:b(\rho)=i} \sum_{z\in X(i)} h_\rho^{-1}\Squares{
			\frac{c_\rho}{\Fmax_{i,\ell(\rho)}} w(A_{\ell(\rho),z}) +
			\frac{c'_\rho}{\Fmax_{i,r(\rho)}} w(A_{r(\rho),z})}
		\\
		&\leq \sum_{\rho\in\calP:b(\rho)=i} h_\rho^{-1} (c_\rho w(A_\ell) + c'_\rho w(A_r))
		\\
		&
		\leq \sum_{\rho\in\calP:b(\rho)=i}(c_\rho T_{\ell(\rho)}(h)+c'_\rho T_{r(\rho)}(h))=T_i(h).\qedhere
	\end{align*}
\end{proof}

\subsection{Exceptional Faces}
\label{subsec:exceptional-faces}

We continue to assume that 
$k\in\{0,\dots,d-1\}$,  $f\in C^k(X,\calF)$,
$\calP$ is 
a $k$-intersection profile for $X$,
and   $h= (h_{\rho})_{\rho\in\calP}\in (0,1]^{\calP}$.

Let $\rho=(t,\ell,r,b) \in \calP$. By a \emph{$\rho$-square} (in $X$), we mean
a quadruple  $(x,y,z,u)$ of faces in $X$ of respective dimensions $(t,\ell,r,b)$ such that
$y,z\leq x$ and $u\in \Inf\{y,z\}$. A $\rho$-square $(x,y,z,u)$
is called $(f,h,\calP)$-\emph{exceptional}, or just exceptional for short, 
if  $y$ and $z$ are heavy, but $u$ is not. 
A $(k+1)$-face $s\in X(k+1)$ is called $(f,h,\calP)$-\emph{exceptional} if there is an exceptional
square $(x,y,z,u)$ with $x\leq s$. The set
of exceptional $(k+1)$-faces will be denoted  by
\[
\Upsilon=\Upsilon(f,h,\calP).
\]
Non-exceptional faces $s\in X(k+1)-\Upsilon$ have the property
that if $(x,y,z,u)$ is a $\rho$-square ($\rho\in\calP$) with $x\leq s$, 
and if  $y$ and $z$
are heavy, then   $u$ is also heavy. 

We will show that the weight of exceptional $(k+1)$-faces is at most proportional
to $\|f\|$, provided that the no-intersection hypergraphs of the links of $X$
are good skeleton expanders (Section~\ref{sec:non-intersect}).
Recall from \S\ref{subsec:polynomials} that given  $\alpha,\beta\in [0,\infty)^{\calP}$, we defined
$S_{\alpha,\beta}\in \R[x_\rho^{\pm1}\where \rho\in\calP]$ by
\[
S_{\alpha,\beta}
=\sum_{\rho=(t,\ell,r,i)\in\calP}
\!\!
\Sci
\Squares{
\Scii{\ell}
T_\ell+
\Scii{r}
T_r
}
(\alpha_\rho +\beta_\rho x_\rho).
\]

\begin{lem}\label{LM:bound-weight-of-exceptional-faces}
	With notation as before, let $\alpha,\beta \in [0,\infty)^{\calP}$. 
	Suppose that for every
	$\rho=(t,\ell,r,b)\in\calP$ and
	$u\in X(b)$, the  no-intersection hypergraph $\NI^{\ell,r,t}_u(X,w)$
	(see \S\ref{subsec:no-intersect} and Notation~\ref{NT:non-intersec-of-link})
	is an $(\alpha_\rho,\beta_\rho)$-skeleton expander.
	Then
	\[
	w(\Upsilon)\leq S_{\alpha,\beta}(h)\|f\|.
	\]
\end{lem}

\begin{proof}
	Fix some $\rho=(t,\ell,r,i)\in\calP$ and $u\in X(i)-A_i$.
	We write $H=\NI^{\ell,r,t}_u(X)$ and denote its weight function
	by $w_H$.
	We claim that 
	\[
	w_u(E_2(A_{\ell,u}\cup A_{r,u})) 
	<
\alpha_\rho  \frac{w_u(A_{\ell,u})+w_u(A_{r,u})}{2} +  \beta_\rho \Circs{\frac{w_u(A_{\ell,u})+w_u(A_{r,u})}{2}}^2 .
	\]
	Indeed, if $\ell=r$, then this holds because $A_{\ell,u}=A_{r,u}$ and  
	$H$ is an $(\alpha_\rho,\beta_\rho)$-skeleton expander,
	and
	if $\ell\neq r$,  then
	$w_H(A_{\ell,u}\cup A_{r,u})=\frac{1}{2}(w_u(A_{\ell,u})+w_u(A_{r,u}))$ while $w_H(E^2(A_{\ell,u}\cup A_{r,u}))=
	w_u(E^2(A_{\ell,u}\cup A_{r,u}))$, and again we reach the same conclusion using our assumption
	on the skeleton expansion of $H$.
	Since $u$ is not heavy, $w(A_{\ell,u})+w_u(A_{r,u})<2 h_\rho$, 
	so the inequality implies that
	\[
	w_u(E^2(A_{\ell,u}\cup A_{r,u}))
	< \frac{1}{2}(w_u(A_{\ell,u})+w_u(A_{r,u}))(\alpha_\rho +  \beta_\rho h_\rho).
	\]
	By applying Lemma~\ref{LM:weight-link-vs-all} to both sides, we get that
	\[
	w(E_2(A_{\ell,u}\cup A_{r,u}))<
	\frac{1}{2}
	(\alpha_\rho +  \beta_\rho h_\rho)
	\AllLinkUp{i}{t}{d}
	\Circs{
		\LinkAllUp{i}{\ell}{d}		
		w(A_{\ell,u})+
		\LinkAllUp{i}{r}{d}
		w(A_{r,u})
	}.
	\]

	Let us abbreviate 	$ E_2(A_{\ell,u}\cup A_{r,u})$ to $E(u,\rho)$.
	Then  $ E(u,\rho)$   is precisely the set of faces $x\in X(t)$
	for which there exists an exceptional $\rho$-square of the form $(x,*,*,u)$.
	As a result, 
	\begin{equation}\label{EQ:exceptional-set-raw-bound}
		w(\Upsilon) \leq \sum_{\rho=(t,\ell,r,i)\in\calP}\sum_{u\in X(i)-A_i}\sum_{x\in E(u,\rho)}
		w(X(k+1)_x).
	\end{equation}
	
	Let   $\rho=(t,\ell,r,i)\in\calP $.  
	By Lemma~\ref{LM:weight-of-j-faces-cont-z} 
	and our upper bound on $w(E(u,\rho))$, we have
	\begin{align*}
		\sum_{u\in X(i)-A_i}
		\!\!\!\!\! & \,\,\,\,\,
		\sum_{x\in E(u,\rho)}
		w(X(k+1)_x)
		\leq 
		\sum_{u\in X(i)-A_i}\sum_{x\in E(u,\rho)}
			\ContZwZUp{t}{k+1}{d}			
			w(x)
		\\
		&=\sum_{u\in X(i)-A_i}
			\ContZwZUp{t}{k+1}{d}
			w(E(u,\rho))\\
		&\leq \sum_{u\in X(i)}
			\frac{1}{2}(\alpha_\rho +  \beta_\rho h_\rho)
			\ContZwZUp{t}{k+1}{d}
			\AllLinkUp{i}{t}{d}
			\Circs{
				\LinkAllUp{i}{\ell}{d}		
				w(A_{\ell,u})+
				\LinkAllUp{i}{r}{d}
				w(A_{r,u})
			}
	\end{align*}	
	By Lemma~\ref{LM:weight-sum-over-i-faces}, this sum is at most
	\[
		(\alpha_\rho +  \beta_\rho h_\rho)
		\Sci		
		\Circs{
			\Scii{\ell} w(A_\ell) +
			\Scii{r} w(A_{r })
		}
	\]
	and by Lemma~\ref{LM:bound-weight-heavy-i-faces}, this is bounded from above
	by
	\[
		(\alpha_\rho +  \beta_\rho h_\rho)
		\Sci
		\Circs{
			\Scii{\ell} T_\ell(h)+
			\Scii{r} T_r(h)
		}\|f\|.
	\]
	Plugging this into \eqref{EQ:exceptional-set-raw-bound} gives the lemma.
\end{proof}

\subsection{Ladders and Terminal Faces}
\label{subsec:ladders}

We continue to use the notation of \S\ref{subsec:exceptional-faces}.

Recall from Definitions~\ref{DF:intersection-profile-abs} and~\ref{DF:intersection-profile} that
$\calP$ comes equipped with a set of $\calP$-admissible pairs, $\operatorname{Ad}(\calP)$,
and that   a pair   $(x,y)\in X\times X$ is called
$\calP$-admissible if $x\geq y$ and $(\dim x,\dim y)$ is $\calP$-admissible.
Given $x,y\in X$, a ($\calP$-)\emph{ladder} 
from $x$ to $y$ is a sequence of faces $x=x_0\geq x_1\geq\dots\geq x_n=y$
such that $(  x_{i-1}, x_i)$ is $\calP$-admissible or $x_{i-1}=x_i$ 
for all $i\in\{1,\dots,n\}$.
We say that the  ladder $x=x_0\geq x_1\geq\dots\geq x_n=y$ is   ($(f,h,\calP)$-)\emph{heavy}
if the faces $x_0,\dots,x_n$ are  heavy or have dimension $k+1$.
If there exists a heavy ladder from $x$ to $y$, we also say that $x$ ($(f,h,\calP)$-)\emph{descends} to $y$.
A face $x\in X$ is called ($(f,h,\calP)$-)\emph{terminal} if it is heavy or $(k+1)$-dimensional
and it does not descend to any of its proper subfaces.
 
It turns out that
a non-exceptional $(k+1)$-face descends to exactly one  terminal face.

\begin{lem}\label{LM:unique-terminal-subface}
	With notation as above, let $x\in X(k+1)-\Upsilon$. Then
	$x$ descends to exactly one terminal face  $u$. 
	Furthermore, any face $y$ descended from $x$
	descends to $u$.
\end{lem}

\begin{proof}
	The second statement follows from the first
	because $y$ descends to some terminal face $u'$ which is descended from $x$
	and must therefore coincide with $u$. We turn to prove the first statement.
	
	Since $x$ descends to itself, it descends to some terminal subface, call it $u$.
	Suppose that $x$ descends to another terminal face $v$. We need to prove that $u=v$.
	Let $x=x_0\geq \dots\geq x_s=u$ and $x=y_0\geq \dots \geq y_t=v$ be   heavy ladders
	from $x$ to $u$ and $v$, respectively.
	
	We claim that for all $i\in\{0,1,\dots,s\}$, $j\in\{0,1,\dots,t\}$,
	there exist   faces $u_{i,j}$ 
	such that:
	\begin{enumerate}[label=(\roman*)]
		\item $u_{i,0}=x_i$ and $u_{0,j}=y_j$ for all $i,j$;
		\item for every $i\geq 1$, we have $u_{i-1,j}\geq u_{i,j}$, and if
		$u_{i-1,j}>u_{i,j}$, then $(  u_{i-1,j},  u_{i,j})$
		is $\calP$-admissible;
		\item for every $j\geq 1$,
		we have  $u_{i,j-1}\geq u_{i,j}$, and if
		$u_{i,j-1}>u_{i,j}$, then $(  u_{i,j-1},  u_{i,j})$
		is $\calP$-admissible;
		\item $u_{i,j}$ is heavy for all $i,j$.
	\end{enumerate} 
	Indeed, 
	if $i=0$ or $j=0$, the we define $u_{i,j}$
	as in (i); conditions (ii)--(iv) hold in this case
	because $ x_0\geq \dots\geq x_s $ and $ y_0\geq \dots \geq y_t $ are heavy ladders.
	We construct the remaining $u_{i,j}$ by induction on $i+j$:
	Assuming
	$s:=u_{i-1,j-1}$, $y:=u_{i-1,j}$  and $z:=u_{i ,j-1}$ were defined
	in such a manner that (ii)--(iv) hold, choose
	$u_{i,j}$ to be some member $b$ of $\Inf\{y,z\}$.
	We need to show that  (ii)--(iv) continue to hold.
	To that end, we split into cases.
	
	If  $b\notin\{y,z\}$,
	then we must have $x\notin\{y,z\}$.
	By the induction hypothesis,  $( s,y)$ and $(s,z)$
	are both $\calP$-admissible. Since  $\calP$ is  a $k$-intersection profile
	for $X$, there is $\rho\in \calP$
	such that $(s,y,z,b)$ or $(s,z,y,b)$ is a $\rho$-square.
	This, means that both $(u_{i-1,j},u_{i,j})=(y,b)$ and $(u_{i,j-1},u_{i,j})=(z,b)$
	are  $\calP$-admissible, proving (ii) and (iii).
	Moreover,
	since $y$ and $z$ are heavy and $x$ is not exceptional,
	$b$ is also heavy.
	
	Suppose next that $b=y$ and $b\neq z$.
	Then (ii) and  (iv) hold and $s\geq z\geq y=b$.
	If one of the last two inequalities is an equality, then (iii) holds by the induction hypothesis.
	Otherwise,   $ (s,z)$
	and $ (s,y)$ are both  $\calP$-admissible,
	and hence $(z,   y)=(u_{i,j-1},u_{i,j})$ must also be $\calP$-admissible
	and  again (iii) holds.

	The case where $b=z$ and $b\neq y$ is handled
	similarly.
	In the remaining case where $b=z=y$,   (ii)--(iv)  follow  readily 
	from the induction hypothesis. This completes the proof of our claim.
	
%%	Furthermore, 
%%	by the induction hypothesis or the definition of $u_{i,0}$ and $u_{0,j}$,
%%	the pairs $(\dim u_{i-1,j-1},u_{i,j-1})$
%%	and $(\dim u_{i-1,j-1},u_{i-1,j})$
%%	are $\calP$-admissible,
%%	so 
%%	\[(\dim u_{i-1,j-1},\dim u_{i,j-1},\dim u_{i-1,j},\dim u_{i,j})\in\calP.\]
	
	To finish observe that $u=u_{s,0}\geq u_{s,1}\geq \dots\geq u_{s,t}$
	is a heavy ladder. Since $u$ is terminal, we must have $u_{s,t}=u$.
	Similarly, $u_{s,t}=v$, and we conclude that $u=v$. 
\end{proof}

Our next goal is to relate the weight of the $(k+1)$-faces descending
to a given terminal face $u$ with the weight of the heavy $k$-faces (i.e.,
faces in $\supp f$) which descend to $u$. To that end, we shall need
the following general lemma:

\begin{lem}\label{LM:restricting-q-minimal-cochain}
	Let $\calF$ be a sheaf on a $d$-poset $X$,
	let $f,g\in C^k(X,\calF)$ ($0\leq k\leq d$),   let $q\in[0,1]$
	and let $u\in X$.
	Suppose that $g(x)\in \{f(x),0\}$ for every $x\in X(k)$. If $f$
	is mock $q$-locally minimal at $u$, then so is $g$. 
\end{lem}

\begin{proof}
	Write $i=\dim u$ and let
	$b\in B^{k-i-1}(X_u,\calF_u)$.  
	We need to prove that $\|g\|\leq \|g+b^u\|+q\,w(u)$.
	Our assumption on $g$ means that $\|f\|=\|g\|+\|f-g\|$.
	Furthermore, since $f$ is mock $q$-locally
	minimal at $u$, we have $\|f\|\leq \|f+b^u\|+q\,w(u)$.
	Together, we get that 
	$\|g+b^u\|=\|(f+b^u)-(f-g)\|\geq \|f+b^u\|-\|f-g\|=\|f+b^u\|-\|f\|+\|g\|\geq \|g\|-q\,w(u)$,
	which is what we want.
\end{proof}

We now return to  use the general notation
introduced so far.

\begin{lem}\label{LM:ladder-weights}
	Let $X,\calF,\calP,k,f,h,\Upsilon$ be as before and let $u\in X$ 
	be a terminal face of dimension $i\in\{0,\dots,k\}$. 
	Define
	\begin{align*}
	D(u)&:=\{x\in X(k)\suchthat \text{$x$ descends to $u$}\},\\
	D'(u)&:=\{x\in X(k+1)\suchthat \text{$x$ descends to $u$}\}.
	\end{align*}
	Suppose also that $f$ is mock $q$-locally minimal at $u$ and let
	$\veps=\cbe_{k-i-1}(X_u,w_u,\calF_u)$. Then
	\begin{align*}
	\veps  \Duc{i} w(D(u))
	&\leq
	w(D'(u)\cap [\supp(df)\cup \Upsilon])
	+
	q \qc \Duc{i} \veps 	
	w(u) 	
	\end{align*}
\end{lem}

\begin{proof}
	Define $g\in C^k(X,\calF)$ by
	\[
	g(x)=\left\{
	\begin{array}{ll}
	f(x) & x\in D(u) \\
	0 & x\notin D(u).
	\end{array}
	\right.
	\]
	Lemma~\ref{LM:restricting-q-minimal-cochain} tells
	us that $g$ is mock $q$-locally minimal at $u$, because $f$ is.
	In addition, $\|g\|=w(D(u))$.
	
	\Step{1} We claim that $\supp(dg)\subseteq  D'(u)\cap (\Upsilon\cup \supp(df))$.
	
	Let $x\in \supp(dg)$. Then there is $y\in x(k)$ with $y\in\supp g$.
	By the definition of $g$, the face $y$ descends to $u$, so $x\in D'(u)$.
	It remains to show that $x\in  \Upsilon\cup \supp(df)$.
	Suppose that $x\notin \Upsilon$. We need to show that $x\in \supp(df)$.
	Let $y\in x(k)\cap \supp(f)$. Then $y$ is heavy and so $x$ descends to $y$.
	By Lemma~\ref{LM:unique-terminal-subface} and our assumption that $x$
	is not exceptional, $y$ descends to $u$. 
	This means that $g(y)=f(y)$. Since $g(y)=0$ whenever $f(y)=0$,
	we conclude that $g$ and $f$ agree on every $y\in x(k)$, so 
	$df(x)=dg(x)\neq 0$, or rather, $x\in \supp (df)$.
	
	\Step{2}
	By Step~1, we have
	\[
	w(\supp(dg))\leq w(D'(u)\cap [\supp(df)\cup \Upsilon]).
	\]
	On the other hand, by 
	the definition of $\cbe_{k-i-1}(X_u,w_u,\calF_u)=\veps$,  
	Lemma~\ref{LM:weight-link-vs-all} and
	our earlier observation that $g $ is mock $q$-locally minimal at $u$
	we have:
	\begin{align*}
	w(\supp(dg)) 
	& 
	\geq w(X(d)_u) 
	\AllLinkLo{i}{k+1}{d}	
	w_u(\supp(dg_u))
	\\
	&\geq 
	w(X(d)_u) 
	\AllLinkLo{i}{k+1}{d}	
	\veps 
	\dist_{w_u}(g_u,B^{k-i-1}(X_u,\calF_u))
	\\
	&\geq 
	w(X(d)_u) 
	\AllLinkLo{i}{k+1}{d}	
	\veps 
	w(X(d)_u)^{-1}
	\LinkAllLo{i}{k}{d}	
	\dist_w(g ,B^{k-i-1}(X_u,\calF_u)^u)
	\\
	&\geq 
	\AllLinkLo{i}{k+1}{d}
	\LinkAllLo{i}{k}{d}	
	\veps  (\|g\|-
	\qc
	q w(u) )
	\\
	& =
	\veps \Duc{i} \|g\|
	- 
	\veps \Duc{i} \qc	
	q w(u) 
	\end{align*}
	Combining this with the last inequality and rearranging gives the lemma.
\end{proof}

\subsection{Completion of The Proof}

We will derive Theorem~\ref{TH:expansion-of-loc-min-cochains}
from the following key lemma.
We continue to assume that $f\in C^k(X,\calF)$, $\calP$
is a $k$-intersection profile, $h\in (0,1]^{\calP}$,
$\Upsilon$ is as in \S\ref{subsec:exceptional-faces},
and   $T_i$, $S_{\alpha,\beta}$,
$U_\calP$
are as in   \S\ref{subsec:polynomials}.

\begin{lem}\label{LM:key}
	With notation as before, let $\veps_0,\dots,\veps_k\geq 0$ and $q\in [0,1]$.
	Suppose that:
	\begin{enumerate}[label=(\arabic*)]
		\item $f$ is mock $q$-locally minimal;
		\item for every $i\in\{0,\dots,k\}$ and $u\in X(i)$, 
		$(X_u,w_u,\calF_u)$ is an $\veps_{i}$-coboundary expander 
		in dimension
		$k-i-1$;
		\item $\emptyset$ is not $(f,h,\calP)$-heavy;
	\end{enumerate}
	and put
	\[
	\tilde{\veps} = \min\left\{
	\Duc{i}  \veps_{i}
	\where i\in\{0,\dots,k\}
	\right\}.
	\]
	Then 	 
	\[
	\tilde{\veps}\|f\|\leq \|d_0 f\|+U_{\calP}w(\Upsilon)+
	q\sum_{i=0}^{k-1} 
	\qc \Duc{i} \veps_i	
	w(A_i)
	\]
\end{lem}

\begin{proof}
	We abbreviate $c_i := \qc \Duc{i} \veps_i	$.

	Let $u$ denote a terminal face with $\dim u\leq k$. Then $u\neq \emptyset$,
	because $\emptyset$ is not heavy,
	and thus $f$ is $q$-locally minimal at $u$.
	Therefore, by Lemma~\ref{LM:ladder-weights}, 
	\begin{align*}
	\textstyle
	\veps_{\dim u}
	\textstyle\Duc{\dim u}  w(D(u))
	& \leq
	w(D'(u)\cap [\supp(df)\cup \Upsilon])
	+
	q   c_{\dim u}
	w(u)
	\\
	=  
	w(D'(u)&\cap [\supp(df)- \Upsilon]) + w(D'(u)\cap \Upsilon)
	+ q  c_{\dim u}
	w(u).
	\end{align*}
	When $\dim u=k$, this moreover holds with $q=0$, because $f$ is $0$-locally minimal at every $u\in X(k)$.
	By summing over all terminal faces $u$ of dimension $k$ or less,
	we get
	\begin{align}\label{EQ:summation-on-terminal-faces}
	& \sum_{u} 
	\veps_{\dim u}
	\textstyle\Duc{\dim u}  
	w(D(u))\leq 
	\\
	\sum_u  w(D'(u)\cap &[\supp(df)- \Upsilon]) +
	\sum_u w(D'(u)\cap \Upsilon) +
	\sum_{u:\dim u<k} q 
	c_{\dim u}
	w(u).
	\nonumber
	\end{align}
	We shall prove the lemma by bounding from above or below the four sums appearing
	in this inequality.
	
	First, we have
	\[
	\sum_{u}
	\veps_{\dim u}
	\textstyle\Duc{\dim u}  
	w(D(u))
	\geq 
	\tilde{\veps}  
	\sum_u  w(D(u))\geq \tilde{\veps}\|f\|,
	\]
	because every face in $\supp f$ descends to some terminal face $u$ with $\dim u\leq k$. Next,
	\[
	\sum_u  w(D'(u)\cap [\supp(df)- \Upsilon]) 
	\leq \|d_0f\|
	\]
	because every non-exceptional $(k+1)$-face     descends to a unique terminal subface 
	(Lemma~\ref{LM:unique-terminal-subface}). On the other hand,
	if $y\in\Upsilon$, then, with notation as in
	\S\ref{subsec:polynomials},
	the set $A$ of heavy faces to which $y$
	descends is in $\calV(y)$. Since the set of terminal faces
	to which $y$ descends is $T(A)$, the face
	$y$   descends to at most $U_\calP$ terminal faces,
	meaning that
	\[
	\sum_u w(D'(u)\cap \Upsilon)\leq U_\calP\|\Upsilon\|.
	\]
	Finally, since every terminal face of dimension $  k$ or less
	is heavy and nonempty, we have
	\[
	\sum_{u:\dim u<k}  q  
	c_{\dim u}
	w(u)
	\leq
	q \sum_{i=0}^{k-1} 
	c_i w(A_i).
	\]
	Plugging these inequalities into \eqref{EQ:summation-on-terminal-faces} gives the lemma.
\end{proof}

\begin{proof}[Proof of Theorem~\ref{TH:expansion-of-loc-min-cochains}]
	Recall that we are given $h\in (0,1]^{\calP}$ such that
	\[
	p:=\tilde{\veps}-U_{\calP}S_{ {\alpha}, {\beta}}(h)-
	q  \sum_{i=0}^{k-1} 
	\qc \Duc{i} \veps_i	
	T_i(h) > 0
	\]
	and $\tilde{\veps}$ is as in Lemma~\ref{LM:key}.
	Furthermore, $f\in C^k(X,\calF)$ is   mock $q$-locally minimal
	and $T_{-1}(h)^{-1}$-small. Again, write  $c_i := \qc \Duc{i} \veps_i	$.
	
	By Lemma~\ref{LM:bound-weight-heavy-i-faces}, $w(A_{-1}) \leq  T_{-1}(h)\|f\|
	<1$. Since $X$ has only one $-1$-face and its weight is $1$, the 
	face $\emptyset$ is not heavy. Now, Lemma~\ref{LM:key} tells us
	that
	\[
	\|df\|\geq \tilde{\veps}\|f\|-U_\calP w(\Upsilon) -
	q\sum_{i=0}^k 
	c_i
	w(A_i)
	\]	
	By Lemmas~\ref{LM:bound-weight-heavy-i-faces} and~\ref{LM:bound-weight-of-exceptional-faces},
	this means that
	\begin{align*}
	\|df\|& \geq 
	\tilde{\veps}\|f\|-U_\calP S_{\alpha,\beta}(h)\|f\|-q 
	\sum_{i=0}^k 
	c_i T_i(h)\|f\|
	= p\|f\|. \qedhere
	\end{align*}
\end{proof}

\appendix

\section{Efficient Decoding Algorithm}
\label{sec:correction-efficient}

The following is a time-efficient variant of 
Algorithm~\ref{AL:correction-algorithm}.
Since Algorithm~\ref{AL:correction-algorithm}   includes 
Algorithm~\ref{AL:correction-algorithm-simplified} as
a special case
(Remark~\ref{RM:relation-between-alg}), it also gives a time-efficient
variant of the latter. 

Throughout,   $R$ is a ring,  
$(X,w)$ is a properly weighted $R$-oriented $d$-poset,
$\calF$ is an $R$-sheaf on $X$ and  
$k\in\{-1,\dots,d-1\}$.

\begin{alg}\label{AL:correction-efficient}
	The algorithm
	takes as input $h\in C^{k+1}(X,\calF)$ and some $q\in [0,\infty)$ 
	and outputs $g\in C^k(X)$ determined as follows.
\begin{enumerate}[label=(\arabic*)]
\item $g\leftarrow 0$ 
\item  $Q \leftarrow$ empty queue
\item  $B \leftarrow $ boolean array indexed
by $X(0)\cup\dots\cup X(k)$ 
\item For each $u\in X(0)\cup\dots\cup X(k)$: 
	\begin{enumerate}
	\item[(4a)] $Q$.push($u$)
	\item[(4b)] $B[u]\leftarrow\mathtt{True}$ // $u$ is in $Q$ 
	\end{enumerate}
\item While $Q$ is not empty:
	\begin{enumerate}[label=(5\alph*)]
	\item $u \leftarrow Q$.pop()
	\item $B[u]\leftarrow\mathtt{False}$ // $u$ is not in $Q$ 
	\item 
		Search for   $g'\in C^{k-\dim u-1}(X_u,\calF_u)$ with $\|(h+dg)+d(g'^u)\|_w<\|h+dg\|_w-q \cdot w(u)$.
	\item If  some $g'$ was found:
		\begin{enumerate}[label=(5d-\roman*)]
		\item \label{item:AL:correction-efficient:found-g} $g \leftarrow g + g'^z$
		\item \label{item:AL:correction-efficient:update-Q}  For every  $y\in X(k+1)_u$ and every $u'\in y(0)\cup\dots \cup y(k) $: 
			\begin{enumerate}[label=(\Roman*)]
			\item  $Q.$push($u'$) 
			\item $B[u']\leftarrow\texttt{True}$ // $u'$ is in $Q$ 
			\end{enumerate}
		\end{enumerate}
	\end{enumerate}
\item Return $g$.
\end{enumerate}
\end{alg}

\begin{prp}
	With notation as in Algorithm~\ref{AL:correction-efficient},
	let $u_i$ and $g'_i$ denote the values of $u$ and $g'$ 
	on the $i$-th time that instruction \ref{item:AL:correction-efficient:found-g} is executed.
	In addition, let $S=\max\{\#\calF(x)\where x\in X(0)\cup\dots\cup X(k)\}$,
	$F=\sum_{i=0}^k\Fmax_{i,k}(X)$ and $D=\sum_{i=0}^kD_{i,k}(X)$. Then:
	\begin{enumerate}[label=(\roman*)]
		\item Algorithm~\ref{AL:correction-efficient} stops and its output
		$g$  is the 
		output of Algorithm~\ref{AL:correction-algorithm} when we choose
		$u=u_i$ and $g'=g'_i$ in the $i$-th iteration of the loop (2). 
		In particular, $h+dg$ is mock $q$-locally minimal.
		\item 
		Suppose that $q>0$ and there is $M\in[1,\infty)$
		such that $w(x)\leq Mw(y)$ for all $x,y\in X(k)$.
		Then the time complexity
		of the algorithm is $O(c|X(k)|)$ for
		$c=DF(S^D+DF)Mq^{-1}$.\footnote{
			Here,  we assume that elementary operations involving $(X,w)$
			and $\calF$, e.g., applying $w$ or the restriction maps $\res^{\calF}_{y\from x}$,
			take $O(1)$ time.
		} In particular, it is linear in $|X(k)|$ when
		$S$, $F$, $D$, $M$  are $O(1)$ and $q=\Theta(1)$.
	\end{enumerate}
\end{prp}

When $w$ is the natural weight function of $X$, we can take $M=L_{k,d}U_{k,d}$
in (ii) (Proposition~\ref{PR:face-weight-ratio-via-irreg}).

\begin{proof}
	(i) 
	Let $A(g)$ denote the set of faces $u\in X(0)\cup\dots\cup X(k)$
	for which 
	$h+dg$ is \emph{not} mock $q$-locally minimal at $u$.
	We first claim that every time the loop (5) is about to be performed,
	every face in $A(g)$ is in $Q$.
	Indeed, this clear on the first iteration of (5). Once we enter
	the loop (5), there are two possibilities:
	If $h+dg$ is  mock $q$-locally minimal
	at $u$,  then no $g'$ will be found in (5c), so $u$ is removed from $Q$ and $h+dg$
	remains unchanged. Thus, the set of faces in $Q$ still contains $A(g)$ when iteration finishes.
	Otherwise, 
	$h+dg$ is not mock $q$-locally minimal	at $u$ and (5c) yields
	some $g'$.
	Since $h+dg+dg'^u$ agrees with $h+dg$ at every $k$-face not-containing $u$,
	we have $A(g+g'^u)\subseteq A(g)\cup (\bigcup_{y\in X(k)_u }\bigcup_{i=0}^k y(i))$,
	so again have that $Q$ includes all faces in (the new) $A(g)$ as the iteration of (5) ends.
	
	It follows from our claim that $h+dg$ is $q$-mock locally minimal when $Q$
	is empty. Moreover, our  discussion above implies that before the $i$-th performance
	of \ref{item:AL:correction-efficient:found-g},
	$h_i:=h+dg$ is not $q$-mock locally minimal at $u_i$.
	This means that we may run Algorithm~\ref{AL:correction-algorithm} 
	(on the same input)
	choosing $u=u_i$ and $g'=g'_i$ in the $i$-th iteration of the loop (2),
	and, provided Algorithm~\ref{AL:correction-algorithm} stops, get the same output
	as Algorithm~\ref{AL:correction-efficient}.
	By Proposition~\ref{PR:alg-complexity}, Algorithm~\ref{AL:correction-algorithm}
	does stop, so we proved (i).
	
	(ii) Thanks to (i), the number of times instructions \ref{item:AL:correction-efficient:found-g} 
	and \ref{item:AL:correction-efficient:update-Q}
	are performed equals
	the number of times the loop (2) in Algorithm~\ref{AL:correction-algorithm}
	is executed, and by Proposition~\ref{PR:alg-complexity}(ii),
	this number does not exceed $M q^{-1}|X(k)|$.
	
	Before (5) is executed for the first time, $Q$ contains at most $F|X(k)|$
	elements, and in every iteration of (5), either 
	\ref{item:AL:correction-efficient:found-g} 
	and \ref{item:AL:correction-efficient:update-Q} are performed and the number
	of elements in $Q$ is increased by at most $DF-1$, or \ref{item:AL:correction-efficient:found-g} 
	and \ref{item:AL:correction-efficient:update-Q}  are not performed
	and the number of elements in $Q$ is decreased by $1$.
	As these instructions cannot run more than $Mq^{-1}|X(k)|$ times, it follows
	that the loop (5) is executed at most $(F+(DF-1) Mq^{-1})|X(k)|\leq 2DFMq^{-1}|X(k)|$ times.
	If the search in (5c) is done naively (by checking all the possibilities), then the time complexity
	of one iteration of (5) is at most $O(S^D+DF)$.
	Thus, the time complexity of the entire algorithm is $O((S^D+DF)DFMq^{-1}|X(k)|)$.
\end{proof}

\newpage

\section*{Notation Index}
\label{sec:not-ind}
\addcontentsline{toc}{section}{\nameref{sec:not-ind}}

%%%%%%%%%%%%%% increase row strech in table
\bgroup
\def\arraystretch{1.1}
%%%%%%%%%%%%%%%
\begin{tabular}{ l p{0.7\textwidth} l }
\hline
Symbol & Short Description & Page \\
\hline
$\mathrm{Cay}(A,G)$ &
	Cayley graph of a group $G$ and a subset $A\subseteq G$ & \pageref{symdef:Cay} \\
$\mathrm{Cay}(A,G,B)$ &
	two-sided Cayley square complex & \pageref{subsec:ltc-poset} \\
$\cbe_i$ & $i$-coboundary expansion & 
	\pageref{DF:cbe} \\
$C(\{C_s\}_{s\in S})$ &
	lifted code induced by small codes $\{C_s\}_{s\in S}$ &
	\pageref{subsec:tanner-codes}\\
$\ccd_i$ & $i$-cocyle distance &
	\pageref{DF:cse} \\
$\cse_i$ & $i$-cosystolic expansion &
	\pageref{DF:cse} \\
$\delta(\cdot )$ &
	relative distance of a code & 
	\pageref{symdef:delta}\\
$\dist $ &
	normalized Hamming distance (possibly weighted) &
	\pageref{symdef:dist-I}, \pageref{symdef:dist-norm-I}\\
$\dist_w $ &
	Hamming distance weighted according to $w$ &
	\pageref{symdef:distw}\\
$D^{\max}_{i,j}$, $D^{\min}_{i,j}$ &
	maximal (minimal) number of $j$-faces containing an $i$-face &
	\pageref{symref:Dmax}\\
$D(\cdot )$ &
	degree of a poset & \pageref{symref:D} \\
$\Fmax_{i,j,k}$, $\Fmin_{i,j,k}$ &
	maximal (minimal) number of $j$-faces between an $i$-face
	and a $k$-face & \pageref{symref:Fmax}\\
$\ugr(X)$ & 
	underlying hypergraph of a poset, i.e., $X(0)\cup X(1)$ &
	\pageref{symdef:ugrXw} \\
$\ugr(X,w)$ & 
	weighted  hypergraph   underlying a weighted poset &
	\pageref{symdef:ugrXw} \\
$g^u$ & cocycle obtained by extending a coycle $g$ on a link & \pageref{symdef:g-up-u}
	\\
$g_u$ & cocycle obtained by restricting a coycle $g$ to a link & \pageref{symdef:g-low-u}
	\\
$L_{i,j,k}$ &
	$(i,j,k)$-lower-irregularity, i.e., $\Fmax_{i,j,k}/\Fmin_{i,j,k}$ &
	\pageref{symref:Lijk} \\
$L(\cdot)$ &
	lower-irregularity of a poset, i.e., $\max_{i,j,k} L_{i,j,k}$ &
	\pageref{symref:Lijk} \\
$\NI^{i,j,k}$ & $(i,j,k)$-no-intersection hypergraph of a poset & 
	\pageref{DF:nih} \\
$\NI^{i,j,k}_z(X)$ & $\NI^{i-\dim z-1,j-\dim z-1,k-\dim z-1}(X_z)$ & 
	\pageref{NT:non-intersec-of-link} \\
$\NNI^{i,j,k}$ & $(i,j,k)$-no-intersection graph of a poset & 
	\pageref{DF:nig} \\
$\NNI^{i,j,k}_z(X)$ & $\NNI^{i-\dim z-1,j-\dim z-1,k-\dim z-1}(X_z)$ & 
	\pageref{NT:non-intersec-of-link} \\
$r(\cdot)$ & rate of a code &
	\pageref{symdef:rate} \\
$U_{i,j}$ & $(i,j)$-upper-irregularity,
	i.e., $D^{\max}_{i,j}/D^{\min}_{i,j}$ &
	\pageref{symref:Uij} \\
$U(\cdot)$ & 
	upper-irregularity of a poset, i.e., $\max_{i,j} U_{i,j}$ &
	\pageref{symref:Uij} \\
$w_z$ & 
	proper weight function induced by $w$ on a link of a poset & \pageref{symdef:wz} \\
$X_z$ & 
	link of a poset $X$ at a face $z$ & \pageref{subsec:links} \\
$\|\cdot\|$ &
	normalized Hamming norm (possibly weighted) &
	\pageref{symdef:norm-I}, \pageref{symdef:dist-norm-I} \\
$\|\cdot\|_w$ &
	Hamming norm weighted according to $w$ &
	\pageref{symdef:normw} \\
\hline
\end{tabular}
%%%%%%%%%%%%%%%%%%%%%%%%% cancel row strech
\egroup
%%%%%%%%%%%%%%%%%%%%%%%%

\newpage

\bibliographystyle{alpha}
\bibliography{MyBib_24_03}

\end{document}

%% file: amsart_header_18_05.tex
\usepackage{amsfonts}
\usepackage{amssymb}
\usepackage{amsmath}
\usepackage{hyperref}
\usepackage{mathrsfs}
\usepackage{centernot}
\usepackage{mathdots}
\usepackage{stmaryrd}
\usepackage[all]{xy}

\usepackage{mathtools}
\usepackage{microtype}

\usepackage{color}

\newtheorem{thm}{Theorem}[section]
\newtheorem{lem}[thm]{Lemma}
\newtheorem{prp}[thm]{Proposition}
\newtheorem{cor}[thm]{Corollary}
\newtheorem{dfn}[thm]{Definition}

\theoremstyle{definition}

\newtheorem{notation}[thm]{Notation}

\newtheorem{example}[thm]{Example}
\newtheorem{remark}[thm]{Remark}

\theoremstyle{plain}

\newcommand{\Step}[1]{\smallskip\noindent {\it Step #1.}} % For steps in proof

\newcommand{\rem}[1]{}

\newcommand{\F}{\mathbb{F}}

\newcommand{\N}{\mathbb{N}}

\newcommand{\R}{\mathbb{R}}
\newcommand{\Z}{\mathbb{Z}}

\newcommand{\HH}{{\mathrm{H}}}

\newcommand{\calA}{{\mathcal{A}}}

\newcommand{\calF}{{\mathcal{F}}}
\newcommand{\calG}{{\mathcal{G}}}

\newcommand{\calP}{{\mathcal{P}}}

\newcommand{\calS}{{\mathcal{S}}}

\newcommand{\calU}{{\mathcal{U}}}
\newcommand{\calV}{{\mathcal{V}}}

\newcommand{\calX}{{\mathcal{X}}}

\newcommand{\veps}{\varepsilon}
\newcommand{\vphi}{\varphi}

\newcommand{\suchthat}{\,:\,}
\newcommand{\where}{\,|\,}

\newcommand{\Circs}[1]{\left( #1 \right)}
\newcommand{\Squares}[1]{\left[ #1 \right]}

\newcommand{\floor}[1]{\lfloor {#1} \rfloor}

\newcommand{\ceil}[1]{\lceil {#1} \rceil}

\DeclareMathOperator{\cores}{res}  
\DeclareMathOperator{\id}{id} %
\DeclareMathOperator{\im}{im} %
\newcommand{\op}{\mathrm{op}} %
\DeclareMathOperator{\supp}{supp} %
\newcommand{\nMat}[2]{\mathrm{M}_{#2}(#1)}

\newcommand{\units}[1]{{#1^\times}}